\documentclass[11pt]{amsart}

\usepackage{geometry}
\usepackage{setspace}

\renewcommand{\arraystretch}{0.6}

\usepackage{amsmath,amsthm}
\usepackage{amssymb}
\usepackage{tipa}       
\usepackage{amsbsy}
\usepackage{graphicx}  
\usepackage{subfigure} 
\usepackage{url} 
\usepackage{tikz-cd}
\usepackage[T1]{fontenc}
\usepackage{inputenc}
\usepackage{adjustbox}
\usepackage{comment}
\usepackage{ stmaryrd }
\usepackage{textcomp}
\usepackage{stackengine}
\usepackage[pdfencoding=auto, psdextra]{hyperref}

\newtheorem{thm}{Theorem}[section]
\newtheorem{cor}[thm]{Corollary}
\newtheorem{lem}[thm]{Lemma}
\newtheorem{prop}[thm]{Proposition}
\theoremstyle{definition} 
\newtheorem{defn}[thm]{Definition}
\newtheorem{rem}[thm]{Remark}

\newcommand{\bA}{\mathbb{A}}
 \newcommand{\bC}{\mathbb{C}}

 \newcommand{\bF}{\mathbb{F}}

 \newcommand{\bR}{\mathbb{R}}
 \newcommand{\bQ}{\mathbb{Q}}
 \newcommand{\bZ}{\mathbb{Z}}

 \newcommand{\ccH}{\mathcal{H}}
 
 \newcommand{\ccF}{\mathcal{F}}
 
 \newcommand{\ccO}{\mathcal{O}}
 \newcommand{\ccI}{\mathcal{I}}
 \newcommand{\fp}{\mathfrak{p}}
 \newcommand{\fq}{\mathfrak{q}}
 
 \newcommand{\fm}{\mathfrak{m}} 
  
 \newcommand{\End}{\text{End}} 
  
 \newcommand{\Hom}{\text{Hom}}

 \newcommand{\GL}{\text{GL}}

 \newcommand{\id}{\mathrm{id}}

 \newcommand{\bD}{\mathbb{D}}

 \newcommand{\adf}{\bA_{F,f}^{(p)}}
 
 \newcommand{\fP}{\mathfrak{P}}
 \newcommand{\ccL}{\mathcal{L}}
 
 \let\emptyset\varnothing
 \newcommand{\dR}{\textrm{dR}}

 \newcommand{\es}{\textrm{es}}
 \newcommand{\Tr}{\textrm{Tr}}
 \newcommand{\Fp}{\mathfrak{P}}
\newcommand{\pn}{(p^n)}
\newcommand{\CR}{\textrm{CRIS}}

\title[A Jacquet-Langlands relation at ramified primes]{ A mod $p$ Geometric Jacquet-Langlands Relation for Quaternionic Shimura Varieties at Ramified Primes}

\author{Gabriel Micolet}
\address{Department of Mathematics King's College London}
\email{gabriel.micolet@kcl.ac.uk}

\begin{document}
\begin{abstract}
    Let $F$ be a totally real field, $p$ a prime that we allow to ramify in $F$, and $B$ a quaternion algebra over $F$ which is split at places over $p$. We consider a smooth $p$-adic integral model, the Pappas-Rapoport model, of the Quaternionic Shimura variety attached to $B$ with prime-to-$p$ level, and the Goren-Oort stratification of its characteristic $p$ fiber. Furthermore, we also introduce Pappas-Rapoport models at Iwahori level $p$ along with a stratification of their characteristic $p$ fiber. We prove that these strata are isomorphic to products of $\mathbb{P}^1$-bundles over auxiliary Quaternionic Shimura varieties, from which we deduce the corresponding description of the Goren-Oort strata, thus extending the results of \cite{tian_xiao_2016} and \cite{2020arXiv200100530D}.
\end{abstract}
\maketitle

\begingroup\singlespacing
\tableofcontents
\endgroup

\section{Introduction}

\subsection{Modular curves}
Let $N \geq 4$, $p$ be a prime which does not divide $N$ and $X_1(N)$ denote the modular curve of level $\Gamma_1(N)$. In \cite[$\mathsection$ IV.2]{DeligneRapoport}, Deligne and Rapoport showed that there exists a smooth model of $X_1(N)$ over $\bZ[\frac{1}{N}]$ obtained as the moduli space of (generalized) elliptic curves with a point of order $N$. Its reduction mod $p$ denoted $\overline{X}_1(N)$ admits a natural decomposition into the open ordinary locus $\overline{X}_1(N)^{\textrm{ord}}$ and the closed supersingular locus $\overline{X}_1(N)^{\textrm{ss}}$. More concretely $\overline{X}_1(N)^{\textrm{ss}}$ is given as the vanishing of the Hasse invariant $h \in H^0(\overline{X}_1(N),\omega^{\otimes p-1})$ where  $\omega^{\otimes p-1}$ denotes the sheaf of weight $p-1$ modular forms, and $\overline{X}_1(N)^{\textrm{ord}}$ is its complement. In a letter to Tate (\cite{Serre1996}), Serre gave the following descripition of the geometric points of $\overline{X}_1(N)^{\textrm{ss}}$:
\begin{thm}[Serre]
    There is a bijection of sets 
    \[\overline{X}_1(N)^{\textrm{ss}}(\overline{\bF}_p) \leftrightarrow B_{p,\infty}^\times \backslash B_{p,\infty}^\times (\bA_f) / O_1(N)B_{p,\infty}^\times(\bZ_p), \]
    equivariant under Hecke correspondences, where $B_{p,\infty}$ denotes the unique Quaternion algebra over $\bQ$ ramified at $p$ and infinity and $O_1(N) \subset B_{p,\infty}^\times(\bA^{(p)}_f)$ is the image of $\Gamma_1(N)$ under the isomorphism $B_{p,\infty}^\times(\bA^{(p)}_f) \simeq \textrm{GL}_2(\bA^{(p)}_f)$
\end{thm}
In this correspondence, we think of the right hand set as the geometric points of the zero-dimensional Shimura variety attached to $B^\times_{p,\infty}$.

One can also consider the modular curve $X_1(N,p)$ of level $\Gamma_1(N) \cap \Gamma_0(p)$. Deligne and Rapoport showed in \cite[$\mathsection$ V.1]{DeligneRapoport} that it admits a semistable model over $\bZ[\frac{1}{N}]$, as the moduli space of (generalized) elliptic curves with a point of order $N$ and an isogeny $\phi$ of degree $p$. Its reduction mod $p$, $\overline{X}_1(N,p)$, is a union of two copies of $\overline{X}_1(N)$, one corresponding to the locus where $\ker \phi$ is the kernel of Frobenius, and the other to the locus where $\ker \phi$ is the kernel of Verschiebung. In particular, these copies of $\overline{X}_1(N)$ intersect transversally precisely at the supersingular points.

\subsection{Hilbert modular varieties at unramified primes}
The analogous situation for Hilbert modular varieties is more complicated: Let $F$ be a totally real field of degree $d$ with discriminant $\Delta_F$. Consider the model $X$ of the Shimura variety attached to $G = \textrm{Res}_{F/\bQ} \textrm{GL}_2$ given by the coarse moduli space of polarized abelian schemes of dimension $d$ with an action of $\ccO_F$ and a suitable level $N$-structure. Deligne and Pappas showed in \cite{deligne1994singularites} that these models are smooth over $\bZ[ \frac{1}{N \Delta_F}]$. In particular, if $p$ is unramified in $F$, then $\overline{X} = X_{\overline{\bF}_p}$ is smooth. 

In this situation, Goren and Oort introduced in \cite{Gorenoort} a stratification of $\overline{X}$, indexed by subsets of $\Theta_F = \{ \ccO_F \to \overline{\bF}_p\}$ as follows: Let $s:A \to \overline{X}$ be the universal abelian scheme and $\omega = s_* \Omega^1_{A/ \overline{X}}$ its sheaf of invariant differentials. The action of $\ccO_F$ splits $\omega = \bigoplus_{\theta \in \Theta_F} \omega_\theta$ into a direct sum of line bundles where $\ccO_F$ acts on each $\omega_\theta$ via $\theta:\ccO_F \to \overline{\bF}_p$. Considering the restriction of Verschiebung $V: \omega \to \omega^{(p)}$ to each component $\omega_\theta$ yields the partial Hasse invariants $h_\theta \in H^0(\overline{X}_\theta,\omega_\theta^{\otimes -1} \otimes \omega^{\otimes p}_{\phi^{-1} \circ \theta})$ for $\theta \in T$. Each closed stratum $\overline{X}_T$ is then given by the vanishing of the each $h_\theta$ for $\theta \in T$. These are smooth subschemes of $\overline{X}$ of codimension $\vert T \vert$. 

Encouraged by the work of Helm (\cite{Helm}) in the context of $U(2)$ Shimura varieties, Tian and Xiao showed in their seminal paper \cite{tian_xiao_2016}, that each stratum $\overline{X}_T$ is isomorphic to a product of $\mathbb{P}^1$-bundles over an auxiliary Quaternionic Shimura variety attached to the quaternion $B_{\Sigma_T}/F$ which is ramified exactly at a set of places $\Sigma_T$, determined by $T$. To do so, they considered a variant of Carayol's (\cite{Carayol1986SurLM}) \textit{mod\`{e}les \'{e}tranges} of Quaternionic Shimura varieties, via auxiliary Unitary Shimura varieties to make use of the moduli interpretation of the latter. In fact, they carried this program, defining a Goren-Oort stratification and giving a global description of each stratum, for any Quaternionic Shimura variety attached to a quaternion algebra $B/F$ which splits at the places above $p$.

In \cite{Pappas1995}, Pappas introduced Iwahori level models $X_0(p)$ over $\bZ_{(p)}$, which he proved were flat local complete intersections of relative dimension $d$, which are given as the coarse moduli space of pairs $(\underline{A_1},\underline{A_1},f)$,
where the $\underline{A_i}$ are points of $X$ and $f$ is an $\ccO_F$-linear isogeny of degree $p^d$ with $\ker f \subset A_1[p]$, that respects the polarizations and the level structures. Its geometric special fiber, $\overline{X}_0(p) = X_0(p) \otimes \overline{\bF}_p$, was studied in detail by Goren and Kassaei in \cite{GorenKassaei+2012+1+63}. There they showed that it admits a natural stratification by closed subschemes $\overline{X}_0(p)_{\phi(I),J}$, indexed by pairs of subsets $I,J \subset \Theta_F$. More precisely, each $\overline{X}_0(p)_{\phi(I),J}$ is given by the vanishing of $\textrm{Lie}(f_\theta)$ for $\theta \in I$, and $\textrm{Lie}(f^\vee_\theta)$ for $\theta \in J$ and is smooth of codimension $\vert I \cap J \vert$ provided that $I \cup J = \Theta_F$. 

Diamond, Kassaei and Sasaki proved in \cite{2020arXiv200100530D} that the irreducible components of $\overline{X}_0(p)$, which are (the components of) the maximal strata $\overline{X}_0(p)_{J} := \overline{X}_0(p)_{\phi(J^c),J}$, are, analogously to the strata downstairs, also isomorphic to products of $\mathbb{P}^1$-bundles over auxiliary Quaternionic Shimura varieties attached to a quaternion algebra $B_{\Sigma_{J}}/F$ determined by $J$. They did this by transferring the question to the context of Unitary Shimura varieties and introducing the technique of splicing; from $A_1$ and $A_2$ they produced an abelian scheme $A_J$ whose Dieudonn\'{e} module at each geometric point is essentially obtained by splicing together the Dieudonn\'{e} modules of $A_1$ and $A_2$. They did not however consider the lower dimensional strata.

Similarly to the case of modular curves, there is a privileged stratum of $\overline{X}_0(p)$, namely $\overline{X}_0(p)_\emptyset$. The natural projection $\pi: \overline{X}_0(p) \to \overline{X}$ mapping $(\underline{A_1},\underline{A_1},f)$ to $\underline{A_1}$ restricts over $\overline{X}_0(p)_\emptyset$ to an isomorphism whose inverse $\overline{X} \to \overline{X}_0(p)_\emptyset$ is given by $\underline{A} \mapsto (\underline{A},\underline{A^{(p)}},Fr)$ where $Fr$ denotes the Frobenius isogeny. Furthermore, these isomorphisms respect the stratifications, mapping each $\overline{X}_0(p)_{\Theta_F,J}$ isomorphically onto $\overline{X}_J$. Therefore a complete description of the smooth strata at Iwahori level would provide a description of the Goren-Oort strata downstairs.

\subsection{Pappas-Rapoport models of Hilbert modular varieties}
As discussed in the previous section, the models $\overline{X}$ are smooth provided that $p$ is unramified in $F$. If $p$ is ramified however, Deligne and Pappas showed in \cite{deligne1994singularites} that $\overline{X}$ is normal, and singular in codimension two. Its smooth locus, which is open and dense, is precisely the Rapoport locus, that is the locus that parameterizes $A/S$ where $\omega_{A/S}$ is locally free of rank one over $\ccO_S \otimes \ccO_F$; this is always the case if $p$ is unramified in $F$.

To resolve these singularities, Pappas and Rapoport introduced in \cite{pappas2005local} so called \textit{splitting models} which were made explicit in \cite{sasaki}. Its fiber over $\overline{\bF}_p$ is given as follows: We consider the set of unramified embeddings $\widehat{\Theta}_F = \{ \theta: \ccO_F/\fp \hookrightarrow \overline{\bF}_p  \, \vert \, \fp \vert p \ccO_F\}$. The action of $\ccO_F$ on $\underline{A} \in \overline{X}(S)$ induces a splitting of $\omega_{A/S} =: \omega = \bigoplus_{\theta \in \widehat{\Theta}_F} \omega_\theta$ into locally free sheaves of $\ccO_S$-modules of rank $e_\fp$, the ramification degree of the corresponding prime $\fp$. The Pappas-Rapoport model $\overline{X}^{\textrm{PR}}$ parameterizes isomorphism classes of tuples $(\underline{A}, \underline{\omega})$ where $\underline{A}$ is an $S$-point of $\overline{X}$ and $\underline{\omega}$ is the data, for each $\theta \in \widehat{\Theta}_F$, of a filtration
\[0 = \omega_\theta(0) \subset \omega_\theta(1) \subset \cdots \subset \omega_\theta(e_\fp-1) \subset \omega_\theta(e_\fp) = \omega_\theta, \]
where each $\omega_\theta(i)$ is a local summand, as an $\ccO_S$-module, of rank $i$ such that each subquotient $\omega_\theta(i)/\omega_\theta(i-1)$ is killed by $[\varpi_\fp]$, where $[\varpi_\fp]$ denotes the action of a uniformizer $\varpi_\fp \in F_\fp$. $\overline{X}^{\textrm{PR}}$ is a smooth scheme of dimension $d$ and admits a natural morphism $\pi: \overline{X}^{\textrm{PR}} \to \overline{X}$ which is projective and induces an isomorphism when restricted to the respective Rapoport loci on both sides. In particular, it is an isomorphism when $p$ is uramified in $F$.

In \cite{reduzzi2016partial}, Reduzzi and Xiao introduced partial Hasse invariants over the Pappas-Rapoport model, defined in terms of the subquotients $\omega_\theta(i)/\omega_\theta(i-1)$, generalizing the ones considered in the unramified case. These partial Hasse invariants allowed them to define a Goren-Oort stratification of $\overline{X}^{\textrm{PR}}$ which they used to prove the existence of Galois representations attached to certain Hecke eigenclasses of Hilbert modular varieties in the ramified case. There are also models at Iwahori level, defined as certain moduli spaces of $p$-isogenies which, crucially, are required to respect filtrations. These models, along with a stratification were introduced in \cite{emerton_reduzzi_xiao_2017} to define Hecke operators at $\fp$ in characteristic $p$. These models were also studied in detail by in \cite{diamond2020cone} and \cite{diamond2021kodairaspencer}.

The goal of this thesis is in some sense to complete these different yet related stories. For a totally real field $F$, a prime $p$ which we allow to ramify in $F$ and every Quaternion algebra $B_\Sigma/F$ ramified at a set of places $\Sigma$ away from $p$, we introduce $p$-adic integral models $X = X_{B_\Sigma}$ at tame level and $X_0(\Fp)$ at Iwahori level.
We then define a Goren-Oort stratification on the geometric special fiber $\overline{X} = X \otimes \overline{\bF}_p$, generalizing both the ones of \cite{tian_xiao_2016} and \cite{reduzzi2016partial}, and a stratification of $\overline{X}_0(\Fp) = X_0(\Fp) \otimes \overline{\bF}_p$ the usual unramified Hilbert case, such that there is a privileged stratum upstairs $\overline{X}_0(\Fp)_\emptyset$ which maps isomorphically downstairs to $\overline{X}$, respecting the stratifications on both sides. Finally, by extending the splice construction of \cite{2020arXiv200100530D}, we show that each smooth stratum at Iwahori level is isomorphic to a product of $\mathbb{P}^1$-bundles over an auxiliary Quaternionic Shimura variety.

\subsection{Overview of the thesis}
 \subsubsection{Tame level Shimura varieties}

We start by fixing a totally real field $F/\bQ$ of degree $d$ and a prime $p \in \bZ$ which we allow to ramify. We write $\Theta_F = \{\beta: F \hookrightarrow \overline{\bQ} \}$ for the set of embeddings which we identify with the set of real embeddings $\{F \hookrightarrow \bR \}$ and $p$-adic embeddings $\{F \hookrightarrow \overline{\bQ}_p \}$ via fixed embeddings $\overline{\bQ} \hookrightarrow \bC$ and $\overline{\bQ} \hookrightarrow \overline{\bQ}_p$. We also write $\widehat{\Theta}_F = \{\theta: \ccO_F/\fp \hookrightarrow \overline{\bF}_p \, \vert \, \, \fp \vert p \ccO_F \}$ for the set of unramified embeddings of $F$ and choose for each $\theta \in \widehat{\Theta}_F$ an arbitrary ordering of the $e_\fp$ embeddings $\theta^1, \cdots, \theta^{e_\fp} \in \Theta_F$ which lift it. 

We also fix a quaternion algebra $B/F$ with ramification set $\Sigma$ such that no place over $p$ in $F$ lies in $\Sigma$ and fix a maximal order $\ccO_B$. We write $\Sigma_\infty $ for the set of archimedean places, which we identify as a subset of $\Theta_F$ by above. We write $G_\Sigma / \bQ$ for the algebraic group given on $R$-points by $G_\Sigma(R) = (B \otimes_\bQ R)^\times$. Let $U \subset G_\Sigma(\bA_f)$ be a sufficiently small open compact subgroup containing $\ccO_{B,p}^\times$ and consider the associated Shimura variety with complex points 
\[ G_\Sigma(\bQ) \backslash (\bC - \bR)^{\Theta_F \setminus \Sigma_\infty} \times G_\Sigma(\bA_f) / U.\]
If we suppose that $p$ is unramified in $F$, then the above has a canonical model integral model $Y_U(G_\Sigma)$ over the ring of integers $\ccO$ of its reflex field $L_\Sigma$. Since $Y_U(G_\Sigma)$ does not have a moduli interpretation, this model is obtained as a quotient of an auxiliary Unitary Shimura variety which is itself a quotient of a PEL moduli problem, as described in \cite{Carayol1986SurLM} or \cite{tian_xiao_2016}. In particular, since the same is true of the Unitary model, it is smooth of dimension $\vert \Theta_F \setminus \Sigma_\infty \vert$ over the localization at any prime above $p$. If $p$ is ramified in $F$ however, the associated integral Unitary model is singular in codimension two and thus so is the Quaternionic model. The first goal of this thesis is dedicated to resolving the singularities of these models.

Our strategy is to obtain smooth Quaternionic integral models is via first resolving the singularities of the associated Unitary models, by generalizing the construction of Pappas and Rapoport, described above in the Hilbert case, to our general Unitary setting. Indeed, the leitmotif of this thesis will be to first review the Hilbert setting, then the Unitary setting and how to transfer to the Quaternionic setting, and finally show the compatibility between the Hilbert setting and the Quaternionic setting with $B = M_2(F)$.

Our Unitary models are roughly obtained as follows: Let $B$ be as above and consider a CM extension $E/F$ which splits $B$, that is $D := B \otimes_F E \simeq M_2(E)$. Fix a set of lifts $\tau$ of each unramified embedding $\theta \in \widehat{\Theta}_F$ and write $\tau^1,\cdots,\tau^{e_\fp}$ for the lifts of the $\theta^1,\cdots,\theta^{e_\fp}$ and $c$ the non trivial element of $Gal(E/F)$. Fix a maximal order $\ccO_D$ of $D$. We also fix a lift $\widetilde{\Sigma}_\infty$ of $\Sigma_\infty$, which needs not be composed of the above fixed lifts, that yields numbers $s_{\tau^i} \in \{0,1,2\}$ with $s_{(\tau^i)^c}=2-s_{\tau^i}$. For an appropriate notion of prime-to-$p$ level structure, we consider the PEL moduli space of of polarized abelian varieties of dimension $4d$ with $\ccO_D$-action, such that the characteristic polynomial of the action of $\alpha \in \ccO_E$ on $\textrm{Lie}(A/S)$ is given by $\prod_{\beta \in \Theta_E} (x-\beta(\alpha))^{2s_\beta}$, and level structure, which we denote here by $Y_\Sigma^{\prime \textrm{DP}}$. It is defined over the ring of integers $\ccO$ of its reflex field and as stated previously, the fiber over primes over $p$ of this moduli space are singular if $p$ ramifies in $F$. We note however that there is a natural generalization of the notion of the usual Rapoport locus in this setting which is precisely the entire smooth locus. 

We thus fix a sufficiently large $p$-adic field $L/\bQ_p$ with ring of integers $\ccO_L$. For an $\ccO$-scheme $S$ and a point $\underline{A} \in Y^{\prime \textrm{DP}}(S)$, we can consider its reduced Hodge bundle $\omega^0_{A/S} := (s_*\Omega^1_{A/S})\cdot e_0$, where $e_0 \in \ccO_D \otimes \ccO_S \simeq M_2(\ccO_S)$ corresponds to the element $\left( \begin{smallmatrix} 1 & 0 \\ 0 & 0 \end{smallmatrix} \right)$. We can furthermore consider for each $\theta \in \widehat{\Theta}_F$ with lift $\tau$, the component $\omega^0_{A/S,\tau}$ and we define a Pappas Rapoport filtration of $\omega^0_{A/S}$, which we denote by $\underline{\omega_{A/S}}$ to be the data of a $\ccO_E$-stable filtration for each $\theta \in \widehat{\Theta}_F$ of the form 
\[0 \subset \omega^0_{A/S,\tau}(1) \subset \cdots \subset \omega^0_{A/S,\tau}(e_\fp) = \omega^0_{A/S,\tau},\]
where each graded piece is locally free of rank $s_{\tau^i}$ and is killed by $\tau^i(\varpi_\fp)$ for any choice of uniformizer $\varpi_\fp \in \ccO_F$ of $F_\fp$. We denote the moduli space of such pairs $(\underline{A},\underline{\omega_{A/S}})$ over $\ccO_L$ by $Y_\Sigma^{\prime \textrm{PR}}$. This space comes with a natural forgetful map to $Y_\Sigma^{\prime \textrm{DP}}$ and we obtain, via crystalline deformation theory, the first main result of this thesis:
\begin{thm}
The splitting model $Y^{\prime \textrm{PR}}_\Sigma$ is smooth over $\ccO$ of relative dimension $\vert \Theta_F \setminus \Sigma_\infty \vert$. Furthermore, the natural forgetful map $\pi: Y^{\prime \textrm{PR}}_\Sigma \to Y^{\textrm{DP}}_\Sigma$ is birational: it is an isomorphism over the generalized Rapoport loci. In particular, it induces an isomorphism over generic fibers.
\end{thm}
Note that a priori, this model depends on the choice of lifts $\tau$; we will show how to remove this dependence by introducing the notion of the dual filtration on $\tau^c$-components. Proceeding as in the unramified setting, we thus obtain our Quaternionic models as quotients of these Unitary splitting models.

\subsubsection{The Goren-Oort stratification}
Having established the models we will work with, we now define the Goren-Oort stratification of their special fiber:
Write $\underline{A}$ for the universal abelian scheme over the special fiber $\overline{Y}^{\prime}_\Sigma$ of $Y^{\prime \textrm{PR}}_\Sigma$, the Pappas-Rapoport filtration gives us a plethora of vector bundles, namely the subquotients of $\omega^0_\tau: = \omega^0_{A/\overline{Y}^{\prime}_\Sigma,\tau}$, $\omega^0_{\tau^i} = \omega^0_\tau(i) / \omega^0_\tau(i-1)$, locally free of rank $s_{\tau^i}$ by definition, and the subquotients of $\ccH^1_{\dR}(A/\overline{Y}^{\prime}_\Sigma)^0_\tau$, $\ccH^1_{\tau^i} = [\varpi_\fp]^{-1}\omega^0_\tau(i-1)/ \omega^0_\tau(i-1)$. The latter are locally free of rank two and can be thought of a way to make sense of de Rham cohomology at $\tau^i$.

Using the action of $[\varpi_\fp]$ and Verschiebung, we can define maps $V_{\es,\tau^i}:\ccH^1_{\tau^i} \to \ccH^1_{\tau^{i-1}}$, for $i >1$, and $V_{\es,\tau^1}:\ccH^1_{\tau^1} \to \ccH^{1 (p)}_{(\phi^{-1}\circ \tau)^{e_\fp}}$, at $i=1$ with $(p)$ denoting a $p$-twist, which we call essential Verschiebung after the corresponding morphisms from \cite{tian_xiao_2016} in the unramified case. These morphisms are defined such that they are isomorphisms when $s_{\tau^{i-1}} \neq 1$, that is when $\theta^{i-1} \in \Sigma_\infty$. We similarly morphisms in the opposite direction which we call essential Frobenius.

Note that we can define a cycle structure $\phi$ on $\Theta_F$ by defining $\phi(\theta^i) = \theta^{i+1}$ for $i < e_\fp$ and $\phi(\theta^{e_\fp})=(\phi \circ \theta)^1$. Furthermore, this induces a cycle structure $\phi^\prime$ on $\Theta_F \setminus \Sigma_\infty$. Given an element $\theta^i \notin \Sigma_\infty$, repeatedly composing essential Verschiebung and restricting yields a morphism $h_{\tau^i}: \omega^0_{\tau^i} \to \omega^{0 (p^n)}_{(\phi^\prime)^{-1}(\tau^i)}$, for the appropriate $p$-twist. The two sheaves are line bundles and we obtain a section $h_{\tau^i} \in H^0(\overline{Y}^\prime_\Sigma, (\omega^0)^{-1}_{\tau^i} \otimes (\omega^0)^{p^n}_{(\phi^\prime)^{-1}(\tau^i)})$ which we call the partial Hasse invariant at $\tau^i$. Restricting to the unramified case, this definition yields the partial Hasse invariant from \cite{tian_xiao_2016}, and restricting to $\Sigma = \emptyset$ yields the partial Hasse invariant from \cite{reduzzi2016partial}.

We can thus define the (closed) Goren-Oort strata $\overline{Y}^\prime_T$ of $\overline{Y}^\prime_\Sigma$ for subsets $T \subset \Theta_F \setminus \Sigma_\infty$ by the vanishing of the sections $h_{\tau^i}$ for $\theta^i \in T$. In particular, this definition agrees with that of \cite{tian_xiao_2016} and \cite{reduzzi2016partial}. Using the usual arguments, we descend this stratification to a stratification at the Quaternionic level, where we denote everything by removing the $\cdot^\prime$ superscript. Deducing the analogous results at the Unitary level, we obtain the following local properties of the strata:

\begin{thm}
For each subset $T \subset \Theta_F \setminus \Sigma_\infty$, the closed Goren-Oort stratum $\overline{Y}_T \subset \overline{Y}_\Sigma$ is smooth of codimension $\vert T \vert$. Furthermore, it is proper if $T \cup \Sigma_\infty \neq \emptyset$.
\end{thm}

We will in fact give a global description of these strata, extending the similar result from \cite{tian_xiao_2016}, that these strata are isomorphic to products of $\mathbb{P}^1$ over (splitting models) of auxiliary Quaternionic Shimura varieties. To do so however, we will go to Iwahori level and prove the analogous result and show how it implies the former.

\subsubsection{Iwahori level models}

As stated above, we now introduce Iwahori level models and their stratifications, in particular, we will consider level $U_0(\Fp) = UI_0(\Fp)$, where $U$ is some prime-to-$p$ level and $I_0(\Fp)$, where we recall that $\Fp$ is the radical of $p\ccO_F$, is given by
\[I_0(\Fp)=\{g \in \textrm{GL}_2(\ccO_{F,\Fp}) \, \vert \, g \equiv \left(\begin{smallmatrix} * & * \\ 0 & * \end{smallmatrix}\right) \, \textrm{mod} \, \fP \},\]
via the isomorphism $\ccO_{B,p} \simeq GL_2(\ccO_{F,p})$. As usual, we define the analogous Unitary models and descend.

To do so, we consider the moduli space of isogenies $Y^\prime_0(\Fp)$ given by isomorphism classes $(\underline{A_1},\underline{A_2},f,g)$, where the $\underline{A_i}$ define points of $Y^\prime_\Sigma$, $f:A_1 \to A_2$ is a $\ccO_D$-linear isogeny with prescribed kernel in $A_1[\Fp]$ which respects the polarizations and level structures, and $g: A_2 \to A_1 \otimes \Fp^{-1}$ is the unique isogeny such that $g \circ f : A_1 \to A_1 \otimes \Fp^{-1}$ is the isogeny induced by the inclusion $\ccO_F \hookrightarrow \Fp^{-1}$. Furthermore, we require both $f$ and $g$ to respect the Pappas-Rapoport filtrations on both sides. Despite having this extra condition on $g$, we will see that it follows automatically from the condition on $f$, we will keep it in the notation however for book-keeping.

There is a natural stratification on the special fiber $\overline{Y}_0^\prime(\Fp)$ given as follows: Write $A_j$ for the universal abelian schemes over $\overline{Y}_0^\prime(\Fp)$, then for any $\theta^i \notin \Sigma_\infty$, the sheaves $\omega^0_{j,\tau^i} = \omega^0_{A_j/\overline{Y}_0^\prime(\Fp),\tau^i}$ are line bundles and $f$ and $g$ respect the filtrations, they induce morphisms $f^*_{\tau^i}: \omega^0_{2,\tau^i} \to \omega^0_{1,\tau^i}$ and $g^*_{\tau^i}: \omega^0_{1,\tau^i} \otimes \Fp \to \omega^0_{2,\tau^i}$. 
It is straightforward to see that the composition $f^*_{\tau^i} \circ g^*_{\tau^i} = 0$. This leads to defining for pairs of subsets $I,J \subset \Theta_F \setminus \Sigma_\infty$, $\overline{Y}^\prime_0(\Fp)_{\phi^\prime(I),J}$ as the vanishing of $f^*_{\tau^i}$ for $\theta^i \in I$ and $g^*_{\tau^i}$ for $\theta^i \in J$. An explicit analysis using crystalline deformation theory and descending the above constructions to the Quaternionic level yields the following result:
\begin{thm}
    $Y_0(\Fp)$ is a reduced complete intersection, flat over $\ccO$ of relative dimension $\vert \Theta_F \setminus \Sigma_\infty \vert$. Furthermore, the strata $\overline{Y}_0(\Fp)_{\phi^\prime(I),J} \subset \overline{Y}_0(\Fp)$ are smooth of codimension $\vert I \cap J \vert$ for the pairs $I,J$ such that $I \cup J = \Theta_F \setminus \Sigma_\infty$.
\end{thm}

The appearance of the shift $\phi^\prime$ in the notation for the stratification is motivated by the following: Note that there is a natural projection $\pi:\overline{Y}_0(\Fp) \to \overline{Y}$ sending a tuple $(\underline{A_1},\underline{A_2},f,g)$ to $\underline{A_1}$. We have the following compatibility with the Goren-Oort strata downstairs.

\begin{thm}\label{goren oort implication}
    The natural projection $\overline{Y}_0(\Fp) \to \overline{Y}$ restricts on strata to morphisms $\pi: \overline{Y}_0(\Fp)_{\phi^\prime(I),J} \to \overline{Y}_{I \cap J}$. In particular, the restriction
    \[ \pi: \overline{Y}_0(\Fp)_{\Theta_F \setminus \Sigma_\infty, \emptyset} \to Y\]
    is an isomorphism which maps the stratum $\overline{Y}_0(\Fp)_{\Theta_F \setminus \Sigma_\infty, T}$ isomorphically onto $\overline{Y}_T$.
\end{thm}

The first part of the statement is a straightforward consequence of the definitions. To prove the second part, which can be seen as a special instance of our main result, we opt to construct, using Raynaud theory, an explicit section of the projection via an isogeny $F_{\es,\Fp}: A \to A^{(\Fp)}$ which we call the essential Frobenius isogeny. Indeed, in the usual unramified Hilbert case, this section is given by the normal Frobenius isogeny. Furthermore, the name is also motivated by the fact that its adjoint isogeny $V_{\es,\Fp}: A^{(\Fp)} \to A \otimes \Fp^{-1} $, which we call the essential Verschiebung is in some sense (made more precise in proposition \ref{essential name justification}) a geometric incarnation of the essential Verschiebung operators defined previously.

\subsubsection{Geometric Jacquet Langlands relations mod $p$}

We can now state the main result of this thesis. We need to consider the geometric special fiber $\overline{Y} = Y_{0}(\Fp)_{\overline{\bF}_p}$ of $Y_0(\Fp)$ and a smooth stratum $\overline{Y}_0(\Fp)_{\phi^\prime(I),J}$. To simplify the statement of the theorem we make the further assumption that at no prime $\fp$ of $F$ over $p$, $I \cap J$ contains $\Theta_{F,\fp} \setminus \Sigma_{\infty,\fp}$, the set of places above $\fp$ unramified in $B$. For each $\beta = \theta^i \in \Theta_F$, we have a locally free sheaf $\ccH^1_\beta$ of rank 2, obtained by descent from the ones at Unitary level defined above, and we define an even set of places $\Sigma_{IJ}$ containing $\Sigma$ and a subset $R \subset \Theta_F$, such that writing $\overline{Y}_{IJ} = Y_{\Sigma_{IJ},\overline{\bF}_p}$, we have the following:
\begin{thm}
 For each pair $I,J \subset \Theta_F \setminus \Sigma_\infty$ as above, and sufficiently small open compact subgroup $U \subset G_\Sigma(\mathbb{A}_f)$ containing $\ccO^\times_{B,p}$, there is a Hecke equivariant isomorphism 
 \[ \overline{Y}_0(\Fp)_{\phi^\prime(I),J} \xrightarrow{\sim} \prod_{\beta \in R} \mathbb{P}^1_{\overline{Y}_{IJ}}(\ccH^1_\beta),\]
 where the fiber product is taken over $\overline{Y}_{IJ}$.
\end{thm}

This result is a direct generalization of Theorem 5.3.1 of \cite{2020arXiv200100530D} in which they only considered the unramified Hilbert setting and the top dimensional strata, that is $I = J^c$. Before describing the proof, which expands on their method, we note that combining the above theorem with Theorem \ref{goren oort implication} immediately yields the promised description of the Goren-Oort strata over $\overline{\bF}_p$ for subsets $T$ such that $T$ does not contain contains $\Theta_{F,\fp} \setminus \Sigma_{\infty,\fp}$ for any prime $\fp$. Writing $\Sigma_T = \Sigma_{(\Theta_F \setminus \Sigma_\infty)\,T}$ we obtain:

\begin{thm}
 For each subset $T \subset \Theta_F \setminus \Sigma_\infty$ as above, and sufficiently small open compact subgroup $U \subset G_\Sigma(\mathbb{A}_f)$ containing $\ccO^\times_{B,p}$, there is a Hecke equivariant isomorphism 
 \[ \overline{Y}_T \xrightarrow{\sim} \prod_{\beta \in R} \mathbb{P}^1_{\overline{Y}_{\Sigma_T}}(\ccH^1_\beta),\]
 where the fiber product is taken over $\overline{Y}_{\Sigma_T}$.
\end{thm}

As usual, we first prove the analogous result at Unitary level and then descend. Our argument expands on the splicing construction from \cite{2020arXiv200100530D}, in the case of the top dimensional strata in the Hilbert setting. It is in essence a pointwise construction. At Iwahori level we are able to construct a third abelian variety whose Dieudonn\'{e} module is obtained by splicing together the Dieudonn\'{e} modules of $A_1$ and $A_2$ from which we can recover both our original abelian varieties using the relations imposed by the stratum we are considering as well as the extra data provided by the $\mathbb{P}^1$-bundles. In the ramified setting, we also obtain the filtration of our third abelian variety by splicing together the filtrations of $A_1$ and $A_2$. For example, in the case that $p$ is totally ramified in $F$ our third abelian variety is either $A_1$ or $A_2$ and our morphism is constructed precisely by carefully forgetting certain parts of the filtration. We also note that what allows us to complete the description of the strata in the unramified Hilbert setting, and extend to the general Quaternionic setting is our systematic use of the essential Frobenius and Verschiebung isomorphisms, described above, which in some sense communicate between the different components of the relevant Dieudonn\'{e} modules.

The machinery which allows us to globalize this pointwise construction is the theory of Raynaud group schemes, introduced in \cite{BSMF_1974__102__241_0}, whose structure allows us to construct specific subgroups of the universal kernel. In fact, to obtain the full flexibility needed in our generality we introduce a mild generalization of this theory in characteristic $p$ and introduce what we call Partial Raynaud group schemes. It is only with this added flexibility that we are able to extend the splice construction in our generality and more crucially, it was the missing ingredient, which combined with crystalline Dieudonn\'{e} theory, to give an explicit, moduli theoretic, construction of the inverse map as opposed to doing pointwise calculations. We hope that this construction can be applied to the setting of more general PEL type Shimura varieties. 

Before giving an explicit example of Theorem 1.7, let us comment on the condition we imposed on $I \cap J$. We will see that the general recipe to construct $\Sigma_{IJ}$ and $R$ is given by applying it independently at each prime $\fp$. Consider then the case that there is only one prime above $p$ in $F$, then we imposed that $I \cap J \neq \Theta_F \setminus \Sigma_\infty$. If we have $I = J = \Theta_F \setminus \Sigma_\infty$, that is the bottom stratum, then as in \cite{tian_xiao_2016} there are two cases: If $\vert \Theta_F \setminus \Sigma_\infty \vert $ is even, the stratum is isomorphic to the Quaternionic Shimura variety whose underlying quaternion algebra has ramification set $\Sigma \sqcup (\Theta_F \setminus \Sigma_\infty)$ but now with Iwahori level. If $\vert \Theta_F \setminus \Sigma_\infty \vert $ is odd, the stratum is isomorphic to the Quaternionic Shimura variety whose underlying quaternion algebra has ramification set $\Sigma \sqcup (\Theta_F \setminus \Sigma_\infty) \sqcup \{\fp\}$. We do not explicitly cover the latter case in this thesis as this requires us, for technical reasons, to modify the definitions of our models. We nonetheless sketch the procedure in section \ref{bottom strata}.

The recipe to obtain $\Sigma_{IJ}$ in general is rather involved so we prefer to give a concrete example. Consider the case that $F/\bQ$ is a totally real extension of degree four with a single prime $\fp$ over $p$ and $\Sigma = \emptyset$. We note that the ramification of $p$ does not change the formulae in the following tables. We write $\Theta_F = \{\theta_1,\theta_2,\theta_3,\theta_4\}$ with the obvious relations under the shift $\phi$. There are $3^4 = 81$ different strata so we choose to present the formulae in three different tables: The first being the top dimensional strata, the second the Goren-Oort strata and finally the remaining intermediate strata. Since swapping $I$ and $J$ produces identical formulae, we only include one of each such pair.
\[\adjustbox{scale = 0.95, center}{
\renewcommand{\arraystretch}{1}%
\begin{tabular}{| c | c | c |}
\hline
Stratum $\overline{Y}_{\phi(J^c),J}$ & Quaternion algebra $B_{\Sigma_J}$ & $\mathbb{P}^1$-bundles  \\ \hline
 $J= \emptyset$     &                   $\Sigma_J = \emptyset$             &   $R = \emptyset$                        \\ \hline
    $J =\{\theta_i\}$    &      $\Sigma_J = \{\theta_{i-1},\theta_i\}$                          &  $R = \{\theta_{i-1},\theta_i\}$                       \\ \hline
    $J = \{\theta_i,\theta_{i+1}\}$    &   $\Sigma_J = \{\theta_{i-1},\theta_{i+1}\}$                       &   $R = \{\theta_{i-1},\theta_{i+1}\}$                       \\ \hline
    $J = \{\theta_i,\theta_{i+2}\}$     &       $\Sigma_J = \Theta_F$                         &    $R = \Theta_F$                       \\ \hline
    $J = \{\theta_i,\theta_{i+1},\theta_{i+2}\}$     &    $\Sigma_J = \{\theta_{i-1},\theta_{i+2}\}$          &   $R = \{\theta_{i-1},\theta_{i+2}\}$                        \\ \hline
\end{tabular}
}\]
Note that in the above table, we have $\Sigma_J = R$, this is expected if only for dimension reasons. Furthermore, the formula for $\Sigma_J$ is given by $\Sigma_J= \{ \theta_i \in J \, \vert \, \phi ( \theta_i) \notin J \} \cup \{ \theta_i \notin J \, \vert \, \phi ( \theta_i) \in J \}.$

\[\adjustbox{scale = 0.95, center}{
\renewcommand{\arraystretch}{1}%
\begin{tabular}{| c | c | c |}
\hline
Stratum $\overline{Y}_{T}$ & Quaternion algebra $B_{\Sigma_T}$ & $\mathbb{P}^1$-bundles  \\ \hline
    $T =\{\theta_i\}$    &      $\Sigma_T = \{\theta_{i-1},\theta_i\}$                          &  $R = \{\theta_{i-1},\theta_i\}$                       \\ \hline
    $T = \{\theta_i,\theta_{i+1}\}$    &   $\Sigma_T = \{\theta_{i},\theta_{i+1}\}$                       &   $R = \emptyset$                       \\ \hline
    $T = \{\theta_i,\theta_{i+2}\}$     &       $\Sigma_T = \Theta_F$                         &    $R = \{\theta_{i+1},\theta_{i+3}\}$                       \\ \hline
    $T = \{\theta_i,\theta_{i+1},\theta_{i+2}\}$     &    $\Sigma_T = \Theta_F $          &   $R = \{\theta_{i+3}\}$                        \\ \hline
    $T= \Theta_F$     &                   $\Sigma_T = \Theta_F$ + Iwahori level             &   $R = \emptyset$                        \\ \hline
\end{tabular}
}\]

We remark that this produces the same varieties as in table 1.5.3 of \cite{tian_xiao_2016}. We also note the appearance of Iwahori level for the bottom stratum $T = \Theta_F$. The formula for $\Sigma_T$ here is given by "evening out" $T$: Write $T$ as a disjoint union of maximal cycles in $\Theta_F$ with respect to $\phi$, $\Sigma_T$ is obtained by appending to $T$ the elements $\phi^{-1}(\theta_i)$, where $\theta_i$ is the first element of a chain of $T$ of odd length. $R$ is then given by $\Sigma_T \setminus T$.

\[\adjustbox{scale = 0.95, center}{
\renewcommand{\arraystretch}{1}
\begin{tabular}{| c | c | c |}
\hline
Stratum $\overline{Y}_{\phi(I),J}$ & Quaternion algebra $B_{\Sigma_{IJ}}$ & $\mathbb{P}^1$-bundles  \\ \hline
    $(I,J) =(\{\theta_i,\theta_{i+1}\}, \{\theta_i,\theta_{i+2},\theta_{i+3}\})$    &      $\Sigma_{IJ} = \{\theta_i,\theta_{i+1}\}$                          &  $R = \{\theta_{i+1}\}$                       \\ \hline
    $(\{\theta_i,\theta_{i+1}\}, \{\theta_{i+1},\theta_{i+2},\theta_{i+3}\})$    &   $\Sigma_{IJ} = \{\theta_{i+1},\theta_{i+3}\}$                       &   $R = \{\theta_{i+3}\}$                       \\ \hline
    $(\{\theta_i,\theta_{i+2}\}, \{\theta_i,\theta_{i+1},\theta_{i+3}\})$     &       $\Sigma_{IJ} = \Theta_F$                         &    $R = \{\theta_{i+1},\theta_{i+2},\theta_{i+3}\}$                       \\ \hline
    $(\{\theta_i,\theta_{i+2}\}, \{\theta_{i+1},\theta_{i+2},\theta_{i+3}\})$     &    $\Sigma_{IJ} = \Theta_F $          &   $R = \{\theta_{i},\theta_{i+1},\theta_{i+3}\}$                        \\ \hline
    $(\{\theta_i,\theta_{i+1},\theta_{i+2}\}, \{\theta_{i},\theta_{i+1},\theta_{i+3}\})$     &                   $\Sigma_{IJ} = \Theta_F$              &   $R = \{\theta_{i+2},\theta_{i+3}\}$                        \\ \hline
    $(\{\theta_i,\theta_{i+1},\theta_{i+2}\}, \{\theta_{i},\theta_{i+2},\theta_{i+3}\})$     &                   $\Sigma_{IJ} = \{\theta_i,\theta_{i+2}\}$              &   $R = \emptyset$                        \\ \hline
    $(\{\theta_i,\theta_{i+1},\theta_{i+2}\}, \{\theta_{i+1},\theta_{i+2},\theta_{i+3}\})$     &                   $\Sigma_{IJ} = \Theta_F$             &   $R = \{\theta_{i},\theta_{i+3}\}$                        \\ \hline
\end{tabular}
}\]

The general formula for $\Sigma_{IJ}$ is a combination of the two previous procedures, note that in every case we have the inclusion $I \cap J \subset \Sigma_{IJ}$ and the equality $R = \Sigma_{IJ} \setminus I \cap J$.

\subsubsection{Overview and advice for the Reader}

In Chapter 2 we define the splitting models of Unitary and Quaternionic Shimura varieties and end with the comparison between the Pappas Rapoport Hilbert model and the Quaternionic model with $\Sigma = \emptyset$. Chapter 3 is concerned with the definition of the Goren-Oort stratification and its basic properties. In Chapter 4 we define Iwahori level models, their stratification, and study their local structure. In Chapter 5, we develop our theory of partial Raynaud group schemes, review Crystalline Dieudonn\'{e} theory as presented in \cite{berthelot1982theorie}, provide the technical setup for Chapter 6, and end with the construction of the Essential Frobenius isogeny. Finally, in Chapter 6 we state and prove the main theorem of this thesis.

We strongly advise the reader to read section \ref{motivating splices} in which we give the intuition behind the splice construction. We also advise on a first reading of Chapter 6 to keep in mind the case $\Sigma = \emptyset$ in order to lighten the notational burden and case analyses.  

\subsection{Acknowledgements}
I would like to thank Payman Kassaei, my Ph.D. advisor, for suggesting this problem and for his trust in me. I would also like to thank Fred Diamond for his helpful advice.
This work was supported by the Engineering and Physical Sciences Research Council [EP/L015234/1],
through the EPSRC Centre for Doctoral Training in Geometry and Number Theory (The London School of
Geometry and Number Theory), University College London, King’s College London, and Imperial College London.

\section{Tame Level Shimura varieties}\label{tame}

\subsection{Pappas-Rapoport models of Hilbert Modular varieties} In this section we define the Pappas-Rapoport models of Hilbert modular varieties introduced in \cite{pappas2005local}. These are smooth integral models which agree with the usual definition in the case that $p$ is unramified in $F$. We tailor our definition in view of our subsequent construction of integral models of Quaternionic Shimura varieties. In particular we obtain for $p$ unramified in $F$, the same models as defined in \cite{2020arXiv200100530D}.

\subsubsection{Notation}\label{hilbert notation} Fix a totally real field $F$ of degree $d=[F:Q]$ and a prime $p$ which we allow to ramify in $F$. For each prime $\fp \vert p$ in $F$, we write $F_\fp$ for the completion of $F$ at $\fp$, with corresponding valuation $v_\fp$ and residue field $\bF_\fp$. Let $f_\fp$ and $e_\fp$ be its inertia and ramification degree respectively. For each $\fp$, fix a totally positive element $\varpi_\fp \in F$ such that $v_\fp(\varpi_\fp)=1$ and $v_{\fq}(\varpi_\fp)=0$ for all primes $\fq \vert p$, $\fq \neq \fp$. It follows that $\fp = (p,\varpi_\fp)$ and $\varpi_\fp$ is a uniformizer of $F_\fp$. Any two such choices of $\varpi_\fp$ differ by an element of $\ccO_{F,(p),+}^\times$. We also write $\bA_{F,f}$ for the finite ad\`{e}les over $F$ and $\bA_{F,f}^{(p)}$ for the prime to $p$ finite ad\`{e}les over $F$.\medskip

We let $\Theta_F$ denote the set of embeddings $F \hookrightarrow \overline{\bQ}$ which we identify with the set of embeddings $\Theta_{F,\infty} = \{F \hookrightarrow \mathbb{R}\}$ and $\Theta_{F,p} =  \{ F \hookrightarrow \overline{\bQ}_p \}$ via fixed embeddings $\overline{\bQ} \hookrightarrow \mathbb{C}$ and $\overline{\bQ} \hookrightarrow \overline{\bQ}_p$. 
From now on, we write $\Theta_F$ for all three sets. 
For each prime $\fp$ over $p$, we let $\Theta_{F,\fp}$ denote the set of of embeddings $\beta : F_\fp \hookrightarrow \overline{\bQ}_p$ and $\widehat{\Theta}_{F,\fp}$ denote the set of of embeddings $\theta : F_\fp^{0} \hookrightarrow \overline{\bQ}_p$ where $F^0_\fp$ denotes the maximal unramified subextension of $F_\fp$. This set is canonically identified with the set of embeddings $\{\bF_\fp \hookrightarrow \overline{\bF}_p\}$.
For each $\theta \in \widehat{\Theta}_\fp$, we choose an arbitrary ordering of the $e_\fp$ embeddings $\theta^1, \cdots , \theta^{e_\fp} \in \Theta_{F,\fp}$ which extend it.
We thus have $\Theta_{F} = \bigsqcup_{\fp \vert p} \Theta_{F,\fp}$ and set $\widehat{\Theta}_{F} = \bigsqcup_{\fp \vert p} \widehat{\Theta}_{F,\fp}$. Note that we will often use $\beta$ to denote an element of $\Theta_F$, and use $\theta$ to denote an element of $\widehat{\Theta}_F$. Given a $\theta$, we will always write $\fp$ for the prime it corresponds to under the previous decomposition. 
We recall that there is a shift operation on $\Theta_F$ given by $\phi(\theta^i) = \theta^{i+1}$ for $1 \leq i < e_\fp$ and $\phi(\theta^{e_\fp})=(\phi \circ \theta)^1$ where $\phi$ is the usual Frobenius acting on $\widehat{\Theta}_{F,\fp}$.

We fix a number field $L \subset \overline{\bQ}$ which contains all of the images $\beta(F)$ for $\beta \in \Theta_F$ and let $\ccO$ be the completion of $\ccO_E$ at the prime determined by the choice of embedding $\overline{\bQ} \to \overline{\bQ}_p$. We let $\bF$ denote its residue field. For each prime $\fp$, we let $E_\fp(x)$ denote the minimal polynomial of $\varpi_\fp$ over $F_\fp^0$, and for each $\theta \in \widehat{\Theta}_{\fp}$ we write
$$E_{\theta}(x) = \theta(E_{\fp}(x)) =(x-\theta^1(\varpi_\fp))\cdots(x-\theta^{e_\fp}(\varpi_\fp)) \in \ccO[x].$$

\subsubsection{The models}\label{hilbert model}
Let $G= \textrm{Res}_{F/\bQ}  \GL_2$, $\mathbb{S}= \textrm{Res}_{\bC/\bR}\mathbb{G}_m$, and let $(G,[h])$ be the usual Shimura datum, that is, $[h]$ is the $G(\bR)$-conjugacy class of the morphism $h:\mathbb{S}\to G(\mathbb{R})$, $x+iy \mapsto \begin{pmatrix}
x & y \\ -y & x
\end{pmatrix}$. 
Let $U \subset G(\bA_f)=GL_2(\bA_{F,f})$ be an open compact subgroup containing $GL_2(\ccO_F \otimes \bZ_p)$ which is small enough in the sense of \cite{DS17}. Write $U=U^p U_p$ with $U^p \subset GL_2(\bA_{F,f}^{(p)})$ and $U_p = GL_2(\ccO_F \otimes \bZ_p)$. The Hilbert modular variety of level $U$ is the Shimura variety whose complex points are given by 
\[\GL_2(F) \backslash (\bC - \bR)^{\Theta_F} \times \GL_2(\bA_{F,f})/U = \GL_2(F)_+ \backslash \ccH^{\Theta_F} \times \GL_2(\bA_{F,f}) / U\]
where $\ccH$ is the upper half plane and $\GL_2(F)_+$ is the subgroup of matrices with totally positive determinant.\medskip

Let $\ccO$ be given as above. We consider the following moduli problem:

\begin{defn}
Let $\widetilde{Y}^{\textrm{DP}}_U(G)$ denote the functor which associates to every locally Noetherian $\ccO$-scheme $S$ the set of isomorphism classes of tuples $\underline{A}=(A,\iota,\lambda,\eta)$ where: 

\begin{itemize}
    \item $A/S$ is an abelian scheme of relative dimension $d$,
    \item $\iota:\ccO_{F} \to \End_S(A)$ is a ring embedding such that the Kottwitz condition holds. That is, for any $\alpha \in \ccO_F$, the action of $\iota(\alpha)$ on $\textrm{Lie}(A/S)$ has characteristic polynomial \[\prod_{\beta \in \Theta_F} (x-\beta(\alpha)) \in \ccO[x],\] 
    \item $\lambda$ is a prime to $p$ quasi-polarization such that the associated Rosati involution fixes $\iota$. That is, $\lambda:A \to A^\vee$ is a quasi-isogeny such that for each connected component of $S$, there is a prime to $p$ integer $n$ such that $n\lambda$ is a genuine polarization,
    \item $\eta$ is a $U^p$ level structure on $A$. That is, for a choice of a geometric point $\overline{s_i}$ on each connected component $S_i$ of $S$, a $\pi_1(S_i,\overline{s_i})$ invariant $U^p$-orbit of isomorphisms \[\eta: \widehat{\ccO}_{F}^{(p)} \times \widehat{\ccO}_{F}^{(p)} \to T^{(p)}A_{\overline{s}_i}.\] 
\end{itemize}
\end{defn}

As in \cite{2020arXiv200100530D}, this functor is representable by an infinite disjoint union of quasi-projective schemes over $\ccO$. More precisely it is representable by an infinite disjoint union of PEL shimura varieties attached to the group $G^*:= \textrm{Res}_{F/\bQ}\GL_2\times_{\mathbb{G}_m} \textrm{Res}_{F/\bQ}\mathbb{G}_m$, indexed over $(\mathbb{A}_{F,f}^{(p)})^\times /\det(U^p)\widehat{\bZ}^{(p)\times}$. It is thus locally of finite type, of relative dimension $d$, and smooth if $p$ is unramified in $F$. If $p$ is ramified however, \cite[Proposition 4.4]{deligne1994singularites} shows that its special fiber is singular, with codimension 2 singular locus. In order to desingularize this "naive" model we introduce the Pappas-Rapoport model:

Let $R$ be an $\ccO$-algebra and $M$ an $\ccO_F \otimes R$ module. The decomposition 
\[\ccO \otimes_{\bZ} \ccO_F = \bigoplus_{\theta \in \widehat{\Theta}_F}  \ccO \otimes_{W(\bF_\fp),\theta} \ccO_{F,\fp} = \bigoplus_{\theta \in \widehat{\Theta}_F} \ccO[x]/E_\theta[x]\]
induces a decomposition 
\[M= \bigoplus_{\fp \vert p} M_\fp = \bigoplus_{\theta \in \widehat{\Theta}_F} M_\theta\] 
into $(\ccO_F \otimes R)$-modules for which the action of $\ccO_F$ on $M_\fp$ extends to a continuous action of $\ccO_{F,\fp}$, and the submodule $M_\theta$ of $M_\fp$ for $\theta \in \widehat{\Theta}_\fp$ corresponds to the submodule on which the action of $W(\bF_\fp)$ is given by scalar multiplication via $\theta : W(\bF_\fp) \to \ccO \to R$.\medskip  

Let $A/S$ be an abelian scheme with real multiplication as above and let $\pi_A:A \to S$ be the structure map. Let $\omega_{A/S} = \pi_{A*} \Omega^1_{A/S}$ and $\ccH^1_{\dR}(A/S) = R^1\pi_{A,*}\Omega^\bullet_{A/S}$.  We thus have a natural decomposition 
\[\omega_{A/S} \simeq \bigoplus_{\theta \in \widehat{\Theta}} \omega_{A/S,\theta}\]
into locally free sheaves of rank $e_\fp$ where $W(\bF_{\fp})$ acts on $\omega_{A/S,\theta}$ via $\theta$. We similarly have a natural decomposition 
\[\ccH^1_{dR}(A/S) \simeq \bigoplus_{\theta \in \widehat{\Theta}} \ccH^1_{dR}(A/S)_{\theta}\]

into locally free sheaves of rank $2e_\fp$. In fact, by \cite[Lemma 1.3]{CM_1978__36_3_255_0}, $\ccH^1_{dR}(A/S)$ is locally free of rank two over $\ccO_F \otimes \ccO_S$.

\begin{defn}
Consider the functor which associates to each $\ccO$-scheme $S$ the set of isomorphism classes of tuples $\underline{A}=(A,\iota,\lambda,\eta, \underline{\omega})$, where $(A,\iota,\lambda,\eta)$ is a point of $\widetilde{Y}^{\textrm{DP}}_U(G)$, and $\underline{\omega}$ is the data for each $\theta \in \widehat{\Theta}_F$ of a filtration 
\[0\subset \omega_\theta(1) \subset \cdots \subset \omega_\theta(e_\fp)=\omega_{A/S,\theta}\]
such that:
\begin{itemize}
    \item Each $\omega_\theta(t)$ is $\ccO_F$-stable,
    \item Each subquotient $\omega_\theta(t)/\omega_\theta(t-1)$ is locally free of rank one over $S$ (hence $\omega_\theta(t)$ is locally free of rank of rank $t$),
    \item The induced action of $\ccO_F$ on $\omega_\theta(t)/\omega_\theta(t-1)$ is given by $\theta^t$. Equivalently, it is killed by $([\varpi_\fp]-\theta^t(\varpi_\fp)) \in \ccO_F \otimes \ccO_S$.
\end{itemize}
\end{defn}

This functor is representable by a scheme $\widetilde{Y}_U(G) = \widetilde{Y}^{\textrm{PR}}_U(G)$, smooth and locally of finite type over $\ccO$ of relative dimension $d$ (see for example \cite{sasaki}). There is a natural forgetful morphism 
\[\pi : \widetilde{Y}^{\textrm{PR}}_U(G) \to \widetilde{Y}^{\textrm{DP}}_U(G)\] 
which is projective, and restricts to an isomorphism over the Rapoport loci of the source and target. That is, the loci corresponding to points $\underline{A}$ where $\omega_{A/S}$ is a locally free $(\ccO_S \otimes \ccO_F)$-module of rank one. We note that over $\widetilde{Y}^{\textrm{DP}}_U(G)$ the Rapoport locus coincides with the smooth locus, and the Rapoport locus is dense in $\widetilde{Y}^{\textrm{PR}}_U(G)$. In particular $\pi$ is birational, inducing an isomorphism on generic fibers, and is an isomorphism if and only if $p$ is unramified in $F$.

There is a natural action of $\ccO^\times_{F,(p),+}$ on $\widetilde{Y}^{DP}_U(G)$, defined by $\theta_\mu(A,\iota,\lambda,\eta)=(A,\iota,\mu \lambda,\eta)$ for $\mu \in \ccO_{F,(p),+}^\times$. This action extends naturally to an action of $\ccO^\times_{F,(p),+}$ on $\widetilde{Y}^{PR}_U(G)$, by ignoring the filtration, and the forgetful morphism is equivariant with respect to the action. It is straightforward to see that $(U \cap \ccO_F^\times)^2$ acts trivially, and \cite[Lemma 2.4.1]{DS17} shows that the resulting action of $\ccO_{F,(p),+}^\times/(U \cap \ccO_{F,(p)}^\times)^2$ is free. We thus obtain quotients $\pi: Y^{\textrm{PR}}_U(G) \to Y^{\textrm{DP}}_U(G)$ where $Y^{\textrm{PR}}_U(G)$ is a smooth quasi-projective scheme over $\ccO$ of relative dimension $d$, and $\pi$ projective, inducing an isomorphism on generic fibers.\medskip

We define an action of $G(\bA^{(p)}_f)$ on $\varprojlim_U Y^{\textrm{PR}}_U(G)$, by which we simply mean an action on the projective system, as follows: Let $U_1,U_2$ be two sufficiently small open compact subgroups as above, and let $g \in G(\bA^{(p)}_f)$ be such that $g^{-1}U_1g \subset U_2$. Let $\underline{A_1} = (A_1,\iota_1,\lambda_1,\eta_1,\underline{\omega_1})$ denote the universal abelian scheme over $S=\widetilde{Y}^{\textrm{PR}}_{U_1}(G)$. 
Consider the abelian scheme $A'/S$ which is prime to $p$ isogeneous to $A_1$ and satisfies 
\[T^{(p)}(A'_{\overline{s}_i})= \eta_{1,i}((\ccO_F^{(p)})^2g^{-1})\]
for all $i$ indexing the connected components of $S$. 
$A'$ then inherits an $\ccO_F$-action $\iota'$ from the canonical quasi-isogeny $\pi \in \Hom(A_1,A')\otimes \bZ_{(p)}$.
In particular, this induces an $\ccO_F$-linear isomorphism $\omega_{A'/S}\to \omega_{A_1/S}$ and so a suitable filtration $\underline{\omega^\prime}$ on $\omega_{A'/S}$. 
Furthermore, this yields a prime to $p$ quasi-polarization $\lambda'= \lambda_1=\pi^\vee \circ \lambda' \circ \pi$ on $A^\prime$ which induces the same Rosati involution as $\lambda_1$. We also obtain a $U_2$ level structure by setting $\eta'= \eta_1 \circ r_{g^{-1}}$ where $r_{g^{-1}}$ denotes right multiplication by $g^{-1}$. The tuple $(A',\iota',\lambda',\eta',\underline{\omega}')$ is then an object over $\widetilde{Y}^{\textrm{PR}}_{U_2}(G)$, and thus defines morphism $\tilde{\rho}_g:\widetilde{Y}_{U_1}(G) \to \widetilde{Y}^{\textrm{PR}}_{U_2}(G)$. Since this morphism commutes with the action of $\ccO_{F,(p),+}^\times$ it descends to a morphism $\rho_g: Y^{\textrm{PR}}_{U_1}(G) \to Y^{\textrm{PR}}_{U_2}(G)$ which is finite and \'{e}tale. Furthermore, it is straightforward to check that this defines an action of $G(\bA_f)$ on $\varprojlim Y^{\textrm{PR}}_U(G)$.\smallskip

From this, one finds that $Y^{\textrm{PR}}_U(G)$ defines an integral model over $\ccO$ of the Hilbert Modular variety with level $U$, and hence the $Y^{\textrm{PR}}_U(G)$ define a system of local models. Furthermore, if $p$ is unramified in $F$, then on finds, as in \cite[$\mathsection$ 2.1.3]{2020arXiv200100530D}, that these are a system of integral canonical models in the sense of \cite{Kisin2010}.

\subsection{Unitary Shimura varieties}\label{Unitary construction}
In this section we construct integral models for Unitary Shimura varieties. We first define "naive" models following \cite[$\mathsection$2.2]{2020arXiv200100530D} which we show are singular if and only if$p$ ramifies in $F$. We then introduce the so called Pappas-Rapoport models as a way to resolve their singularities. 
\subsubsection{Shimura data}\label{unitary notation}
We keep the notation from section \ref{hilbert notation}. In particular, we have a fixed totally real field $F$ and fixed rational prime $p$ which we allow to ramify in $F$. 
Let $\Sigma$ be any set of places of $F$, of even cardinality, and which does not contain any place above $p$. For such a $\Sigma$, we write $\Sigma_\infty$ for the archimedean places of $\Sigma$, which we view as a subset of $\Theta_F$. Let $B = B_\Sigma$ denote the quaternion algebra over $F$ ramified at the places in $\Sigma$ and split everywhere else. In particular, $B$ is split at all places over $p$, and $B = M_2(F)$ if $\Sigma = \emptyset$. Denote by $G_\Sigma$ the algebraic group over $\bQ$ given by $G_\Sigma(R) = (B \otimes R)^\times$. For any $\beta \in \Theta_F$ such that $\beta \notin \Sigma_\infty$, we fix an isomorphism $B \otimes_{F,\beta} \bR \simeq M_2(\bR)$ and define $h_\Sigma: \mathbb{S} \to G_{\Sigma,\bR}$ by $x+iy \mapsto \left( \begin{smallmatrix} x & y \\ -y & x \end{smallmatrix}\right)_{\beta \notin \Sigma_\infty}$. In section \ref{quaternionic models} we will define smooth integral models of the quaternionic Shimura varieties attached to the Shimura datum $(G_\Sigma,[h_\Sigma])$, by associating them to certain unitary groups which we now define. 

\label{CM definition}
For a given $\Sigma$, fix a quadratic CM extension $E/F$ such that every prime $\fp \vert p$ in $F$ splits and such that every prime $\mathfrak{l}$ corresponding to a place of $\Sigma$ is inert. In particular, this choice ensures that there exists an isomorphism $B \otimes_F E \simeq M_2(E)$. We write $c$ for the nontrivial element of $\textrm{Gal}(E/F)$ and write $\Theta_E= \{  E \hookrightarrow \overline{\bQ}\}$ which we identify as in \ref{hilbert notation} with $\Theta_{E,\infty}= \{ E \hookrightarrow \bC \}$ and $\Theta_{E,p}= \{ E \hookrightarrow \overline{\bQ}_p\}$. For each prime $\fq \vert p$ we similarly define the sets $\Theta_{E,\fq}$ and $\widehat{\Theta}_{E,\fq}$. If $\fp \vert p$ is a prime of $F$ with $\fp \ccO_E = \fq \fq^c$, we write $\widehat{\Theta}_{E,\fp} = \widehat{\Theta}_{E,\fq} \sqcup \widehat{\Theta}_{E,\fq^c}$.

From now on, we fix for each prime $\fp \vert p$ in $F$, a prime $\mathfrak{q}\vert \fp$ in $E$. We write 
\[\widetilde{\Theta}_E=\bigsqcup_i \{\beta: E \to \overline{\bQ}_p \, \vert \, \tau^{-1}(\mathfrak{m}_p) = \mathfrak{q}\},\]
where $\mathfrak{m}_p$ is the maximal ideal of $\overline{\bZ}_p$. We note that $\widetilde{\Theta}_E$ is in bijection with $\Theta_F$ via $\beta \mapsto \beta \vert_F$. 
For each embedding $\theta^\ell \in \Theta_F$ we  write $\tilde{\theta}^\ell$ for its lift in $\widetilde{\Theta}_E$. 
From now on, $\tau$ will always denote an element of $\widehat{\Theta}_E$ and we will write $\tau \vert_F$ for its image in $\widehat{\Theta}_F$. 
For any $1 \leq \ell \leq e_\fp$, $\tau^\ell$ will denote the element $\Theta_E$ that lifts $\theta^\ell$, where $\tau \vert_F = \theta$ and write $\tau^\ell \vert_F = \theta^\ell$. 
Finally, a general element of $\Theta_E$ will be denoted by $\beta$. It will generally be clear from context whether we are considering $\beta \in \Theta_E$ or $\beta \in \Theta_F$; we shall clarify otherwise.

Let $\Sigma$ be as above and let $\Tilde{\Sigma}_\infty \subset \Theta_{E}$ be any lift of $\Sigma_\infty$. We stress that we do not assume that $\widetilde{\Sigma}_\infty \subset \widetilde{\Theta}_E$. For $\beta \in \Theta_E$, we define $s_\beta$ by
\[s_\beta = \begin{cases} 0 \,\textrm{  if } \beta \in \widetilde{\Sigma}_\infty \\ 1 \, \textrm{  if } \beta\vert_F \notin \Sigma_\infty\\ 2 \, \textrm{  if }  \beta^c \in \widetilde{\Sigma}_\infty. \end{cases}\]
In particular, $s_\beta + s_{\beta^c} = 2$ for all $\beta \in \Theta_E$.

Let $T_F$ denote the torus $\textrm{Res}_{F/\bQ}\mathbb{G}_m$ and $T_E$ the torus $\textrm{Res}_{E/\bQ}\mathbb{G}_m$. We define $G'_\Sigma$ as the quotient $(G_\Sigma \times T_E) / T_F$, where $T_F$ is embedded into $(G_\Sigma \times T_E)$ via $z \mapsto (z,z^{-1})$. 
Let $i_{\widetilde{\Sigma}_\infty}: \mathbb{S} \to T_{E,\bR}$ be the homomorphism given on $\bR$-points by 
\[
\bC^\times \to \prod_{\beta \in \Sigma_\infty} \bC^\times \simeq \prod_{\beta \in \Sigma_\infty} (E \otimes_{F,\beta} \bR)^\times \hookrightarrow \prod_{\beta \in \Theta_E} (E \otimes_{F,\beta}\bR)^\times,\] where the first map is the diagonal map, the second is given on $\beta$-components by $\gamma^c$ where $ \gamma \in \widetilde{\Sigma}_\infty$ and $\gamma \vert_F = \beta$, and the third is the inclusion map. Finally, we define $h'_{\widetilde{\Sigma}}: \mathbb{S}\to G'_{\Sigma,\bR}$ to be the composite of $h_\Sigma \times i_{\widetilde{\Sigma}_\infty}$ with the projection map $(G_\Sigma \times T_E)_\bR \to G'_{\Sigma,\bR}$. We let $L_\Sigma$ and $L_{\widetilde{\Sigma}}$ denote the reflex field of $(G_\Sigma,h_\Sigma)$ and $(G'_\Sigma,h'_{\Tilde{\Sigma}})$ respectively.

\subsubsection{Naive models}\label{unitary moduli}
In this section we define a PEL moduli problem  as in \cite{2020arXiv200100530D} which will give us "naive" models of the Shimura varieties attached to the Shimura datum $(G^\prime_\Sigma,[h^\prime_{\widetilde{\Sigma}_\infty}])$.

Keep the same notation as above and write $D=D_\Sigma$ for $ E \otimes_F B_\Sigma$, with the anti-involution $u \mapsto \overline{u} = (c \otimes \iota)(u)$, where $\iota$ is the standard anti-involution on $B$. We can then identify $G'_\Sigma$ with 
\[G'_\Sigma(R) = \{ g \in (D\otimes R) \, \vert \, g\overline{g} \in (F\otimes R)^\times\}.\]

Recall that our choice of extension $E/F$ ensures that $D$ is isomorphic to $M_2(E)$ $E/F$.
Fix an order $\ccO_D$ of $D$ such that $\ccO_{D,p} = \ccO_E \otimes_{\ccO_F} \ccO_{B,p}$ for some maximal order $\ccO_{B,p}$ of $B_p$, and fix once and for all an isomorphism $\ccO_{D,p} \simeq M_2(\ccO_{E,p})$. We choose an element $\delta \in D^\times$ such that: 
\begin{itemize}
    \item $\delta \in \ccO_{D,p}^\times$ and $\overline{\delta}=-\delta$,
    \item The pairing on $D \otimes \bR$ defined by \[(v,w) \mapsto \Tr_{E/\bQ}(\Tr_{D/E}(v\overline{h}^\prime_{\widetilde{\Sigma}}(i) \overline{w} \delta))\] is positive definite.
\end{itemize}

\noindent We define an anti-involution $u \mapsto u^*$ on $D$ by $u^*=\delta^{-1}\overline{u}\delta$. We also define a pairing $\psi_E: D \times D \to E$ by
\[\psi_E(v,w)= \Tr_{D/E}(v\overline{w}\delta) = \Tr_{D/E}(v\delta w^*).\]
Then $\psi_E(u,v) = - \psi_E(v,w)^c$ and $\psi(uv,w) = \psi_E(v,u^*w)$ for all $u,v,w \in D$. We thus set $\psi_F = \Tr_{E/F} \circ \psi_E$; it is alternating, satisfies $\psi_F(uv,w) = \psi_F(v,u^*w)$, and is thus $F$-bilinear.

Now, let $U' \subset G'_\Sigma(\bA_{F,f})$ be a sufficiently small open compact subgroup which contains the image $U'_p$ of $\ccO_{B,p}^\times \times \ccO^\times_F$ under the natural map $G_\Sigma \times T_E \to G'_\Sigma$. Write $U^\prime = (U^\prime)^p U^\prime_p$ where $(U^\prime)^p \subset G_\Sigma^\prime(\mathbb{A}_f^{(p)})$.

We now define the "naive" models of Shimura varieties attached to $(G^\prime_\Sigma,[h^\prime_{\widetilde{\Sigma}}])$. Since we will need to extend scalars to define the Pappas-Rapoport models in section \ref{PR unitary}, we do not worry about minimal fields of definition. In particular, let $L \subset \overline{\bQ}$ be a sufficiently large number field which contains all images of the embeddings $E \hookrightarrow \overline{\bQ}$ (it thus contains the reflex field $L_{\widetilde{\Sigma}}$ of the Shimura datum $(G^\prime_\Sigma,[h^\prime_{\widetilde{\Sigma}}])$). Write $\ccO$ for the ring of integers of the completion of $L$ at the prime determined by the embedding $\overline{\bQ} \to \overline{\bQ}_p$.

Consider the functor which associates to every locally noetherian $\ccO$-scheme $S$, the set of isomorphism classes of tuples $\underline{A}=(A,\iota,\lambda,(\eta,\epsilon))$ where

\begin{itemize}
    \item $A/S$ is an abelian scheme of dimension $4d$,
    \item $\iota: \ccO_D \hookrightarrow \End_S(A)$ is a ring embedding such that for any $\alpha \in \ccO_E$, the characteristic and minimal polynomial of $\iota(\alpha)$ on $\textrm{Lie}(A/S)$ are given by 
    \[\prod_{\beta \in \Theta_E}(x-\beta(\alpha))^{2s_\beta} \in \ccO[x],\prod_{s_{\beta} \neq 0}(x-\beta(\alpha)) \in \ccO[x]\]
    respectively,
    \item $\lambda$ is a prime to $p$ quasi-polarization whose associated Rosati involution is compatible with the involution $u \mapsto u^*$ on $D$,
    \item $(\eta,\epsilon)$ is a $(U^\prime)^p$ level structure on $A$. That is, for a choice of geometric point $\overline{s}_i$ for each connected component $S_i$ of $S$, a $\pi_1(S_i,\overline{s}_i)$ invariant $(U^\prime)^p$ orbit of $\widehat{\ccO}_D^{(p)}:=\ccO_D \otimes \widehat{\bZ}^{(p)}$-linear isomorphisms \[\eta_i: \widehat{\ccO}_D^{(p)} \times \widehat{\ccO}_D^{(p)} \to T^{(p)}A_{\overline{s}_i},\] and $\epsilon_i: \bA^{(p)}_{F,f} \to \bA^{(p)}_{F,f}(1)$ is an $\bA^{(p)}_{F,f}$-linear isomorphism such that the following diagram:
    
\[\begin{tikzcd}
(\widehat{\ccO}_D^{(p)})^2 \times (\widehat{\ccO}_D^{(p)})^2 \arrow[dd, "{(\eta_i,\eta_i)}"'] \arrow[rr, "\psi_F"] &  & {\bA_{F,f}^{(p)}} \arrow[d, "\epsilon_i"]    \\
                                                                                                                         &  & {\bA_{F,f}^{(p)}(1)} \arrow[d, "\Tr_{F/\bQ}"] \\
T^{(p)}A_{\overline{s}_i}\times T^{(p)}A_{\overline{s}_i} \arrow[rr, "e^\lambda"]                                        &  & \bA^{(p)}_f(1),                              
\end{tikzcd}\]
where $e^\lambda$ denotes the Weil pairing induced by the quasi-polarization $\lambda$, commutes.
\end{itemize}

Note that in the case that $p$ is unramified in $F$, the condition on the minimal polynomial of the action of $\alpha$ is redundant as it is implied by the shape of the characteristic polynomial. This is no longer true if we do allow $p$ to ramify however.

As in \cite[$\mathsection$2.2.2]{2020arXiv200100530D}, this moduli problem is representable by a scheme $\widetilde{Y}^{\textrm{DP}}_{U'}(G_\Sigma')/\ccO$ which is an infinite disjoint union of PEL Shimura varieties attached to the group $G'_\Sigma \times_{\nu,\textrm{Res}_{E/\bQ}\mathbb{G}_m}\mathbb{G}_m$, indexed over $\epsilon \in \bA^{(p)}_{F,f}/(\nu((U^\prime)^p)\widehat{\bZ}^{(p)})$ where $\nu:G^\prime_\Sigma \to T_F$ is given by $g \mapsto g\overline{g}$. It is thus quasi-projective and hence locally of finite type over $\ccO$. Additionally, it is smooth if $p$ is unramified in $F$. Otherwise, it has smooth generic fiber and singular special fiber with codimension 2 singular locus.

\subsubsection{Singularities of the naive model}
In this section we study the singularities of the "naive" models $S = \widetilde{Y}^{DP}_{U'}(G'_\Sigma)$ introduced above, in the case that $p$ is ramified in $F$. Adapting \cite{deligne1994singularites}, in the context of Hilbert modular varieties, we introduce local models as follows: 

Recall that we fixed an isomorphism $\ccO_{D,p} \simeq M_2(\ccO_{E_p})$. For any locally free sheaf $\ccF$ over an $\ccO$-scheme $S$ with a right $\ccO_S$-linear action of $\ccO_D$, the induced decomposition 
\[\ccO_S \otimes \ccO_D \simeq M_2(\ccO_S \otimes \ccO_E) \simeq \bigoplus_{\tau \in \widehat{\Theta}_E} M_2(\ccO_S[u]/E_{\tau}(x))\]
where we recall that $E_\fp(u)$ is the minimal polynomial of the fixed uniformizer $\varpi_\fp$ of $F_\fp$ and $E_\tau(u) = \tau(E_\fp(u))= \prod_i (u-\tau^i(\varpi_\fp)) \in \ccO[u]$, gives a decomposition 
\[\ccF = \bigoplus_{\tau \in \widehat{\Theta}_E} \ccF_{\tau}\]
where each $\ccF_{\tau}$ is a $M_2\left(\ccO_S[u]/E_{\tau}(u)\right)$-module such that the action of $\ccO_E$ factors through $\tau$. We also define 
\begin{equation}\label{reduced sheaf}
\ccF^0_\tau = \ccF_{\tau} \cdot e_0,
\end{equation}
where $e_0 \in M_2(\ccO_S[u]/E_{\tau}(u))$ denotes the matrix $\left(\begin{smallmatrix} 1 & 0 \\ 0 & 0 \end{smallmatrix}\right)$. We note that, by Morita equivalence, $\ccF_{\tau} \simeq \ccF_{\tau}^0 \oplus \ccF_{\tau}^0$ with their respective module structures. We note that we can also consider the subsheaf $\ccF^0 = \ccF \cdot e_0 $ where $e_0$ is now considered as an element of $M_2(\ccO_S \otimes \ccO_E)$ and its decomposition $\ccF^0 = \bigoplus (\ccF^0)_\tau$ under the action of $\ccO_E$. Then $\ccF^0_\tau = (\ccF^0)_\tau$.

Let $A/S$ denote the universal abelian variety over $S$. Consider the Hodge bundle $e^*\Omega^1_{A/S} \subset \ccH^1_{\dR}(A/S)$, where $e:S \to A$ is the identity section. Since $\ccH^1_{\dR}(A/S)$ is a locally free $\ccO_S \otimes \ccO_D$-module of rank two, consider for any closed point $x$ of $S$ an open neighborhood $V$ of $x$ over which the restriction $ \ccH = \ccH^1_{\dR}(A/S)\vert_V$ is free as an $\ccO_V \otimes \ccO_D$-module. Write $\omega = e^*\Omega^1_{A/S} \vert_V$. For any $\tau \in \widehat{\Theta}_E$, $\ccH^0_\tau$ is thus a free $\ccO_V[u]/E_\tau(u)$-module of rank $2$ and $\omega^0_\tau$ is a locally free $O_V$-module of rank $\sum_i s_{\tau^i}$. In particular, the characteristic polynomial of the action of $\alpha \in \ccO_E$ on $\omega^0_\tau$ is given by $\prod_i (x - \tau^i(\alpha))^{s_{\tau^i}}$ and its minimal polynomial by $\prod_{s_{\tau^i} \neq 0} (x - \tau^i(\alpha))$. Define the polynomials in $\ccO[u]$: $g_\tau(u) = \prod_{s_{\tau^i} \neq 2} (u- \tau^i(\varpi_\fp))$, $h_\tau(u) = \prod_{s_{\tau^i} = 0} (u- \tau^i(\varpi_\fp))$ and $f_\tau(u) = \prod_{s_{\tau^i}=1}(u-\tau^i(\varpi_\fp))$ so that $f_\tau(u)g_\tau(u) = h_\tau(u)$. It follows from the minimal polynomial that $\omega^0_\tau \subset h_\tau(u)\ccH^0_\tau$. Furthermore the compatibility between the $\ccO_D$-action and the prime-to-$p$ quasi-polarization $\lambda: A \to A^\vee$ implies that the induced pairing on $\ccH^1_{\dR}(A/S)$ restricts to a perfect pairing on $\ccH^0_\tau \times \ccH^0_{\tau^c}$ under which the orthogonal complement of $\omega^0_\tau$ is given by $\omega^0_{\tau^c}$. Furthermore, the orthogonal complement of $h_{\tau^c}(u)\ccH^0_{\tau^c}$ under this pairing is $g_\tau(u)\ccH^0_\tau$ from which we deduce the inclusion $g_\tau(u)\ccH^0_\tau \subset \omega^0_\tau$. 

It follows that $h_\tau(u)\ccH^0_\tau/g_\tau(u)\ccH^0_\tau$ is a free 
\[h_\tau(u)(\ccO_V[u]/E_\tau(u))/ g_\tau(u)(\ccO_V[u]/E_\tau(u)) \simeq \ccO_V[u]/f_\tau(u)\]
module of rank two. Furthermore the image of $\omega^0_\tau$ is an $\ccO_V[u]/f_\tau(u)$-stable $\ccO_V$-summand such that the characteristic polynomial of the action of $\alpha \in \ccO_E$ is given by $\prod_{s_{\tau^i}=1}(x-\tau^i(\alpha))$. Let $M_\tau$ denote the Grassmannian over $\ccO$ that classifies such objects. 

We thus obtain a morphism $V \to M = \prod_{\theta \in \widehat{\Theta}_F} M_{\tilde{\theta}}$ by considering for each $\theta$ the image of $\omega^0_{\tilde{\theta}}$ in $h_{\tilde{\theta}}(u)\ccH^0_{\tilde{\theta}}/g_{\tilde{\theta}}(u)\ccH^0_{\tilde{\theta}}$.
A boot-strapping argument as in \cite{deligne1994singularites} furnishes an isomorphism of completed local rings at $x$ and so the morphism is \'{e}tale at $x$. Therefore, to study the singularities of $S$ it suffices to study the singularities of $M$.

For a fixed $\tau = \tilde{\theta}$ consider the special fiber of $\overline{M}_\tau$ of $M_\tau$.
Let $\overline{y}: T=\textrm{Spec}(k) \to \overline{M}_\tau$ be a geometric point corresponding to $\ccF \subset (\ccO_T[u]/f_\tau(u))^2$. Then $ R = (\ccO_T[u]/f_\tau(u))^2 \simeq (k[u]/u^d)^2$ (where $d$ is the degree of $f_\tau$) and there is a basis $e_1,e_2$ of $R$ such that $\ccF$ is generated by $u^i e_1$ and $u^j e_2$ for $i+j = d$ and $0 \leq i \leq j \leq d$.

We define a stratification on $\overline{M}_\tau$ based on the value of $i$ as described above. Consider the closed subfunctors $G_n$ which classify the $\ccF \subset (\ccO_T[u]/f_\tau(u))^2=R_T$ such that $ u^{d-n}R_T \subset \ccF \subset u^{n} R_T$  for $0 \leq n \leq d/2$. Then the complement $G^0_n$ of $G_{n+1}$ in $G_n$ classifies the $\ccF$ such that $ u^{d-n }R_T \subset \ccF \subset u^{n} R_T$ and that the image of  $\ccF$ in $u^{n} R_T / u^{d-n}R_T$ is locally a direct factor of $u^{n} R_T/u^{d-n}R_T$ as a $R_T/u^{d-2n}$-module. This is represented by the Weil restriction of $\mathbb{P}^1$ over $\textrm{Spec}(k[u]/u^{d-2n})$ to $k$ and is thus smooth of dimension $d -2n$. It follows that $\overline{M}_\tau$ is smooth outside of a codimension two locus. In particular, the smooth open subscheme $G^0_0$ of $\overline{M}^\prime_\tau$ consists of those $\ccF$ such that locally everywhere, there is a basis $e_1,e_2$ of $R_T$ such that $\ccF$ is generated by $e_1$. Translating this back to $\overline{M}_\tau$ this corresponds to submodules such that locally everywhere, there is a basis $e_1,e_2$ of $R_T$ such that $\ccF$ is generated by $u^{b_\tau}e_1,u^{e_\fp -d_\tau}e_2$ where $b_\tau$ is the number of $i$ such that $s_{\tau^i}=0$ and $d_\tau$ is the number of $i$ such that $s_{\tau^i}=2$. This corresponds to the usual Rapoport locus in the case $\Sigma = \varnothing$.\medskip

We now give an explicit proof that $G^0_0$ is precisely the whole smooth locus. This can safely be skipped by the uninterested reader. Take a point $p$ in the complement of $G^0_0$. Then $\ccF$ is generated by $u^i e_1$ and $u^j e_2$ where $i+j =d$ and $1 \leq i \leq j <e$. We define a lift of $\tilde{p}$ of $p$ to $k[y]/y^2$ which cannot be lifted to a $k[y]/y^3$-point.\medskip

Keep the notation above and consider the submodule $\ccF' \subset (\ccO_F \otimes k[y]/y^2)^2$ generated over $\ccO_F \otimes k[y]/y^2$ by $u^i e_1 + ye_2$ and $u^j e_2 + ye_1$. That is, $\ccF$ is generated over $k[y]/y^2$ by the $u^a(u^i e_1 + ye_2)$ and $u^b(u^j e_2 + ye_1)$ for $0\leq a < j-1$ and $0\leq b <i-1$. These are clearly linearly independent over $k[y]/y^2$, and can be extended to a basis of $(\ccO_F \otimes k[y]/y^2)^2$ by adding the $e_1, \cdots u^{i-1} e_1, e_2, \cdots, u^{j-1}e_2$ and $\ccF'$ satisfies the Kottwitz condition. Therefore $\ccF'$ defines a point $\tilde{p}$ lifting $p$.\medskip

We now show that $\tilde{p}$ cannot lift to a $k[y]/y^3$ point. Suppose it does. Then pick any lifts  $u^i e_1 + ye_2 + v$ and $u^j e_1 +yu^{j-i}e_2 + w$, of the corresponding elements downstairs, in $\ccF''$. Then $u^a(u^i e_1 + ye_2 +v)$ and $u^b(u^j e_2 + ye_1+w)$ for $0\leq a,b \leq e$ must generate $\ccF''$ over $k[y]/y^3$. We now show that $\ccF''$ is not $k[y]/y^3$-projective by computing the Smith normal form of the matrix given by the above generators and the ordered basis $e_1, u e_1, \cdots , u^{d-1} e_1 , e_2, u e_2, \cdots , u^{d-1} e_2$.\medskip

Notice that every coefficient of $v$ and $w$ is in $y^2k[y]/y^3$ by definition. Let $\alpha$ be the $u^i e_1 + ye_2$ coefficient of $v$, then by multiplying $u^i e_1 + ye_2 +v$ by $(1-\alpha)$ we can assume wlog that $\alpha =0$. Similarly, we can assume that the coefficient of $u^j e_2$ of $w$ is 0. Now, the matrix determined by the generators can be decomposed into block matrices: $$ R = \begin{pmatrix}
A & B \\
C & D 
\end{pmatrix}$$

where:
\begin{itemize}
    \item each entry $\alpha_{k \ell}$ of $A$ is the $u^{\ell-1} e_1$ coefficient of $u^{k-1}(u^i e_1 + ye_2 +v)$,
    \item each entry $\alpha_{k \ell}$ of $B$ is the $u^{\ell-1} e_2$ coefficient of $u^{k-1}(u^i e_1 + ye_2 +v)$,
    \item each entry $\alpha_{k \ell}$ of $C$ is the $u^{\ell-1} e_1$ coefficient of $u^{k-1}(u^j e_2 + ye_1+w)$,
    \item each entry $\alpha_{k \ell}$ of $D$ is the $u^{\ell-1} e_2$ coefficient of $u^{k-1}(u^j e_2 + ye_1+w)$.
\end{itemize}

In particular, each matrix is upper triangular. Since we assumed that the $u^i e_1 + ye_2$ coefficient of $v$ is $0$, then the entries $\alpha_{k(k+i)}$ of $A$, for $ 1 \leq k \leq e-i = j$ are equal to $1$. Similarly the entries  $\alpha_{k(k+j)}$ of $D$, for $ 1 \leq k \leq e-j = i$ are equal to $1$. Furthermore, the diagonal entries of $B$ and $C$ are of the form $y+\beta$ and $y+\gamma$ for some $\beta,\gamma \in y^2k[y]/y^3$. Every other entry of the whole matrix lies in $y^2k[y]/y^3$.\medskip

We now reduce the matrix: Take an entry $\alpha_{k \ell}$ of $A$ with $k \leq \ell$ and $\ell \neq k+i$.  If $k \leq j $, then subtract $\alpha_{k \ell}$ times column $k+i$ of $R$ from column $\ell$. Since $\alpha_{k\ell}$ is divisible by $y^2$ and every entry of column $\ell$ is divisible by $y$ except for $\alpha_{k(k+i)}=1$, $\alpha_{k \ell}$ times the column has every entry $0$ except for the $k$-th entry equal to $\alpha_{k \ell}$. This column operation thus only has the effect of deleting the entry $\alpha_{k \ell}$ from $A$.\smallskip

Similarly, if $\ell \geq k > j $, we subtract $\alpha_{k\ell}$ times row $\ell-i$ from row $k$, this has the only effect of deleting the entry $\alpha_{k\ell}$. We have now reduced $A$ to an off diagonal matrix with entries $1$ since $A$ was upper diagonal.\medskip

A similar analysis, by also making use of the off diagonal of $1$s in $D$, shows that we can reduce $B$ and $C$ to the form $(y+\beta)I_e$ and  $(y+\gamma)I_e$ respectively. Furthermore, making use of the off diagonal of $1$s in $D$ and the first $i$ columns of $C$ and last $j$ rows of $B$, we can also reduce $D$ to an off diagonal matrix of $1$s. The resulting matrix is of the form: $$
\begin{pmatrix}
I_{e,i-1} & (y+\beta)I_e \\ (y+\gamma)I_e & I_{e,j-1}
\end{pmatrix}$$

where $I_{e,n}$ denotes the $e \times e$ matrix with entries $a_{k \ell}=1$ if $\ell = k +n$ and $0$ otherwise.\medskip

To finish the reduction, we start by subtracting to each column (of the full matrix) $k$, for $k \leq i$ the $(y+\gamma)$ times the column $e+k+j$ and then subtract to each row $e+k$, for $k > i$, $(y+\gamma)$ times the row $i$. The resulting matrix has the form: $$
\begin{pmatrix}
I_{e,i-1}-y^2 I_{e,1-j} & (y+\beta)I_e \\ 0 & I_{e,j-1} -y^2 I_{e,1-i}
\end{pmatrix}.$$
Finally, it is straightforward to delete $B$ and the final reduced matrix is of the form $$\begin{pmatrix}
I_e & 0 \\ 0 & y^2 I_{e}
\end{pmatrix}.$$

Thus $\ccF''$ is clearly not projective as a $k[y]/y^3$-module and thus does not define a point of $M$ so that $\ccF'$ cannot be lifted. This shows that the original point $p$ is singular.\medskip

\subsubsection{Pappas-Rapoport models of Unitary Shimura varieties}\label{PR unitary}
We now introduce smooth models of Unitary Shimura varieties that resolve the singularities of the "naive" models  $\widetilde{Y}^{\textrm{DP}}_{U^\prime}(G^\prime_\Sigma).$ 

Consider the functor which associates to every locally noetherian $\ccO$-scheme $S$, the set of isomorphism classes of tuples $\underline{A}=(A,\iota,\lambda,(\eta,\epsilon), \underline{\omega})$ where the first five elements give a point of $\widetilde{Y}^{\textrm{DP}}_{U'}(G_\Sigma')(S)$, and $\underline{\omega}$ is the data of, for each $\theta \in \widehat{\Theta}_F$ and $\tau = \tilde{\theta}$, a full filtration of $\omega^0_{A/S,\tau} = \omega^0_{\tau}$ 
\[0 = \omega^0_{\tau}(0) \subset \omega^0_{\tau}(1) \subset \cdots \subset \omega^0_{\tau}(e_\fp-1) \subset \omega^0_{\tau}(e_\fp) = \omega^0_{\tau} \]
such that:
\begin{itemize}
    \item Each $\omega^0_{\tau}(j)$ is $\ccO_E$-stable,
    \item Each subquotient $\omega^0_\tau(j) / \omega^0_\tau(j-1)$ is locally free of rank $s_{\tau^j}$ over $S$ (hence $\omega^0_\tau(j)$ is locally free of rank $s_\tau(j) = \sum_{i \leq j} s_{\tau^i}$),
    \item The induced action of $\ccO_E$ on $\omega^0_\tau(j)/\omega^0_\tau(j-1)$ factors through $\tau^j$. Equivalently, it is killed by $([\varpi_\fp]-\tau^j(\varpi_\fp))$.
\end{itemize}

This functor is representable, as shown in \cite{pappas2005local}, over $\ccO$ by a scheme which we denote $\widetilde{Y}^{\textrm{PR}}_{U'}(G_\Sigma')$: The natural forgetful morphism to $\pi:\widetilde{Y}^{\textrm{PR}}_{U'}(G_\Sigma') \to \widetilde{Y}^{\textrm{DP}}_{U'}(G_\Sigma')$ that forgets the filtration is relatively representable over $\widetilde{Y}^{\textrm{DP}}_{U'}(G_\Sigma')$ and $\widetilde{Y}^{\textrm{PR}}_{U'}(G_\Sigma')$ is thus given as closed subscheme of a product of Grassmannians.

Note that we only require a filtration on $\tilde{\theta}$-components. We could have equivalently defined the Pappas-Rapoport model by requiring filtrations $\tilde{\theta}$ and $\tilde{\theta}^c$ components which are dual in some sense. We explain how to do this in section \ref{dual filtration}.  

It is clear that $\pi$ induces an isomorphism on generic fibers. Furthermore it is an isomorphism over the generalized Rapoport loci of the special fibers of the respective schemes: Recall that this is the locus where, for each $\tau =\tilde{\theta} \in \widehat{\Theta}_E$ there is everywhere locally, an $\ccO_E \otimes \ccO_U$-linear trivialization $\ccH^1_{\dR}(A/U)_\tau^0 \simeq (\ccO_U[x]/E_\tau(x))^2$, such that $\omega^0_{\tau}$ is generated by $x^{b_\tau}e_1$ and $x^{e_\fp-d_\tau}e_2$ where where $b_\tau$ is the number of $i$ such that $s_{\tau^i}=0$ and $d_\tau$ is the number of $i$ such that $s_{\tau^i}=2$.

We show that there is a unique choice of filtration on $\omega^0_{\tau}$ by starting from the top. Write as above $s_\tau(j) = \sum_{i \leq j} s_{\tau^i}$, and let $d_{\tau,j}$ denote the number of $i \leq j$ such that $s_{\tau^i}=2$. By definition of the filtration, for any $j$, $\omega^0_\tau(j)$ contains 
\[\ker \prod_{i \leq j \, , \, s_{\tau^i}=2} (x-\tau^i(\varpi_\fp)) = \ker x^{d_{\tau,j}}=\langle x^{e-d_{\tau,j}}e_1\, , x^{e-d_{\tau,j}}e_2 \rangle,\] 
where the $e_i$ are the same basis as above.
Suppose now that for some $j$ we have $\omega^0_\tau(j)=\langle x^{e-s(j)+d_{\tau,j}}e_1\, , x^{e-d_{\tau,j}}e_2 \rangle $ (this is clearly true for $j=e_\fp$). Suppose that $s_{\tau^j}=0$ (respectively $s_{\tau^j}=2$), then we have $\omega^0_\tau(j-1)=\omega^0_\tau(j)$ (respectively $\omega^0_\tau(j-1)=x\omega^0_\tau(j))$ and the result follows immediately. Suppose otherwise then that $s_{\tau^j}=1$, then $s_\tau(j-1)=s_\tau(j)-1$, $d_{\tau,j-1}=d_{\tau,j}$ and we know that $\omega^0_\tau(j-1)$ contains both $\langle x^{e-s(j)+d_{\tau,j}+1}e_1\, , x^{e-d_{\tau,j}+1}e_2 \rangle$ and $\langle x^{e-s(j)+d_{\tau,j}}e_1\, , x^{e-d_{\tau,j}}e_2 \rangle$. The result follows by dimension count.

We now show that $\widetilde{Y}_{U'}(G_\Sigma')$ is smooth over $\ccO$. We first record the following lemma: 

\begin{lem}\label{filtration torsion}
Let $S$ be scheme over $\ccO$, $\mathcal{G}$ a locally free $\ccO_S \otimes \ccO_E$-module of rank $2$, and let $\ccF \subset \mathcal{G}_\tau$, for some $\tau \in \widehat{\Theta}_E$, be a $\ccO_S$-locally free subsheaf,local summand, which is stable under the action of $\ccO_E$ and killed by $a_{\tau,i-1} = ([\varpi_\fp]-\tau^1(\varpi_\fp)) \cdots ([\varpi_\fp]-\tau^{i-1}(\varpi_\fp))$ for some $1 \leq i \leq e_\fp$. Then the sheaf $([\varpi_\fp]-\tau^i(\varpi_\fp))^{-1} \ccF /\ccF$ is locally free of rank $2$.
\end{lem}

\begin{proof}
This follows similarly to the proof of \cite[Proposition 5.2 b]{pappas2005local}. Write $b_{\tau,i-1} = ([\varpi_\fp]-\tau^i(\varpi_\fp)) \cdots ([\varpi_\fp]-\tau^{e_\fp}(\varpi_\fp))$. Then $\textrm{Im} \, b_{\tau,i-1} = \ker a_{\tau,i-1}$ is locally free of rank $2(e_\fp-i)$ and the morphism $([\varpi_\fp]-\tau^i(\varpi_\fp)): \textrm{Im} \, b_{\tau,i} \to \textrm{Im} \, b_{\tau,i-1}$ is surjective. Since $\ccF \subset \ker a_{\tau,i-1}$, we can consider the short exact sequence 
\[0 \to  \ker ([\varpi_\fp]-\tau^i(\varpi_\fp)) \to ([\varpi_\fp]-\tau^i(\varpi_\fp))^{-1}\ccF \xrightarrow{([\varpi_\fp]-\tau^i(\varpi_\fp))} \ccF \to 0.\] 
This shows that $([\varpi_\fp]-\tau^i(\varpi_\fp))^{-1}\ccF$ is locally free as the other two are. Similarly, the short exact sequence 
\[0 \to \ker a_{\tau,i-1} /\ccF \to\mathcal{G}_\tau/\ccF \to \mathcal{G}_\tau/ \ker a_{\tau,i-1} \to 0\]
shows that $\ker a_{\tau,i-1} /\ccF \to \mathcal{G}_\tau/\ccF$ is locally free. Considering the following short exact sequence then yields the result
\[0 \to ([\varpi_\fp]-\tau^i(\varpi_\fp))^{-1}\ccF/\ccF \to  \ker a_{\tau,i}/\ccF \xrightarrow{([\varpi_\fp]-\tau^i(\varpi_\fp))}\ker a_{\tau,i-1}/\ccF \to 0.\]
To check the rank, it suffices to look at fibers: Let $k$ be a geometric point, if it has characteristic zero, this is clear. Suppose that it has characteristic $p$. Then $\mathcal{G}_{\tau,k} \simeq (k[x]/x^{e_\fp})^2$. For any $\ccF$ as above, we can choose a basis such that $\ccF = x^a k[x]/x^{e_\fp} \oplus x^b k[x]/x^{e_\fp}$ with $a+b = 2e_\fp - \dim \ccF$. Since $\ccF \subset x\ccH^1_{dR}(A/k)^0_{\tau}$, we also see that $a,b \geq 1$ and so $[\varpi_\fp]^{-1}\ccF/\ccF \simeq (x^{a-1} k[x]/x^{e_\fp} \oplus x^{b-1} k[x]/x^{e_\fp}) / (x^a k[x]/x^{e_\fp} \oplus x^b k[x]/x^{e_\fp})$ has dimension 2.
\end{proof}

\begin{prop}\label{unitary smooth}
 The moduli space $\widetilde{Y}^{\textrm{PR}}_{U'} (G'_\Sigma)$ is smooth over $\ccO$ of relative dimension $\vert \Theta_F \setminus \Sigma_\infty \vert$.
\end{prop}

Before proving the proposition, we recall a version of Grothendieck-Messing theory (\cite[V.4]{alma990002183250306161}) which we shall use. Let $S_0$ be a scheme over $\ccO$ and consider the nilpotent crystalline of $S_0$ over $\bZ_p$. Let $f:A_0 \to S_0$ be an abelian scheme and consider the crystalline cohomology sheaf $\ccH^1_{\textrm{cris}}(A_0/S_0) = R^1f_{*,\textrm{cris}}\ccO_{A_0,\textrm{cris}}$. For a nilpotent divided power thickening $S_0 \hookrightarrow S$, there is a canonical isomorphism $\ccH^1_{\textrm{cris}}(A_0/S_0)_S \otimes_{\ccO_{S}} \ccO_{S_0} \simeq \ccH^1_{\dR}(A_0/S_0) $. Under this isomorphism, deformations of $A_0$ to $S$, that is abelian schemes $A/S$ such that $A \times_S S_0 = A_0$, correspond to lifts $\omega_S \subset \ccH^1_{\textrm{cris}}(A_0/S_0)_S$ of the hodge filtration $\omega_{A_0/S_0} \subset \ccH^1_{\dR}(A_0/S_0)$. If $A_0$ comes with an action of a ring $R$, the action extends to a unique action on $R$ on $A$ if and only if $\omega_S$ is $R$-stable. Similarly, if $A_0$ has a polarization $\lambda_0$, it extends to a unique polarization $\lambda$ of $A$ if and only if $\omega_S$ is isotropic with respect to the pairing induced by $\lambda_0$ on $\ccH^1_{\textrm{cris}}(A_0/S_0)_S$.

\begin{proof}
This proof is similar to the proof of Theorem 2.9 of \cite{reduzzi2016partial}. 
 Furthermore, since $\widetilde{Y}^{\textrm{PR}}_{U'}(G_\Sigma')_\bF$ is locally of finite presentation, we only need to show that it is formally smooth.

Let $S_0$ be a scheme over $\ccO$ and $S_0 \hookrightarrow S$ be a closed immersion with square zero ideal of definition $\ccI$. Let $x_0 = (A_0,\iota_0,\lambda_0,\eta_0,\epsilon_0,\underline{\omega}_0) \in \widetilde{Y}^{\textrm{PR}}_{U'}(G_\Sigma')(S_0)$. By Grothendieck-Messing theory, to lift $A_0$ to an abelian scheme $A/S$ with an $\ccO_D$ action and polarization, it suffices to lift $\omega_{A_0/S_0} \subset \ccH^1_{\dR}(A_0/S_0)$ to a $\ccO_D$-stable, maximal isotropic, submodule $\widetilde{\omega} \subset \ccH^1_{\textrm{cris}}(A_0/S_0)_S$.

By \cite{CM_1978__36_3_255_0}, $\ccH^1_{\textrm{cris}}(A_0/S_0)_S$ is locally free of rank two over $\ccO_D \otimes \ccO_S$. In particular, it suffices lift each $\omega_{A_0/S_0,\tau}^0$ to a submodule $\widetilde{\omega}^0_\tau \subset \ccH^1_{\textrm{cris}}(A_0/S_0)_{S_0,\tau}^0$ for $\tau = \tilde{\theta} \in \widehat{\Theta}_F$ as we then take $\tilde{\omega}_{\tau^c}$ to be the orthogonal complement of $\tilde{\omega}_{\tau}$ in $\ccH^1_{\textrm{cris}}(A_0/S_0)_{S,\tau^c}$.

We proceed inductively: For $\theta \in \widehat{\Theta}_F$ write $\tau = \tilde{\theta}$, and write $\omega^0_{\tau}(i) = \omega^0_{A_0/S_0,\tau}(i)$.  Suppose we have fixed a lift $\tilde{\omega}^0_{\tau}(i-1)$ of $\omega^0_{\tau}(i-1)$. If $s_{\tau^i}=0$ (respectively $s_{\tau^i}=2$), then we are forced to take $\tilde{\omega}^0_{\tau}(i)=\tilde{\omega}^0_{\tau}(i-1)$ (respectively $\tilde{\omega}^0_{\tau}(i)= ([\varpi_\fp] -\tau^i(\varpi_\fp))^{-1}\tilde{\omega}^0_{\tau}(i-1)$). Suppose now that $s_{\tau^i}=1$, and define 
\[\widetilde{\ccH}^1_{\tau^i}=([\varpi_\fp] -\tau^i(\varpi_\fp))^{-1}\tilde{\omega}^0_{\tau}(i-1) / \tilde{\omega}^0_{\tau}(i-1).\]
This is a locally free sheaf of rank 2 over $\ccO_S$ by lemma \ref{filtration torsion}.
Furthermore, over $S_0$, we have
\[\widetilde{\ccH}^1_{\tau^i} \otimes_{\ccO_S} \ccO_{S_0} = ([\varpi_\fp] -\tau^i(\varpi_\fp))^{-1}\omega^0_{\tau}(i-1) /\omega^0_{\tau}(i-1)=: \ccH^1_{\tau^i}.\]
It follows that lifting $\omega^0_{\tau}(i)$ is equivalent to lifting 
\[\omega^0_{\tau^i} := \omega^0_{\tau}(i)/\omega^0_{\tau}(i-1) \subset \ccH^1_{\tau^i}\]
to a rank one $\ccO_S$-subbundle of $\widetilde{\ccH}^1_{\tau^i}$. The set of such lifts forms a torsor under
\[
 \mathcal{H}om_{\ccO_{S_0}}\left(\omega^0_{\tau^i}, \ccH^1_{\tau^i} / \omega^0_{\tau^i}\right) \otimes_{\ccO_{S_0}} \ccI \simeq (\omega^0_{\tau^i} )^{\otimes -2} \otimes_{\ccO_{S_0}} \wedge^2 \ccH^1_{\tau^i} \otimes_{\ccO_{S_0}} \ccI.\]
It follows that we can lift $\omega^0_{\tau}$ and, by construction, we have even lifted the Pappas-Rapoport filtration. Noting as in \cite{reduzzi2016partial} that the level structure lifts uniquely, the proposition then follows.
\end{proof}

As in \cite{2020arXiv200100530D}, we define an action of $\ccO_{F,(p),+}^\times$ on both $\widetilde{Y}^{DP}_{U'}(G_\Sigma')$ and $\widetilde{Y}^{\textrm{PR}}_{U'}(G_\Sigma')$ via the map $\theta_\mu(A,\iota,\lambda,\eta,\epsilon,\underline{\omega})=(A,\iota,\mu\lambda,\eta,\mu\epsilon,\underline{\omega})$. This is clearly compatible with the forgetful map. Furthermore, by \cite[Lemma 2.2.1]{2020arXiv200100530D}, the action of $\ccO_{F,(p),+}^\times$ factors through $Nm_{E/F}(U'\cap E^\times)$, and the induced action of $\ccO_{F,(p),+}/Nm_{E/F}(U'\cap E^\times)$ is free. We obtain quotients $\pi: Y^{\textrm{PR}}_{U'}(G_\Sigma')\to Y^{DP}_{U'}(G_\Sigma')$ over $\ccO$, where $Y^{\textrm{PR}}_{U'}(G_\Sigma')$ is an infinite disjoint union of smooth quasi-projective schemes of relative dimension $\vert \Theta_F \setminus \Sigma_\infty \vert$. 

We now define an action of $G'_\Sigma(\bA_f)$ on the projective system $\varprojlim_{U^\prime} Y^{\textrm{PR}}_{U'}(G_\Sigma')$ as follows:
Take two sufficiently small open compact subgroups $U'_1,U'_2 \subset G(\bA^{(p)}_f)$ as above, and $g\in G_\Sigma'(\bA_f)$ such that $g^{-1}U_1g \subset U_2$. Given the universal object $(A_1,\iota_1,\lambda_1,\eta_1,\epsilon_1,\underline{\omega_1})$ over $S = \widetilde{Y}^{\textrm{PR}}_{U'_1}(G'_\Sigma)$, consider the abelian scheme $A'/ S$ which is prime to $p$ isogeneous to $A_1$ and satisfies 
\[T^{(p)}(A'_{\overline{s}_i})= \eta_{1,i}(\ccO_D^{(p)}g^{-1}),\]
for all $i$, indexing over connected components of $S$. Then $A'$ inherits an $\ccO_D$-action $\iota'$ via the canonical quasi-isogeny $\pi\in\textrm{Hom}(A_1,A')\otimes \bZ_{(p)}$. Furthermore $\pi$ induces a $\ccO_D$-linear isomorphism $\omega_{A'/S} \to \omega_{A_1/S}$ and so a suitable filtration on $\underline{\omega^\prime}$ on $\omega_{A^\prime/S}$. We define a polarization $\lambda'$ on $A'$ by $\lambda_1=\pi^\vee \circ \lambda' \circ \pi$ and a $U_2$ level structure by $\eta'= \eta_1 \circ r_{g^{-1}}$ where $r_{g^{-1}}$ denotes right multiplication by $g^{-1}$, and $\epsilon'=\nu(g)^{-1}\epsilon_1$. The tuple $(A',\iota',\lambda',\eta',\underline{\omega}')$ is then an object classified by $\widetilde{Y}^{\textrm{PR}}_{U'_2}(G'_\Sigma)$ and thus defines a morphism $\tilde{\rho}_g:\widetilde{Y}^{\textrm{PR}}_{U'_1}(G'_\Sigma) \to \widetilde{Y}^{\textrm{PR}}_{U'_2}(G'_\Sigma)$. Since this morphism commutes with the action of $\ccO_{F,(p),+}$, it descends to a morphism $\rho_g: Y^{\textrm{PR}}_{U'_1}(G'_\Sigma) \to Y^{\textrm{PR}}_{U'_2}(G'_\Sigma)$ which is finite and \'{e}tale. Furthermore, it is straightforward to check that this defines an action action of $G(\bA_f)$ on the projective system $\varprojlim Y^{\textrm{PR}}_{U'}(G'_\Sigma)$.\smallskip

From this, as in the Hilbert case, it can be checked  that the $Y_{U'}(G'_\Sigma)$ form a system of smooth integral models over $\ccO$ for the Shimura variety whose complex points are given by 
\[Y_{U'}(G'_\Sigma)(\bC)= G'_\Sigma(\bQ) \backslash( (\bC-\bR)^{\Theta-\Sigma} \times G'_\Sigma(\bA_f))/ U' = G'_\Sigma(\bQ)_+ \backslash( \ccH^{\Theta-\Sigma} \times G'_\Sigma(\bA_f))/ U',\]
where $G'_\Sigma(\bQ)_+$ denotes the subgroup of elements $g$ such that $\nu(g)$ is totally positive.

Although this will follow from the results of section \ref{Go properties}, we now show that the generalized Rapoport locus is dense in $\widetilde{Y}_{U^\prime}(G^\prime_\Sigma)$. Recall that this is the locus such that, for every $\tau =\tilde{\theta} \in \widehat{\Theta}_E$, there is everywhere locally an $\ccO_E \otimes \ccO_U$-linear trivialization $\ccH^1_{\dR}(A/U)_\tau^0 \simeq (\ccO_U[x]/E_\tau(x))^2$, $\omega^0_{\tau}$ is generated by $x^{b_\tau}e_1$ and $x^{e_\fp-d_\tau}e_2$ where where $b_\tau$ is the number of $i$ such that $s_{\tau^i}=0$ and $d_\tau$ is the number of $i$ such that $s_{\tau^i}=2$.

\begin{prop}\label{dense rapoport}
The generalized Rapoport locus is dense in the special fiber $\widetilde{Y}^{\textrm{PR}}_{U^\prime}(G^\prime_\Sigma)_\bF$.
\end{prop}

\begin{proof}
We will prove this by showing that for any closed point $x \in \widetilde{Y}^{\textrm{PR}}_{U^\prime}(G^\prime_\Sigma)(\bF)$ with residue field $k$, there exists a a morphism $\textrm{Spec}(k\llbracket t \rrbracket) \to  \widetilde{Y}^{\textrm{PR}}_{U^\prime}(G^\prime_\Sigma)_\bF$ extending $x$ and such that the generic point lands in the generalized Rapoport locus.

Consider $\Delta = H^1_{\dR}(A/k)^0 \otimes_k k\llbracket t \rrbracket$, where $A$ is the abelian variety corresponding to $x$. 
Then $\Delta$ is a free $\ccO_E \otimes k\llbracket t \rrbracket$-module of rank $2$ and for any $\tau = \tilde{\theta} \in \widehat{\Theta}_E$, its $\tau$-component $\Delta_\tau$ is free of rank $2$ over $k\llbracket t \rrbracket[x]/x^{e_\fp}$. 
If for every $1 \leq i \leq e_\fp$, $s_{\tau^i} \neq 1$, we have a unique lift of the unique filtration of $\omega^0_\tau$ to $\Delta_\tau$. 

Otherwise, let $i$ be the smallest $i$ such that $s_{\tau^i}=1$ and pick a $k[x]/x^{e_\fp}$ basis $\overline{e_1},\overline{e_2}$ of $ H^1_{\dR}(A/k)^0_\tau$ such that $\omega^0_\tau(i) = \langle x^{a-1}\overline{e_1},x^{a}\overline{e_2} \rangle$ for a suitable choice of $a \leq e_\fp$. Furthermore, we may choose a series of vectors $0 = v_0, v_1, \cdots , v_{s_\tau(e_\fp)}$ such that $\omega^0_{\tau}(i) = \langle v_0, \cdots, v_{s_\tau(i)} \rangle$ and $xv_i \in \{v_{i-2},v_{i-1}\}$. Furthermore, if $s_{\tau^i} =2$ we impose the relations $xv_{s_\tau(i)-1} = xv_{s_\tau(i)} = v_{s_\tau(i-1)}$ and $xv_{s_\tau(i)+1} = v_{s_\tau(i)-1}$. Let $e_1,e_2$ be any lifts of $\overline{e_1},\overline{e_2}$ to $\Delta_\tau$ which thus form a basis. We inductively define a series of lifts $0= L_{\tau,0} \subset L_{\tau,1} \subset \cdots \subset L_{\tau,e_\fp}=L_\tau$ of $0 \subset \omega^0_\tau(1) \subset \cdots \subset \omega^0_\tau$ which are local summands of $\Delta_\tau$ as follows:

For each $\overline{v}_i$ write $\tilde{v}_i$ for its canonical lift to $\Delta_\tau$ with respective to the basis $e_1,e_2$. For the base case, if $s_{\tau^1} \neq 1$, the lift is imposed, otherwise if $s_{\tau^1}=1$ put $L_{\tau,1} = \langle x^{e_\fp-1}e_1 \rangle$. Suppose we have constructed $L_{\tau,i-1}$ by determining a set of lifts $v_1, \cdots , v_{s_\tau(i-1)}$. 
If $s_{\tau^i}=0$, we set $L_{\tau,i} = L_{\tau,i-1}$ and if $s_{\tau^i}=2$, we set $L_{\tau,i} = x^{-1}L_{\tau,i-1}$ and we set $v_{s_{\tau}(i)-1}$ to be a lift of $\overline{v}_{s_{\tau}(i)-1}$ such that $xv_{s_\tau(i)-1} = v_{s_\tau(i-1)}$. We can do this because $x\tilde{v}_{s_\tau(i)-1} - v_{s_\tau(i-1)} \in x \Delta_\tau$. We define $v_{s_\tau(i)}$ similarly. Suppose now that $s_{\tau^i}=1$, if $x \overline{v}_{s_\tau(i)} = \overline{v}_{s_\tau(i-1)}$, let $v_{s_{\tau(i)}}$ be a lift such that  $xv_{s_{\tau(i)}} = v_{s_{\tau(i-1)}}$. If $x \overline{v}_{s_\tau(i)} = \overline{v}_{s_\tau(i)-2}$, then we set $v_{s_{\tau(i)}}$ to be a lift such that $xv_{s_{\tau(i)}} = v_{s_{\tau(i)-2}}+tv_{s_{\tau(i)-1}}$.

It follows from construction that over the generic fiber $k((t))$, $L_\tau \otimes_{k \llbracket t \rrbracket } k((t))$ defines a point of the generalized Rapoport locus.

We now define inductively for $n \geq 1$, a sequence of compatible lifts $\underline{A_n}$ over $k \llbracket t \rrbracket/t^n$ such that $\underline{A_1}=\underline{A}$. Suppose we have constructed $\underline{A_n}$ and consider the thickening $\textrm{Spec}(k \llbracket t \rrbracket/t^{n+1}) \to \textrm{Spec}(k \llbracket t \rrbracket/t^n)$, with trivial divided powers, we can consider the data 
\[L_\tau \otimes_{k \llbracket t \rrbracket} k \llbracket t \rrbracket/t^{n+1} \subset \Delta_\tau \otimes_{k \llbracket t \rrbracket} k \llbracket t \rrbracket/t^{n+1} \simeq H^1_{cris}(A_n/ k \llbracket t \rrbracket/t^n)^0_\tau \otimes_{k \llbracket t \rrbracket/t^n} k \llbracket t \rrbracket/t^{n+1},\]
which by Grothendieck-Messing theory, as in the proof of proposition \ref{unitary smooth} yields a lift $A_{n+1}$ of $A_n$, the base case being clear for $n=1$. Let $U = \textrm{Spec}(A)$ be an open affine containing $x$, then we obtain a sequence of compatible morphisms $A \to k\llbracket t \rrbracket /t^n$ which thus yield a morphism $\textrm{Spec}(k \llbracket t \rrbracket) \to \widetilde{Y}^{\textrm{PR}}_{U^\prime}(G^\prime_\Sigma)$. By construction, its special fiber maps to $x$ and its generic fiber maps into the generalized Rapoport locus. This completes the proof.
\end{proof}

As a consequence, we see that $\widetilde{Y}^{\textrm{PR}}_{U^\prime}(G^\prime_\Sigma)$ has the same connected components as $\widetilde{Y}^{\textrm{DP}}_{U^\prime}(G^\prime_\Sigma)$.

\subsubsection{The dual filtration}\label{dual filtration}
Note that the models defined above depend on the choice of lifts of $\Theta_F$ to $\Theta_E$, namely in terms of the filtration. To remove this dependence, we introduce the notion of the dual filtration. 

Let $S = \widetilde{Y}^{\textrm{PR}}_{U^\prime}(G^\prime_\Sigma)$, as defined in the previous section. 
For any $\tau \in \widehat{\Theta}_E$ and $i \leq e_\fp$, set 
\[a_{\tau,i}(x) = \prod_{j \leq i}(x-\tau^j(\varpi_\fp)),\,  b_{\tau,i}(x) = \prod_{j \geq i+1}(x-\tau^j(\varpi_\fp)),\]
and $a_{\tau,i}=a_{\tau,i}([\varpi_\fp])$, $b_{\tau,i}=a_{\tau,i}([\varpi_\fp]) \in \ccO \otimes_\bZ \ccO_E$. Let $T$ be an $\ccO$-scheme and $\underline{A}$ a $T$-point of $S$. Then the polarization $\lambda$ induces a perfect $\ccO_T$-linear perfect pairing $\langle \cdot,\cdot \rangle$ on $\ccH^1_{\textrm{dR}}(A/T)$ such that the actions of $\alpha$ and $\alpha^*$, for any $\alpha \in \ccO_D$, form an adjoint pair. We can thus restrict it, for any $\tau \in \widehat{\Theta}_E$ with $\tau \vert_F =\theta$, to a perfect pairing $\langle \cdot,\cdot \rangle_\theta:\ccH^1_{\textrm{dR}}(A/T)^0_{\tau} \times \ccH^1_{\textrm{dR}}(A/T)^0_{\tau^c} \to \ccO_T$. It is straightforward to check, since $\ccH^1_{\textrm{dR}}(A/T)$ is locally free of rank of rank 2 as an $\ccO_T \otimes_\bZ  \ccO_E$-module, that $\ker a_{\tau,i} = \textrm{Im} \, b_{\tau,i}$ is locally free of rank $2i$ over $\ccO_T$, and that $\ker a_{\tau,i}$ and $\ker b_{\tau^c,i}$ are orthogonal.\smallskip 

We can thus define a pairing \[\langle \cdot,\cdot \rangle_{\theta^i} : \ker a_{\tau,i} \times \ker a_{\tau^c,i} \to \ccO_T, \] given on sections (by abuse of notation) by 
\[\langle b_{\tau,i} \cdot x, b_{\tau^c,i} \cdot y \rangle_{\theta^i} = \langle x, b_{\tau^c,i}\cdot y \rangle_\theta = \langle b_{\tau,i} \cdot x,  y \rangle_\theta.\]
It follows, since the original pairing  $\langle \cdot,\cdot \rangle$ is perfect, that so is $\langle \cdot,\cdot \rangle_{\theta^i}$. For a subsheaf $\ccF \subset \ker a_{\tau,i}$ we write $\ccF^{\perp_i}$ for its orthogonal complement in $\ker a_{\tau^c,i}$ with respect to $\langle \cdot,\cdot \rangle_{\theta^i}$.

Now, let $\tau = \tilde{\theta} \in \widehat{\Theta}_E$. By definition of the filtration $\underline{\omega}$, $\omega^0_{A/T,\tau}(i) = \omega^0_\tau(i) \subset \ker a_{\tau,i}$ has dimension $s_\tau(i) = \sum_{j \leq i} s_{\tau^j}$. We thus set
\[\omega^0_{\tau^c}(i) \subset \ker a_{\tau^c,i} = \omega^0_\tau(i)^{\perp_i}.\]
It follows that $\omega^0_{\tau^c}(i)$ is locally free of rank $2i - s_\tau(i) = s_{\tau^c}(i)$. We remark that this definition is independent of the choice of uniformizer $\varpi_\fp$ and depends only on the $\ccO_{F,(p)}^+$-orbit of $\lambda$.
\begin{lem}
For $\tau$ as above, the collection of $\omega^0_{\tau^c}(i)$ defines an $\ccO_T \otimes \ccO_E$-stable filtration of $\omega^0_{\tau^c}$. Furthermore, for each $1 \leq i \leq e_\fp$, the graded piece $\omega^0_{\tau^c}(i)/\omega^0_{\tau^c}(i-1)$ is locally free of rank $s_{(\tau^c)^i}$, killed by $[\varpi_\fp]-(\tau^c)^i(\varpi_\fp)$.
\end{lem}

\begin{proof}
For any $1 \leq i \leq e_\fp$, $\omega^0_{\tau^c}(i)$ is locally free of rank $ s_{\tau^c}(i)$ and is $\ccO_T \otimes \ccO_E$-stable by the compatibility of $\langle \cdot,\cdot \rangle_{\theta^i}$ and the $\ccO_D$-action. Now, note that since $b_{\tau^c,i-1} = ([\varpi_\fp]-(\tau^c)^i(\varpi_\fp)) b_{\tau^c,i}$, if $\ccF \subset \ker a_{\tau,i-1}$ is locally free, then $\ccF^{\perp_i} = \left([\varpi_\fp]-(\tau^c)^i(\varpi_\fp)\right)^{-1} \ccF^{\perp_{i-1}}$. Indeed,
\begin{equation*}
\begin{split}
\left\langle \ccF,\left([\varpi_\fp]-(\tau^c)^i(\varpi_\fp)\right)^{-1} \ccF^{\perp_{i-1}} \right\rangle_{\theta^{i}} 
& = \left\langle \ccF, (b_{\tau^c,i})^{-1}\left([\varpi_\fp]-(\tau^c)^i(\varpi_\fp)\right)^{-1} \ccF^{\perp_{i-1}} \right\rangle_\theta \\
& = \left\langle \ccF,\ccF^{\perp_{i-1}}\right\rangle_{\theta^{i-1}} \\
& = 0.
\end{split}
\end{equation*}
In particular, since $([\varpi_\fp]-\tau^i(\varpi_\fp))\omega^0_{\tau}(i) \subset \omega^0_{\tau}(i-1) \subset \omega^0_\tau(i)$, taking orthogonal complements with respect to $\langle \cdot , \cdot \rangle_{\theta^i}$ and applying $\left([\varpi_\fp]-(\tau^c)^i(\varpi_\fp)\right)$ yields  \[([\varpi_\fp]-(\tau^c)^i(\varpi_\fp))\cdot\omega^0_{\tau^c}(i) \subset \omega^0_{\tau^c}(i-1) \subset \omega^0_{\tau^c}(i).\] 
So that the $\omega^0_{\tau^c}(i)$ give a filtration of $\omega^0_{\tau^c}$. We also deduce that 
\[\omega^0_{\tau^c}(i)/\omega^0_{\tau^c}(i-1) \xrightarrow{\sim} \left(([\varpi_\fp]-\tau^i(\varpi_\fp))^{-1}\omega^0_{\tau}(i-1)/\omega^0_{\tau}(i)\right)^\vee \]
which is locally free of rank $2-s_{\tau^i}=s_{(\tau^c)^i}$.

\end{proof}

Suggested by the proof of proposition \ref{unitary smooth} and the above lemma, we now write for all $\tau \in \widehat{\Theta}_E$, $\ccH^1_{\tau^i} := ([\varpi_\fp]-\tau^i(\varpi_\fp))^{-1}\omega^0_{\tau}(i-1)/\omega^0_{\tau}(i-1) $, $\omega^0_{\tau^i}:=\omega^0_{\tau}(i)/\omega^0_{\tau}(i-1)$, and $v^0_{\tau^i}:=([\varpi_\fp]-\tau^i(\varpi_\fp))^{-1}\omega^0_{\tau}(i-1)/\omega^0_{\tau}(i)$. We obtain

\begin{cor}\label{filtration duality}
The pairing $\langle \cdot , \cdot \rangle_{\theta^i}$ induces a perfect pairing on $\ccH^1_{\tau^i} \times \ccH^1_{(\tau^c)^i}$ from which we obtain isomorphisms

\begin{center}
\begin{tikzcd}
0 \arrow[r] & \omega^0_{\tau^i} \arrow[r] \arrow[d,"\wr"] & \ccH^1_{\tau^i} \arrow[r] \arrow[d,"\wr"] & v^0_{\tau^i} \arrow[r] \arrow[d,"\wr"] & 0 \\
0 \arrow[r] & v^{0 \vee}_{(\tau^c)^i} \arrow[r] & \ccH^{1 \vee}_{(\tau^c)^i} \arrow[r] & \omega^{0 \vee}_{(\tau^c)^i} \arrow[r] & 0.
\end{tikzcd}
\end{center}
\end{cor}

We remark that whereas the sheaves $\ccH^1_{(\tilde{\theta}^c)^i}$ are independent of the specific choice of quasi-polarization $\lambda$ in its $\ccO_{F,(p),+}^\times$-orbit, the above isomorphism does. Finally, one can check that \begin{equation}\label{dual formula} \omega^0_{\tau^c}(i)=b_{\tau^c,i} \cdot \omega^0_{\tau}(i)^\perp,
\end{equation}
where the orthogonal complement here is taken with respect to the original pairing $\langle \cdot, \cdot \rangle$. One may simply take this as the definition of the dual filtration.

As mentioned at the beginning of the section, the dual filtration can be used to remove the dependence of the Pappas-Rapoport model on the choice of lift of $\Theta_F$ to $\Theta_E$: Define the Pappas-Rapoport model by adding the data, for each $\tau \in \widehat{\Theta}_E$, of a suitable filtration of $\omega^0_\tau$ such that for each $\tau$ and $i$, $\omega^0_\tau(i)$ and $\omega^0_{\tau^c}(i)$ are orthogonal with respect to the pairing $\langle \cdot , \cdot \rangle_{\theta^i}$. In particular, from now on when we write $\omega^0_\tau(i)$, we can mean the filtration with respect to either lift of $\theta$.

\begin{rem}
Consider the case that $T$ is a scheme over  the fraction field of $\ccO$. Then the splitting $\ccO_E \otimes_\bZ \ccO_T \simeq \bigoplus_{\beta \in \Theta_E} \ccO_T$, where we stress that the splitting is now over all ramified embeddings $\beta \in \Theta_E$, gives a splitting $\ccH^1_{\textrm{dR}}(A/T)^0 \simeq \bigoplus_{\beta \in \Theta_E} \ccH^1_{\textrm{dR}}(A/T)^0_{\beta}$, where each component $\ccH^1_{\textrm{dR}}(A/T)^0_{\beta}$ is locally free of rank $2$ such that the action of $\ccO_E$ factors through $\beta$. For a fixed $\tau \in \widehat{\Theta}_E$, we have $\ker a_{\tau,i} = \bigoplus_{j \leq i} \ccH^1_{\textrm{dR}}(A/T)^0_{\tau^j}$ and the pairing $\langle \cdot , \cdot \rangle_{\theta^i}$ is equal to the restriction of the original pairing $\langle \cdot , \cdot \rangle$ to $\ker a_{\tau,i} \times \ker a_{\tau^c,i}$ and so $\omega^0_{\tau^c}(i)$ as defined above really is $\omega^0_{\tau^c}(i)$.
\end{rem}

\subsection{Quaternionic Shimura varieties}\label{quaternionic models}
In this section, we recall how to construct integral models of the Shimura varieties attached to the Shimura datum $(G_\Sigma,[h_\Sigma])$, defined in section \ref{unitary notation}, from the integral Unitary models we defined in section \ref{PR unitary}. We again follow the strategy of  \cite[$\mathsection$2.3]{2020arXiv200100530D}. 
\subsubsection{Pappas-Rapoport models of Quaternionic Shimura varieties}\label{quat def}

Keep the notation from the section \ref{unitary notation}. In particular, let $L \subset \overline{\bQ}$ be a sufficiently large number field such that it contains the images of $\beta(E)$ for all $\beta \in \Theta_E$, and let $\ccO$ denote the completion of its ring of integers with respect to the prime determined by the inclusion $\overline{\bQ} \hookrightarrow \overline{\bQ}_p$. 

Let $T'$ be the abelian quotient of $G'_\Sigma$; it is isomorphic to $(T_F \times T_E)/T_F$, where $T_F$ is embedded via the map $x \mapsto (x^2,x^{-1})$. Let $\nu':G'_\Sigma \to T'$ be the morphism induced by $G_\Sigma \times T_E \to T_F \times T_E$, $(g,y) \mapsto (\det(g),y)$ where $\det$ denotes the reduced norm. 

Consider the Shimura data $(T_E,i_{\tilde{\Sigma}})$, $(T_F,i_{\Sigma})$, $(T',i'_{\Tilde{\Sigma}})$, where $i_{\tilde{\Sigma}}$ was defined in section \ref{unitary notation}$, i_{\Sigma}=\det \circ h_\Sigma$, and $i'_{\Tilde{\Sigma}}$ is the composition of $(i_{\Sigma},i_{\Tilde{\Sigma}})$ with the projection map to $T'$. These Deligne homomorphisms will be implicit whenever we talk about the Shimura varieties associated with the above Shimura data.

For a suitable small open compact $V_E=\ccO_{E,p}^\times V_E^p$ with $V_E^p \subset (\mathbb{A}_{E,f}^{(p)})^\times$, the geometric points of the zero dimensional Shimura variety $Y_{T_E}(V_E)$ are identified with the finite set 
\[C_{V_E} = (\bA_{E,f}^{(p)})^\times / \ccO_{E,(p)}^\times V^p_E = E^\times \backslash \bA_{E,f}^\times / V_E.\]
Write $L_{\widetilde{\Sigma}}$ for its reflex field $\ccO_{\widetilde{\Sigma}}$ its ring of integers. As a scheme over $\ccO_{\widetilde{\Sigma}}$, $Y_{V_E}(T_E)$ is characterized by descent from the isomorphism 
\[Y_{V_E}(T_E) \times_{\ccO_{\widetilde{\Sigma}}} \ccO_M \simeq \bigsqcup_{c \in C_{V_E}} \textrm{Spec}(\ccO_M),\] where $\ccO_M$ is the ring of integers of a finite abelian extension $M$ of $L_{\tilde{\Sigma}}$, and the Galois action of $\textrm{Gal}(M/L_{\widetilde{\Sigma}})$ is determined by Shimura reciprocity for $i_{\tilde{\Sigma}}$ as in \cite{tian_xiao_2016}. We can assume that the extension $M/L_{\widetilde{\Sigma}}$ is unramified at the primes of $L_{\tilde{\Sigma}}$ over $p$ since $\ccO_{E,p}^\times \subset V_E$.

 In the same way, for $V_F = V_F^p \ccO_{F,p}^\times$, we have a description for $Y_{T_F}(V_F)$ whose geometric points are given by 
\[C_{V_F}= (\adf)^\times / \ccO_{F,(p)}^\times = (F_+)^\times \backslash \bA_{F,f}^\times / V_F\]
with Galois action determined by Shimura reciprocity for $i_\Sigma$. We will however, for our purposes, view $Y_{V_F}(T_F)$ as a finite \'{e}tale scheme over $\ccO$. Similarly, for $V' \subset T'(\bA^{(p)}_f)$ whose $p$-component contains the image of $\ccO_{F,p}^\times \times \ccO_{E,p}^\times$, $Y_{T'}(V')$ is a finite \'{e}tale-$\ccO$-scheme whose geometric points are given by $C_{V'}=T'(\bQ)_+ \backslash T'(\bA_f)/V'$.\medskip

We use the following assumptions of \cite[$\mathsection$2.3]{2020arXiv200100530D}: Let $U \subset G_{\Sigma}(\bA_f)$ be an open compact subgroup containing $\ccO_{B,p}^\times$. We say that $V_E$ is sufficiently small with respect to $U$ if 

\begin{itemize}
    \item $E^\times \cap V^1_E=\{1\}$, where $V^1_E := \{y/y^c \, \vert \, y \in V_E \}$,
    \item $V_E \cap \bA_{F,f}^\times \subset U$,
    \item $\textrm{Nm}_{E/F}(V_E)\subset \det(U)$.
\end{itemize}

Such a $V_E$ always exists. Let $U' \subset G'_\Sigma(\bA_f)$ denote the image of $U\times V_E$ for $V_E$ sufficiently small with respect to $U$. 

We have a commutative diagram of algebraic groups

\begin{center}
    \begin{tikzcd}
        G_\Sigma \times T_E \arrow[r] \arrow[d] & G^\prime_\Sigma \arrow[d] \\
        T_F \times T_E \arrow[r] & T^\prime,
    \end{tikzcd}
\end{center}
which is compatible with the Deligne homomorphisms $(h_\Sigma,i_{\widetilde{\Sigma}})$, $h^\prime_{\widetilde{\Sigma}}$, $(i_\Sigma,i_{\widetilde{\Sigma}})$, and $i^\prime_{\widetilde{\Sigma}}$. It therefore yields, for compatible compact open subgroups, a commutative a diagram of Shimura varieties. Furthermore, we have
\begin{lem}[Lemma 2.3.1 \cite{2020arXiv200100530D}] If $V_E$ is sufficiently small with respect to $U$ and $U^\prime $ is the image of $U \times V_E$, then the diagram of complex manifolds 
\begin{equation}\label{complex cartesian}
\begin{tikzcd}
(G_\Sigma(\bQ)_+ \backslash \mathcal{H}^{\Theta_F \setminus \Sigma_\infty} \times G_\Sigma(\mathbb{A}_f^{p}) / U) \times C_{V_E} \arrow[d] \arrow[r] & 
G_\Sigma^\prime (\bQ)_+ \backslash \mathcal{H}^{\Theta_F \setminus \Sigma_\infty} \times G^\prime_\Sigma(\mathbb{A}_f^{p}) / U^\prime \arrow[d] \\
C_{\det(U)}\times C_{V_E} \arrow[r]         & C_{\nu'(U')},           
\end{tikzcd}
\end{equation}
where $G_\Sigma(\bQ)_+$ denotes the subgroup of elements with totally positive reduced norm, is Cartesian. 
\end{lem}
This brings us to the following definition: Set 
\begin{equation}\label{Y UxV (G xT)}
Y_{U \times V_E}(G_\Sigma \times T_E) = Y_{U'}(G'_\Sigma) \times_{Y_{\nu'(U')}(T')} (Y_{\det(U)}(T_F) \times_\ccO Y_{V_E}(T_E)),
\end{equation}
where $Y_{U'}(G'_\Sigma)$ denotes the Pappas-Rapoport model defined in section \ref{PR unitary}. We remark that from now on, we will only consider Pappas-Rapoport models, unless stated otherwise, and thus remove the PR superscript.

Suppose now that $g \in G_\Sigma(\bA^{(p)}_f)$, and $U_1,U_2 \subset G_\Sigma(\mathbb{A}_f)$ are two sufficiently small open compact subgroups of level prime to $p$ such that $g^{-1}U_1g \subset U_2$. For $y \in (A_{E,f}^{(p)})^\times$ and $V_E$ sufficiently small relative to $U_1$ and $U_2$, of level prime to $p$, we have a  commutative diagram 
\[\begin{tikzcd}
Y_{U_1'}(G'_\Sigma) \arrow[d, "\rho_{gy}", shift right] \arrow[r] & Y_{\nu'(U_1')}(T') \arrow[d, "\nu'(gy)"] & Y_{\det(U_1)}\times_\ccO Y_{V_E}(T_E) \arrow[d, "{(\det g,y)}"] \arrow[l] \\
Y_{U_2'}(G'_\Sigma) \arrow[r]                                     & Y_{\nu'(U_2')}(T')                       & Y_{\nu'(U_2')}(T'), \arrow[l]                                         
\end{tikzcd}\]
which yields a morphism $\rho_{(g,y)}: Y_{U_1 \times V_E}(G_\Sigma \times T_E) \to Y_{U_2 \times V_E}(G_\Sigma \times T_E)$. The $\rho_{(g,y)}$ satisfy the usual compatibilities. Therefore, for $U=U_1=U_2$ and $g=1$, the automorphisms $\rho_{(1,y)}$ define a free action of the group $C_{V_E}=(\bA_{E,f}^{(p)})^\times / \ccO_{E,(p)}^\times V_E^p$. We define $Y_U(G_\Sigma)$ as the quotient of $Y_{U \times V_E}(G_\Sigma \times T_E)$ by this action. The resulting scheme is independent of the choice of $V_E$. Since $Y_{U^\prime}(G^\prime_\Sigma)$ is smooth, quasi-projective of relative dimension $ \vert \Theta_F \setminus \Sigma_\infty \vert$, so is $Y_{U \times V_E}(G_\Sigma \times T_E)$ and thus so is the quotient $Y_U(G_\Sigma)$. Moreover, the natural projections induce an isomorphism
\begin{equation}\label{quat split}
Y_{U \times V_E}(G_\Sigma \times T_E) \xrightarrow{\sim} Y_U(G_\Sigma) \times Y_{V_E}(T_E).     
\end{equation}

Given general $g$, $U_1$, $U_2$ with $g^{-1}U_1g_2$, the map $\rho_{(g,1)}$ commutes with the action of $C_{V_E}$ and thus defines a morphism
\[\rho_g: Y_{U_1}(G_\Sigma) \to Y_{U_2}(G_\Sigma),\]
independent of $V_E$. Furthermore, if we have $h$ and $U_3$, again sufficiently small of level prime to $p$, with $h^{-1}U_2 h \subset U_3$, then we have the equality $\rho_{gh} = \rho_g \circ \rho_h$. We thus obtain a Hecke action on the projective system $\varprojlim_U Y_U(G_\Sigma)$ and we note that under the isomorphism \ref{quat split}, $\rho_{(g,y)} = (\rho_g,y)$.

In this thesis, we will not attempt to consider the descent of $Y_U(G_\Sigma)$ to a smaller ($p$-adic) field. Instead, for our purposes, it will be convenient to extend scalars even further to obtain our main result. We do so uniformly by setting $W$ to be the ring of integers of the maximal unramified extension $L^{nr}$ of $L$. Over $W$, the identity $1 \in C_{V_E}$ gives a section $\textrm{Spec} ( W) \to Y_{V_E}(T_E) \simeq \bigsqcup_{c \in C_{V_E}} \textrm{Spec} ( W)$. This yields a morphism 
\[
Y_U(G_\Sigma)_W \to Y_U(G_\Sigma)_W \times_W Y_{V_E}(T_E)_W \to Y_{U \times V_E}(G_\Sigma \times T_E)_W \to Y_{U^\prime}(G^\prime)_W 
\]
for any $U,V_E,U^\prime$ as in lemma \ref{complex cartesian}. We immediately obtain the analog of lemma 2.3.2 of \cite{2020arXiv200100530D}:

\begin{lem}\label{compatibility components unitary quaternionic}
For $U,V_E$ and $U'$ as above, we have a Cartesian Diagram 
\begin{center}
    \begin{tikzcd}
Y_U(G_\Sigma)_W \arrow[d] \arrow[r] & Y_{U'}(G'_\Sigma)_W \arrow[d] \\
C_{\det(U)} \arrow[r]               & C_{\nu'(U')}                  
\end{tikzcd}

\end{center}
which identifies $Y_U(G_\Sigma)$ with an open and closed subscheme of $Y_{U'}(G'_\Sigma)$. Furthermore, this diagram is Hecke equivariant in the sense that for any 
$U_1,U_2 \subset G_\Sigma(\bA_f)$, $g \in G_\Sigma(\bA^{(p)}_f)$ 
with $g^{-1}U_1 g \subset U_2$, 
writing $U^\prime_i$ for the image of $U_i \times V_E$ in
$G^\prime_\Sigma(\bA_f)$ (for sufficiently small $V_E$ with respect to both subgroups), then the following diagram commutes.
\begin{center}
\begin{tikzcd}
    Y_{U_1}(G_\Sigma)_W \arrow[r] \arrow[d,"\rho_g"'] & Y_{U^\prime_1}(G^\prime_\Sigma)_W  \arrow[d,"\rho_g"] \\
    Y_{U_2}(G_\Sigma)_W \arrow[r] & Y_{U^\prime_2}(G^\prime_\Sigma)_W.
\end{tikzcd}
\end{center}
\end{lem}

We finish this section by remarking that in the case that $p$ is unramified in $F$, the $Y_U(G)_\Sigma$ define a system of integral canonical models for $(G_\Sigma, [h_\Sigma])$.

\subsubsection{Comparing models}\label{model comparison}
We have given two different integral models over $\ccO = \ccO_L$ for the Shimura variety attached to the group $G=\textrm{Res}_{F/\bQ}GL_2$, the first as a Hilbert modular variety $Y_U(G)$ and the second as a quaternionic Shimura variety $Y_{U}(G_\emptyset)$. The latter is obtained from $Y^\prime_{U^\prime}(G^\prime_\emptyset)$, where we choose $\ccO_D = M_2(\ccO_E)$ and $\delta = \left( \begin{smallmatrix} 0 & 1 \\ -1 & 0 \end{smallmatrix} \right)$. In the case that $p$ is unramified in $F$, these are necessarily isomorphic as they are both integral canonical models. Since we allow $p$ to ramify in $F$, we cannot appeal to this fact, since the models are in general not canonical. We however give an explicit isomorphism between both models under the assumption that that we can choose $E=F(\sqrt{d}) $ such that the image of (relative) Trace is principal. 

Let $U \subset G(\bA^\times_f)$ and $U' \subset G_\varnothing'(\bA^\times_f)$ be sufficiently small subgroups such that $U'$ contains the image of $U$. Let $(A,\iota,\lambda,\eta,\underline{\mathcal{\omega}})$ be the universal object over $\widetilde{Y}_U(G)$. We define a morphism $\tilde{i}: \widetilde{Y}_U(G) \to \widetilde{Y}_{U'}(G_\varnothing')$ as follows:\medskip

Temporarily write $S = \widetilde{Y}_U(G)$. Let $A'= A \otimes_{\ccO_F} \ccO_E^2$ with its natural $M_2(\ccO_E)$-action $\iota'$ given by left multiplication on $\ccO_E^2$. The dual of $A'$ is canonically isomorphic to $A^\vee \otimes_{\ccO_F} \textrm{Hom}_{\ccO_F}(\ccO_E^2,\ccO_F)$ so we define a quasi-polarization $\lambda^\prime$ as the tensor product of $\lambda$ with the Trace map  $\textrm{Tr}:\ccO_E^2 \to \textrm{Hom}_{\ccO_F}(\ccO_E^2,\ccO_F)$ given by 
$\Tr(x,y)(v_1,v_2)= \Tr_{E/F}(x\overline{v_2})-Tr_{E/F}(y\overline{v_1})$. We define a level structure 
\[\eta_i':M_2(\bA_{E,f}^{(p)}) \to T^{(p)}A'_{\overline{s}_i} \simeq T^{(p)}A_{\overline{s}_i}\otimes_{\adf}(\bA_{E,f}^{(p)})^2\]
by requiring that 
$\eta_i'\left(
(a,b)\left(\begin{smallmatrix}c \\
d \end{smallmatrix} \right) \right)=\eta_i(a,b)\otimes\left(\begin{smallmatrix}c \\ d\end{smallmatrix}\right)$ for $a,b\in \adf$ and $c,d \in \bA_{E,f}^{(p)}$.
Finally, we construct the filtration on 
$\omega_{A'/S} \simeq \omega_{A/S} \otimes \textrm{Hom}_{\ccO_F}(\ccO_E^2,\ccO_F)$ 
by taking for each $\tau = \tilde{\theta} \in \widehat{\Theta}_E$
\[\omega^0_{A'/S,\tau}(j) = \omega_{A/S,\theta}(j) \otimes_{\ccO_S[x]/E_{\theta}(x)}\textrm{Hom}_{\ccO_F}(\ccO_E,\ccO_F)_{\tau}.\]
The tuple $(A^\prime, \iota^\prime,\lambda^\prime,\eta^\prime,\underline{\omega^0_{A^\prime}})$ thus defines a point of $\widetilde{Y}_{U'}(G_\varnothing')$ and thus a morphism
\[\tilde{i}: \widetilde{Y}_U(G)_\ccO \to \widetilde{Y}^\prime_{U'}(G_\varnothing').\]

We now show that $\tilde{i}$ is a closed and open immersion. It suffices to do so after extending $\ccO$ to $W = \ccO_{L^{nr}}$, the ring of integers of the maximal unramified extension of $L$. By proposition \ref{dense rapoport}, $\widetilde{Y}$ has the same set of connected components as the naive model and, by \cite[$\mathsection$2.7]{tian_xiao_2016}, we have the equalities
\[\pi_0(\widetilde{Y}_{\overline{\bF}_p}) = \pi_0(\widetilde{Y}_W) = \pi_0(\widetilde{Y}_{L^{nr}}) = \pi_0(\widetilde{Y}_{\bC}). \]
That is, every geometric connected component of $\widetilde{Y}_{\overline{\bF}_p}$ is contained in a unique geometric connected component of $\widetilde{Y}_W$, with non-empty generic fiber.

Let $S$ be a connected component of $\widetilde{Y}_U(G)_W$ and $S^\prime$ be the connected component of $\widetilde{Y}^\prime_{U^\prime}(G^\prime_\emptyset)$ that $\tilde{i}$ maps $S$ to. One can check that on complex points, $\widetilde{i}$ is given by the map 
\[\textrm{SL}_2(\ccO_{F,(p)}) \backslash (\ccH^{\Theta_F} \times \textrm{GL}_2(\mathbb{A}^{(p)}_{F,f}))/U^p
\to G^1_{\emptyset} \backslash (\mathbb{Z}_{(p)}) (\ccH^{\Theta_F} \times G_\emptyset(\mathbb{A}^{(p)}_{E,f}))/(U^{\prime})^p,\]
which as in \cite[$\mathsection$ 2.3.4]{2020arXiv200100530D} is an isomorphism onto its image which is open and closed.

Take $S$ and $S^\prime$ as above. Then, by Galois descent, we see that $\tilde{i}$ is an isomorphism on generic fibers and we thus have an open dense subscheme $V \subset S^\prime$ and a morphism $q: V \to S$ such that $\widetilde{i} \circ q = \id$. In particular, over $V$, the abelian scheme $B = q^* A$ satisfies $ B \otimes_{\ccO_F} \ccO_E^2 = A_\emptyset \vert_V$ where $A_\emptyset$ is the universal abelian scheme over $S^\prime$. Recall that we made the assumption that $E=F(\sqrt{d})$ where the image of Trace is a principal ideal of $\ccO_F$. Letting $\alpha^\prime \in \ccO_E$ generate the image of Trace, this gives us an $\ccO_F$-linear splitting $\ccO_E = \sqrt{d}\ccO_F \oplus \alpha^\prime\ccO_F \simeq \ccO_F^2$. We set $\alpha = \Tr(\alpha^\prime)^{-1}\alpha^\prime$ and define the $M_2(\ccO_F)$-linear endomorphism of $A_\emptyset \vert_V = B^2 \otimes_{\ccO_F} \ccO_E$ given on points by $x \otimes \gamma \mapsto x \otimes \Tr(\alpha \gamma)$. This is an idempotent operator since 
\[\Tr(\alpha \Tr(\alpha \gamma)) = \Tr(\alpha \gamma) \Tr(\alpha) =  \Tr(\alpha \gamma),      \]
and has image $B^2 \otimes 1 \subset B^2 \otimes_{\ccO_F} \ccO_E$. By \cite[Proposition 2.11]{Milne1992ThePO}, since $S^\prime$ is smooth, we may uniquely extend the endomorphism $\Tr$ to a endomorphism $\Tr: A_\emptyset \to A_\emptyset$ over to all of $S^\prime$. 
Since this extension is unique it follows that it is also idempotent and $M_2(\ccO_F)$-linear. 
We let $A^\prime_\emptyset$ denote the scheme theoretic image of $\Tr$ and $A^{\prime \prime}_\emptyset$ the scheme theoretic image of $(\id_{A_\emptyset} - \Tr)$, which is also an idempotent $M_2(\ccO_F)$-linear endomorphism. 
Since $S^\prime$ is Noetherian, both $A^{\prime}_\emptyset$ and $A^{\prime \prime}_\emptyset$ are abelian schemes over $S^\prime$ by, for example, \cite{achter2023images}. Furthermore, we find that the morphism $A_\emptyset \xrightarrow{\Tr \times (\id_{A_\emptyset} - \Tr)} A^\prime_\emptyset \times_{S^\prime} A^{\prime \prime}_\emptyset$ is an isomorphism. Finally, $A^{\prime}_\emptyset$ and $A^{\prime \prime}_\emptyset$ are $M_2(\ccO_F)$-linearly isomorphic, since over $V$ $A^{\prime}_\emptyset$ is given by $B^2 \otimes 1$ and $A^{\prime \prime}_\emptyset$ by $B^2 \otimes \delta$ where $\delta$ generates the image $\gamma \mapsto \gamma - \Tr(\alpha\gamma)$. Therefore, applying the idempotent $e_0$ to $A^\prime_\emptyset$, we obtain an abelian scheme with $\ccO_F$-action that extends $B/V$ that we also denote $B$. Furthermore we obtain an $M_2(\ccO_E)$-linear isomorphism $A_\emptyset \simeq B \otimes_{\ccO_F} \ccO_E^2 $ extending the one on $V$.

The action $\iota_B$ of $M_2(\ccO_F)$ satisfies the Kottwitz condition over all of $S$, since it does over $V$ and it is a closed condition. We can also uniquely extend the quasi-polarization $\lambda_B$ over $V$ to all of $S^\prime$. It is a quasi-polarization since being a polarization is a closed condition. The level structure structure extends uniquely to a level structure $\eta_B$ as well since $V$ is dense. We define a filtration on $\omega^0_B$ via the isomorphism \[\omega^0_{B,\theta} \otimes_{\ccO_F} \ccO_E \simeq \omega^0_{A_\emptyset,\tau} \oplus \omega^0_{A_\emptyset,\tau^c}.\]
It follows that the tuple $(B,\iota_B,\lambda_B,\eta_B,\underline{\omega_B})$ defines a point of $\overline{Y}_U(G)$ and so yields a morphism $q: S^\prime \to \overline{Y}_U(G)$ extending the one on $V$. Since $S^\prime$ is connected, it must factor through $S$, and since the composite $\widetilde{i} \circ q$ is the identity on $V$, which is open and dense, then $\widetilde{i} \circ q = \id_{S^\prime}$. We deduce similarly that $ q \circ \widetilde{i} = \id_{S}$. Therefore $\widetilde{i}: S \to S^\prime $ is an isomorphism.

By the previous discussion, we see that $\tilde{i}: \widetilde{Y}_U(G)_\ccO \to \widetilde{Y}^\prime_{U'}(G_\varnothing')$ is an isomorphism onto its image. Furthermore, $\tilde{i}$ is clearly compatible under the action of $\ccO_{F,(p),+}^\times$ hence descends to a morphism $i: Y_U(G) \to Y_{U'}(G_\varnothing')$ which is also an isomorphism onto its image. Let $V_E$ now be an open compact subgroup of $\mathbb{A}_{E,f}^\times$ contained in $U^\prime$ of level prime to $p$. Since $\Sigma = \emptyset$, the Galois action on $C_{V_E}$ given by Shimura reciprocity is trivial, so we may identify $Y_{V_E}(T_E)$ over $\ccO$ with $\bigsqcup_{C_{V_E}} \textrm{Spec}(\ccO)$ and we extend $i$ to a morphism
\[ i^\prime : Y_U(G) \times Y_{V_E}(T_E) = \bigsqcup_{C_{V_E}} Y_U(G) \to Y^\prime_{U^\prime}(G^\prime_\emptyset),\]
where $i^\prime = \rho_y \circ i$ on the component indexed by $y \in \mathbb{A}_{E,f}^{(p) \, \times}$. It follows that the diagram
\[
\begin{tikzcd}
Y_U(G) \times Y_{V_E}(T_E) \arrow[r] \arrow[d] & Y_{U^\prime}(G^\prime_\emptyset) \arrow[d] \\
Y_{\det(U)}(T_F) \times Y_{V_E}(T_E) \arrow[r] & Y_{\nu(U^\prime)}
\end{tikzcd}
\]
is Cartesian. Therefore we have an isomorphism \[Y_U(G) \times Y_{V_E}(T_E) \xrightarrow{\sim} Y_{U \times V_E} (G_\emptyset \times V_E) \xrightarrow{\sim} Y_{U}(G_\emptyset) \times Y_{V_E}(T_E),\]
which is compatible with the action of $G(\mathbb{A}^{(p)}_f) \times (\mathbb{A}^{(p)}_{E,f})^\times$. Taking quotients by the action of $C_{V_E}$ we obtain an isomorphism 
\[Y_{U}(G) \xrightarrow{\sim} Y_{U}(G_\emptyset) \]
which is compatible with the Hecke action and whose composite with the inclusion from \ref{compatibility components unitary quaternionic} is given by the (base change to $W$ of the) morphism $i: Y_U(G) \to Y^\prime_{U^\prime}(G^\prime_\emptyset)$ defined above.

\subsection{Automorphic vector bundles}\label{bundles}
In this section we define automorphic vector bundles on the special fibers of the (non PEL) Shimura varieties $Y_U(G_\Sigma)$, for sufficiently small $U$, of level prime to $p$. From now on, we let $L \subset \overline{\bQ}$ be a sufficiently large number field containing all images of embeddings of $E$. Let $\ccO$ denote the completion of its ring of integers with respect to the prime determined by the embedding $\overline{\bQ} \to \overline{\bQ}_p$ and write $\bF$ for its residue field. In particular, all of our models are defined over $\ccO$. In the following, we will also write $\Fp = \fp_1 \cdots \fp_n$ for the radical of $p\ccO_F$.
\subsubsection{The Hilbert setting}\label{automorphic bundles hilbert}

Consider first the case $G =  \textrm{Res}_{F/\bQ} \, \textrm{GL}_2$ and consider a sufficently small, in the usual sense, open compact subgroup 
$U \subset G(\mathbb{A}_f)$ which is $\Fp$-neat. That is, for any $\alpha \in U \cap \ccO_F^\times$, $\alpha - 1 \in \Fp$. If $p$ is unramified in $F$, then $\Fp =p \ccO_F$ and this is the same as being $p$-neat, as in the terminology of \cite{DS17}. 
Write $\overline{Y} = Y_U(G)_\bF$, where $Y_U(G)$ is the Pappas-Rapoport model defined in section \ref{hilbert model}, $S = \widetilde{Y}_U(G)_\bF$ and $A/S$ the universal abelian variety. Write $\widetilde{\omega} = \omega_{A/S}$. We insert the $\tilde{\cdot}$ to differentiate sheaves on $S$ and sheaves on $\overline{Y}$. Recall that the Pappas-Rapoport filtration is the data for each $\theta \in \widehat{\Theta}_F$ of a filtration:

\[0 \subset \widetilde{\omega}_{\theta}(1) \subset \cdots \subset \widetilde{\omega}_\theta(e_\fp) = \widetilde{\omega}_\theta \subset \ccH^1_{dR}(A/S)_\theta\]

by locally free subsheaves such that, in particular, for any $1 \leq i \leq e_\fp$, the sheaves $\widetilde{\omega}_{\theta^i} := \widetilde{\omega}_\theta(i)/ \widetilde{\omega}_\theta(i-1)$ are line bundles. Furthermore, by lemma \ref{filtration torsion}, the sheaves $\widetilde{\ccH}^1_{\theta^i} = [\varpi_\fp]^{-1}\widetilde{\omega}_\theta(i-1)/ \widetilde{\omega}_\theta(i-1)$ are locally free of rank two. We thus have, for each $\theta^i \in \Theta_F$, a short exact sequence of locally free sheaves 
\[ 0 \to \widetilde{\omega}_{\theta^i} \to \widetilde{\ccH}^1_{\theta^i} \to \widetilde{v}_{\theta^i} \to 0,\]
where $\widetilde{v}_{\theta^i} = [\varpi_\fp]^{-1}\widetilde{\omega}_\theta(i-1)/ \widetilde{\omega}_\theta(i)$ is also a line bundle. One may think of $\widetilde{\ccH}^1_{\theta^i}$ as a way of making sense of the phrase "de Rham cohomology at $\theta^i$" and the above short exact sequence as the Hodge filtration at $\theta^i$.

Recall that for $\mu \in \ccO_{F,(p),+}^\times$, the automorphism $\theta_\mu:S \to S$ is given by $(A,\iota,\lambda,\eta) \mapsto (A,\iota,\mu \lambda,\eta)$. It follows that the abelian variety $\theta_{\mu}^*A$ is canonically $\ccO_F-$linearly isomorphic to $A$ and we have an induced $\ccO_F \otimes \ccO_S$-linear isomorphism $\ccH^1_{dR}(\theta^*_\mu A/S)=\theta_\mu^* \ccH^1_{dR}(A/S)\to \ccH^1_{dR}(A/S)$ which preserves the filtration. In particular, for $\mu=\alpha^2 \in (U \cap \ccO_F^\times)^2$, this isomorphism $A \to \theta_\mu^*A$ is given by $\iota(\alpha^{-1})$. Hence, since $U$ is $\Fp$-neat, $\alpha \equiv 1 \mod \Fp$ and the induced action on $\ccH^1_{dR}(A/S)$ is trivial. In particular we obtain an action of $\ccO_{F,(p),+}^\times/ (U \cap \ccO_F^\times)^2$ on the sheaves $\widetilde{\omega}_{\theta^i}$, $\widetilde{\ccH}^1_{\theta^i}$ and $\widetilde{v}_{\theta^i}$ which descend to sheaves over $\overline{Y}$, fitting in the exact sequence 
\[0 \to \omega_{\theta^i} \to \ccH^1_{\theta^i} \to v_{\theta^i} \to 0.\]

Furthermore, for suitable $U_1$, $U_2$ and $g^{-1}U_1g \subset U_2$, the Hecke morphism $\tilde{\rho}_g:S_1 \to S_2$ given by the quasi-isogeny $\pi_g \in \Hom_{S_1}(A_1,\tilde{\rho}_2^*A_2) \otimes \bZ_{(p)}$ induces an $\ccO_F \otimes \ccO_{S_1}$-linear isomorphism $\tilde{\rho}_g^*\ccH^1_{dR}(A_2 / S_2) \to \ccH^1_{dR}(A_1 / S_1)$ which is compatible with the $\ccO_{F,(p),+}^\times$-action and filtrations. In particular, we obtain isomorphisms 
\[ \tilde{\rho}_g^*\widetilde{\omega}_{2,\theta^i} \xrightarrow{\sim} \widetilde{\omega}_{1,\theta^i}, \, \tilde{\rho}_g^*\widetilde{\ccH}^1_{2,\theta^i} \xrightarrow{\sim} \widetilde{\ccH}^1_{1,\theta^i}, \, \tilde{\rho}_g^*\widetilde{v}_{2,\theta^i} \xrightarrow{\sim} \widetilde{v}_{1,\theta^i},\]

\noindent which also descend to isomorphisms, which we also denote $\pi_g^*$,
\[ \rho_g^*\omega_{2,\theta^i} \xrightarrow{\sim} \omega_{1,\theta^i}, \, \rho_g^*\ccH^1_{2,\theta^i} \xrightarrow{\sim} \ccH^1_{1,\theta^i}, \, \rho_g^*v_{2,\theta^i} \xrightarrow{\sim} v_{1,\theta^i},\]
over $\overline{Y}_1$. Finally, we note that if $U_3$ is as above and we have $h \in G(\mathbb{A}_f)$ such that $h^{-1}U_2h \subset U_3$, then the relation $\pi_{gh} = \tilde{\rho_g}^*(\pi_h) \circ \pi_g$ yields the equality of morphisms of sheaves on $S_1$ $\pi_{gh}^*= \pi_g^* \circ \rho_g^*(\pi_h^*)$ and thus the same equality over $\overline{Y}_1$

\subsubsection{The Unitary setting}\label{automorphic bundles unitary}
We work similarly as above, let $\Sigma$ be an even subset of places, away from $p$, and $U'\subset G'_\Sigma(\bA_f)$ a prime to $p$ open compact subgroup such that $\alpha-1 \in \Fp$ for all $\alpha \in \ccO_E^\times \cap U'$. Write $\overline{Y}^\prime = Y_{U^\prime}(G^\prime_\Sigma)_\bF$, $S = \widetilde{Y}_{U^\prime}(G^\prime_\Sigma)_\bF$, as defined in section \ref{PR unitary}, and let $A/S$ be the universal abelian scheme over $S$.

Write $\widetilde{\omega} = \omega_{A/S}$. Recall from section \ref{dual filtration} that we have, for every $\tau \in \widehat{\Theta}_E$, a Pappas-Rapoport filtration of $\omega^0_\tau$, whose graded pieces are locally free of rank $s_{\tau^i}$. Furthermore, the construction of the filtration on $\tilde{\theta}^c$-components is independent of the specific choice of polarization $\lambda$ in its $\ccO_{F,(p),+}^\times$-orbit. 

As in the statement of corollary \ref{filtration duality}, we have for every $\tau \in \widehat{\Theta}_E$ and $1 \leq i \leq e_\fp$, the sheaves $\widetilde{\ccH}^1_{\tau^i} = [\varpi_\fp]^{-1}\widetilde{\omega}^0_\tau(i-1) / \widetilde{\omega}^0_\tau(i-1)$, $\widetilde{\omega}^0_{\tau^i} = \widetilde{\omega}^0_{\tau}(i)/\widetilde{\omega}^0_\tau(i-1)$, and $v^0 = [\varpi_\fp]^{-1}\widetilde{\omega}^0_\tau(i-1) / \widetilde{\omega}^0_\tau(i)$, which are locally free of rank 2, $s_\tau^i$, and $2-s_{\tau^i} = s_{(\tau^c)^i}$ respectively. They sit in the short exact sequence
\[ 0 \to \widetilde{\omega}^0_{\tau^i} \to \widetilde{\ccH}^1_{\tau^i} \to \widetilde{v}^0_{\tau^i} \to 0.\]
In particular, if $\tau^i \vert_F = \theta^i \notin \Sigma_\infty$, then both $\widetilde{\omega}^0_{\tau^i}$ and $\widetilde{v}^0_{\tau^i}$ are line bundles.

Similarly to the Hilbert case, we have a free action of $\ccO_{F,(p),+}^\times/(U' \cap \ccO_F^\times)^2$ on $S$ given by $\theta_\mu(A,\iota,\lambda,\eta,\epsilon,\underline{\omega})=(A,\iota,\mu\lambda,\eta,\mu\epsilon,\underline{\omega})$ for $\mu \in \ccO_{F,(p),+}^\times$: The abelian scheme $\theta_\mu^*A$ is canonically $\ccO_D$-linearly isomorphic to $A$, and we get a $\ccO_D \otimes \ccO_{S}$-linear isomorphism $\ccH^1_{dR}(\theta_\mu^*A/S) = \theta_\mu^*\ccH^1_{dR}(A/S) \xrightarrow{\sim} \ccH^1_{dR}(A'/S')$ which preserves the filtrations. As before, by our assumption on $U'$, this isomorphism is given by the identity if $\mu \in (U^\prime \cap \ccO_F)^2$. We thus obtain for all $\tau^i \in \Theta_E$, sheaves $\ccH^1_{\tau^i}$, $\omega^0_{\tau^i}$, and $v^0_{\tau^i}$, locally free of rank $2$, $s_{\tau^i}$ and $s_{(\tau^c)^i}$ respectively. Furthermore, they sit in the short exact sequence

\[ 0 \to \omega^0_{\tau^i} \to \ccH^1_{\tau^i} \to v^0_{\tau^i} \to 0.\]
We note that whereas corollary \ref{filtration duality} provides an isomorphism $\widetilde{\omega}^0_{\tau^i} \xrightarrow{\sim} (\widetilde{v}_{(\tau^c)^i})^\vee$, this isomorphism depends on the specific choice of polarization $\lambda$ and so does not descend to the quotient.

Similarly to the Hilbert case, for suitable open compact subgroups $U^\prime_1,U^\prime_2 \subset G^\prime_\Sigma(\mathbb{A}_f)$ and $g \in G^\prime_\Sigma(\mathbb{A}_f)$ such that $g^{-1}U^\prime_1g \subset U_2$, we have the Hecke morphisms $\tilde{\rho}_g: S_1 \to S_2$ and $\rho_g: \overline{Y}_1 \to \overline{Y}_2$ induced by the quasi-isogeny $\pi_g \in \Hom_{S_1}(A_1,\tilde{\rho}_g^*A_2) \otimes \bZ_{(p)}$. These induce isomorphisms $\pi_g^*:\tilde{\rho}_g^*\widetilde{\ccH}^1_{2,\tau^i} \to \widetilde{\ccH}^1_{1,\tau^i}$ which descend to isomorphisms $\pi_g^*:\rho_g^*\ccH^1_{2,\tau^i} \to \ccH^1_{1,\tau^i}$. Furthermore, we similarly obtain isomorphisms $\pi_g^*:\rho_g^*\omega^0_{2,\tau^i} \to \omega^0_{1,\tau^i}$ and $\pi_g^*:\rho_g^*v^0_{2,\tau^i} \to v^0_{1,\tau^i}$. 
Finally, if $h \in G^\prime_\Sigma(\mathbb{A}_f)$ and $U^\prime_3$ are as above, such that $h^{-1}U_2 h \subset U_3$, then we have $\pi_{gh}^*= \pi_g^* \circ \rho_g^*(\pi_h^*)$.

\subsubsection{The Quaternionic setting}\label{quaternion bundles}\label{automorphic bundles quaternionic}

Again, let $\Sigma $ be an even set of places away from $p$, and  $U\subset G_\Sigma(\bA_f^{(p)})$ a sufficiently small open compact subgroup. Let $V_E$ be sufficiently small with respect to $U$, $U^\prime = U V_E \subset G^\prime(\mathbb{A}_f^{(p)})$, and write $\overline{Y}=Y_{U}(G_\Sigma)_{\overline{\bF}_p}$, as defined in section \ref{quaternionic models}, and $\overline{Y}^\prime = Y_{U'}(G'_\Sigma)_{\overline{\bF}_p}$.

Recall from lemma \ref{compatibility components unitary quaternionic} that we have the morphism $i: \overline{Y}\to \overline{Y}^\prime$ which identifies $\overline{Y}$ as an open and closed subscheme of $\overline{Y}^\prime$. For every $\beta = \theta^i \in \Theta_F$, we define the sheaf $\ccH^1_\beta = i^*\ccH^1_{\tilde{\theta}^i}$; it is locally free of rank two over $\overline{Y}$. Similarly, for $\beta = \theta^i \notin \Sigma$, we define the line bundles over $\overline{Y}$: $\omega_\beta = i^*\omega^0_{\tilde{\theta}^i}$ and $v_\beta = i^*v^0_{\tilde{\theta}^i}$. These bundles sit in the short exact sequence:
\[0 \to \omega_\beta \to \ccH^1_\beta \to v_\beta \to 0.\]

These vector bundles are independent of the choice of $V_E$. Furthermore, if $U_1,U_2$ are sufficiently small, we can choose $V_E$ sufficiently small for both $U_1$ and $U_2$ and set $U^\prime_i = U_i V_E$. Given $g \in G_\Sigma(\mathbb{A}_f)$ such that $g^{-1} U_1 g \subset U_2$, so that $g^{-1} U^\prime_1 g \subset U^\prime_2$, we can define the isomorphism $\pi_g^*: \rho_g^*\ccH^1_{1,\beta} \to \ccH^1_{2,\beta}$ via pulling back the corresponding isomorphism at the Unitary level by $i_1^*$. We similarly have, for $\beta = \theta^i \notin \Sigma$, isomorphisms $\pi_g^*: \rho_g^*\omega_{1,\beta} \to \omega_{2,\beta}$ and $\pi_g^*: \rho_g^*v_{1,\beta} \to \omega_{2,\beta}$, which do not depend depend on $V_E$. Finally, if $U_3$ and $h$ are such that $h^{-1} U_2 h \subset U_3$, then we have the equality of isomorphisms $\pi_{gh}^*= \pi_g^* \circ \rho_g^*(\pi_h^*)$.

We have now given two definitions for the above automorphic vector bundles on $\overline{Y}$ in the case that $\Sigma = \varnothing$. The first via descent from the corresponding sheaves on the geometric special fiber $S$ of $\widetilde{Y}_U(G)$ as a Hilbert modular variety, and the second via the restriction of the descent of sheaves on $S^\prime$, the geometric special fiber of $\widetilde{Y}^\prime_{U^\prime}(G^\prime_\emptyset)$. Recall that we obtained, in section \ref{model comparison}, the comparison of these two different models by the (base change of the) morphism $\widetilde{i}:S \to S^\prime$ given by $A \mapsto A \otimes_{\ccO_F} \ccO_E^2$. In turn, we saw that this induced, for every $\theta \in \widehat{\Theta}_F$, a $\ccO_F \otimes \ccO_F$-linear isomorphism
\[ \widetilde{i}^*\ccH^1_{\dR}(A^\prime/S^\prime)^0_{\tilde{\theta}} \xrightarrow{\sim} \ccH^1_{\dR}(A/S)_\theta,\]
compatible with filtrations and the respective action of $\ccO_{F,+,(p)}^\times$ on $S$ and $S^\prime$.
It follows that the two different models of vector bundles $\ccH^1_\beta$, $\omega_\beta$ and $v_\beta$ are canonically isomorphic over $\overline{Y}$, compatibly with the Hecke actions.

\begin{rem}
    Note that the definition of these vector bundles depends on several choices, namely the choice of $E$, the set of lifts and even of isomorphism $\ccO_{B,p} \simeq M_2(\ccO_{F,p})$. We do not consider the independence of these choices in this thesis.
\end{rem}

\section{The Goren-Oort Stratification }
In this chapter we define the Goren-Oort stratification of the special fibers of all the various integral models of Shimura varieties we have just defined. In the case that $p$ is unramified, we will recover the Goren-Oort stratification for Hilbert modular varieties introduced in \cite{Gorenoort} and the Goren-Oort stratification of Unitary/Quaternionic Shimura varieties introduced in \cite{tian_xiao_2016}. In the case that $p$ is ramified, we will recover the Goren-Oort stratification of the Pappas-Rapoport models of Hilbert modular varieties given in \cite{reduzzi2016partial}.
\subsection{The Goren-Oort Stratification of Hilbert modular varieties}\label{gohilb}

In this section, we recall the definition of the Goren-Oort stratification of Hilbert modular varieties in the ramified case as introduced by \cite{reduzzi2016partial}.

Let $F $ and $p$ be as usual and $G = \textrm{Res}_{F/\bQ} \textrm{GL}_2$. Let $U$ be a sufficiently small, in the usual sense, open compact subgroup of $G(\mathbb{A}_f)$ and let $\overline{Y} = Y_U(G)_\bF$ be the special fiber of the Pappas-Rapoport defined in section \ref{hilbert model}. We recall that from now on, we will systematically remove the PR superscript. Similarly, write $\widetilde{Y} = \widetilde{Y}_U(G_\Sigma)_{\bF}$, and let $A/\widetilde{Y}$ be the universal abelian variety. For every $\beta = \theta^i \in \Theta_F$, we defined in section \ref{automorphic bundles hilbert} the bundles 
\[\ccH^1_\beta  = [\varpi_\fp]^{-1}\omega_{A/\widetilde{Y},\theta}(i-1)/\omega_{A/\widetilde{Y},\theta}(i-1)\] and
\[\omega_{\beta} = \omega_{A/\widetilde{Y},\theta}(i)/\omega_{A/\widetilde{Y},\theta}(i-1),\]
locally free of rank two and one respectively. 

Recall the shift operation $\phi$ on $\Theta_F$ given by $\phi(\theta^i)=\theta^{i+1}$ if $i < e_\fp$ and $\phi(\theta^{e_\fp}) = (\phi \circ \theta)^1$ where $\phi \circ \theta$ is the usual action of Frobenius on $\widehat{\Theta}_{F,\fp}$. By definition of the filtration, multiplication by $\varpi_\fp$ induces a morphism $[\varpi_\fp]: \ccH^1_{\theta^i} \to \ccH^1_{\theta^{i-1}}$ for all $\theta \in \widehat{\Theta}_F$ and $i>1$. If $i= 1$ however, note that $\ccH^1_{\theta^1} = [\varpi_\fp]^{e_\fp-1}\ccH^1_{\dR}(A/\widetilde{Y})_\theta$. Given an open $W$ of $\widetilde{Y}$ and a section $s \in \ccH^1_{\theta^1}(W)$, pick a section $s' \in \ccH^1_{\dR}(A/\widetilde{Y})_\theta(W)$ such that $[\varpi_\fp]^{e_\fp-1}s'=s$. Define $V^\prime(s)$ to be the image of $V(s')$, where $V$ denotes the usual Verschiebung, in $\omega^p_{(\phi^{-1}\circ \theta)^{e_\fp}}(W)$. It is straightforward to show that this is a well defined and so yields a morphism $V^\prime:\ccH^1_{\theta^1} \to \ccH^{1 \,(p)}_{(\phi^{-1}\circ \theta)^{e_\fp}}$.

We thus have morphisms $V_{\beta}: \ccH^1_{\beta} \to \ccH^{1 (p^\delta)}_{\phi^{-1}(\beta)}$, where $\delta =1$ if $\beta = \theta^1$ and $0$ otherwise, given by :
\[
V_{\beta} = 
\begin{cases}
    [\varpi_\fp] & \textrm{if } \beta = \theta^i, i>1,\\
    V^\prime & \textrm{if } \beta = \theta^1.
\end{cases}
\]
Furthermore, the $V_\beta$ restrict to morphisms $V_\beta: \omega_\beta \to \omega^{p^\delta}_{\phi^{-1}(\beta)}$ and we call the resulting section $h_\beta \in H^0(\widetilde{Y}, \omega_{\beta}^{\otimes -1} \otimes \omega_{\phi^{-1}(\beta)}^{p^\delta})$ the partial Hasse invariant $\beta$.

For any $\beta \in \Theta_F$, we denote the vanishing locus of $h_\beta$ by $\widetilde{Y}_\beta$. For any subset $T \subset \Theta_F$ we define the closed Goren-Oort stratum $\widetilde{Y}_T = \bigcap_{ \beta \in T} \widetilde{Y}_\beta$ and open Goren-Oort stratum $\widetilde{W}_T = \widetilde{Y}_T \setminus \bigcup_{\beta \notin T} \widetilde{Y}_\beta$.

We note that although the construction of the $h_\beta$ depends on the choice of uniformizer $\varpi_\fp$, its vanishing locus does not. 

Consider now the action of $\ccO^\times_{F,(p),+}$. Whereas the line bundles $\omega_\beta$ do not necessarily descend to $\overline{Y}$ (without the assumption of $\Fp$-neatness of $U$), the line bundles $\omega_{\beta}^{\otimes -1} \otimes \omega_{\phi^{-1}(\beta)}^{p^\delta}$ do: Let $\mu = \alpha^2$ for $\alpha \in U \cap \ccO_F$; then the morphism $\theta_\mu(A,\iota,\lambda,\eta) = (A,\iota,\mu \lambda,\eta)$ is realized via $\iota(\alpha^{-1}): A \to \theta_\mu^*A = A$. The resulting action on $\omega^{-1}_\beta \otimes \omega^{(p^\delta)}_{\phi^{-1}(\beta)}$, where $\delta = 1$ if $\beta = \theta^1$ and zero otherwise, is therefore trivial. Therefore the action of $\ccO_{F,(p),+}^\times$ factors through covering group $\ccO_{F,(p),+}^\times/(U \cap \ccO_F^\times)^2$ of $\widetilde{Y} \to \overline{Y}$, and the bundles $\omega_{\beta}^{\otimes -1} \otimes \omega_{\phi^{-1}(\beta)}^{p^\delta}$ and sections $h_\beta$ descend $\overline{Y}$. We also call the resulting section $h_\beta$.

For any $\beta \in \Theta_F$, we thus define $\overline{Y}_\beta$ to be the vanishing locus $h_\beta$. For any subset $T \subset \Theta_F$ we define the closed Goren-Oort stratum $\overline{Y}_T = \bigcap_{ \beta \in T} \overline{Y}_\beta$ and the open Goren-Oort stratum $W_T = \overline{Y}_T \setminus \bigcup_{\beta \notin T} \overline{Y}_\beta$.

By \cite{reduzzi2016partial}, the subschemes $Y_T$ are smooth. Furthermore, they are proper provided that $T \neq \emptyset$. We will re-obtain this result in the subsequent sections.

Let $U_1,U_2 \subset G(\mathbb{A}_f)$ be two sufficiently small open compact subgroups and let $g \in G(\mathbb{A}^{(p)}_f)$ be such that $g^{-1} U_1 g \subset U_2$. Recall that the quasi-isogeny $\pi_g$ from section \ref{automorphic bundles hilbert} induces isomorphisms $\tilde{\rho}_g^*\omega_{2,\beta} \xrightarrow{\sim} \omega_{1,\beta}$ where $\omega_{i,\beta}$ is the corresponding line bundle on $\widetilde{Y}_i$ for any $\beta \in \Theta_F$. It is straightforward to see that these isomorphisms are compatible with the formation of Hasse invariants in the sense that the following diagram

\begin{center}
\begin{tikzcd}
    \tilde{\rho}_g^*\omega_{2,\beta} \arrow[r,"\sim"] \arrow[d,"\tilde{\rho}_g^*V_{2,\beta}"'] & \omega_{1,\beta} \arrow[d,"V_{1,\beta}"]\\
    \tilde{\rho}_g^*\omega^{p^\delta}_{2,\phi^{-1}(\beta)} \arrow[r,"\sim"] & \omega^{p^\delta}_{1,\phi^{-1}(\beta)}
\end{tikzcd}
\end{center}
commutes. In particular, for every $T \subset \Theta_F$, $\tilde{\rho}_g$ restricts to a morphism $\tilde{\rho}_g: \widetilde{Y}_{1,T} \to \widetilde{Y}_{2,T}$ and so induces a morphism $\rho_g: \overline{Y}_{1,T} \to \overline{Y}_{2,T}$.

\begin{rem}
We consider it instructive to explain what the vanishing of the Hasse invariants at a point tells us about the abelian variety $A$ it classifies. In the unramified case, this is classical: The Hasse invariants $h_\theta : \omega_\theta \to \omega_{\phi^{-1} \circ \theta}^{\otimes p}$ are given by Verschiebung. It follows that $A$ is ordinary if and only if none of the $h_\theta$ vanish, ie it lies in $U_\emptyset$. In general, $A$ lies in $Y_T$ if and only if $T$ is in the type of $A$ (see for example \cite{GorenKassaei+2012+1+63} for the definition of the type).\medskip

Consider now the ramified case. By definition the vanishing of $h_{\theta^i}$ for $i > 1$ at a point $A$ is by definition equivalent to the inclusion $[\varpi_\fp] \omega_{\theta}(i) \subset \omega_\theta (i-2)$. In particular, $A$ cannot lie in the Rapoport locus. Conversely, if for all $\theta \in \widehat{\Theta}$ and $i > 1$, none of the $h_{\theta^i}$ vanish, it is straightforward to show that $A$ lies in the Rapoport locus. Therefore the Rapoport locus is the complement of the union of all $Y_{\theta^i}$ running over all $\theta \in \widehat{\Theta}_F$ and $i > 1$. In general, the vanishing of the $h_{\theta^i}$ tells us about the position of $A$ in the singular stratification on the naive model. If we restrict ourselves to the Rapoport locus, then the situation becomes analogous to the unramified situation, that is the vanishing of the $h_{\theta^1}$ tell us about the type of $A[\mathfrak{P}]$, in particular, $W_\emptyset$ lies in the ordinary locus of $Y^{\textrm{ord}}$ of $Y$. Conversely, if $A$ is ordinary, the Serre-Tate theorem yields a canonical lift of $A$ from which we deduce that $A$ lies in the Rapoport locus. We therefore get the equality $Y^{\textrm{ord}} = W_\emptyset$.

\end{rem}

\subsection{The Goren-Oort Stratification of Unitary Shimura Varieties}

We now show how to define partial Hasse invariants on the special fibers of the Unitary models. In particular, we start by adapting the notions of essential Frobenius and Verschiebung as introduced by \cite{tian_xiao_2016} in the unramified setting. 

\subsubsection{Essential Frobenius and Essential Verschiebung}\label{Fesves}
Let $\Sigma$ and $U^\prime \subset G_\Sigma^\prime(\mathbb{A}_f)$ be as in \ref{unitary notation}. Write $\overline{Y}^\prime = Y_{U^\prime}(G^\prime_\Sigma)_\bF$ for the Pappas-Rapoport model defined in section \ref{PR unitary}, $\widetilde{Y}^\prime = \widetilde{Y}_{U'}(G'_\Sigma)_\bF$, and let $A/\widetilde{Y}^\prime$ be the universal abelian scheme. Write $\omega^0 = \omega^0_{A/\widetilde{Y}^\prime}$. Recall from section \ref{automorphic bundles unitary}, that we have for any $\beta = \tau^i \in \Theta_E$, the locally free sheaves of rank two $\ccH^1_\beta = [\varpi_\fp]^{-1}\omega^0_\tau(i-1)/\omega^0_\tau(i-1)$ and the locally free sheaves of rank $s_\beta = s_{\tau^i}$, $\omega^0_\beta = \omega^0_\tau(i)/\omega^0_\tau(i-1)$. 

For $\beta = \tau^i \in \Theta_E$ with $i > 1$, we define essential Frobenius at $\tau^i$ as:
\begin{align*}
F_{\textrm{es},\tau^i}:\ccH^1_{\tau^{i-1}} \to \ccH^1_{\tau^i}\\
x \mapsto 
\begin{cases}
x & \textrm{if } s_{\tau^{i-1}} = 1,0,\\
[\varpi_\fp]^{-1}(x) & \textrm{if } s_{\tau^{i-1}} = 2.
\end{cases}
\end{align*}
where $x$ denotes the map induced by the inclusion $[\varpi_\fp]^{-1}\omega^0_\tau(i-2) \hookrightarrow [\varpi_\fp]^{-1}\omega^0_\tau(i-1)$, and essential Verschiebung at $\tau^i$ to be 

\begin{align*}
V_{\textrm{es},\tau^i}:\ccH^1_{\tau^{i}} \to \ccH^1_{\tau^{i-1}}\\
x \mapsto 
\begin{cases}
[\varpi_\fp]x & \textrm{if } s_{\tau^{i-1}} = 1,2,\\
x & \textrm{if } s_{\tau^{i-1}} = 0,
\end{cases}
\end{align*}
Where $x$ denotes the map induced by $\omega^0_\tau(i-1)=\omega^0_\tau(i-2)$. It is straightforward to see from the definition of the filtration and the numbers $s_{\tau^i}$ that both $F_{\textrm{es}}$ and $V_{\textrm{es}}$ are isomorphisms if $s_{\tau^{i-1}} \neq 1$. If $s_{\tau^{i-1}}=1$, then it is straightforward to see that $\ker F_{\textrm{es},\tau^i} = \omega^0_{\tau^{i-1}} = \textrm{Im}\, V_{\textrm{es},\tau^i}$  and $\ker V_{\textrm{es},\tau^i} = [\varpi_\fp]^{-1}\omega^0_{\tau}(i-2)/\omega^0_\tau(i-1) = \textrm{Im}\, F_{\textrm{es},\tau^i}$.\medskip

\noindent For $i=1$, we define Essential Frobenius  at $\tau^1$ as
\begin{align*}
F_{\textrm{es},\tau^1}:\ccH^{1 \, (p)}_{(\phi^{-1} \circ \tau)^{e_\fp}} \to \ccH^1_{\tau^1}\\
x \mapsto 
\begin{cases}
F(x) & \textrm{if } s_{(\phi^{-1} \circ \tau)^{e_\fp}} = 1,0,\\
(V')^{-1}(x) & \textrm{if } s_{(\phi^{-1} \circ \tau)^{e_\fp}} = 2.
\end{cases}
\end{align*}
Similarly we define Essential Verschiebung at $\tau^1$ as
\begin{align*}
V_{\textrm{es},\tau^1}:\ccH^1_{\tau^{1}} \to \ccH^{1 \, (p)}_{(\phi^{-1}\circ\tau)^{e_\fp}}\\
x \mapsto 
\begin{cases}
V'(x) & \textrm{if } s_{(\phi^{-1} \circ \tau)^{e_\fp}} = 1,2,\\
F^{-1}(x) & \textrm{if } s_{(\phi^{-1} \circ \tau)^{e_\fp}} = 0.
\end{cases}
\end{align*}

Here, $F$ denotes the map induced by the usual Frobenius, and $V'$ is the map $x \mapsto V(x')$ for any $[\varpi_\fp]^{e_\fp-1}x'=x$, as defined in the previous section. It is straightforward to show that these maps make sense and induce isomorphisms if $s_{(\phi^{-} \circ \tau)^{e_\fp}} \neq 1$. As before, if $s_{(\phi^{-} \circ \tau)^{e_\fp}} = 1$, then  we have the relations $\ker F_{\textrm{es},\tau^1} = (\omega^0_{(\phi^{-1} \circ \tau)^{e_\fp}})^p = \textrm{Im}\, V_{\textrm{es},\tau^1}$  and $\ker V_{\textrm{es},\tau^1} = \textrm{Im}\, F \cap \ker [\varpi_\fp] = \textrm{Im}\, F_{\textrm{es},\tau^1}$\medskip

If $p$ is unramified in $F$, we obtain the essential Frobenius and Verschiebung morphisms of \cite{tian_xiao_2016}.

\subsubsection{Partial Hasse invariants and the Goren-Oort stratification}\label{unitary hasse}

\noindent For any $\beta \in \Theta_E$ and $1 \leq n \leq e_\fp f_\fp$, we define the composites 
\[
F^{\phi^n(\beta)}_{\textrm{es},\beta} : \left(\ccH^1_{\phi^{n}(\beta)} \right)^{(p^{m_n})} \xrightarrow{F^{(p^{m_{n-1}})}_{\textrm{es},\phi^{n-1}(\beta)}} \cdots \xrightarrow{F^{(p^{m_1})}_{\textrm{es},\phi(\beta)}} \left(\ccH^1_{\phi^{-1}(\beta)} \right)^{(p^{m_1})} \xrightarrow{F_{\textrm{es},\beta}} \ccH^1_\beta
\]
\[
V^{\phi^{-n}(\beta)}_{\textrm{es},\beta} : \ccH^1_\beta \xrightarrow{V_{\textrm{es},\beta}} \left(\ccH^1_{\phi^{-1}(\beta)} \right)^{(p^{m_1})}  \xrightarrow{V^{(p^{m_1})}_{\textrm{es},\phi^{-1}(\beta)}} \cdots \xrightarrow{V^{(p^{m_{n-1}})}_{\textrm{es},\phi^{-n+1}(\beta)}} \left(\ccH^1_{\phi^{-n}(\beta)} \right)^{(p^{m_n})}
\]

\noindent where the (different) $m_i$ are chosen such that the compositions make sense.

Now, let $\beta \in \Theta_E$ with $s_{\beta} =1$. Let $\phi'$ be the induced by the cycle structure on $\Theta_F \setminus \Sigma_\infty$. That is, $\phi'^{-1} ( \beta ) = \phi^{-n} (\beta)$ where $n \geq 1$ is chosen to be the smallest such that $\phi^{-n}(\beta) \vert_F \notin \Sigma_\infty$ (equivalently $s_{\phi^{-n}(\beta)}=1$). We will often use this map, defined similarly, over all of $\Theta_F$. 
Then for any $\beta$ such that $s_\beta = 1$ we have the map 
$V_{\es,\beta}^{(\phi^\prime)^{-1}(\beta)} : 
\ccH^1_\beta \to \left( \ccH^1_{\phi'^{-1}(\beta)} \right)^{(p^{m_\beta})}$ 
which, by restricting to $\omega^0_\beta$, yields a section \[h_\beta \in H^0(\widetilde{Y}', (\omega^0_\beta)^{-1} \otimes (\omega^0_{\phi'^{-1}(\beta)})^{ p^{m_\beta}}).\] 
Note that if $(\phi^\prime)^{-1}(\beta) \neq \phi^{-1}(\beta)$, then $V_{\es,\beta}^{(\phi^\prime)^{-1}(\beta)}$ factors as 
$  V_{\es, \phi((\phi^\prime)^{-1}(\beta))} \circ V^{\phi((\phi^\prime)^{-1}(\beta))}_{\textrm{es}, \beta }$ 
(where we have removed powers of $p$), where $V^{\phi((\phi^\prime)^{-1}(\beta))}_{\textrm{es}, \beta }$ is an isomorphism.

\begin{lem}\label{hasse vanish}
The Hasse invariant $h_\beta$ vanishes if and only if $h_{\beta^c}$ vanishes.
\end{lem}

\begin{proof}
Keep the notation from above. Notice first that $\ker h_\beta = \textrm{Im}\, F^{\beta}_{\textrm{es},(\phi^\prime)^{-1}(\beta)}$. So $h_\beta$ vanishes if and only if $\omega^0_\beta = \textrm{Im} \, F^{\beta}_{\textrm{es},(\phi^\prime)^{-1}(\beta)}$. The same holds with $\beta$ replaced by $\beta^c$. Consider now the commutative diagram: 

\begin{center}
\begin{tikzcd}
\left(\ccH^1_{\phi'^{-1}(\beta)} \right)^{(p^{n_\beta})} \arrow[d, "h_{\beta}", bend left] & \times & \left(\ccH^1_{\phi'^{-1}(\beta)^c}  \right)^{(p^{n_\beta})} \arrow[rr, "{\langle \cdot {,} \cdot \rangle^{p^{n_\beta}} }"] \arrow[d, "h_{\beta^c}", bend left] &  & \ccO_S \arrow[d] \\
\ccH^1_{\beta} \arrow[u, "{F^\beta_{\textrm{es},(\phi^\prime)^{-1}(\beta)}}", bend left]   & \times & \ccH^1_{\beta^c} \arrow[u, "{F^{\beta^c}_{\textrm{es},(\phi^\prime)^{-1}(\beta^c)}}", bend left] \arrow[rr, "{\langle \cdot {,} \cdot \rangle }"]              &  & {\ccO_S,}       
\end{tikzcd}
\end{center}

\noindent where the pairings were defined in \ref{dual filtration}, and under which $\omega^0_\beta$ and $\omega^0_{\beta^c}$ are orthogonal. Furthermore, one can check that $F^\beta_{\textrm{es},(\phi^\prime)^{-1}(\beta)}$ and $h_{\beta^c}$ form an adjoint pair so that we get the equality $(\textrm{Im}\, F^\beta_{\textrm{es},(\phi^\prime)^{-1}(\beta)})^\perp = \textrm{Im}\, F^n_{\textrm{es},\beta^c}$ and so \[ \omega^0_\beta = \textrm{Im}\, F^\beta_{\textrm{es},(\phi^\prime)^{-1}(\beta)}  \Leftrightarrow (\omega^0_\beta)^\perp = (\textrm{Im}\, F^\beta_{\textrm{es},(\phi^\prime)^{-1}(\beta)})^\perp \Leftrightarrow \omega^0_{\beta^c} = \textrm{Im}\, F^{\beta^c}_{\textrm{es},(\phi^\prime)^{-1}(\beta^c)}. \]
\end{proof}

Similarly as in the previous section, for any $\beta \in \Theta_E$ with $s_\beta = 1$ we call the section $h_\beta$, the partial Hasse invariant at $\beta$.

For any $\beta \in \Theta_F \setminus \Sigma_\infty$, we let $\widetilde{Y}^\prime_\beta$ denote the vanishing locus of $h_\beta$. We note by the previous lemma, that this does not depend on the choice of lift of $\beta$ hence the ambiguous notation. For any $T \subset \Theta_F \setminus \Sigma_\infty$, we define the closed Goren-Oort stratum to be $\widetilde{Y}^\prime_T = \bigcap_{\beta \in T} \widetilde{Y}^\prime_\beta$ and the open Goren-Oort stratum to be $\widetilde{W}_T = \widetilde{Y}^\prime_T \setminus \bigcup_{\beta \notin T} \widetilde{Y}^\prime_\beta$. The collection $\widetilde{W}_T$ defines a stratification, in the usual sense, on $\widetilde{Y}^\prime$.

Similarly as in section \ref{gohilb}, we find that $(U^\prime \cap \ccO^\times_F)^2$ acts trivially on the line bundles $(\omega^0_\beta)^{-1} \otimes (\omega^0_{(\phi^\prime)^{-1}(\beta)})^{p^{m_\beta}}$. Therefore, they descend to line bundles on $\overline{Y}^\prime$ and so do the sections $\tilde{h}_\beta$. 

For any $\beta \in \Theta_F \setminus \Sigma_\infty$, we let $\overline{Y}^\prime_\beta$ denote the vanishing locus of $h_\beta$. We note by the previous lemma, that this does not depend on the choice of lift of $\beta$ hence the ambiguous notation. For any $T \subset \Theta_F \setminus \Sigma_\infty$, we define the closed Goren-Oort stratum to be $\overline{Y}^\prime_T = \bigcap_{\beta \in T} \overline{Y}^\prime_\beta$ and the open Goren-Oort stratum to be $W_T = \overline{Y}^\prime_T \setminus \bigcup_{\beta \notin T} \overline{Y}^\prime_\beta$. The collection $W_T$ defines a stratification, in the usual sense, on $\overline{Y}^\prime$.

Let $U^\prime_1,U^\prime_2 \subset G^\prime_\Sigma(\mathbb{A}_f)$ be two sufficiently small open compact subgroups and let $g \in G^\prime_\Sigma(\mathbb{A}^{(p)}_f)$ be such that $g^{-1} U^\prime_1 g \subset U^\prime_2$. Recall that the quasi-isogeny $\pi_g$ from section \ref{automorphic bundles unitary} induces isomorphisms $\tilde{\rho}_g^*\omega^0_{2,\beta} \xrightarrow{\sim} \omega^0_{1,\beta}$ where $\omega_{i,\beta}$ is the corresponding line bundle on $\widetilde{Y}^\prime_i$ for any $\beta \in \Theta_E$ with $\beta \vert_F \notin \Sigma_\infty$. It is straightforward to see that these isomorphisms are compatible with the formation of Hasse invariants in the sense that the following diagram

\begin{center}
\begin{tikzcd}
    \tilde{\rho}_g^*\omega^0_{2,\beta} \arrow[r,"\sim"] 
    \arrow[d,"{\tilde{\rho}_g^*V^{(\phi^\prime)^{-1}(\beta)}_{2,\es,\beta}}"'] & \omega^0_{1,\beta} \arrow[d,"{V^{(\phi^\prime)^{-1}(\beta)}_{1,\es,\beta}}"]\\
    \tilde{\rho}_g^*(\omega^0_{2,(\phi^\prime)^{-1}(\beta)})^{p^{m_\beta}} \arrow[r,"\sim"] & \omega^0_{1,(\phi^\prime)^{-1}(\beta)})^{p^{m_\beta}}
\end{tikzcd}
\end{center}
commutes. In particular, for every $T \subset \Theta_F \setminus \Sigma_\infty$, $\tilde{\rho}_g$ restricts to a morphism $\tilde{\rho}_g: \widetilde{Y}^\prime_{1,T} \to \widetilde{Y}^\prime_{2,T}$ and so induces a morphism $\rho_g: \overline{Y}^\prime_{1,T} \to \overline{Y}^\prime_{2,T}$.

\subsubsection{Basic Properties of the Goren-Oort Strata}\label{Go properties}

We now describe the basic properties of the Goren-Oort stratification. Since $\widetilde{Y}^\prime$ is a disjoint union of finite \'{e}tale covers of $\overline{Y}^\prime$, it suffices, for our purposes, to study the Goren-Oort stratification of $\widetilde{Y}^\prime$. We start by showing that the closed strata $\widetilde{Y}^\prime_T$ are non-empty. 

In \cite[Proposition 4.7]{tian_xiao_2016}, the authors claim that they can prove the nonemptiness of the closed strata $\widetilde{Y}^\prime_T$ by using the same argument as \cite[Lemma 3.7]{Helm}. However, this argument seems to only applies to the case $\Sigma_\infty = \emptyset$, and we were not able to see how it applies to a general $\Sigma$. We thus give an alternate proof, based on the strategy of \cite{diamond2020cone}.

\begin{lem}\label{omega trivial}
Let $T \subset \Theta_F \setminus \Sigma_\infty$ be a subset and let $\beta \in T$. Write $U= \widetilde{Y}^\prime \setminus \widetilde{Y}^\prime_\beta$. Then we have isomorphisms

\[\omega^0_{\tilde{\beta}} \vert_{\widetilde{Y}^\prime_T} \simeq \left(\omega^0_{(\phi^\prime)^{-1}(\tilde{\beta}^c)}\right)^{-p^{m_\beta}} \vert_{\widetilde{Y}^\prime_T}. \]

and

\[\omega^0_{\tilde{\beta}} \vert_U \simeq \left(\omega^0_{(\phi^\prime)^{-1}(\tilde{\beta})}\right)^{p^{m_\beta}} \vert_U, \]
where $m_\beta$ was defined in the previous section. The result also holds when swapping $\tilde{\beta}$ and $\tilde{\beta}^c$.
\end{lem}

\begin{proof}
Recall that $\widetilde{Y}^\prime_T$ is given by the vanishing of the partial Hasse invariants $h_{\tilde{\beta}}$ for $\beta \in T$, and that by lemma \ref{hasse vanish}, this is the same as the vanishing of the partial Hasse invariants $h_{\tilde{\beta}^c}$. Therefore, $h_{\tilde{\beta}^c}$ vanishes on $\omega^0_{\tilde{\beta}^c}$ over $\widetilde{Y}^\prime_T$, and rank considerations show that the restriction of  
\[ 0 \to \omega^0_{\tilde{\beta}^c} \to \ccH^1_{\tilde{\beta}^c} \to \left( \omega^0_{(\phi^\prime)^{-1}(\tilde{\beta})^c}\right)^{p^{m_\beta}} \to 0 \]
to $\widetilde{Y}^\prime_T$ is exact. 

Now, recall from section \ref{bundles} that we have the following (non-canonical) isomorphisms induced by the pairing on de Rham cohomology:

\begin{center}
    \begin{tikzcd}
    
     0 \arrow[r] & \omega^0_{\tilde{\beta}} \arrow[r] \arrow[d,"\sim" {anchor = south, rotate=90}] & \ccH^1_{\tilde{\beta}}  \arrow[r] \arrow[d,"\sim" {anchor = south, rotate=90}] & v^0_{\tilde{\beta}} \arrow[r] \arrow[d,"\sim" {anchor = south, rotate=90}] & 0 \\
      0 \arrow[r] & \left(v^0_{\tilde{\beta}^c}\right)^{-1} \arrow[r] & \left( \ccH^1_{\tilde{\beta}^c} \right)^\vee \arrow[r]  & \left(\omega^0_{\tilde{\beta}^c} \right)^{-1}  \arrow[r] & 0.
    \end{tikzcd}
\end{center}

Therefore 
\[ \omega^0_{\tilde{\beta}} \otimes v^0_{\tilde{\beta}} \simeq \wedge^2 \ccH^1_{\tilde{\beta}}  \simeq \wedge^2\left( \ccH^1_{\tilde{\beta}^c} \right)^\vee  \simeq  \left(\omega^0_{\tilde{\beta}^c} \right)^{-1} \otimes \left( \omega^0_{\phi'^{-1}((\tilde{\beta})^c)}\right)^{-p^{m_\beta}},\]

and the first isomorphism follows. The second isomorphism also follows from the definition of $U$.

\end{proof}

We immediately have the two following corollaries:

\begin{cor}
Let $T \subset \Theta_F \setminus \Sigma_\infty$, then for any $\beta \in \Theta_F \setminus \Sigma_\infty$,
\[ { \omega^0_{\tilde{\beta}} }^{ \otimes (p^{2f_\fp}-1)} \vert_{\widetilde{W}_T} \simeq \ccO_{\widetilde{W}_T}. \]
\end{cor}

\begin{proof}
    Recall that $\widetilde{W}_T = \widetilde{Y}^\prime_T \setminus \bigcup_{\beta \notin T} \widetilde{Y}^\prime_\beta \subset \widetilde{Y}^\prime_T$. Therefore, if $\beta \in T$ then, by the previous lemma, we have $\omega^0_{\tilde{\beta}} \vert_{\widetilde{W}_T} \simeq \left(\omega^0_{(\phi^\prime)^{-1}(\tilde{\beta}^c)}\right)^{-p^{m_\beta}} \vert_{\widetilde{W}_T}$. We have a similar isomorphism when swapping $\tilde{\beta}$ and $\tilde{\beta}^c$.

    Similarly, if $\beta \notin T$, then $\widetilde{W}_T \subset U = \widetilde{Y}^\prime \setminus \widetilde{Y}^\prime_\beta$ and, also by the previous lemma, we have $\omega^0_{\tilde{\beta}} \vert_{\widetilde{W}_T} \simeq \left(\omega^0_{(\phi^\prime)^{-1}(\tilde{\beta})}\right)^{p^{m_\beta}} \vert_{\widetilde{W}_T}$. We have a similar isomorphism when swapping $\tilde{\beta}$ and $\tilde{\beta}^c$.

    It follows that, for any $\beta \in \Theta_F \setminus \Sigma_\infty$, by chaining the above isomorphisms we have $\omega^0_{\tilde{\beta}} \vert_{\widetilde{W}_T} \simeq  \omega^{0 \, \otimes p^{2f_\fp}}_{\tilde{\beta}}  \vert_{\widetilde{W}_T} $, and the result follows.
\end{proof}

\begin{cor}
Let $ T \subset \Theta_F \setminus \Sigma_\infty$, and $\fp$ be a prime such that $\Theta_{F,\fp} \subset T$. Then for any $\beta \in \Theta_{F,\fp}$,
\[ { \omega^0_{\tilde{\beta}} }^{ \otimes (p^{2f_\fp}-1)} \vert_{\widetilde{Y}^\prime_T} \simeq \ccO_{\widetilde{Y}^\prime_T}. \]
\end{cor}

\begin{proof}
    The result follows from the same considerations as in the previous corollary.
\end{proof}

\begin{lem}\label{Picard group calculations}
Let $\textrm{Pic}(X)_\bQ = \textrm{Pic}(\widetilde{Y}^\prime) \otimes_\bZ \bQ $ denote the rational Picard group of $\widetilde{Y}^\prime$. For any line bundle $\mathcal{L}$ on $\widetilde{Y}^\prime$, let $[\mathcal{L}]$ denote its image in $\textrm{Pic}(\widetilde{Y}^\prime)_\bQ$. Then 
\begin{itemize}
    \item Let $\beta \in \Theta_F \setminus \Sigma_\infty$, we have $-[v^0_{\tilde{\beta}^c}] = [\omega^0_{\tilde{\beta}}]   = [\omega^0_{\tilde{\beta}^c}] = - [v^0_{\tilde{\beta}}]$. 
    \item For any $\beta \in \Theta_E$, we have $[\wedge^2 \ccH^1_{\beta}]=0.$ That is, $\wedge^2 \ccH^1_{\beta}$ is torsion.
\end{itemize}
\end{lem}

\begin{proof}
The relations $[\omega^0_{\tilde{\beta}}] = -[v^0_{\tilde{\beta}^c}]$ and $[\omega^0_{\tilde{\beta}^c}] =-[v^0_{\tilde{\beta}}]$ come from the isomorphisms induced by the pairings. Now, since the Hasse invariants $h_{\tilde{\beta}}$ and $h_{\tilde{\beta}^c}$ define the same divisor, we have the equality \[ [\omega^0_{\tilde{\beta}}]-p^{m_\beta}[\omega^0_{\phi'^{-1}(\tilde{\beta})}] = [\omega^0_{\tilde{\beta}^c}]-p^{m_\beta}[\omega^0_{\phi'^{-1}(\tilde{\beta}^c)}] .\]
Let $A$ be the matrix labeled by $\beta \in \Theta_F \setminus \Sigma_\infty$ such that the $\beta$ component of $Av$, where $v$ is the vector whose $\beta$ component is $[\omega^0_{\tilde{\beta}}]$, is $[\omega^0_{{\tilde{\beta}}}]-p^{n_\beta}[\omega^0_{\phi'^{-1}({\tilde{\beta}})}]$. Then one can show that $A$ is invertible and so $ [\omega^0_{\tilde{\beta}}]   = [\omega^0_{{\tilde{\beta}}^c}]$.

Suppose now that $\beta \in \Theta_E$ lifts an element of $\Theta_F \setminus \Sigma_\infty$. By the previous discussion, we have \[ [\wedge^2 \ccH^1_{\beta}] = [\omega^0_\beta] + [v^0_\beta] = [\omega^0_\beta] - [\omega^0_{\beta^c}]=0. \]

If $\beta \in \Theta_E$ lifts an element of $\Sigma_\infty$ however, we consider two cases. If $\Theta_{F,\fp} \subset \Sigma_\infty$, then the composition of all essential Verschiebung induces an isomorphism $\wedge^2 \ccH^1_{\beta} \simeq \left( \wedge^2 \ccH^1_{\beta} \right)^{\otimes p^{f_\fp}}$ and thus $[\wedge^2 \ccH^1_{\beta}]=0$. Otherwise, essential Verschiebung induces an isomorphism $\wedge^2 \ccH^1_{\beta} \simeq \left( \wedge^2 \ccH^1_{\phi'^{-1}(\beta)} \right)^{\otimes p^{m_\beta}}$ so $[\wedge^2 \ccH^1_{\beta}] = p^{m_\beta}[\wedge^2 \ccH^1_{\phi'^{-1}(\beta)}] = 0$.

\end{proof}

\begin{lem}\label{omega rel}
The sheaf $\omega^{\textrm{rel}} := \bigotimes_{\theta \in \widehat{\Theta}_F} \bigotimes_{i=1}^{e_\fp} {\omega^0_{\tilde{\theta}^i}}^{\otimes (i-1-e_\fp)}$ is relatively ample with respect to the forgetful morphism $\pi: \widetilde{Y}^\prime \to \widetilde{Y}^{\textrm{DP}}_{U'}(G'_\Sigma)_\bF$. Furthermore, there exist integers $n_\beta$ such that $\omega^{\textrm{amp}}:=\bigotimes_{ \beta \notin \Sigma_\infty} {\omega^0_{\tilde{\beta}}}^{\otimes n_\beta}$ is ample.
\end{lem}

\begin{proof}
Since $\pi$ is projective, it suffices to check on fibers. Let $y$ be a closed point of $\widetilde{Y}^{'\textrm{DP}}_{U'}(G'_\Sigma)_\bF$ with residue field $k$, then the fiber $\pi^{-1}(y)$ is the space of all suitable Pappas-Rapoport filtrations on $\omega^0_{A_y / k}$. Fix for each $\theta \in \widehat{\Theta}_F$ a trivialization $\ccH^1_{\textrm{dR}}(A / \pi^{-1}(y))^0_{\tilde{\theta}} \simeq \ccO_{\pi^{-1}(y)}[[u]]/u^{e_\fp}$. This defines, for each $1 \leq i \leq e_\fp$, a morphism $\pi^{-1}(y) \to \textrm{Gr}_i$, where $\textrm{Gr}_i$ denotes the closed subscheme of the affine Grassmannian classifying lattices $L$ such that $u^d L_0 \subset L \subset u^{-d} L_0$ where $L_0$ is the standard lattice, by looking at the (lift of) $u^{i-e}\omega^0_{\tilde{\theta}}(i)$. The collection of these maps defines a closed immersion $\pi^{-1}(y) \to X = \prod_{\theta \in \widehat{\Theta}_F} \textrm{Gr}_1 \times \textrm{Gr}_2 \times \cdots \times \textrm{Gr}_{e_\fp}$. Furthermore, each $\textrm{Gr}_i$ has a canonical ample line bundle which pulls back to $(\omega_{\tilde{\theta}^1} \otimes \cdots \otimes \omega_{\tilde{\theta}^i})^{-1}$ on $\pi^{-1}(y)$. It follows that the product of all of these line bundles on $X$ is ample and so is its pullback to $\pi^{-1}(y)$. It also follows that $(\omega^{\textrm{rel}})^{\otimes b}$ is relatively ample for any $b \geq 1$.\medskip 

Now, we also know that $\det \omega$ is ample on $\widetilde{Y}^{'\textrm{DP}}_{U'}(G'_\Sigma)_\bF$ by \cite[Theorem 4.3.(ix)]{Faltingschai}. Therefore, by \cite[\href{https://stacks.math.columbia.edu/tag/0892}{Lemma 0892}]{stacks-project}, $(\omega^{\textrm{rel}})^{\otimes b} \otimes \left(\pi^*(\det \omega)\right)^{\otimes a}$ is ample for $a \gg 0$. Furthermore, 
\[\pi^*(\det \omega) = \bigotimes_{\theta \in \widehat{\Theta}_F}(\det \omega^0_{\tilde{\theta}} \otimes \det \omega^0_{\tilde{\theta}^c})^2 = \bigotimes_{\beta \in \Theta_F \setminus \Sigma_\infty} (\omega^0_{\tilde{\beta}} \otimes \omega^0_{\tilde{\beta}})^2 \otimes \bigotimes_{s_{\beta^\prime}=2} (\Lambda^2\ccH^1_{\beta^\prime})^2. \]

However, by the previous lemma, the $\Lambda^2\ccH^1_{\beta^\prime}$ are torsion and $\omega^0_{\tilde{\beta}}$ and $\omega^0_{\tilde{\beta}^c}$ differ by a torsion line bundle. We can thus pick $a$ such that 
\[\pi^*(\det \omega)^a =  \bigotimes_{\beta \in \Theta_F \setminus \Sigma_\infty} (\omega^0_{\tilde{\beta}})^{4a}.\]
Performing a similar analysis for $\omega^{rel}$, we can pick $b$ and $a \gg 0$ such that
\[
(\omega^{\textrm{rel}})^{\otimes b} \otimes \left(\pi^*(\det \omega)\right)^{\otimes a} =  \bigotimes_{\beta \notin \Sigma_\infty}  {\omega^0_{\beta}}^{\otimes 4a+b m_\beta}, 
\]
for the appropriate $m_\beta \geq 0$, is ample. The result follows.

\end{proof}

\begin{cor}
The open strata $\widetilde{W}_T$ are quasi-affine. In particular, for each irreducible component $C$ of $\widetilde{Y}^\prime$, there exists $\beta$ such that $\widetilde{Y}^\prime_\beta \cap C \neq \emptyset$.
\end{cor}

\begin{proof}
Since $\widetilde{W}_T$ is (a disjoint union of) quasi-compact schemes, it suffices to show that there exist a torsion ample line bundle over it. By lemma \ref{omega rel}, the sheaf $\omega^{rel}$ is ample on $\widetilde{W}_T$ and by lemma \ref{omega trivial}, it is torsion. The open strata $\widetilde{W}_T$ are therefore quasi-affine.

For $\widetilde{Y}^\prime$ proper, which is the case if $\Sigma_\infty \neq \emptyset$, the second statement follows immediately from the fact that $W_\emptyset \cap C = C$ would imply that $C$ is finite, as it is quasi-affine and proper, a contradiction. 

In the case that $\Sigma_\infty  = \emptyset$, we can prove this by defining suitable toroidal compactifications of $\widetilde{Y}^\prime$, for example as in \cite[section 7]{Diamond_2023}, and extending the Goren-Oort stratification there.
\end{proof}

\begin{cor}\label{non empty goren strata}
The closed strata $\widetilde{Y}^\prime_T$ are non-empty and intersect every component $C$ of $\widetilde{Y}^\prime$. The same is true of the strata $Y^\prime_T$.
\end{cor}

\begin{proof}
By the previous corollary there exists $\beta \in \Theta_F \setminus \Sigma_\infty$ such that $C \cap \widetilde{Y}^\prime_\beta$ is non empty. 
Suppose that there for all $\beta^\prime \neq \beta$, $C \cap \widetilde{Y}^\prime_\beta \cap \widetilde{Y}^\prime_{\beta^\prime}$ is empty. 
Then $C \cap \widetilde{Y}^\prime_\beta = C \cap \widetilde{W}_\beta$, and we conclude as before that this is a contradiction (We will show in the next proposition that $\widetilde{Y}_\beta$ is in fact proper itself). 
Therefore, by induction, $\widetilde{Y}^\prime_{\Theta_F \setminus \Sigma_\infty} \cap C$ is non empty. 
Since $\widetilde{Y}^\prime_{\Theta_F \setminus \Sigma_\infty} \subset \widetilde{Y}^\prime_T$ for all $T$, the result follows. 
\end{proof}

\begin{prop}\label{strata proper}
The closed strata $Y^\prime_T$ are non empty and smooth of codimension $\vert T \vert$. Furthermore, if $T \sqcup \Sigma_\infty \neq \emptyset$, then $Y^\prime_T$ is proper.
\end{prop}

\begin{proof}

As usual, it suffices to obtain analogous results for the $\widetilde{Y}^\prime_T$. We have just shown that the $\widetilde{Y}^\prime_T$ are non-empty and we can easily adapt the proof of proposition \ref{unitary smooth} to show that they are smooth of codimension $\vert T \vert $.

Note that the $\widetilde{Y}^\prime_T$ are not proper since they are not quasi-compact (they are an infinite union of varieties), but we will show that they satisfy the valuative criterion of properness if $T \sqcup \Sigma_\infty \neq \emptyset$. In particular, we show how the proof \cite[Proposition 4.7]{tian_xiao_2016} extends to our situation. For the sake of completeness, we write out the entire argument.

Let $R$ be a DVR with fraction field $K$ and $x_K = \underline{A_K}$ be a $K$-point of $\widetilde{Y}_T$. We need to show that $\widetilde{Y}^\prime_K$ extends to a unique $R$-point (up to a finite extension $M$ of $K$). 
By Grothendieck's semi-stable reduction theorem, we may assume that $A_K$ has semi-stable reduction. Let $A_R$ denote the N\'{e}ron model of $A_K$, $\overline{A}$ its special fiber, and $\mathbb{T}$ its torus part. Consider the rational cocharacter group $X_{*}(\mathbb{T})_\bQ = \Hom ( \mathbb{G}_m, \mathbb{T}) \otimes_\bZ \bQ$. It is a $\bQ$-vector space of dimension at most $\dim(\overline{A}) = \frac{1}{2}\dim(D)$ and comes with an action of $D \simeq \textrm{M}_2(E)$. It must then either be 0 or isomorphic to $E^2$. In the latter case, we have an isomorphism $X_{*}(\mathbb{T}) \otimes K \simeq \textrm{Lie}(A_K)$ and the trace of any element $\alpha \in \ccO_E$ is given by $2 \sum_{\beta \in \Theta_E} \beta(\alpha)$. Hence $\Sigma_\infty = \emptyset$ by the Kottwitz condition.\medskip 

Hence if $\Sigma_\infty \neq \emptyset$, $A_R$ is an abelian variety and the polarization and level structure extend uniquely by canonicity of the N\'{e}ron model. Thus we obtain a point of $\widetilde{Y}^{\textrm{DP}}$. It only rests to show now that there is a unique filtration on $\omega^0_{A_R/R}$ extending the one on $\omega^0_{A_K/K}$.\medskip

Fix $\tau \in \widehat{\Theta}_E$. Then $H^1_{\textrm{dR}}(A_R / R)^0_{\tau}$ is free of rank $2e_\fp$ over $R$. Write $\textrm{Fil}^i_K = \omega^0_{A_K/K,\tau} \subset H^1_{\textrm{dR}}(A_R / R)^0_{\tau}[1/t]$ where $t$ denotes a uniformizer of $R$. Consider $\textrm{Fil}^{i} := \textrm{Fil}_K^{i} \cap \omega^0_{A_R/R}$. Then $\textrm{Fil}^i[1/t] = \textrm{Fil}_K^{i}$ is free (over $R$) of rank $s_{\tau}(i)$, 
\[ ([\varpi_\fp]-\tau^i(\varpi_\fp))\textrm{Fil}^{i} \subset \textrm{Fil}_K^{i-1} \cap \omega^0_{A_R/R} = \textrm{Fil}^{i-1},\] and $\textrm{Fil}^{i}/\textrm{Fil}^{i-1}$ clearly has no $t$-torsion so is free of rank $s_{\tau^i}$. This defines a suitable lift of $\textrm{Fil}_K^{i}$. Suppose now that ${\textrm{Fil}^{i}}'$ is another lift. Then since they both localize to $\textrm{Fil}_K^{i}$ we have, without loss of generality, ${\textrm{Fil}^{i}}' = t^n \textrm{Fil}^{i}$ for some $n \geq 0$. However, since $\textrm{Fil}^{i}/t \neq 0$, we have $n=0$\medskip

Hence $x_K$ extends uniquely to an $R$-point of $\widetilde{Y}^\prime$ and because $\widetilde{Y}^\prime_T$ is closed, it lands in $\widetilde{Y}^\prime_T$. $\widetilde{Y}^\prime_T$ is thus proper. \medskip

Suppose now that $\Sigma_\infty = \emptyset$. If $X_*(\mathbb{T})_\bQ \simeq E^2$, then Verschiebung is an isomorphism and so $x_L$ cannot lie in any $\widetilde{Y}^\prime_T$. Thus if furthermore $T \neq 0$, $X_*(\mathbb{T})_\bQ = 0$ and we conclude as before. 
\end{proof}

\begin{rem}
Note that we showed that each $\widetilde{Y}^\prime_T$ satisfies the valuative criterion of properness. We stress again that they are not however proper; they are an infinite disjoint union of projective schemes.
\end{rem}

\subsection{The Goren-Oort stratification of Quaternionic Shimura varieties}\label{quaternionic goren oort}
For our purposes, we will only need to define the Goren-Oort stratification on the geometric special fiber. Keep the notation from the previous section. Let $U \subset G(\mathbb{A}_f)$, $V_E \subset (\mathbb{A}_{E,f}^{(p)})^\times$ and $U^\prime \subset G^\prime(\mathbb{A}_f)$ be as in lemma \ref{compatibility components unitary quaternionic} and write $\overline{Y} = Y_{U}(G_\Sigma)_{\overline{\bF}_p}$ and $\overline{Y} = Y_{U^\prime}(G^\prime_\Sigma)_{\overline{\bF}_p}$.

Lemma \ref{compatibility components unitary quaternionic} yields a Hecke equivariant embedding $i: \overline{Y} \to \overline{Y}^\prime$ which identifies $\overline{Y}$ with an open and closed subscheme of $\overline{Y}^\prime$. For any subset $T \subset \Theta_F \setminus \Sigma_\infty$, we thus set
\[ \overline{Y}_T = \overline{Y} \times_{\overline{Y}^\prime} \overline{Y}^\prime_T.\]
We note that this definition is independent of the choice $V_E$.
Since $i$ is Hecke equivariant and the Hecke action of $G^\prime_\Sigma(\mathbb{A}^{(p)}_f)$ on the $\overline{Y}^\prime$, varying over $U^\prime$, respects the stratification, it follows that the action of $G_\Sigma(\mathbb{A}^{(p)}_f)$ on the $\overline{Y}$ restricts to an action on the $\overline{Y}_T$, varying over $U$.

Recall that in the case that $\Sigma = \emptyset$, we have two models: $\overline{Y} =  Y_U(G)_{\overline{\bF}_p}$ viewed as a Hilbert modular variety and $\overline{Y}_\emptyset = Y_{U_\emptyset}(G_\emptyset)_{\overline{\bF}_p}$ as above.
Furthermore, in section \ref{model comparison} we defined a morphism $\tilde{i}: \widetilde{Y}_U(G)_{\overline{\bF}_p} \to \widetilde{Y}_{U_\emptyset^\prime}(G^\prime_\emptyset)_{\overline{\bF}_p}$ by sending $A$ to $A^\prime  = A \otimes_{\ccO_F} \ccO_E^2$ and setting for each $\tau \in \widehat{\Theta}_E$
\[\omega^0_{A^\prime,\tau}(i) = \omega_{A,\theta}(i) \otimes_{\ccO_F} \Hom(\ccO_E,\ccO_F)_\tau.\]
Furthermore, we showed that this induced an isomorphism 
\[\tilde{i}: \widetilde{Y}_U(G)_{\overline{\bF}_p} \to \widetilde{Y}_{U^\prime}(G^\prime)_{\overline{\bF}_p} \times_{C^\prime} C,\]
where $C^\prime = C_{\nu^\prime(U^\prime)}$ and $C = C_{\det U}$ index connected components of $\overline{Y}^\prime_\emptyset$ and $\overline{Y}$ respectively. Now, it is straightforward to see that under the isomorphisms $\tilde{i}^*\ccH^1_{A^\prime,\tau^i} \xrightarrow{\sim} \ccH^1_{A,\theta^i}$, the partial Hasse invariant $h_{\tau^i}$ is sent to the partial Hasse invariant $h_{\theta^i}$. Therefore, for any $T \subset \Theta_F$, we obtain isomorphisms $\tilde{i}: \widetilde{Y}_U(G)_{\overline{\bF}_p,T} \xrightarrow{\sim} \widetilde{Y}^\prime_{U^\prime}(G^\prime)_{\overline{\bF}_p,T} \times_{C^\prime} C$ and in particular we obtain the isomorphism 
\[i: \overline{Y}_T \xrightarrow{\sim} \overline{Y}^\prime_{\emptyset,T} \times_{C^\prime} C.\]
The two definitions of the Goren-Oort stratification agree.

\begin{rem}
   In the case of totally indefinite Quaternion algebras $B_\Sigma/F$, ie $\Sigma_\infty = \emptyset$, one may define integral models $Y_\Sigma$ as in \cite[$\mathsection$2.5]{liu:hal-02967482}. That is, analogously to the Hilbert case, as a coarse moduli space of polarized abelian varieties of dimension $2d$ with an action of $\ccO_B$ such that the characteristic polynomial of the action of $\alpha \in \ccO_B$ is given by $N_{F/\bQ}(N^0_{B/F}(T-\alpha))$ where $N^0_{B/F}$ is the reduced norm. One may also define a Pappas-Rapoport model by introducing a filtration on $\omega^0_{A/S}$, where the reduced component is taken with respect to the isomorphism $\ccO_{B,p} \simeq M_2(\ccO_{F,p})$.
   We obtain, mutatis mutandis, the same comparison to the Unitary Shimura variety as in section \ref{model comparison}, where the map $i$ is now given by $A \mapsto A \otimes_{\ccO_F} \ccO_E$.

   One may also define the Goren-Oort stratification of the special fiber $\overline{Y}_\Sigma$ as in section \ref{gohilb} which agrees with the stratification defined above. In this situation however, we can prove that $\overline{Y}_\Sigma$ itself, not just the strata, is proper if $\Sigma \neq \emptyset$: Repeating the argument from proposition \ref{strata proper} we see that the rational cocharacter group $X_*(\mathbb{T})_\bQ$ of the Torus part of the special fiber of the N\'{e}ron model of $A$ has dimension at most $2d$. However it comes with an action of $B$ which is a division algebra, since $\Sigma \neq \emptyset$, of dimension $4d$ over $\bQ$. It follows that $\mathbb{T}$ is trivial. 
\end{rem}

\section{Iwahori Level Structures}
In this chapter we define Iwahori models of the various Shimura varieties defined in chapter \ref{tame}. In particular, we will be considering Shimura varieties at level $U_0(\mathfrak{P})$ where $\mathfrak{P}= \fp_1 \cdots \fp_n$ is a product of distinct primes $\fp_i \vert p$ of $F$. In the case of Hilbert modular varieties, these have already been extensively studied in \cite{diamond2021kodairaspencer}.

\subsection{The Hilbert setting}\label{hilbert iwahori}
\subsubsection{The moduli problem}
Keep the notation from section \ref{hilbert model}. Let $\mathfrak{P} = \fp_1 \cdots \fp_n$ be the product of distinct primes $\fp_i \vert p$. Set $f_\Fp = f_{\fp_1} \cdots f_{\fp_n}$, $d_{\fp} = \sum_{\fp \vert \Fp} e_\fp d_\fp$ and recall that for each prime $\fp_i$, we have an element $\varpi_{\fp_i} \in \ccO_F$ such that $v_{\fp'}(\varpi_{\fp_i})=1$ if $\fp'=\fp_i$ and $0$ otherwise. We set $\varpi_{\mathfrak{P}} = \varpi_{\fp_1} \cdots \varpi_{\fp_n}$. For a prime $\fp \vert p$, define the Iwahori subgroup 
\[I_0(\fp)=\{g \in \textrm{GL}_2(\ccO_{F,\fp}) \, \vert \, g \equiv \left(\begin{smallmatrix} * & * \\ 0 & * \end{smallmatrix}\right) \, \textrm{mod} \, \fP \}.\]

For a sufficiently small open compact subgroup $U =U^p U_p \subset G(\mathbb{A}_f)$ with $U_p = GL_2(\ccO_{F,p})$ we set 
\[U_0(\Fp) = \{ g \in U \, \vert \, g_\fp \in I_0(\fp) \textrm{ for all } \fp \vert \Fp \}.\]

Consider the functor which associates to any $\ccO$-scheme $S$ the isomorphism classes of tuples $(\underline{A_1},\underline{A_2},f,g)$ where $\underline{A_i}$ define elements of $\widetilde{Y}_U(G)(S)$ and $f:A_1 \to A_2$, $g:A_2 \to A_1\otimes\mathfrak{p}$ are isogenies of degree $p^{f_\Fp}$ such that
\begin{itemize}
\item $f$ and $g$ are $\ccO_F$-linear,
\item $H = \ker f \subset A_1[\fp]$ is totally isotropic with respect, and $f^\vee \circ \lambda_2 \circ f = \lambda_1 \circ \iota_1( \varpi_\fp)$,
\item $f \circ \eta_1 = \eta_2$ as $U^p$-orbits on each connected components,
\item $f$  and $g$ respect the Pappas-Rapoport filtrations on $A_1$ and $A_2$, that is for every $\theta \in \widehat{\Theta}_F$ and $1 \leq i \leq e_\fp$, $f^* \omega_{2,\theta}(i) \subset \omega_{1,\theta}(i)$ and $g^* \omega_{A_1 \otimes \fP^{-1},\theta}(i) \subset \omega_{2,\theta}(i)$,
\item The composition $g \circ f: A_1 \to  A_1\otimes \fp^{-1}$ is equal to the canonical isogeny $\iota_{1,\fP^{-1}}$ induced by the inclusion of $\ccO_F$ into $\fP^{-1} $.
\end{itemize}

As in \cite{diamond2021kodairaspencer}, this functor is representable by a scheme $\widetilde{Y}_{U_0(\Fp)}(G)$, syntomic over $\ccO$ of relative dimension $d$. Furthermore, both of the forgetful morphisms to $\widetilde{Y}_U(G)$ are projective. 

We have a natural action of $\ccO_{F,(p),+}^\times$ on $\widetilde{Y}_{U_0(\fp)}(G)$ by acting on the polarizations of both $A_1$ and $A_2$, which again factors through a free action of $\ccO_{F,(p),+}^\times / (U \cap F^\times)^2$. We denote the quotient by $Y_{U_0(\fp)}(G)$. This is again a quasi-projective flat local complete intersection over $\ccO$ of constant relative dimension $d$. Note that for a different choice of uniformizer $\varpi_\Fp^\prime$ we obtain a different model $\widetilde{Y}^\prime_{U_0(\fp)}(G)$, this model is however isomorphic to the one given by $\varpi_\fp$ via the morphism induced by $(\underline{A_1},\underline{A_2},f,g) \mapsto (\underline{A^\prime_1},\underline{A^\prime_2},f^\prime,g^\prime)$ where all the data is the same except for the quasi polarization $\lambda^\prime_1 = \alpha^{-1} \lambda_1$, where $\alpha = \varpi_\Fp / \varpi_\Fp^\prime \in \ccO_{F,(p),+}^\times$. Therefore, the quotient $Y_{U_0(\fp)}(G)$ and the forgetful morphisms $\pi_i:Y_{U_0(\fp)}(G) \to Y_{U}(G)$ for $i=1,2$ are independent of this choice.

\begin{rem}
Note that the definition of $\widetilde{Y}_{U_0(\Fp)}(G)$ given in \cite[$\mathsection$ 2.4]{diamond2021kodairaspencer} does not include $g$. Since $g$ is uniquely determined by $\underline{A_1}$ and $f$, we only include it in the definition because we require it to preserve the filtrations. However, it is shown in loc cit. that this is automatically the case. We will continue to include $g$ in our notation, despite being redundant, as we will often refer to it in the following.
\end{rem}

Finally, we remark that the tower $Y_{U_0(\Fp)}(G)$ (ranging over $U^p$) comes with a natural Hecke action of $G(\mathbb{A}^{(p)}_f)$, induced by the natural action of $G(\mathbb{A}^{(p)}_f)$ on $\widetilde{Y}_{U_0(\Fp)}(G)$. This action is compatible with the forgetful morphisms $\pi_i$, for $i=1,2$ and at the level of complex points we have
\[Y_{U_0(\Fp)}(G)(\bC) \simeq \textrm{GL}_2(F)_+ \backslash \ccH^{\Theta_F} \times \textrm{GL}_2(\mathbb{A}_f) / U_0(\Fp).\]

We finish with the following straightforward proposition. Although some care is needed with regards to the Pappas-Rapoport filtrations, the proof follows almost verbatim from the unramified case.
\begin{prop}\label{iwahori fiber product}
    Let $\Fp = \fp_1 \cdots \fp_n$ be a product of distinct primes of $F$ over $p$. Write $Y = Y_U(G)$, $Y_{U_0(\Fp)}(G) = Y_0(\Fp)$, and $Y_{U_0(\fp_i)}(G) = Y_0(\fp_i)$. We have a Hecke equivariant isomorphism over $Y$
    \[ Y_0(\Fp) \xrightarrow{\sim} Y_0(\fp_1) \times_Y \cdots \times_{Y} Y_0(\fp_n),\]
    where the fiber product is taken with respect to the natural projection induced by $(A_1,A_2,f,g) \mapsto A_1$.
\end{prop}

\subsubsection{Stratifications in the Hilbert Setting}\label{hilbert strata}

In this section we introduce a stratification of the special fiber of $Y_{U_0(\fp)}(G)$ as in \cite{sasaki} by descending it from one on $\widetilde{Y}_{U_0(\fp)}(G)$.\medskip

Let $G$, $U$ and $\Fp$ be as before and set $\Theta_{F,\Fp} = \bigcup \Theta_{F,\fp}$. Write $S = \widetilde{Y}_{U_0(\Fp)}(G)_\bF$ and $\overline{Y}_0(\Fp) = Y_{U_0(\Fp)}(G)_\bF$. Let $(\underline{A_1},\underline{A_2},f,g)$ denote the universal object over $S$ and for any $\beta \in \Theta_F$ write $\omega_{1,\beta} = \omega_{A_1/S,\beta}$ and $\omega_{2,\beta} = \omega_{A_2/S,\beta}$, as defined in section \ref{automorphic bundles hilbert}. Since $f$ and $g$ respect the filtrations, we have morphisms 
\[f^*_\beta :\omega_{2,\beta} \to \omega_{1,\beta} \qquad \textrm{and} \qquad g^*_\beta: \omega_{A_1 \otimes \fP^{-1},\beta} \to \omega_{2,\beta}.\]

\noindent Note that for $\beta \notin \Theta_{F,\Fp}$, $f^*_\beta$ and $g^*_\beta$ are isomorphisms. For any two subsets $I,J \subset \Theta_{F,\Fp}$, we let $S_{\phi(I),J}$ denote the closed subscheme of $S$ given by the vanishing of the sections
\[\{ f^*_\beta \, \vert \, \beta \in I \} \cup \{ g^*_\beta \, \vert \, \beta \in J \}.\]
We immediately remark that we use the shift $\phi(I)$ in our notation in order to relate these closed subschemes to the Goren-Oort stratification downstairs. Note that since $f^*_\beta$ and $g^*_\beta$ form an exact pair, that is $\ker f^*_\beta = \textrm{Im} g^*_\beta$ and $\ker g^*_\beta = \textrm{Im} (f \otimes \Fp^{-1})^*_\beta$, that the vanishing of $g^*_\beta$ is equivalent to the vanishing of $f^{* \prime}_\beta : v_{2,\beta} \to v_{1,\beta}$, where we recall that $v_{i,\beta} = \ccH^1_{i,\beta} / \omega_{i,\beta}$. Indeed, $g^*_\beta \omega_{A_1 \otimes \Fp^{-1},\beta} = 0$ implies that $f^*_\beta \ccH^1_{2,\beta} = \omega_{1,\beta}$ hence $f^{*\prime}_\beta v_{2,\beta} = 0$ and vice versa.\medskip

For a closed point $P$ of $S$, let $I_P$ be the set of $\beta$ such that $f^*_\beta$ vanishes at $P$, and $J_P$ the set of $\beta$ such that $g^*_\beta$ vanish at $P$. Then $I_P \cup J_P = \Theta_{F,\Fp}$ and $J_P$ is also the set of $\beta$ such that $f^{* \prime}_\beta$ vanish at $P$. After taking trivializations of the relevant bundles in a neighbourhood $U$ of $P$, we obtain the following description of the completed local ring $\widehat{\ccO}_{S,P}$ from  \cite[$\mathsection$2.5]{diamond2021kodairaspencer}

\begin{thm}
Let $P$ be a closed point of $S$. Then there is an isomorphism 
\[ \widehat{\ccO}_{S,P} \simeq k_P \llbracket X_\beta, X^\prime_\beta \rrbracket_{\beta \in \Theta_F} / \langle g_\beta \rangle_{\beta \in \Theta_F} \]
where $f^*_\beta \mapsto X_\beta$, $f^{* \prime}_\beta \mapsto X^\prime_\beta$ and
\[ g_\beta = 
\begin{cases}
X_\beta  & \textrm{ if } \beta \in I_P \setminus J_P,\\
X^\prime_\beta &\textrm{ if } \beta \in J_P \setminus I_P,\\
X_\beta X^\prime_\beta &\textrm{ if } \beta \in I_P \cap J_P,\\
X_\beta - \beta(\varpi_\Fp)X_\beta^ \prime & \textrm{ if } \beta \notin \Theta_{F,\Fp}.
\end{cases}\]
\end{thm}

As in \cite[ $\mathsection$4.3.2]{2020arXiv200100530D} we have the following:

\begin{cor}
$S_{\phi(I),J}$ is a reduced local complete intersection of pure dimension $d- \vert I \cap J\vert$. Furthermore, it is smooth if $I \cup J = \Theta_{F,\Fp}$. 
\end{cor}

Now, the natural action of $\ccO_{F,(p),+}^\times$ on the line bundles 
\[\ccH om_{\ccO_S}\left(\omega_{2,\beta},\omega_{1,\beta}\right) \quad \textrm{and} \quad \ccH om_{\ccO_S}\left(\omega_{A_1 \otimes \Fp^{-1},\beta},\omega_{2,\beta}\right)\] 
factors through $\ccO_{F,(p),+}^\times/(\ccO_F \cap U)^2$ so that the above line bundles and the accompanying sections $f^*_\beta$ and $g^*_\beta$ all descend to $\overline{Y}_0(\Fp)$.

This allows to define the closed subschemes $\overline{Y}_0(\fp)_{\phi(I),J}$ of $\overline{Y}_0(\fp)$ as the vanishing of the appropriate sections or, equivalently, as the quotient of by the $\ccO_{F,(p),+}^\times$-action of $S_{\phi(I),J}$. It follows from above that the reduced closed subschemes $\overline{Y}_0(\fp)_{\phi(I),J}$ satisfy the same relations as the $S_{\phi(I),J}$. It is also straightforward to check that the stratification is compatible with the Hecke action for varying $U$. That is, if we take suitable $U_1$ and $U_2$ and $g \in GL_2(\bA_{F,f}^{(p)})$ such that $g^{-1}U_1g \subset U_2$, then the morphisms $\tilde{\rho}_g$ and $\rho_g$ restrict to morphisms on the corresponding closed subschemes for any $I,J\subset \Theta_{F,\Fp}$.

\begin{prop}
Let $I,J \subset \Theta_{F,\Fp}$ and write for each prime $\fp_i \vert \Fp$, $I_{\fp_i}, J_{\fp_i} = I \cap \Theta_{F,\fp_i}, J \cap \Theta_{F,\fp_i}$. The isomorphism $\overline{Y}_0(\Fp) \xrightarrow{\sim} \overline{Y}_0(\fp_1) \times_{\overline{Y}} \cdots \times_{\overline{Y}} \overline{Y}_0(\fp_n)$ restricts to a Hecke equivariant isomorphism 
\[\overline{Y}_0(\Fp)_{\phi(I),J} \simeq \overline{Y}_0(\fp_1)_{\phi(I_{\fp_1}),J_{\fp_1}} \times_{\overline{Y}} \cdots \times_{\overline{Y}} \overline{Y}_0(\fp_n)_{\phi(I_{\fp_n}),J_{\fp_n}}.\]
\end{prop}

\subsection{The Unitary setting}
We now define Iwahori level Shimura varieties and their stratifications in the Unitary setting.
\subsubsection{The moduli problem}\label{iwahori unitary def}
As usual, let $\Sigma$ and $G^\prime_\Sigma$ be as in section \ref{unitary moduli}. Let $\mathfrak{P} = \fp_1 \cdots \fp_n$ be the product of distinct primes $\fp_i \vert p$ of $F$ and let $\mathfrak{Q} = \fq_1 \cdots \fq_n$ so that $\Fp \ccO_E = \mathfrak{Q}\mathfrak{Q}^c$. Set $f_\Fp = f_{\fp_1} \cdots f_{\fp_n}$, $d_{\fp} = \sum_{\fp \vert \Fp} e_\fp d_\fp$ and $\varpi_{\mathfrak{P}} = \varpi_{\fp_1} \cdots \varpi_{\fp_n}$. Recall that we fixed an isomorphism $\ccO_{D,p} \simeq M_2(\ccO_{E,p})$. Via this isomorphism, define for each prime $\fp \vert p$, the Iwahori subgroup $I_0(\fp) \subset G^\prime_\Sigma$ 
\[I_0(\fp)=\{g \in \textrm{GL}_2(\ccO_{E,\fp}) \, \vert \, g \equiv \left(\begin{smallmatrix} * & * \\ 0 & * \end{smallmatrix}\right) \, \textrm{mod} \, \fp\ccO_E \}.\]

For a sufficiently small open compact subgroup $U^\prime = (U^\prime)^pU ^\prime_p \subset G_\Sigma(\mathbb{A}_f)$ with $U^\prime = GL_2(\ccO_{E,p})$, we set 
\[U^\prime_0(\Fp) = \{ g \in U^\prime \, \vert \, g_\Fp \in I_0(\fp) \textrm{ for all } \fp \vert \Fp \}.\]
Hence if we write $U^\prime=(U^\prime)^p U^\prime_p$ with $U^\prime_p = GL_2(\ccO_{E,p})$, we have $U^\prime_0(\Fp)= U^p I_0(\Fp)$

Consider the functor which associates to every locally Noetherian scheme $S/\ccO$ the set of isomorphism classes of tuples $(\underline{A_1},\underline{A_2},f,g)$ where 
\begin{itemize}
    \item $\underline{A_1}$ and $\underline{A_2}$ define points of $\widetilde{Y}_{U^\prime}(G^\prime_\Sigma)(S)$,
    \item $f:A_1 \to A_2$ and $g:A_2 \to A_1 \otimes \Fp^{-1}$ are $\ccO_D$-linear isogenies of degree $p^{4f_\Fp}$, with $\ker f \subset A_1[\Fp]$ and $\ker g \subset A_2[\Fp]$, whose composition $g \circ f: A_1 \to A_1 \otimes \Fp^{-1}$ is equal to the canonical map $\iota_{\Fp^{-1}}$ induced by the inclusion $\ccO_E \hookrightarrow \Fp^{-1}$,
    \item $H = \ker f \subset A_1[p]$ is totally isotropic and splits as $\prod_{\fp \vert \Fp} H_{\fq} \oplus H_{\fq^c} \subset A_1[\fq] \oplus A_1[\fq^c]$ where each subgroup has degree $p^{2f_\fp}$ and $f^\vee \circ \lambda_2 \circ f =  \lambda_1 \circ \iota_1( \varpi_\Fp)$,
    \item $f \circ \eta_1$ as $(U^\prime)^p$-orbits and $\epsilon_2 = \varpi_{\Fp} \epsilon_2$,
    \item $f$ respects the Pappas-Rapoport filtration on $\tilde{\theta}$ components. That is, for every $\theta \in \widehat{\Theta}_F$ and $1 \leq i \leq e_\fp$, $f^* \omega^0_{2,\tilde{\theta}}(i) \subset \omega^0_{1,\tilde{\theta}}(i)$.
\end{itemize}

Standard arguments show that this functor is representable by a scheme $\widetilde{Y}_{U^\prime_0(\Fp)}(G_\Sigma^\prime)$ over $\ccO$ such that both forgetful morphisms to $\widetilde{Y}_{U^\prime}(G^\prime_\Sigma)$ are projective. As before, $\widetilde{Y}_{U^\prime_0(\Fp)}(G^\prime_\Sigma)$ comes with a natural action of $\ccO_{F,(p),+}^\times$, acting diagonally on $(\underline{A_1},\underline{A_2})$, that factors through a free action of $\ccO_{F,(p),+}^\times/ (\ccO_F \cap U^\prime)^2$. We write $Y_{U'_0(\fp)}(G^\prime_\Sigma)$ for the quotient. Again, $Y_{U'_0(\fp)}(G'_\Sigma)$ and the natural forgetful morphisms $\pi_i: Y_{U'_0(\fp)}(G^\prime_\Sigma) \to Y_{U'}(G^\prime_\Sigma)$ for $i=1,2$ are independent of the choice of $\varpi_\Fp$. It is straightforward to check that $Y_{U'_0(\fp)}(G^\prime_\Sigma)$ comes with a Hecke action of $G^\prime_\Sigma(\mathbb{A}_f^{(p)})$ such that the forgetful morphisms are compatible with the Hecke action downstairs.

\noindent Finally, at the level of complex points, we have 
\[Y_{U'_0(\fp)}(G'_\Sigma)(\bC) \simeq G^\prime(\bQ)_+ \backslash \ccH^{\Theta_F \setminus \Sigma_{\infty}} \times G^\prime_\Sigma(\mathbb{A}_f) / U^\prime_0(\Fp). \]

Recall that given an $S$-point $\underline{A}$ of $\widetilde{Y}_{U'}(G'_\Sigma)(S)$ and any $\tau = \tilde{\theta}$, we required the filtrations on $\omega^0_{\tau}$ and $\omega^0_{\tau^c}$ to be dual. In practice, via the formula \ref{dual formula}, this means that \[\omega^0_{\tau^c}(i)=t_{\tau^c,i}\omega^0_{\tau}(i)^\perp,\]
where $t_{\tau^c,i} = \prod_{j > i} ([\varpi_\fp]-(\tau^i)^c(\varpi_\fp))$. 

\begin{lem}\label{respect filtrations}
Let $S$  be a locally Noetherian scheme over $\ccO$ and $(\underline{A_1},\underline{A_2},f,g)$ be a tuple such as above, minus the assumption that $f$ and $g$ respect the filtrations. Suppose that for some $\tau \in \widehat{\Theta}_E$, $f$ respects the filtrations on $\tau$ components. Then $f$ respects the filtrations on $\tau^c$ components, and $g$ respects the filtrations on $\tau^c$ components. In particular $g$ respects the filtrations on $\tau$ components.
\end{lem}

\begin{proof}
The question is local, therefore we may take $S = \textrm{Spec}(A)$ for a connected Noetherian ring $A$. Fix $\tau \in \widehat{\Theta}_E$. We will need to consider for each $j$ whether the image of $(\tau^c)^j(\varpi_\fp) \in A$ is zero or not. Let $N$ denote the set of $j$ such that $(\tau^c)^j(\varpi_\fp)$ is not zero in $A$. For such a $j$, the component of $\ccH^1_{\dR}(A_i/S)^0_{\tau^j}$ of $\ccH^1_{\dR}(A_i/S)^0_{\tau}$ where $\ccO_E$ acts via $(\tau^c)^j$ makes sense, ie for any other $k \neq j$, if we define $\ccH^1_{\dR}(A_i/S)^0_{\tau^k}$ similarly, then $\ccH^1_{\dR}(A_i/S)^0_{\tau^j} \cap \ccH^1_{\dR}(A_i/S)^0_{\tau^k} = 0$. Note that this is not at all the case for the elements outside of $N$, hence the need of Pappas-Rapoport models. Then for each $j \in N$, $\ccH^1_{\dR}(A_i/S)^0_{\tau^j}$ is a summand of $\ccH^1_{\dR}(A_i/S)^0_{\tau}$ of rank $2$ and we obtain a splitting $\omega^0_{i,\tau^c}(j) = \omega^0_{i,\tau^c}(j-1) \oplus \omega^0_{i,(\tau^j)^c}$ and $f^*$ must map $\omega^0_{2,(\tau^j)^c}$ to $\omega^0_{1,(\tau^j)^c}$ by $\ccO_E$-linearity.

We now show the first implication by induction. If $i = 0$ or $i = e_\fp$, this is trivial. Let $1 \leq i < e_\fp$ and suppose that $f^*\omega^0_{2,\tau^c}(i-1) \subset \omega^0_{1,\tau^c}(i-1)$.
By the discussion above if $i \in N$, then we automatically have $f^*\omega^0_{2,\tau^c}(i) \subset \omega^0_{1,\tau^c}(i)$. The case $s_{(\tau^i)^c} \neq 1$ also follows automatically so suppose that $i \notin N$ and $s_{(\tau^i)^c}=1$. For such an $i$, we have inclusions 
\[f^*\omega^0_{2,\tau}(i) \subset \omega^0_{1,\tau}(i) \subset f^* ( [\varpi_\fp]^{-1}\omega^0_{2,\tau}(i) ).\]
The compatibility of $f$ with the polarizations then shows that 
\[ \langle f^* ( [\varpi_\fp]^{-1}\omega^0_{2,\tau}(i) ), f^* (\omega^0_{2,\tau}(i)^\perp)  \rangle_1 = \langle [\varpi_\fp]( [\varpi_\fp]^{-1}\omega^0_{2,\tau}(i) ), \omega^0_{2,\tau}(i)^\perp \rangle_2 = 0.\]
So that $f^* \omega^0_{2,\tau}(i)^\perp \subset \omega^0_{1,\tau}(i)^\perp$. The result follows by applying $b_{\tau^c,i}$ to both sides.

For the second implication, note that the the compatibility of $f$ with the polarizations implies that the following diagram 
\[
\begin{tikzcd}
A_2 \arrow[d, "\lambda_2"] \arrow[r, "g"] & A_1 \otimes \Fp^{-1} \arrow[d, "\varpi_\Fp \otimes \lambda_1"] \\
A_2^\vee \arrow[r, "f^\vee"]                   & A_1^\vee                                                      
\end{tikzcd}
\]
is commutative and thus so is the diagram
\[
\begin{tikzcd}
\Fp \otimes \ccH^1_{\dR}(A_1 / S)_\tau \arrow[r,"g^*"] \arrow[d,"\varpi_\Fp^{-1} \otimes (\lambda^*_1)^{-1}"'] &  \ccH^1_{\dR}(A_2 / S)_\tau \arrow[d,"(\lambda^*_2)^{-1}"] \\
\ccH^1_{\dR}(A_1 / S)^\vee_{\tau^c} \arrow[r,"(f^*)^\vee"] & \ccH^1_{\dR}(A_1 / S)^\vee_{\tau^c}.
\end{tikzcd}
\]
Hence $f^*\omega^0_{2,\tau}(i) \subset \omega^0_{1,\tau}(i)$ and so $g^*(\omega^0_{A_1 \otimes \Fp^{-1},\tau}(i)^\perp) \subset \omega^0_{2,\tau}(i)^\perp$. The result follows by applying $b_{\tau^c,i}$ to both sides. The last implication follows similarly.
\end{proof}

As in the Hilbert case, it follows that the addition of $g$ in the data is redundant. We will however keep it as we will often refer to it.

\subsubsection{Local Structure}\label{iwahori unitary local}

In this section, we follow the strategy of \cite[$\mathsection$2.5]{diamond2021kodairaspencer} to study the local structure of $\widetilde{Y}_{U'_0(\fp)}(G'_\Sigma)$. Let $y$ be a closed point of $\widetilde{Y}'_{U'_0(\fp)}(G'_\Sigma)$ with residue field of characteristic $p$. Let $S$ denote its local ring with maximal ideal $\fm$. Set $S_0 = S/\fm$, $S_1 = S/ (\fm^2, \varpi)$, where $\varpi$ denotes a uniformizer of $\ccO$, and $S_n = S/ \fm^n$ for $n>1$. Let $(\underline{A}, \underline{A'},f,g)$ be the corresponding tuple over $S$, and for any $n \geq 0 $, let $(\underline{A_n}, \underline{A'_n},f_n,g_n)$ be the pullback of $(\underline{A}, \underline{A'},f,g)$ to $S_n$. For any sheaf $\ccF$ on $S$, we write $\ccF_n = \ccF \otimes_S S_n $ for its pullback to $S_n$. For any $\tau \in \widehat{\Theta}_E$ and $0 \leq i \leq e_\fp$, we have the Pappas-Rapoport filtration $\omega^0_\tau(i) \subset H^0(A, \Omega^1_{A/S})^0_\tau$ where $\tau = \tilde{\theta}$. We also have for $\beta = \tau^i$, the sheaves $\omega^0_\beta = \omega^0_\tau(i)/\omega^0_\tau(i-1)$ and $\ccH^1_\beta = ([\varpi_\fp]-\beta(\varpi_\fp))^{-1}\omega^0_\tau(i-1)/\omega^0_\tau(i-1)$, locally free of rank $s_\beta$ and $2$ respectively. We have similar sheaves defined for $\underline{A'}$ that we denote with a $\cdot '$.\medskip

By definition of the moduli problem $f^*$ respects the filtrations on both sides and so induces morphisms, for any $\beta \in \Theta_E$, $f^*_\beta: {\ccH^1_\beta}' \to \ccH^1_\beta$ such that $f^*_\beta ( {\omega^0_\beta}') \subset \omega^0_\beta$. These are isomorphisms if $\beta \notin \Theta_{E,\fp}$. Similarly $g^*$ respects the filtrations, either by definition or as a consequence of lemma \ref{respect filtrations}, so induces, for each $\beta \in \Theta_E$, morphisms $g^*_\beta : \ccH^{1 \prime}_\beta \otimes \fp \to \ccH^1_\beta$. Furthermore, the composition $f^*_\beta \circ g^*_\beta : \ccH^{1 \prime}_\beta \otimes \fp \to \ccH^{1 \prime}_\beta$ is the natural morphism $x \otimes \alpha \mapsto \alpha x$ which over $S_0$ is identically zero in particular, at least one of $f^*_{\beta,0}$ and $g^*_{\beta,0}$ must not be invertible. We will show that ethey both have rank one.

Suppose that $f^*_{\tau^1,0}$ has non-trivial kernel and consider $W = W(S_0)$, and the module $D = H^1_{\textrm{crys}}(A_0/W)^0$, free of rank 2 over $\ccO_E \otimes W$. We write $D = \oplus D_\tau$ for its decomposition into isotypic components; each $D_\tau$ is free of rank two over $W[x]/E_\tau$. We put $L_{\tau^i}$ for the preimage of $[\varpi_\fp]^{-1}\omega^0_\tau(i)_0$ under the canonical projection $D_\tau / p D_\tau \simeq H^1_{\textrm{dR}}(A_0/S_0)^0_\tau$ and consider the chain 
\[x^{e_\fp-1} D_\tau = L_{\tau^1} \subset L_{\tau^2} \subset \cdots \subset L_{\tau^{e_\fp}} \]
with graded pieces of dimension $s_{\tau^i}$ over $S_0$ such that $L_{\tau^i}/[\varpi_\fp] L_{\tau^i} \simeq \ccH^1_{\tau^i}$. 
We similarly have $D'_\tau = H^1_{\textrm{crys}}(A'_0/W)^0_\tau$ and submodules $L'_{\tau^i}$ defined analogously, enjoying the same properties. 
Now, the isogeny $f_0 : A \to A'$ induces a $\ccO_D \otimes W$-linear injection $f^*: D' \to D$ which is an isomorphism on $\tau$-components for $\tau \notin \widehat{\Theta}_{E,\Fp}$,
and whose cokernel has dimension $2f_\fp$ for the $\tau \in \widehat{\Theta}_{E,\Fp}$. 
More precisely, for $\fp \vert \Fp$, writing $D^{(\prime)}_{\fq}$, respectively $D^{(\prime)}_{\fq^c}$, 
for the components where $\ccO_E$ acts through $\ccO_{E_\fq}$, 
respectively $\ccO_{E_{\fq^c}}$, then $f^*$ respects these components and has cokernel of dimension $f_\fp$. 
Now, for $\tau \in \widehat{\Theta}_{E,\fp}$, since $f^*$ commutes with $\Phi$, we have 
\[
\begin{split}
\dim \Delta_{\phi \circ \tau}/ \Phi \Delta_\tau + \dim \Delta_\tau / f^* \Delta^\prime_\tau 
& = \dim \Delta_{\phi \circ \tau}/ \Phi \Delta_\tau + \dim \Phi\Delta_\tau / \Phi f^* \Delta^\prime_\tau\\
& = \dim \Delta_{\phi \circ \tau} /  f^* \Delta^\prime_{\phi \circ \tau} + \dim f^* \Delta^\prime_{\phi \circ \tau} / 
 f^* \Phi \Delta^\prime_\tau \\
 & =  \dim \Delta_{\phi \circ \tau} /  f^* \Delta^\prime_{\phi \circ \tau} + \dim \Delta^\prime_{\phi \circ \tau} / 
  \Phi \Delta^\prime_\tau,
\end{split}
\]
where the dimension is taken over $\overline{\bF}_p$. 
Since 
$\dim \Delta_{\phi \circ \tau}/ \Phi \Delta_\tau = \dim \Delta^\prime_{\phi \circ \tau} / \Phi \Delta^\prime_\tau$ for each $\tau$, we deduce that the
$D_{\tau}/f^*D^{\prime}_{2,\tau}$ all have the same dimension which must then be one. Furthermore, we deduce that the $f^*_{\tau^i,0}$ has non trivial kernel for any $1\leq i \leq e_\fp$.

Now, $f^*$ restricts to maps $L'_{\tau^i} \to L_{\tau^i}$ whose cokernel surjects onto that of $f^*_{\tau^i,0}$. We now show that the nontrivial cokernel of $L'_{\tau^i} \to L_{\tau^i}$ has length one. This holds for $i=1$ via the identification $L^{(')}_{\tau^1}= x^{e_\fp-1} D^{(')}_\tau$. For $i >1$, we obtain the result by induction. More precisely, if $s_{\tau^i} = 0$ (respectively $s_{\tau^{i}}=2$), then $L^{(')}_{\tau^i} = L^{(')}_{\tau^{i-1}}$ (respectively $x L^{(')}_{\tau^i} = L^{(')}_{\tau^{i-1}}$) so the result follows. Now if $s_{\tau^i}=1 $, $L^{(')}_{\tau^i} / L^{(')}_{\tau^{i-1}}$ has length one and the result follows since $f^*_{\tau^i}$ is injective.

Going back, we have just shown that each $f^*_{\tau^i,0}$ have one dimensional kernel and so each $g^*_{\tau^i,0}$ must have at least one dimensional kernels. Repeating the same argument shows that they are in fact one dimensional. If we suppose however that $g^*_{\tau^1,0}$ has non-trivial kernel, we can repeat the entire argument, swapping $f$ and $g$, and conclude, in any case, that $f^*_{\beta,0}$ has rank one for every $\beta \in \Theta_E$.

Now, define the subsets \[ \Theta_y = \{ \beta \in \Theta_F \setminus \Sigma_\infty \, \vert \, \ker f^*_{\tilde{\beta},0} = {\omega^{0 \prime}_{\tilde{\beta},0}} \}  \textrm{   and   }  \Theta'_y = \{ \beta \in \Theta_F \setminus \Sigma_\infty \, \vert \, \textrm{Im}\, f^*_{\tilde{\beta},0} = {\omega^{0 \prime}_{\tilde{\beta},0}} \}. \]

\noindent By the discussion above, we have $\Theta_y \cup \Theta'_y = \Theta_{F,\Fp}$. 
We define 
\[ R =  \ccO[X_\beta, X'_{\beta}] / (h_\beta)_{\beta \in \Theta_F \setminus \Sigma_\infty}, \]
where $h_\beta = \begin{cases}
X_\beta - \beta(\varpi_\Fp)X'_\beta & \textrm{ if } \beta \in \Theta_F \setminus (\Sigma_\infty \cup \Theta_y), \\
X'_\beta - \beta(\varpi_\Fp)X_\beta & \textrm{ if } \beta \in \Theta_y \setminus \Theta'_y, \\
X_\beta X'_\beta + \beta(\varpi_\Fp) & \textrm{ if } \beta \in \Theta_y \cap \Theta'_y.
\end{cases}$\medskip

\noindent We will show that there is an isomorphism $W \otimes_{W(\kappa)} \widehat{R}_\mathfrak{n} \xrightarrow{\sim} \widehat{S}_\mathfrak{m}$, where $\mathfrak{n}$ is the ideal generated by $\varpi$ and the variables $X_\beta$ and $X^\prime_\beta$.

Let $\alpha^{(')}$ denote the composite of the canonical isomorphisms 
\[ H^1_{\textrm{dR}}(A^{(')}_0/S_0) \otimes_{S_0} S_1 \xrightarrow{\sim} H^1_{\textrm{crys}}(A^{(')}_0/S_1) \xrightarrow{\sim} H^1_{\textrm{dR}}(A^{(')}_1/S_1) \]
where $H^1_{\textrm{crys}}(A^{(')}_0/S_1)$ denotes the evalution of the crystalline cohomology sheaf of $A^{(')}_0/S_0$ at $S_1$. Therefore, $\alpha$ and $\alpha'$ are $\ccO_D \otimes S_1$ linear, and compatible with $f^*_0$ and $f^*_1$. We thus obtain for each $\tau \in \widehat{\Theta}_{E,\fp}$  $S_1[x]/x^{e_\fp}$ linear isomorphisms $\alpha_\tau$ and $\alpha'_\tau$ which sit in the commutative diagram 

\begin{center}
\begin{tikzcd}
H^1_{\textrm{dR}}(A^{'}_0/S_0)^0_\tau \otimes_{S_0} S_1 \arrow[r, "f^*_{\tau,0} \otimes 1"]  \arrow[d, "\alpha'_\tau"']  & H^1_{\textrm{dR}}(A_0/S_0)^0_\tau \otimes_{S_0} S_1 \arrow[d, "\alpha_\tau"] \\
H^1_{\textrm{dR}}(A^{'}_1/S_1)^0_\tau  \arrow[r, "f^*_{\tau,1}"]  &  H^1_{\textrm{dR}}(A_1/S_1)^0_\tau.
\end{tikzcd}
\end{center}

Consider now the chain of ideals $$ (\fm^2, \varpi) = I_\tau^{(0)} \subset I_\tau^{(1)} \subset \cdots \subset I_\tau^{(e_\fp-1)} \subset I_\tau^{(e_\fp)} \subset \fm$$
 where each $I_\tau^{(i)}$ is given by the vanishing of \[ \omega^0_\tau(\ell)_0 \otimes_{S_0} S_1 \xrightarrow{\alpha_\tau} H^1_{\textrm{dR}}(A_1/S_1)^0_\tau / \omega^0_\tau(\ell)_1  \textrm{   and   } \omega_\tau^0(\ell)'_0 \otimes_{S_0} S_1 \xrightarrow{\alpha'_\tau} H^1_{\textrm{dR}}(A'_1/S_1)^0_\tau / \omega_\tau^0(\ell)'_1 \] for all $1 \leq \ell \leq i$.
It is straightforward to see that if $s_{\tau^i} \neq 1$, then $I^{(i)}_\tau = I^{(i-1)}_\tau$: for $s_{\tau^i}=0$, this is clear. For $s_{\tau^i} = 2$, this follows from the fact that $\omega^0_\tau(i)_0 \subset [\varpi_\fp]H^1_{\textrm{dR}}(A_0 / S_0)^0_\tau$ and similarly for $A'$. Similarly, one can see that, since the $\alpha_\tau$ are compatible with the pairings, that $I^{(i)}_\tau = I^{(i)}_{\tau^c}$.

Thus, for each $\beta = \theta^i \in \Theta_{F,\fp} \setminus \Sigma_{\infty,\fp}$, we set $S_{\beta} = S / I^{i-1}_\theta$; for any lift $\tau$ of $\theta$, $\alpha_\tau$ and $\alpha'_\tau$ induce $S_\beta[x]/x^{e_\fp}$-linear isomorphisms \[ \omega^0_\tau(i-1)_0 \otimes_{S_0} S_{\beta} \xrightarrow{\sim} \omega^0_\tau(i-1) \otimes_S S_\beta \, , \, \omega^0_\tau(i-1)'_0 \otimes_{S_0} S_{\beta} \xrightarrow{\sim} \omega^0_\tau(i-1)' \otimes_S S_\beta. \]

In particular we obtain $S_\beta$-linear isomorphisms sitting in the following commutative diagram:
\begin{center}
\begin{tikzcd}
\ccH^1_{\beta,0} \otimes_{S_0} S_\beta \arrow[r, " f^*_{\beta,0} \otimes \id"] \arrow[d, "\alpha_\beta"] & \ccH^{1'}_{\beta,0} \otimes_{S_0} S_\beta \arrow[d, "\alpha'_\beta"] \\

\ccH^1_{\beta} \otimes_{S_0} S_\beta \arrow[r, " f^*_{\beta} \otimes \id"]  & \ccH^{1'}_{\beta} \otimes_{S_0} S_\beta.
\end{tikzcd}
\end{center}
 
 In particular, the $\alpha_\beta$ are compatible with the pairings $\langle \cdot , \cdot \rangle_{\beta,0}$ and $\langle \cdot , \cdot \rangle_{\beta}$, and similarly the $\alpha'_\beta$ are compatible with the pairings $\langle \cdot , \cdot \rangle'_{\beta,0}$ and $\langle \cdot , \cdot \rangle'_{\beta}$. We now claim that there exist bases $(e_{1,\beta},e_{2,\beta})$ and $(f_{1,\beta},f_{2,\beta})$ of $\ccH^1_{\tilde{\beta}}$ and $\ccH^1_{\tilde{\beta}^c}$ respectively, and bases $(e'_{1,\beta},e'_{2,\beta})$ and $(f'_{1,\beta},f'_{2,\beta})$ of $\ccH^{1 \prime}_{\tilde{\beta}}$ and $\ccH^{1 \prime}_{\tilde{\beta}^c}$ respectively such that 
\begin{itemize}
    \item $e_{1,\beta} \otimes_S S_\beta = \alpha_{\tilde{\beta}}(\omega^0(\tilde{\beta})_0 \otimes_{S_0} S_\beta)$ and $e'_{1,\beta} \otimes_S S_\beta = \alpha'_{\tilde{\beta}}(\omega^0(\tilde{\beta})'_0 \otimes_{S_0} S_\beta)$,
    \item $\langle e_{1,\beta},f_{1,\beta} \rangle_{\beta} = \langle e_{2,\beta},f_{2,\beta} \rangle_{\beta} = \langle e'_{1,\beta},f'_{1,\beta} \rangle'_{\beta} = \langle e'_{2,\beta},f'_{2,\beta} \rangle'_{\beta} =1$ and any other pairing is 0,
    \item The matrix of $f^*_{\tilde{\beta}}$ is given in these bases by 
    \begin{equation}\label{matrix lift}
    \begin{pmatrix} 1 & 0 \\ 0 & \beta(\varpi_\fp) \end{pmatrix} \, , \,   \begin{pmatrix} \beta(\varpi_\fp) & 0 \\ 0 & 1 \end{pmatrix} \, \textrm{or} \,   \begin{pmatrix}0 & 1 \\ -\beta(\varpi_\fp)  & 0 \end{pmatrix}
    \end{equation}  if $\beta \in \Theta_F \setminus \Theta_y$, $\beta \in \Theta_y \setminus \Theta'_y$ or $\beta \in \theta_y \cap \Theta'_y$ respectively,
    \item The matrix of $f^*_{\tilde{\beta}^c}$ is given in these bases by 
    \begin{equation}\label{matrix lift conjugate}
    \begin{pmatrix} \beta(\varpi_\fp) & 0 \\ 0 & 1 \end{pmatrix} \, , \,   \begin{pmatrix} 1 & 0 \\ 0 & \beta(\varpi_\fp) \end{pmatrix} \, \textrm{or} \,   \begin{pmatrix}0 & -1 \\ \beta(\varpi_\fp)  & 0 \end{pmatrix}
    \end{equation}
    if $\beta \in \Theta_F \setminus \Theta_y$, $\beta \in \Theta_y \setminus \Theta'_y$ or $\beta \in \theta_y \cap \Theta'_y$ respectively.
\end{itemize}
 
 To see this, first pick bases  $(\overline{e}_{1,\beta},\overline{e}_{2,\beta})$ and $(\overline{f}_{1,\beta},\overline{f}_{2,\beta})$ of $\ccH^1_{\tilde{\beta},0}$ and $\ccH^1_{\tilde{\beta}^c,0}$ such that $\omega^0(\tilde{\beta})_0 = S_0 \overline{e}_{1,\beta}$, $\omega^0(\tilde{\beta}^c)_0 = S_0 \overline{f}_{2,\beta}$, $\langle \overline{e}_{1,\beta},\overline{f}_{1,\beta} \rangle_{{\beta},0} = \langle \overline{e}_{2,\beta},\overline{f}_{2,\beta} \rangle_{{\beta},0}$ and similar bases $(\overline{e}'_{1,\beta},\overline{e}'_{2,\beta})$ and $(\overline{f}'_{1,\beta},\overline{f}'_{2,\beta})$ such that $f^*_{\tilde{\beta},0}$ and $f^*_{\tilde{\beta}^c,0}$ are in the given forms above (recall that the vanishing of $f^*_{\tilde{\beta}^c,0}$ is the same as that of $f^*_{\tilde{\beta},0}$). Now, choose arbitrary lifts of the images of the bases under $\alpha$ and $\alpha'$ such that the condition on the pairings is satisfied. Under these bases, the pairings on $\ccH^1_{\tilde{\beta}} \oplus \ccH^1_{\tilde{\beta}^c}$ and $\ccH^{1 \prime}
_{\tilde{\beta}} \oplus \ccH^{1 \prime}_{\tilde{\beta}^c}$ are given in standard symplectic form, and the compatibility of $f$ and the polarizations implies that $AB^T=B^TA= \beta(\varpi_\fp) I_2$ where $A$ and $B$ denote the matrices of $f^*_{\tilde{\beta}}$ and $f^*_{\tilde{\beta}^c}$ in these bases respectively. Suppose now that $\beta \in \Theta_F \setminus \Theta_y$, then by scaling the bases by appropriate elements of $\ker \left( \textrm{Sp}_4\left( S \right) \to \textrm{Sp}_4\left( S_\beta \right) \right)$ we have the matrices $A$ and $B$ (as defined above) given in the form 
 \[ \begin{pmatrix} 1 & 0 \\ 0 & \beta(\varpi_\fp)+d \end{pmatrix} \textrm{   and   } \begin{pmatrix} \beta(\varpi_\fp)+a' & b' \\ c' & 1 \end{pmatrix} \]
 where $a',b',c',d \in I_\beta$. However we immediately see from the relation $AB^T=B^TA= \beta(\varpi_\fp) I_2$ that $a'=b'=c'=d=0$, as desired. The other two cases follow similarly.
 
We thus have unique elements $t_\beta,t'_\beta \in \fm$ such that $\omega^0(\tilde{\beta})= S(e_{1,\beta} - t_\beta e_{2,\beta})$ and $\omega^0(\tilde{\beta})'= S(e'_{1,\beta} - t'_\beta e'_{2,\beta})$ and furthermore, since $f^*_{\tilde{\beta}}$ preserves the filtrations we have 
 \begin{itemize}
     \item $t_\beta = \beta(\varpi_\fp)t'_\beta$ if $\beta \in \Theta_F \setminus \Theta_y$,
     \item $t'_\beta = \beta(\varpi_\fp)t_\beta$ if $\beta \in \Theta_y \setminus \Theta'_y$,
     \item $t_\beta t'_\beta = -\beta(\varpi_\fp)$ if $\beta \in \Theta_y \cap \Theta'_y$,
 \end{itemize}

We now define a map $\ccO[X_\beta,X'_\beta]_{\left\{ \beta \in \Theta_F \setminus \Sigma_\infty \right\}} \to S$ by $X_\beta \mapsto t_\beta$ and $X'_\beta \mapsto t'_\beta$. By the above, this map factors through $R$ and induces a morphism \[ \rho: W \otimes_{W(k)} \ccO[[X_\beta,X'_\beta]] / (h_\beta)_{\beta \in \Theta_F \setminus \Sigma_\infty } \to \ccO^\wedge_{\widetilde{Y}'_{U^\prime_0(\Fp)}(G'_\Sigma),y}. \]

Write $R_1 = W \otimes_{W(k)} R / (\mathfrak{n}^2, \varpi)$ and $R_n = W \otimes_{W(k)} R / \mathfrak{n}^n$ for $n >1$, where $\mathfrak{n}$ denotes the maximal ideal generated by $\varpi$ and the $X_\beta,X'_\beta$, so that we have morphisms $\rho_n : R_n \to S_n$ for all $n$. We will show that they are all isomorphisms by induction on $n$. We start by showing that $\rho_1$ is surjective. To do so, first remark that $I_\tau^{(i)} = (I_\tau^{(i-1)}, t_{\tau^i},t'_{\tau^i})$. Indeed, for $\beta = \tau^i$, $I_\tau^{(i)}/I_\tau^{(i-1)}$ is given by then vanishing of the maps \[ \omega^0_{\beta,0} \otimes_{S_0} S_{\beta} \xrightarrow{\alpha_\tau} \left( \ccH^1_\beta / \omega^0_{\beta,1} \right) \otimes_S S_\beta \textrm{   and   } \omega^{0'}_{\beta,0} \otimes_{S_0} S_\beta \xrightarrow{\alpha'_\tau} \left( {\ccH^1}'_\beta / \omega^{0'}_{\beta,1} \right) \otimes_S S_\beta\]
which, relative to the (preimage by the $\alpha$ and $\alpha'$ of the) bases constructed above is given precisely by $t_\beta$ and $t'_\beta$. In particular, the image of $\rho_1$ contains all of the $I^{(e_\fp)}_\tau$, therefore it suffices to show that the quotient $T \to S_0$ where $T$ is the quotient of $S$ by all the $I^{(e_\fp)}_\tau$ is an isomorphism. In other words, that the tuple $(\underline{A},\underline{A'},f)\otimes_S T$ is isomorphic to $(\underline{A_0},\underline{A'_0},f) \otimes_{S_0} T$. However, by construction, both tuples induce the same Pappas-Rapoport flags in $H^1_{\textrm{crys}}(A^{(')}_0/T)^0$ from which we deduce the result by Grothendieck-Messing Theory.\smallskip

To show that $\rho_1$ is an isomorphism now, it suffices to show that $\textrm{lg}(R_1) \leq \textrm{lg}(S_1)$, which follows from the existence of a surjection $S \to R_1$ (which factors through $S_1$). To this end we construct Pappas-Rapoport flags 
\[ F^\bullet \subset H^1_{\textrm{crys}}(A_0/R_1)^0 \xleftarrow[\sim]{\gamma} H^1_{\textrm{dR}}(A_0/S_0)^0 \otimes_{S_0} R_1\]
and \[ F^\bullet \subset H^1_{\textrm{crys}}(A'_0/R_1)^0 \xleftarrow[\sim]{\gamma'} H^1_{\textrm{dR}}(A'_0/S_0)^0 \otimes_{S_0} R_1\]
such that
\begin{itemize}
    \item The flags are compatible with $f^*_{0, \textrm{crys}}$,
    \item For any $\beta = \tau^i \in (\Theta_F \setminus \Sigma_\infty)^\sim$, $\gamma$ identifies $F_{\tau}(i-1) \otimes_{R_1} R_\beta$ with $\omega^0_{\tau}(i-1)_0 \otimes_{S_0} R_\beta$ and  $\gamma'$ identifies $F'_{\tau}(i-1) \otimes_{R_1} R_\beta$ with $\omega^{0'}_{\tau}(i-1)_0 \otimes_{S_0} R_\beta$),
    \item $(F_\tau(i)/F_\tau(i-1)) \otimes_{R_1} R_\beta$ is generated by $\gamma(\overline{e}_{1,\beta} \otimes 1 - \overline{e}_{2,\beta} \otimes X_\beta)$ and $(F'_\tau(i)/F'_\tau(i-1)) \otimes_{R_1} R_\beta$ is generated by $\gamma'(\overline{e}'_{1,\beta} \otimes 1 - \overline{e}'_{2,\beta} \otimes X'_\beta )$
\end{itemize}
where the rings $R_\beta$ are defined inductively as $R_{\tau^1}=R_1$ and $R_{\tau^{i+1}} = R_{\tau^{i}} / (X_{\tau^i},X'_{\tau^i})$.\smallskip

Suppose we have constructed flags as above until the term $i-1$. Write $\beta= \tau^i$. If $s_\beta=0$ or $2$, then the last two bullets do not apply and the term $F_\tau(i)$ is uniquely determined by $F_\tau(i-1)$ and similarly for $F'_\tau(i-1)$. 
If $s_\beta = 1$, we define $\widetilde{\ccH}_\beta = [\varpi_\fp]^{-1} F_\tau(i-1) / F_\tau(i-1)$ 
and $\widetilde{\ccH}'_\beta= [\varpi_\fp]^{-1} F^\prime_\tau(i-1) / F^\prime_\tau(i-1)$. 
Since $F_\tau(i-1) \subset [\varpi_\fp]^{e-i+1} H^1_{\textrm{crys}}(A_0/R_1)$, 
where $H^1_{\textrm{crys}}(A_0/R_1)$ is free of rank $2$ over $R_1[x]/x^{e_\fp}$, $\widetilde{\ccH}_\beta$ is free of rank $2$ over $R_1$ by lemma \ref{filtration torsion}. 
Similarly for $\widetilde{\ccH}'_\beta$. Hence $\gamma$ and $\gamma'$ induce isomorphisms $\gamma_\beta: \ccH^1_{\beta,0} \otimes_{S_0} R_\beta \xrightarrow{\sim} \widetilde{\ccH}_\beta \otimes_{R_1} R_\beta$ and $\gamma'_\beta: \ccH'^1_{\beta,0} \otimes_{S_0} R_\beta \xrightarrow{\sim} \widetilde{\ccH}'_\beta \otimes_{R_1} R_\beta$. Furthermore since $f^*_{0, \textrm{crys}}$ respects the partially constructed flag, it induces an $R_1$-linear morphism $f^*_\beta: \widetilde{\ccH}_\beta \to \widetilde{\ccH}'_\beta$ which is compatible with $\gamma_\beta$ and $\gamma'_\beta$.

Now, since $H^1_{\textrm{crys}}(A_0/R_1)$ comes with a perfect alternating under which the action of $\ccO_F$ is self adjoint, we can define $F_{\tau^c}(i-1) = F_{\tau}(i-1)^{\perp_{i-1}}$ where $\cdot^{\perp_{i-1}}$ denotes the orthogonal complement with respect to the induced perfect pairing on $[\varpi_\fp]^{e-i+1} H^1_{\textrm{crys}}(A_0/R_1)$. This yields $\widetilde{\ccH}_{\beta^c}$ and a perfect pairing on $\widetilde{\ccH}_{\beta} \times \widetilde{\ccH}_{\beta^c}$. Now, since the pairing on $H^1_{\textrm{crys}}(A_0/R_1)$ is compatible with the one on $H^1_{\textrm{dR}}(A_0/S_0)$ it follows that in fact $\gamma$ identifies $\omega_{\tau^c}(i-1)_0 \otimes_{S_0} R_\beta$ with $F_{\tau^c}(i-1) \otimes_{R_1} {R_\beta}$ and so induces an isomorphism $\gamma_{\beta^c}: \ccH^1_{\beta^c,0} \otimes_{S_0} R_\beta \xrightarrow{\sim} \widetilde{\ccH}_{\beta^c} \otimes_{R_1} R_\beta$ compatible with the various pairings. We obtain a similar $\widetilde{\ccH}'_{\beta^c}$. Using the same argument as before we obtain bases $(\tilde{e}_{1,\beta},\tilde{e}_{2,\beta})$ and $(\tilde{e}'_{1,\beta},\tilde{e}'_{2,\beta})$ which lift $\gamma_\beta(\overline{e}_{1,\beta} \otimes 1,\overline{e}_{2,\beta} \otimes 1)$ and $\gamma'_\beta(\overline{e}'_{1,\beta} \otimes 1,\overline{e}'_{2,\beta} \otimes 1)$ such that $f^*_\beta$ is given in the appropriate form \ref{matrix lift}. We thus define $F_\tau(i)$ and $F'_{\tau}(i)$ by taking the preimages of $L_\beta = R_\beta(\tilde{e}_{1,\beta} -X_\beta \tilde{e}_{2,\beta})$ and $L'_\beta = R_\beta(\tilde{e}'_{1,\beta} -X'_\beta \tilde{e}'_{2,\beta})$. This yields Pappas-Rapoport flags $F^\bullet $ and $F'^\bullet$ satisfying our required conditions.\medskip

Let $(\underline{\widetilde{A}},\underline{\widetilde{A}'},\tilde{f})$ be the lift over $R_1$ of $(\underline{A_0},\underline{A'_0},f_0)$ defined by the flags constructed above. For each $\tau$ and $\beta=\tau^i$, by construction of the flags and induction on $i$, the induced morphism $S \to R_1 \to R_\beta$ factors through $S_\beta$ and induces isomorphisms 
\begin{center}
\begin{tikzcd}
\ccH^1_\beta \otimes_S R_\beta \arrow[r,"\sim"] & \widetilde{\ccH}_\beta \otimes_{R_1} R_\beta  &&  \ccH'^1_\beta \otimes_S R_\beta \arrow[r,"\sim"] & \widetilde{\ccH}'_\beta \otimes_{R_1} R_\beta\\

\omega_\beta \otimes_S R_\beta \arrow[r,"\sim"] \arrow[hookrightarrow]{u} & L_\beta \otimes_{R_1} R_\beta \arrow[hookrightarrow]{u} && \omega'_\beta \otimes_S R_\beta \arrow[r,"\sim"] \arrow[hookrightarrow]{u} & L'_\beta \otimes_{R_1} R_\beta \arrow[hookrightarrow]{u}

\end{tikzcd}
\end{center}

\noindent sending $(e_{1,\beta},e_{2,\beta})$ to $(\tilde{e}_{1,\beta},\tilde{e}_{2,\beta})$ and $(e'_{1,\beta},e'_{2,\beta})$ to $(\tilde{e}'_{1,\beta},\tilde{e}'_{2,\beta})$. In particular (in $R_\beta$) we get $t_\beta \mapsto X_\beta$ and $t'_\beta \mapsto X'_\beta$. Hence $S \to R_1$ is surjective and $\rho_1 : R_1 \to S_1$ is an isomorphism.\medskip

It remains to show that $\rho_n$ is an isomorphism for $n \geq 2$. Surjectivity follows immediately from that of $\rho_1$. As before injectivity will follow, by induction, from showing that $\textrm{lg}(R_n) \leq \textrm{lg}(R_n)$ which itself follows from constructing a lift of $\rho_{n-1}^{-1}$ to a morphism $S \to R_n$ which must necessarily factor through $S_n$. This will follow from constructing, similarly to above, a lift of $(\underline{A_{n-1}},\underline{A'_{n-1}},f_{n-1})$ to $R_n$. 

Suppose that $n\geq 1$ and that we already have established that $\rho_n:R_n \to S_n$ is an isomorphism (the base case $n=1$ having been done above), then we can consider the nilpotent thickening $R_{n+1} \to R_n \xrightarrow{\sim} S_n$ with trivial divide powers. We inductively define lifts of the Pappas-Rapoport flag:

\[ \widetilde{F}^\bullet \subset H^1_{\textrm{crys}}(A_0/R_{n+1})^0 \]
and \[ \widetilde{F}^{\prime \bullet} \subset H^1_{\textrm{crys}}(A'_0/R_{n+1})^0\]
such that for each $\theta \in \Theta_{F,\fp}$, $\tau = \tilde{\theta}$, and $0 \leq i \leq e_\fp$
\begin{itemize}
    \item $\widetilde{F}_\tau(i)$ and $\widetilde{F}^\prime_\tau(i)$ are $R_{n+1}[x]/E_\tau(x)$-modules, free over $R_{n+1}$ of rank $s_\tau(i)$,
    \item $([\varpi_\fp] - \tau^i(\varpi_\fp))\widetilde{F}_\tau(i) \subset \widetilde{F}_\tau(i-1)$, $([\varpi_\fp] - \tau^i(\varpi_\fp))\widetilde{F}^\prime_\tau(i) \subset \widetilde{F}^\prime_\tau(i-1)$ and $f^*_{n,crys}\widetilde{F}^\prime_\tau(i) \subset \widetilde{F}_\tau(i)$,
    \item $\widetilde{F}_\tau(i) \otimes_{R_{n+1}} S_n = \omega^0_{A_n/S_n,\tau}(i)$ and $\widetilde{F}^\prime_\tau(i) \otimes_{R_{n+1}} S_n = \omega^0_{A^\prime_n/S_n,\tau}(i)$ under the canonical isomorphisms $H^1_{\dR}(A_n/S_n) \cong H^1_{crys}(A_n/R_{n+1}) \otimes_{R_{n+1}} S_n$ and $H^1_{\dR}(A^\prime_n/S_n) \cong H^1_{crys}(A^\prime_n/R_{n+1}) \otimes_{R_{n+1}} S_n$.
\end{itemize}
Suppose that we have constructed the flags to the term $i-1$. We let $\widetilde{\ccH}^1_{\tau^i} = ([\varpi_\fp]-\tau^i(\varpi_\fp))^{-1}\widetilde{F}_\tau(i-1) / \widetilde{F}_\tau(i-1)$, which is,again by lemma \ref{filtration torsion}, free of rank $2$ over $R_{n+1}$. We similarly have $\widetilde{\ccH}^{1 \prime}_{\tau^i}$, and we also have canonical isomorphisms
\[\gamma_{\tau^i}: \widetilde{\ccH}^1_{\tau^i} \xrightarrow{\sim} \ccH^1_{\tau^i,n} \, \textrm{ and } \, \gamma^\prime_{\tau^i}: \widetilde{\ccH}^{1 \prime}_{\tau^i} \xrightarrow{\sim} \ccH^{1 \prime}_{\tau^i,n}. \]
Also, $f^*_{crys,n+1}$ induces an $R_{n+1}$-linear homomorphism $\tilde{f}^*_{\tau^i}: \widetilde{\ccH}^{1 \prime}_{\tau^i} \to \widetilde{\ccH}^{1}_{\tau^i}$, which is compatible with $\tilde{f}^*_{\tau^i,n}$.
The construction of $\widetilde{F}_\tau(i)$ and $\widetilde{F}^\prime_\tau(i)$ becomes equivalent to the construction of invertible $R_{n+1}$ submodules $\widetilde{F}_{\tau^i} \subset \widetilde{\ccH}^{1 }_{\tau^i}$ and $\widetilde{F}^\prime_{\tau^i} \subset \widetilde{\ccH}^{1 \prime}_{\tau^i}$ such that 
\begin{itemize}
    \item $f^*_{\tau^i}( \widetilde{F}^\prime_{\tau^i}) \subset \widetilde{F}_{\tau^i}$,
    \item $\widetilde{F}_{\tau^i} \otimes_{R_{n+1}} S_n = \gamma_{\tau^i}(e_{1,\tau^i,n}-t^\prime_{\tau^i}e_{2,\tau^i,n})$,
    \item $\widetilde{F}^\prime_{\tau^i} \otimes_{R_{n+1}} S_n = \gamma^\prime_{\tau^i}(e^\prime_{1,\tau^i,n}-t_{\tau^i}e^\prime_{2,\tau^i,n})$.
\end{itemize}
This follows from the same arguments as for the case $n=0$, replacing the polynomial $x^{e_\fp-i}$ with $(x- \tau^{i+1}(\varpi_\fp)) \cdots (x- \tau^{e_\fp}(\varpi_\fp))$. We thus obtain a lift of $(\underline{A_n},\underline{A^\prime_n},f_n,g_n)$ to a triple over $R_{n+1}$ and conclude that $\rho_{n+1}: R_{n+1} \to S_{n+1}$ is an isomorphism.

This completes the proof that $\rho$ is an isomorphism and that $\ccO^\wedge_{\widetilde{Y}'_{U^\prime_0(\Fp)}(G'_\Sigma),y}$ is isomorphic to \[ W \otimes_{W(k)} \ccO[[X_\beta,X'_\beta]] / (h_\beta)_{\beta \in \Theta_F \setminus \Sigma_\infty }, \]

\noindent which is a reduced complete intersection, flat over $\ccO$. In particular, as the generic fiber of $\widetilde{Y}'_{U^\prime_0(\Fp)}(G'_\Sigma)$ is smooth, $\widetilde{Y}'_{U^\prime_0(\Fp)}(G'_\Sigma)$ is reduced and syntomic over $\ccO$, of relative dimension $\vert \Theta_F \setminus \Sigma_\infty \vert$. Since the quotient map $\widetilde{Y}'_{U^\prime_0(\Fp)}(G'_\Sigma) \to Y'_{U^\prime_0(\Fp)}(G'_\Sigma)$ is locally on $\widetilde{Y}'_{U^\prime_0(\Fp)}(G'_\Sigma)$ finite \'{e}tale, the same is true of $Y'_{U^\prime_0(\Fp)}(G_\Sigma)$.

\subsubsection{Stratifications in the Unitary Case}\label{unitary stratification}
We now proceed to define a stratification of $\overline{Y}_0(\Fp) = Y_{U^\prime_0(\Fp)}(G^\prime_\Sigma)_\bF$  by first defining one on $S^\prime = \widetilde{Y}_{U^\prime}(G^\prime_\Sigma)_\bF$ and descending it. 

Keep the notation from section \ref{iwahori unitary def}. Write $(\underline{A_1}, \underline{A_2},f,g)$ for the universal object on $S^\prime$, and for any $\beta \in \Theta_E$ and $j =1,2$, write $\omega^0_{j,\beta} = \omega^0_{A_j/S^\prime,\beta}$, $\ccH^1_{j,\beta} = \ccH^1_{A_j/S^\prime,\beta}$ and $v^0_{j,\beta} = v^0_{A_j/S^\prime_\beta}$ for the sheaves defined in section \ref{automorphic bundles unitary}. By the definition of the moduli problem and lemma \ref{respect filtrations}, for any $\beta \in \Theta_E$, we have morphisms $f^*_\beta: \ccH^1_{2,\beta} \to \ccH^1_{1,\beta}$ and $g^*_\beta : \ccH^1_{A_1 \otimes \Fp^{-1},\beta} \to \ccH^1_{2,\beta}$ which restrict to morphisms: 
\[ f^*_\beta : \omega^0_{2,\beta} \to \omega^0_{1,\beta}, \, \, \quad g^*_\beta:\omega^0_{A_1 \otimes \Fp^{-1},\beta} \to \omega^0_{2,\beta},\]
and \[ f^{*\prime}_\beta : v^0_{2,\beta} \to v^0_{1,\beta}, \, \, \quad g^{*\prime}_\beta:v^0_{A_1 \otimes \Fp^{-1},\beta} \to v^0_{2,\beta}.\]

We determined in the previous section that these morphisms are isomorphisms for $\beta \notin \Theta_{E,\Fp}$ and, if $\beta \in \Theta_{E,\Fp}$, then $g^*_\beta(\ccH^1_{A_1 \otimes \Fp^{-1},\beta}) = \ker f^*_\beta$ and $(f \otimes \Fp^{-1})^*_\beta(\ccH^1_{A_2 \otimes \Fp^{-1}}) = \ker g^*_\beta$ are all locally free of rank one. Coupling this with the compatibilities between polarizations, we immediately obtain for $\beta \vert_F \notin \Sigma_{\infty,\Fp}$
\[f^*\omega^0_{2,\beta} = 0 \Leftrightarrow g^{*\prime} v^0_{A_1 \otimes \Fp^{-1},\beta} = 0 \Leftrightarrow f^*\omega^0_{2,\beta^c} = 0 \Leftrightarrow g^{*\prime} v^0_{A_1 \otimes \Fp^{-1},\beta^c} = 0. \]
We also get similar implications by swapping $1$ and $2$.

For any two subsets $I,J \subset \Theta_{F,\Fp} \setminus \Sigma_{\infty,\Fp}$, we thus define $S^\prime_{\phi^\prime(I),J}$ to be given by the vanishing of the sections 
\[ \{ f^*_{\tilde{\beta}}: \omega^0_{2,\tilde{\beta}} \to \omega^0_{1,\tilde{\beta}} \, \vert \, \beta \in I \} \cup \{ g^*_{\tilde{\beta}} : \omega^0_{A_1 \otimes \Fp^{-1},\tilde{\beta}} \to \omega^0_{2,\tilde{\beta}} \, \vert \, \beta \in J  \}.\]
Note that by above, we could have equivalently considered the conjugate lifts if we wanted or even replaced $g^*_\beta$ by $f^{*\prime}_\beta$. We also remark that, as in the Hilbert case, we use the shift $\phi^\prime(I)$ in our notation for when we relate these closed subschemes to the Goren-Oort stratification downstairs in section \ref{Iwahori goren relation}.

Now, let $P$ be a closed point of $S^\prime$ and let $I_P \subset \Theta_{F,\Fp} \setminus \Sigma_{\infty,\Fp}$ be the set of $\beta$ such that $f^*_\beta$ vanishes and $J_P \subset \Theta_{F,\Fp} \setminus \Sigma_{\infty,\Fp}$ the set of $\beta$ such that $g^*_\beta$ vanishes. Note that $I_P \cup J_P = \Theta_{F,\Fp} \setminus \Sigma_{\infty,\Fp}$. We obtained in the last section, after taking local trivializations $\ccH om_{\ccO_S}(\omega^0_{2,\tilde{\beta}}, \omega^0_{1,\tilde{\beta}})$ and $\ccH om_{\ccO_S}(v^0_{2,\tilde{\beta}}, v^0_{1,\tilde{\beta}})$, the following description of the completed local ring $\widehat{\ccO}_{S^\prime,P}$:

\begin{thm}\label{unitary iwahori structure thm}
Let $P$ be a closed point of $S^\prime$. Then we have an isomorphism
\[ \widehat{\ccO}_{S^\prime,P} \simeq k_P \llbracket X_\beta, X^\prime_\beta \rrbracket_{\beta \in \Theta_F} / \langle h_\beta \rangle_{\beta \in \Theta_F} \]
where $f^*_\beta \mapsto X_\beta$, $f^{* \prime}_\beta \mapsto X^\prime_\beta$ and
\[ h_\beta = 
\begin{cases}
X_\beta  & \textrm{ if } \beta \in I_P \setminus J_P,\\
X^\prime_\beta &\textrm{ if } \beta \in J_P \setminus I_P,\\
X_\beta X^\prime_\beta &\textrm{ if } \beta \in I_P \cap J_P,\\
X_\beta - \beta(\varpi_\Fp)X_\beta^ \prime & \textrm{ if } \beta \notin \Theta_{F,\Fp}.
\end{cases}\]
\end{thm}

We thus obtain, as in \cite[$\mathsection$4.3.2]{2020arXiv200100530D}, 

\begin{cor}\label{corollary strata dim smooth}
 The scheme $S^\prime_{\phi^\prime(I),J}$ over $\bF$ is a reduced local complete intersection of dimension $\vert \Theta_F \setminus \Sigma_\infty\vert - \vert I \cap J \vert$ and is smooth if $I \cup J = \Theta_{F,\Fp} \setminus \Sigma_{\infty,\Fp}$.   
\end{cor}

As in the Hilbert case, the natural action of $\ccO_{F,(p),+}^\times$ on the line bundles 
\[\ccH om_{\ccO_S}\left(\omega^0_{2,\beta},\omega^0_{1,\beta}\right) \quad \textrm{and} \quad \ccH om_{\ccO_S}\left(\omega^0_{A_1 \otimes \Fp^{-1},\beta},\omega^0_{2,\beta}\right)\] 
factors through $\ccO_{F,(p),+}^\times/(\ccO_F \cap U^\prime)^2$ so that the above line bundles and the accompanying sections $f^*_\beta$ and $g^*_\beta$ all descend to $\overline{Y}_0(\Fp)$.

This allows us to define the closed subschemes $\overline{Y}_0(\fp)_{\phi^\prime(I),J}$ of $\overline{Y}_0(\fp)$ as the vanishing of the appropriate sections or, equivalently, as the quotient of $S^\prime_{\phi^\prime(I),J}$ by the action of $\ccO_{F,(p),+}$. It follows from above that the closed subschemes $\overline{Y}_0(\fp)_{\phi(I),J}$ satisfy the same properties as the $S^\prime_{\phi^\prime(I),J}$; they are therefore reduced complete intersections of dimension $\vert \Theta_F \setminus \Sigma_\infty \vert - \vert I \cap J \vert$ and smooth if $I \cup J$. Furthermore, it follows that every irreducible component of $\overline{Y}_0(\fp)$, with its reduced structure, is smooth of dimension  $\vert \Theta_F \setminus \Sigma_\infty \vert$ and is contained in a unique $\overline{Y}_0(\fp)_{\phi^\prime(J^c),J}$ for $J \subset \Theta_{F,\Fp} \setminus \Sigma_{\infty,\Fp}$. These irreducible components are precisely the connected components of $\overline{Y}_0(\fp)_{\phi^\prime(J^c),J}$.

It is also straightforward to check that the stratification is compatible with the Hecke action for varying $U^\prime$. That is, if we take suitable $U^\prime_1$ and $U^\prime_2$ and $g \in G^\prime_\Sigma(\bA_{F}^{(p)})$ such that $g^{-1}U^\prime_1g \subset U^\prime_2$, then the morphisms $\tilde{\rho}_g$ and $\rho_g$ restrict to morphisms on the corresponding closed subschemes for any $I,J\subset \Theta_{F,\Fp}$.

As in the Hilbert case, we have the following proposition:

\begin{prop}
    Let $\Fp= \fp_1 \cdots \fp_n$ be a product of primes $\fp_i \vert p$ such that $\Sigma_{\infty,\fp_i} \neq \Theta_{F,\fp_i}$ for all $i$, writing $I_{\fp_i} = I \cap \Theta_{F,\fp_i}$ and $J_{\fp_i} = J \cap \Theta_{F,\fp_i}$, we have a Hecke equivariant isomorphism 
    \[ {\overline{Y}^\prime}_0(\Fp)_{\phi^\prime(I),J} \xrightarrow{\sim} {\overline{Y}^\prime}_0(\Fp)_{\phi^\prime(I_{\fp_1}),J_{\fp_1}} \times_{\overline{Y}^\prime} \cdots \times_{\overline{Y}^\prime} {\overline{Y}^\prime}_0(\Fp)_{\phi^\prime(I_{\fp_n}),J_{\fp_n}}.\]
\end{prop}

We will also show in corollary \ref{iwahori strata non empty} that the strata are non-empty by showing that the the natural projection $(\underline{A_1},\underline{A_2},f,g) \mapsto \underline{A_1}$ induces an isomorphism ${\overline{Y}^\prime}_0(\Fp)_{\Theta_{F,\Fp} \setminus \Sigma_{\infty,\Fp},\emptyset} \to \overline{Y}^\prime$ which maps the stratum ${\overline{Y}^\prime}_0(\Fp)_{\Theta_{F,\Fp} \setminus \Sigma_{\infty,\Fp},T}$ isomorphically onto the Goren-Oort stratum $\overline{Y}^\prime_T$.

\subsection{The Quaternionic setting}\label{quaternionic iwahori}
\subsubsection{Iwahori level Quaternionic Shimura varieties}
Let $\Sigma$, $G_\Sigma$ and $G^\prime_\Sigma$ be as usual and let $\Fp = \fp_1 \cdots \fp_n$ be a product of distinct primes above $p$. Recall that by our assumption on $\Sigma$ we have a fixed isomorphism $\ccO_{B,p} \simeq M_2(\ccO_{F,p})$. Via this isomorphism, we define for each $\fp \vert p$, the Iwahori subgroup $I_0(\fp) \subset G_\Sigma$ given by
\[ I_0(\fp) = \{ g \in GL_2(\ccO_{F,\fp}) \, \vert \, g \equiv \left( \begin{smallmatrix}
    * & * \\ 0 & *
\end{smallmatrix}\right) \mod \fp \}.\]
For a sufficiently small open compact subgroup $U = U^p U_p \subset G(\mathbb{A}_f)$ with $U_p = GL_2(\ccO_{F,p})$, we let $U_0(\Fp) \subset G_\Sigma$ denote the subgroup
\[U_0(\Fp) = \{ g \in U \, \vert \, g_\fp \in I_0(\fp) \textrm{ for all } \fp \vert \Fp \}.\]

Let $V_E \subset (\mathbb{A}^{(p)}_{E,f})^\times$ be a sufficiently small open compact subgroup relative to $U$ and let $U^\prime \subset G^\prime(\mathbb{A}_f)$ be the image of $U \times V_E$ so that $U^\prime_0(\Fp)$ is the image of $U_0(\Fp) \times V_E$. 

We note that in this setting, the analogous diagram to \ref{complex cartesian} is still Cartesian; we thus set:
\[Y_{U_0(\Fp) \times V_E}(G_\Sigma \times T_E) = Y_{U^\prime_0(\Fp)}(G_\Sigma^\prime) \times_{Y_{\nu^\prime(U^\prime)}(T^\prime)} (Y_{\det(U)}(T_F) \times Y_{V_E}(T_E)),\]
where $T_F$, $T_E$, and $T^\prime$ are as in section \ref{quaternionic models}. The projections $Y_{U^\prime_0(\Fp)}(G_\Sigma^\prime) \to Y_{U^\prime}(G_\Sigma^\prime)$ yield projections 
$\pi_j: Y_{U_0(\Fp) \times V_E}(G_\Sigma \times T_E) \to Y_{U_0 \times V_E}(G_\Sigma \times T_E)$. 
Just as in Tame level, we obtain for sufficiently small open compact subgroups $U_1,U_2 \subset G_\Sigma(\mathbb{A}_f)$, $g \in G_\Sigma(\mathbb{A}_f)$ with $g^{-1}U_1 g \subset U_2$, and $y \in (\mathbb{A}_{E,f}^{(p)})^\times$, morphisms $\rho_{(g,y)} : Y_{U_{1,0}(\Fp) \times V_E}(G_\Sigma \times T_E) \to Y_{U_{2,0}(\Fp) \times V_E}(G_\Sigma \times T_E)$ such that the diagram
\begin{center}
\begin{tikzcd}
    Y_{U_{1,0}(\Fp) \times V_E}(G_\Sigma \times T_E) \arrow[r,"\rho_{(g,y)}"] \arrow[d,"\pi_{1,j}"] & Y_{U_{2,0}(\Fp) \times V_E}(G_\Sigma \times T_E) \arrow[d,"\pi_{2,j}"]\\
    Y_{U_{1} \times V_E}(G_\Sigma \times T_E) \arrow[r,"\rho_{(g,y)}"] & Y_{U_{2} \times V_E}(G_\Sigma \times T_E)
\end{tikzcd}
\end{center}
commutes.

As in Tame level, the automorphisms $\rho_{(1,y)}$ define a free action of  $ C_{V_E} = (\mathbb{A}_{E,f}^{(p)})^\times/ \ccO_{E,(p)}^\times V^p_E$. We thus define $Y_{U_0(\Fp)}(G_\Sigma)$ to be the quotient of $Y_{U_0(\Fp) \times V_E}(G_\Sigma \times T_E)$ by the action of $C_{V_E}$. The natural projection splits and yields an isomorphism 
\[Y_{U_0(\Fp) \times V_E}(G_\Sigma \times T_E) \xrightarrow{\sim} Y_{U_0(\Fp)}(G_\Sigma) \times_\ccO Y_{V_E}(T_E),\]
compatibly, via projections, with the similar splitting at Tame level. This definition is independent of the choice of $V_E$.

We say that $Y_{U_0(\Fp) }(G_\Sigma )$ is the integral Quaternionic model of Iwahori level $U_0(\Fp)$. It comes with projections $\pi_i: Y_{U_0(\Fp) }(G_\Sigma ) \to Y_{U }(G_\Sigma )$ for $i=1,2$ and an action of $G_\Sigma(\mathbb{A}_f^{(p)})$ such that the projections are compatible with the action downstairs. Furthermore, after base-changing to $W = \ccO_{L^{nr}}$, the Cartesian diagram from lemma \ref{compatibility components unitary quaternionic} extends a Cartesian diagram
\[
\begin{tikzcd}
Y_{U_0(\Fp)}(G_\Sigma)_W \arrow[r] \arrow[d,"\pi_i"] & Y_{U^\prime_0(\Fp)}(G_\Sigma^\prime)_W \arrow[d,"\pi_i"]\\
Y_{U}(G_\Sigma)_W \arrow[r] \arrow[d,] & Y_{U^\prime}(G_\Sigma^\prime)_W \arrow[d,]\\
C \arrow[r] & C^\prime,
\end{tikzcd}
\]
where $C = C_{\det(U)}$ and $C^\prime = C_{\nu^\prime(U^\prime)}$, which identifies $Y_{U_0(\Fp)}(G_\Sigma)_W$ with an open and closed subscheme of $Y_{U^\prime_0(\Fp)}(G_\Sigma^\prime)_W$. Furthermore, the diagram is Hecke equivariant in the sense that the analogous diagram \ref{compatibility components unitary quaternionic} commutes.

Over $\overline{\bF}_p$, write $\overline{Y}_0(\Fp) = Y_{U_0(\Fp)}(G_\Sigma)_{\overline{\bF}_p} $ and $\overline{Y}^\prime_0(\Fp) = Y_{U^\prime_0(\Fp)}(G^\prime_\Sigma)_{\overline{\bF}_p} $. We thus have the map $i: \overline{Y}_0(\Fp) \to \overline{Y}^\prime_0(\Fp)$ which identifies $\overline{Y}_0(\Fp)$ with $\overline{Y}^\prime_0(\Fp) \times_{C^\prime} C$ (we could have taken this as a definition). For any pair of subsets $I,J \subset \Theta_{F,\Fp} \setminus \Sigma_{\infty,\Fp}$, we write 
\[\overline{Y}_0(\Fp)_{\phi^\prime(I),J} = \overline{Y}^\prime_0(\Fp)_{\phi^\prime(I),J} \times_{C^\prime} C.\]
$\overline{Y}_0(\Fp)$ and its strata thus enjoy all the same properties as those of $\overline{Y}^\prime_0(\Fp)$. In particular, $\overline{Y}_0(\Fp)$ is a reduced local complete intersection of dimension $\vert \Theta_F \setminus \Sigma_\infty \vert$, and the strata $\overline{Y}_0(\Fp)_{\phi^\prime(I),J}$ are reduced complete intersections of dimension $\setminus \Sigma_\infty \vert - \vert I \cap J \vert$, smooth if $I \cap J = \Theta_{F,\Fp} \setminus \Sigma_{\infty,\Fp}$. Every irreducible component, with its reduced structure, is smooth of dimension $\vert \Theta_F \setminus \Sigma_\infty \vert$ and is contained in a unique stratum $\overline{Y}_0(\Fp)_{\phi^\prime(J^c),J}$. These irreducible components are precisely the connected components of $\overline{Y}_0(\Fp)_{\phi^\prime(J^c),J}$.
Finally, the Hecke action of $G_\Sigma(\mathbb{A}_f^{(p)})$ restricts to an action on the strata.

\subsubsection{Comparison to the Hilbert case}\label{quaternion hilb iwahori}
Similarly to chapter \ref{tame}, we now have two different models for the $U_0(\Fp)$ Iwahori level Hilbert modular variety, the first one being the model introduced in section \ref{hilbert iwahori} and the second in the previous section \ref{quaternionic iwahori} with $\Sigma = \emptyset$. 

Let $U \subset G(\bA_f) = G_\emptyset(\bA_f)$ be suitably small open compact subgroup of level prime to $p$ and $U^\prime \subset G^\prime_\emptyset(\bA_f)$ such that $U^\prime$ contains the image of $U$. We defined in section \ref{model comparison} a morphism $\Tilde{i}:\widetilde{Y}_U(G) \to \widetilde{Y}_{U^\prime}(G^\prime_\emptyset)$, which sends $A$ to $A' = A \otimes_{\ccO_F} \ccO_E^2$, and showed that it induced an isomorphism 
\[\widetilde{Y}_U(G) \to \widetilde{Y}_{U^\prime}(G^\prime_\emptyset) \times_{C'} C,\]
where $C = C_{\det(U)}$ and $C^\prime = C_{\nu^\prime(U^\prime)}$. The map $\Tilde{i}$ extends to a morphism $\widetilde{Y}_{U_0(\Fp)} \to \widetilde{Y}'_{U'_0(\Fp)}\times_{C'} C$, by mapping $(\underline{A_1},\underline{A_2},f,g)$ to $(\underline{A^\prime_1},\underline{A^\prime_2},f \otimes 1,g \otimes 1)$ and we have the following result:

\begin{prop}\label{iwahori quat hilb}
The extended map $\Tilde{i}$ given above is a Hecke equivariant isomorphism.
\end{prop}

\begin{proof}
The same proof of the corresponding result of \cite[lemma 4.2.1]{2020arXiv200100530D} holds in this case. There are no extra complications which arise from considering the Pappas-Rapoport filtrations.
\end{proof}

Furthermore, for each $\tau \in \widehat{\Theta}_E$ lifting some $\theta \in \widehat{\Theta}_F$, we have canonical isomorphisms \[\ccH^1_{A_i,\tau^j} \xrightarrow{\sim} \ccH^1_{A^\prime_i,\tau^j} = \tilde{i}^* \ccH^1_{i,\tau^j},\]
such that the sections $f^{\prime *} \omega^0_{2,\tau^i} \to \omega^0_{1,\tau^i}$ and $g^{\prime *} \omega^0_{1,\tau^i} \otimes \Fp \to \omega^0_{2,\tau^i}$ correspond to $f^{*} \omega_{A_2,\theta^i} \to \omega_{A_1,\theta^i}$ and $g^{\prime *} \omega_{A_1 \otimes \Fp^{-1},\theta^i} \to \omega_{A_2,\theta^i}$. In particular $\tilde{i}$ restricts to isomorphisms 

\[\widetilde{Y}_U(G)_{\phi(I),J} \to \widetilde{Y}'_{U'}(G^\prime_\emptyset)_{\phi^\prime(I),J} \times_{C'} C,\]

and descends over $\overline{\bF}_p$ to Hecke equivariant isomorphisms  

\[\overline{Y}_0(\Fp)_{\phi(I),J} \xrightarrow{\sim} \overline{Y}_{\emptyset,0}(\Fp)_{\phi(I),J}.\]

\section{Partial Raynaud \texorpdfstring{$\mathbb{F}_q$}{Fq}-schemes and Partial Essential Frobenius isogenies }
In the next chapter we will heavily make use of Crystalline Dieudonn\'{e} theory over smooth bases as described in \cite{BerthelotMessing} and of Raynaud's theory (\cite{BSMF_1974__102__241_0}) of $\bF_q$-vector schemes. We thus opt to dedicate an entire chapter to introduce the objects and results we will need in the following. Furthermore, at the end of the chapter, we will make use of this theory to construct a section of the forgetful morphism $\overline{Y}_0(\Fp) \to \overline{Y}$ which will restrict to an isomorphism $\overline{Y}_0(\Fp)_{\Theta_F \setminus \Sigma_\infty,\emptyset} \xrightarrow{\sim} \overline{Y} $, compatibly with the stratifications. This can be seen as a first case of our main theorem \ref{unitary thm}.

\subsection{Raynaud Theory}

\subsubsection{Raynaud $\bF_q$-vector schemes}\label{full raynaud}
We start by briefly recalling Raynaud's theory (\cite{BSMF_1974__102__241_0}) of $\bF_q$-vector schemes, tailored to our situation. Let $\bF = \bF_q$ be a finite field of characteristic $p$, and $S$ be a scheme over  $\textrm{Spec}(\ccO)$, the spectrum of a discrete valuation ring $ \ccO$, with residue characteristic $p$, that contains all $(q-1)$-th roots of unity. We apologize for this temporay different use of $\ccO$ and $\bF$. Fix an identification $\mu_{q-1} \subset \ccO$ and let $M =  \Hom(\bF^\times,\mu_{q-1})$ be the character group of $\bF^\times$. Every embedding $\theta \in \Theta =  \{ \bF \to \overline{\bF}_p \}$ yields, via Teichm\"{u}ller lifts, a character that we call the fundamental character $\chi_\theta$. In particular, for any $\theta \in \Theta$, we have $\chi_{\theta}^p = \chi_{\phi \circ \theta}$ and any nontrivial character $\chi$ has a unique $p$-adic expression: 
\[\chi = \prod_\theta \chi_\theta^{n_\theta} \qquad 0 \leq n_\theta \leq p-1,\]
where not all $n_\theta$ are identically $0$ or $p-1$. In both of these case, $\chi$ is the trivial character.

Let $G/S$ be a finite flat $\bF$-vector scheme, that is a finite flat group scheme with an action of $\bF$. If we write $\mathcal{I}$ for the augmentation ideal of $G/S$, the action $\bF$ yields a decomposition
\[ \mathcal{I} = \bigoplus_{\chi \in M} \mathcal{I}_\chi\]
into locally free sheaves $I_\chi$ where the action of $\bF^\times$ factors through $\chi$.

We say that $G = \underline{\textrm{Spec}}_S(\mathcal{A})$ is a Raynaud $\bF$-vector scheme, or a Raynaud group scheme, if it satisfies condition $( \star \star)$ of \cite{BSMF_1974__102__241_0}. That is, for each fundamental character $\chi_\theta$, with $\theta \in \Theta$, the locally free sheaf $\mathcal{I}_{\chi_\theta}$ is a line bundle. For such a group scheme, and any $\theta \in \Theta$, the multiplication and comultiplication on $\mathcal{A}$ yield morphisms 
\[s_\theta: \mathcal{I}_{\chi_\theta}^p \to \mathcal{I}_{\chi_{\phi \circ \theta}} \quad \textrm{ and } \quad t_\theta: \mathcal{I}_{\chi_{\phi \circ \theta}} \to \mathcal{I}_{\chi_\theta}^p,\]
such that $s_\theta \circ t_\theta$ and $t_\theta \circ s_\theta$ are given by multiplication by $w$ for some fixed $w \in p\ccO^\times$. 

Conversely, consider any set of tuples $(\ccL_{\theta},s_\theta,t_\theta)_{\theta \in \Theta}$ where each $\ccL_\theta$ is a line bundle, $s_\theta: \ccL_\theta^p \to \ccL_{{\phi \circ \theta}}$, $t_\theta: \ccL_{\phi \circ \theta} \to \ccL_\theta^p$ are morphisms such that $s_\theta \circ t_\theta$ and $t_\theta \circ s_\theta$ are given by multiplication by $w$ for some fixed $w \in p\ccO^\times$. We call such a set of tuples a Raynaud datum. Set $\mathcal{L} = \bigoplus_{\theta \in \Theta} \ccL_\theta$ and $ \mathcal{A} = \underline{\textrm{Sym}}_{\ccO_S}(\ccL) / \mathcal{I} $ where $\mathcal{I}$ is the ideal generated by $(s_\theta-1)\ccL^p_\theta$ ranging over $\theta \in \Theta$. Then $\mathcal{A}$ comes with an action of $\bF$ by letting $\bF$ act on each $\ccL_\theta$ via $\chi_\theta$ and is given the structure of a Hopf algebra by declaring its Cartier dual to be $\textrm{Sym}_{\ccO_S} (\mathcal{L}^\vee)/ \mathcal{J}$, where $\mathcal{J}$ is the ideal generated by $(t^\vee_{\theta}-1)\mathcal{L}_\theta^{-p}$. It can be shown that 
\[G = \underline{\textrm{Spec}}_S(\mathcal{A})\]
is a Raynaud group scheme and Raynaud shows that these two operations yield a one to one correspondence between isomorphism classes of Raynaud group schemes over $S$ and isomorphism classes of Raynaud data. Under this correspondence, the dual of a Raynaud group scheme, which is also Raynaud, is given by the Raynaud datum $(\ccL^\vee_\theta, t^\vee_\theta,s_\theta^\vee)$.

Back to our situation, we let $\ccO$ be as in \ref{unitary moduli} and let $\Fp = \fp_1 \cdots \fp_n$ be a product of distinct primes $\fp_i \vert p$ of $F$. We say that a finite flat group scheme $G/S$ with an action of $\ccO_E/\Fp \simeq \bigoplus \ccO_E/\fq_i \oplus \bigoplus \ccO_E / \fq^c_i$, is a Raynaud $\ccO_E / \Fp$-scheme if each component $G[\fq_i]$, respectively $G[\fq^c_i]$, is a Raynaud $\ccO_E / \fq_i$, respectively  $\ccO_E / \fq^c_i$ vector scheme. Let $S = \widetilde{Y}_{U^\prime_0(\Fp)}(G^\prime_\Sigma)$, with $\Sigma$, $G^\prime_\Sigma$ and $U^\prime$ as usual. Write $H = e_0 \ker f$, be the reduced kernel of the universal isogeny $f$ over $S$. We now prove the following lemma which justifies the prior discussion:

\begin{lem}\label{kernel is raynaud}
$H/S$ is a $\ccO_E/\Fp$-Raynaud scheme.
\end{lem}

\begin{proof}
We first remark that $H$ is $\ccO_E/\Fp$-module scheme. Write $\mathcal{I}$ for the augmentation ideal of $H$ and pick a prime $\fp_i \vert \Fp$ and consider the prime $\fq_i$ over $\fp_i$. By the definition of the moduli problem in section \ref{iwahori unitary def}, $H[\fq_i]$ has rank $f_{\fq_i}$. Since the sheaves $\mathcal{I}_\tau = \mathcal{I}_{\chi_\tau}$ for $\tau \in \widehat{\Theta}_{E,\fq_i}$ are locally free, they must have constant rank on connected components. Furthermore, since $S$ is flat over $\ccO$, which is complete, every connected component of the special fiber $S_\bF$ of $S$ lies in a connected component of $S$ which thus has non-empty generic fiber. Therefore, we may check on the generic fiber $S_L$ of $S$ and in fact, we may check after taking the base change to an algebraic closure. Let $x \in S(\overline{L})$ be a closed point, then $H_x$ is \'{e}tale over $x$ since $L$ is of characteristic zero. Since $H_x$ has rank one over $\ccO_E/\fq_i$, its underlying Hopf algebra must be isomorphic to the Hopf algebra of functions $\ccO_E/\fq_i \to \overline{L}$. The result then follows and the case of $\fq^c_i$ is entirely analogous.
\end{proof}

\begin{rem}
The same proof shows that the universal kernel $I / \widetilde{Y}_{U_0(\Fp)}(G)$ over the Iwahori level Hilbert modular variety is also a Raynaud $\ccO_F/\Fp$-module scheme given by the Raynaud datum $(\ccL_\theta,s_\theta,t_\theta)$. Recall from section \ref{quaternion hilb iwahori} that we have the morphism $\tilde{i} : \widetilde{Y}_{U_0(\Fp)}(G) \to \widetilde{Y}^\prime_{U^\prime_0(\Fp)}(G^\prime_\emptyset)$ which, over $W$, identifies the former variety as a collection of connected components of the latter. Under this morphism, we have $\tilde{i}^*H = I \otimes_{\ccO_F} \ccO_{E}$ and so for every prime $\fp \vert \Fp$ an $\ccO_F/\fp = \ccO_E/\fq$-linear isomorphism $\tilde{i}^*(H[\fq]) \to I[\fp]$. In particular, we obtain for every $\theta \in \widehat{\Theta}_{F,\fp}$ an isomorphism $\tilde{i}^*\ccL_{\tilde{\theta}} \to \ccL_\theta$.
\end{rem}

\subsubsection{Partial Raynaud $\bF_q$-schemes}\label{partial raynaud}

In this section we provide a mild generalization of the notion of Raynaud group schemes in characteristic $p$, that we will need in the next chapter.

We keep the same notation as in the previous section. From now on $S$ will denote a scheme over $\ccO/\mathfrak{m}$, in particular the element $w \in p \ccO^\times$ mentioned above is zero over $S$. 

\begin{defn}
Let $J \subset \Theta$ be any non empty subset, a partial Raynaud datum for $J$ is a collection of triples $(\mathcal{L}_\theta, s_\theta,t_\theta)_{\theta \in \Theta} $ where $\mathcal{L}_\theta$ is a line bundle over $S$ if $\theta \in J$ and the zero sheaf otherwise, and \[s_\theta: \ccL_\theta^p \to \ccL_{\phi \circ \theta} \qquad t_\theta: \ccL_{\phi \circ \theta} \to \ccL_\theta^p\]
are morphisms of sheaves such that $s_\theta \circ t_\theta = 0$, $t_\theta \circ \theta = 0$. 
\end{defn} 

\begin{rem}
Note that if either $\theta \notin J$ or $\phi \circ \theta \notin J$, then both $s_\theta$ and $t_\theta$ are automatically zero. If $J = \Theta$, then we recover the notion of a Raynaud datum as defined above in section \ref{full raynaud}, hence the name "Partial Raynaud datum". 
\end{rem}

Given a partial Raynaud datum $(\mathcal{L}_\theta, s_\theta,t_\theta)$, we define a finite flat group scheme $H/S$ of rank $p^{\vert J \vert}$ and with an action of $\bF$ as follows:

For each $\theta \in \Theta$, let $\bF$ act on $\ccL_\theta$ via $\chi_\theta$ and set $\ccL = \bigoplus \ccL_\theta$, and $\mathcal{A}_H = (\textrm{Sym}_{\ccO_S} \mathcal{L}) / \mathcal{I}$, where $\mathcal{I}$ is the ideal of $\textrm{Sym}_{\ccO_S} \mathcal{L}$ generated by $(s_\theta - 1)\ccL^p_\theta$. $\mathcal{A}$ then comes with a natural action of $\bF$. As a sheaf with $\bF$ action, without considering the algebra structure, we have 
\[\mathcal{A}_H = \ccO_S \oplus \bigoplus \left( \bigotimes \ccL_\theta^{m_\theta} \right) = \ccO_S \oplus \bigoplus_{\chi \in M} \ccL_\chi,\] where the first sum runs over tuples $(m_\theta)_{\theta \in \Theta_F}$ such that $0 \leq m_\theta \leq p-1$ and not all $m_\theta = 0$, and $\mathcal{L}_\chi = \bigotimes \bigotimes \ccL_\theta^{m_\theta}$ where the $m_\theta$ are given uniquely by $\chi = \prod_\theta \chi_\theta^{m_\theta}$. $\bF^\times$ acts on $\ccL_\chi$ via $\chi$. It follows that $\mathcal{A}_H$ is locally free of rank $p^{\vert J \vert}$ and we set 
\[H = \underline{\textrm{Spec}}_S \, \mathcal{A}_H.\]

We give $\mathcal{A}_H$ a bialgebra structure by declaring its dual to be $\mathcal{A}_{H^\vee} = (\textrm{Sym}_{\ccO_S} \mathcal{L}^\vee) / \mathcal{J}$, where $\mathcal{J}$ is the ideal of $\textrm{Sym}_{\ccO_S} \mathcal{L}^\vee$ generated by $(t_\theta^\vee - 1)(\ccL^\vee_\theta)^p$. Similarly as before, as a sheaf, we have 
\[\mathcal{A}_{H^\vee} = \ccO_S \oplus \bigoplus_{\chi \in M} \ccL^\vee_\chi = \mathcal{A}^\vee_H.\]
The comultiplication $\mathcal{A}_H \to \mathcal{A}_H \otimes_{\ccO_S} \mathcal{A}_H$ is thus given by the dual of the multiplication map $\mathcal{A}_{H^\vee} \otimes_{\ccO_S} \mathcal{A}_{H^\vee} \to \mathcal{A}_{H^\vee}$. The unit and counit $\ccO_S \to \mathcal{A}_H \to \ccO_S$ are given by the natural copy of $\ccO_S$ in $\mathcal{A}_H$ and the antipode $\iota$ is given by $\iota(1) = 1$ where $1$ is in the copy of $\ccO_S$ and $\iota: \mathcal{L}_\theta \to \ccL_\theta$ is multiplication by $-1$.

\begin{prop}\label{proposition is a group scheme}
 The data defined above defines a Hopf algebra structure on $\mathcal{A}$ and so gives $H$ the structure of a finite flat group scheme with an action of $\bF$.
\end{prop}

\begin{proof}
One can check with explicit calculations that $\mathcal{A}$ is a Hopf algebra but we prefer to give a slicker proof, inspired by \cite[Lemma 4.1.1]{2020arXiv200100530D}.

Define a full Raynaud datum $(\ccL^\prime_\theta, s^\prime_\theta,t^\prime_\theta)$ by setting $(\ccL^\prime_\theta, s^\prime_\theta,t^\prime_\theta) = (\ccL_\theta,s_\theta,t_\theta)$ for $\theta \in J$ and  for $(\ccL^\prime_\theta, s^\prime_\theta,t^\prime_\theta) = (\ccO_S,0,0)$ for $\theta \notin J$. Then, by the previous section, this gives rise to a Raynaud group scheme $H^\prime = \underline{\textrm{Spec}}_S (\mathcal{A}^\prime)$ and 
\[\mathcal{A}^\prime = (\textrm{Sym}_{\ccO_S} \mathcal{L}^\prime) / \mathcal{I}^\prime = \mathcal{A} \oplus \bigoplus_{\theta \notin J}  \textrm{Sym}_{\ccO_S} (\ccO_S) / \ccO_S^{\otimes p},\]
Since $s_\theta = 0$ whenever $\theta \notin J$, the natural projection $\mathcal{A}^\prime \to \mathcal{A}$ is compatible with the ring structures and this yields a closed immersion of schemes with $\bF$-action $H \hookrightarrow H^\prime$. To show that $H$ is a group, it suffices to show that the comultiplication on $H$ is compatible via the embedding above to the one on $H^\prime$. That is, the diagram
\[
\begin{tikzcd}
    (\mathcal{A}^\prime)^\vee \otimes_{\ccO_S} (\mathcal{A}^\prime)^\vee \arrow[r] & (\mathcal{A}^\prime)^\vee \\
    \mathcal{A}^\vee \otimes_{\ccO_S} \mathcal{A}^\vee \arrow[r] \arrow[u] & \mathcal{A}^\vee \arrow[u]
\end{tikzcd}
\]
is a commutative. Since $t_\theta = 0$ for all $\theta$ such that $\phi \circ \theta \notin J$, we conclude that the multiplication on $H$ is compatible with the one on $H^\prime$ hence $H^\prime$ is a subgroup and thus a group.
\end{proof}

\begin{defn}
A morphism of Raynaud data $f: (\ccL_\theta,s_\theta,t_\theta) \to (\ccL'_\theta,s'_\theta,t'_\theta)$ is a collection of morphisms $f_\theta: \ccL_\theta \to \ccL^\prime_\theta$ such that $s'_\theta \circ f^p_\theta = f_{\phi \circ \theta} \circ s_\theta$ and $t'_\theta \circ f_{\phi \circ \theta} = f^p_{ \theta} \circ t_\theta$.
\end{defn}

It is clear that such a morphism yields an $\bF$-linear homomorphism $f:H' \to H$ of associated groups and any such homomorphism of groups comes from a morphism of Raynaud data. We immediately have the following corollary:

\begin{cor}\label{produce subgroup}
Let $(\ccL_\theta,s_\theta,t_\theta)$ be a Raynaud datum for $J$ and let $J' \subset J$ be a subset such that 
\[
  \theta \notin J^\prime,\phi \circ \theta \in J' \Rightarrow s_\theta = 0 \, \textrm{   and   } \, \theta \in J^\prime,\phi \circ \theta \notin J' \Rightarrow  t_\theta = 0.
\]
Let $(\ccL'_\theta,s'_\theta,t'_\theta) = (\ccL_\theta,s_\theta,t_\theta)_{\theta \in J'}$ be the sub-Raynaud datum for $J'$. Write $H$ for the group scheme attached to $(\ccL_\theta,s_\theta,t_\theta)$ and $H^\prime$ the group for $(\ccL'_\theta,s'_\theta,t'_\theta)$. Then the natural map $(\ccL_\theta,s_\theta,t_\theta) \to (\ccL'_\theta,s'_\theta,t'_\theta)$ is a morphism of Raynaud data which induces the closed immersion $H' \hookrightarrow H$, compatibly with all structures.
\end{cor}

\begin{proof}
    We prefer to spell this out to make the required conditions clear. Write $f_\theta: \ccL_\theta \to \ccL^\prime_\theta$ for the natural morphism, i.e. $f_\theta$ is the identity if $\theta \in J^\prime$ and zero otherwise. By the above definition, we need to show that  $s'_\theta \circ f^p_\theta = f_{\phi \circ \theta} \circ s_\theta$ and $t'_\theta \circ f_{\phi \circ \theta} = f^p_{ \theta} \circ t_\theta$ for all $\theta \in \Theta$. Because both Raynaud data only have non-zero entries in $J$, it suffices to check for $\theta \in J$.
    
    Consider the equation $s'_\theta \circ f^p_\theta = f_{\phi \circ \theta} \circ s_\theta$. If both $\theta, \phi \circ \theta \in J^\prime$, then both $f_\theta$ and $f_{\phi \circ \theta}$ are the identity and $s^\prime_\theta = s_\theta$, the equality is clear. If neither are in $J^\prime$, then both $f_\theta$ and $f_{\phi \circ \theta}$ are zero and the equation is clear. If $\theta \in J^\prime$ and $\phi \circ \theta \notin J^\prime$, then $s^\prime_\theta = 0 $ and $f_{\phi \circ \theta} = 0$, this is again clear. Finally, if $\theta \notin J^\prime$ and $\phi \circ \theta \in J^\prime$, then $f_\theta=0$, $f_{\phi \circ \theta}= \id$, $s^\prime_\theta = 0$ and the equation thus becomes $s_\theta = 0$. Similar considerations for $t'_\theta \circ f_{\phi \circ \theta} = f^p_{ \theta} \circ t_\theta$ show that we require $t_\theta = 0$ for $\theta$ such that $\theta \in J^\prime$ and $\phi \circ \theta \notin J^\prime$. The result follows.
\end{proof}

\subsubsection{Subgroups of the Universal Kernel}\label{subgroups of kernel}
We now apply the notions introduced above to the universal kernel of $\widetilde{Y}_{U^\prime_0(\Fp)}(G^\prime_\Sigma)$ over a smooth stratum $S^\prime_{\phi^\prime(I),J}$.

Let $\Sigma, G^\prime_\Sigma, U^\prime$ and $\Fp$ be as usual. Consider $S^\prime = \widetilde{Y}_{U^\prime_0(\Fp)}(G^\prime_\Sigma)_\bF$ and let $I,J \subset \Theta_F \setminus \Sigma_\infty$ be such that $I \cup J = \Theta_F \setminus \Sigma_\infty$. Then we have the stratum $S^\prime_{\phi^\prime(I),J}$ defined in section \ref{unitary stratification} which, because $I \cup J = \Theta_F \setminus \Sigma_\infty$, is smooth of codimension $\vert I \cap J \vert$ by corollary \ref{corollary strata dim smooth}. Furthermore, over $S^\prime_{\phi^\prime(I),J}$, we have the reduced universal kernel $H_1 = e_0 \ker f$, where $f$ is the universal isogeny. By lemma \ref{kernel is raynaud}, it is a Raynaud $\ccO_E/ \Fp$-scheme and so is attached to a Raynaud datum $(\ccL_\tau,s_\tau,t_\tau)_{\tau \in \widehat{\Theta}_{E,\Fp}}$. By the same argument, the reduced kernel $H_2 = e_0 \ker g \subset A_2[\Fp]$ is also a Raynaud $\ccO_E/ \Fp$-scheme. In fact, the moduli problem gives us an $\ccO_E$-antilinear isomorphism 
\[H_1 \xrightarrow{\sim} (e_0A_1[\Fp]/H_1)^\vee \simeq H_2^\vee\]
which restricts, for each prime $\fq \vert \Fp \ccO_F$ to an $\ccO_E/\fq = \ccO_F/\fp$-linear isomorphism $H_1[\fq] \simeq H_2[\fq^c]^\vee$. Hence $H_2[\fq]$ is given by the dual Raynaud datum of $H_1[\fq^c]$.

We now show that certain $s_\tau$ and $t_\tau$ vanish over $S^\prime_{\phi^\prime(I),J}$. Note that $s_\tau$, which is part of the Raynaud datum for $H_1$, should not be confused with the numbers $s_{\tau^i} \in \{0,1,2\}$. We also denote the dimension counting function by $d_\tau(i)$ to not clash with our current use of $s_\tau$. It is given by $d_\tau(i) = \sum_{j \leq i} s_{\tau^i}$.

We first record the following lemma:

\begin{lem}\label{free surjective is isomorphism}
Let $A$ be a commutative ring and $M$ an $A$-module such that $M$ is simultaneously isomorphic to $A/sA$ for some $s \in A$, and $A^n/f(A^{n-1})$, where $f: A^{n-1} \to A^n$ is injective. Then $M$ is free of rank one.
\end{lem}

\begin{proof}
Choose a lift $x \in A^n$ of $1 \in A/sA$ under the surjection $A^n \to M \simeq A/sA$. Define the map $h:A^n = A^{n-1} \times A \to A^n$ by $(a,b) \mapsto f(a) + bx$. Then $h$ is surjective and is automatically an isomorphism as it is an endomorphism of a free module of finite rank. Therefore $sx \in f(A^{n-1})$ implies $s=0$ and $M \simeq A$ is free of rank one as desired. 
\end{proof} 

\begin{lem}\label{raynaud vanishing}
Let $\tau \in \widehat{\Theta}_{E, \Fp}$ and suppose that there exists an $i$, which we pick smallest, such that $\tau^i \vert_F = \theta^i \notin \Sigma_{\infty,\Fp}$. If $\theta^i \in I$, then $s_{\phi^{-1}(\tau)}=0$. Similarly, if $\theta^i \in J$ then $t_{\phi^{-1}(\tau)}=0$
\end{lem}

\begin{proof}
Recall that the line bundles $\ccL_\tau$ are given by the $\tau$ components of the augmentation ideal $\ccI$ of $H_1$. In particular, via the isomorphism 
\[\omega_{H_1} \simeq \ccI / \ccI^2 \simeq \bigoplus_{\tau \in \widehat{\Theta}_{E,\Fp}} \ccL/ s_{\phi^{-1} (\tau)}\ccL^p_{\phi^{-1} (\tau)},\]
we see that $s_{\phi^{-1}(\tau)} = 0$ if and only if $\omega_{H_1,\tau}$ is a line bundle. Similarly, $t_{\phi^{-1}(\tau)} = 0$ if and only if $\omega_{H_1^\vee,\tau} \simeq \omega_{H_2,\tau^c}$ is a line bundle.

Suppose that $\theta^i \in I$. Then, by definition of the stratification, and by section \ref{iwahori unitary local}, the kernel of $f^*_{\tau^i}: \ccH^1_{2,\tau^i} \to \ccH^1_{1,\tau^i}$ is $\omega^0_{2,\tau^i}$. Recall from section \ref{unitary hasse} that we have the morphism $V_{\textrm{es},\tau^i}^{\tau^1}: \ccH^1_{2,\tau^i} \to \ccH^1_{2,\tau^1}$ which is an isomorphism by our assumption on $i$. We thus see that the kernel of $f^* : \omega^0_{2,\tau} \to \omega^0_{1,\tau}$ is equal to $V_{\textrm{es},\tau^1}^{\tau^i}\omega^0_{2,\tau^i}$. By definition of the Pappas-Rapoport filtration, $f^* \omega^0_{2,\tau} \simeq \omega^0_{2,\tau} / V_{\textrm{es},\tau^1}^{\tau^i}\omega^0_{2,\tau^i}$ is thus locally free of rank $d_\tau(e_\fp)-1$ and, $\ccL/ s_{\phi^{-1} (\tau)}\ccL^p_{\phi^{-1} (\tau)} \simeq \omega_{H_1,\tau}$ sits in the following exact sequence:
\[\begin{tikzcd} 0 \arrow[r] & f^* \omega^0_{2,\tau} \arrow[r] &  \omega^0_{1,\tau} \arrow[r] & \ccL/ s_{\phi^{-1} (\tau)}\ccL^p_{\phi^{-1} (\tau)} \arrow[r] & 0.\end{tikzcd}\]
We conclude by lemma \ref{free surjective is isomorphism}, that it is a line bundle and that $s_{\phi^{-1}(\tau)}=0$. The case of $t_{\phi^{-1}(\tau)}=0$ is entirely analogous, since $\theta^i \in J$ implies that $\ker g^*_{\tau^i} = \omega^0_{A_1 \otimes \Fp^{-1},\tau^i}$.
\end{proof}

We end with the following important lemma. Let $T \subset \Theta_{F,\Fp} \setminus \Sigma_{\infty,\Fp}$ be a subset such that $I^c \subset T \subset J$. Suppose further that for each prime $\fp \vert \Fp$, $( I \cap J ) \cap \Theta_{F,\fp} \neq (\Theta_{F,\fp} \setminus \Sigma_{\infty,\fp})$ (in particular $ \Sigma_{\infty,\fp} \neq \Theta_{F,\fp}$). Extend $T$ to a subset $T^\prime \subset \Theta_{F,\Fp}$ by appending to $T$, the maximal chains $\{\phi^{-n}(\beta), \cdots, \beta \}$ such that $\phi(\beta) \in T$. Define the set 
\[T^1 = \{ \theta \, \vert \, \theta^1 \in T^\prime \}, \]
and write $\mathfrak{Q} = \fq_1 \cdots \fq_n$ where $\Fp = \fp_1 \cdots \fp_n$.
\begin{lem}\label{C_T}
    There are closed subgroups $C_T \subset H_1[\mathfrak{Q}]$ and $C_{T^c} \subset H_2[\mathfrak{Q}^c]$ corresponding to the partial Raynaud data $(\ccL_{\tilde{\theta}}, s_{\tilde{\theta}}, t_{\tilde{\theta}})_{\phi^{-1}(T^1)}$ and $(\ccL^\vee_{\tilde{\theta}^c}, t^\vee_{\tilde{\theta}^c},s^\vee_{\tilde{\theta}^c})_{\phi^{-1}((T^1)^c)}$ respectively.
\end{lem}

\begin{proof}
To prove the existence of $C_T$, it suffices to show, by corollary \ref{produce subgroup} that $s_{\phi^{-1}(\tilde{\theta})} = 0$ if $\theta \notin T^1$ and $\phi(\theta) \in T^1 $, and that $t_{\phi^{-1}(\tilde{\theta})} = 0$ if $\theta \in T^1$ and $\phi(\theta) \notin T^1 $. 

Suppose that $\theta \notin T^1$ but $\phi(\theta) \in T^1 $. Then by definition of $T^1$, there must be an $i$ such that $\theta^i \notin \Sigma_{\infty,\fP}$ and in particular, the smallest such $i$ must satisfy $\theta^i \in T^c \subset I$. Then $s_{\phi^{-1}(\tau)}=0$ by lemma \ref{raynaud vanishing}. Similarly, if $\theta \in T^1, \phi(\theta) \notin T^1 $, there is an $i$, smallest such that $\theta^i \notin \Sigma_{\infty,\Fp}$ and $\theta^i \in T \subset J$. Hence $t_{\phi^{-1}(\tilde{\theta})}=0$ by lemma \ref{raynaud vanishing}.

The case for $C_{T^c}$ is similar.
\end{proof}

\subsection{Dieudonn\'{e} modules, crystals, and Raynaud subgroups}
We establish our conventions and collect the results we need from Dieudonn\'{e} theory, and establish a procedure to create partial Raynaud subgroups of abelian varieties.

\subsubsection{Dieudonn\'{e} modules}\label{Dieudonne normal}
Let $A$ be an abelian variety of dimension $g$ over $\overline{\bF}_p$, we denote by $\mathbb{D} = \mathbb{D}(A[p^\infty])$ its contravariant Dieudonn\'{e} module. It is a free module over $W = W(\overline{\bF}_p)$ of rank $2g$. It comes with a $\phi$-semilinear endomorphism $\Phi$ induced by the Frobenius on $A$ such that $p\mathbb{D} \subset \Phi( \bD)$ and we denote by $V$ the $\phi^{-1}$-semilinear map such that $\Phi \circ V = V \circ \Phi = p$.

There is a canonical isomorphism $\bD/p \bD \simeq H^1_{\dR}(A/\overline{\bF}_p)$ such that $H^0(A,\Omega^1_{A/ \overline{\bF}_p})$ is identified with $V \bD / \bD$.  Furthermore, there is a canonical isomorphism $\bD(A^\vee[p^\infty]) \simeq \Hom(\bD, W) = \bD^\vee$ such that the induced perfect pairing $\bD \times \bD^\vee \to W$ is anti-symmetric under the canonical isomorphism $(A^\vee)^\vee \simeq A$. 

The Dieudonn\'{e} functor restricted over the category of finite flat commutative group schemes killed by $p$ defineds an equivalence of categories to the category of $\overline{\bF}_p$-vector spaces with a $\phi$-semilinear endomorphism $\Phi$ and $\phi^{-1}$-semilinear endomorphism $V$ such that $V \Phi = \Phi V =0$. For any such group scheme $G$, $\textrm{rank}(G) = p^{\dim \mathbb{D}(G)}$. The functor $\bD$ is compatible in the obvious way with Cartier Duality and for any $G$, we have canonical isomorphisms
\[\bD(G) / \Phi \bD(G) \simeq \Hom_{\overline{\bF}_p}(\textrm{Lie}(G)^{(p)}, \overline{\bF}) \quad \textrm{ and } \quad \ker V \simeq \textrm{Lie}(G^\vee)^{(p)}.\]

We recall that any $p$-isogeny $f: A_1 \to A_2$ induces an injective morphism $f^* : \bD_1 \to \bD_1$ such that $\bD_1 / f^* \bD_1 \simeq \bD( \ker f)$.

Aplying the $\bD$ to the canonical isomorphism $A^\vee[p] \simeq (A[p])^\vee$ gives the reduction of the isomorphism $\bD(A^\vee[p^\infty]) \simeq \bD^\vee$ described above. In particular if $\lambda: A \to A^\vee$ is a prime to $p$ quasi-polarization, then the $\lambda$-Weil pairing on $A[p^\infty]$ yields an isomorphism $\bD^\vee \simeq \bD$ corresponding to an alternating perfect pairing on $\bD$ whose reduction modulo $p$ defines the isomorphism $\Hom_{\overline{\bF}_p}(\bD/p\bD, {\overline{\bF}_p})  \simeq \bD/p\bD$.

\subsubsection{Crystalline Dieudonn\'{e} theory}\label{dieudonne crystal}
We will need the relative version of Dieudonn\'{e} theory in the case of finite flat group schemes killed by $p$ over a smooth base of characteristic $p$. We follow the conventions of \cite{berthelot1982theorie}.

Let $S$ be a scheme of characteristic $p$. Write $\Sigma_0 = \bZ_p$ and $\Sigma_1 = \bF_p$; for $n =0,1$, write $\CR(S/\Sigma_n)$ for the big crystalline site of $S/\Sigma_n$ with the fppf topology and $(S/\Sigma_n)_{\CR}$ its category of sheaves. We note $\CR(S/\bF_p)$ sits naturally inside $\CR(S/\bZ_p)$ so that we obtain a functor $i^*_{1\CR} : (S/\bZ_p)_{\CR} \to (S/\bF_p)_{\CR}$ given by restriction. For any sheaf $\ccF \in (S/\Sigma_n)_{\CR}$ and object $(U,T,\delta) \in \CR(S/\Sigma_n)$, we write $\ccF_T$ for the sheaf on $T$ (with the usual zariski topology) given by $\ccF_T(V) = \ccF(U \times_T V, V, \delta \vert_V)$.
\cite[D\'{e}f 3.1.5]{berthelot1982theorie} associates to every finite flat commutative group scheme $G$ over $S$ a crystal of $\ccO_{S/\bZ_p}$-modules $\bD = \bD(G)$, called its Dieudonn\'{e} crystal, which comes equipped with morphisms $\Phi: \bD^{(p)} \to \bD$ and $V: \bD \to \bD^{(p)}$ such that $V \circ \Phi $ and $\Phi \circ V$ are multiplication by $p$. 

The functor is contravariant, $\Phi$ and $V$ are induced by Frobenius and Verschiebung on $G$, and is compatible with base change $S^\prime \to S$. When restricted to a perfect field $k$, the Dieudonn\'{e} functor reduces to the usual Dieudonn\'{e} theory. In fact, in the case that $S= \overline{\bF}_p$, the $\bD(H)$ we considered above corresponds to $\varprojlim \Gamma (\textrm{Spec}(W_n),\bD(G))$. If $G$ is killed by $p$, this is just $\Gamma(S,\bD(G))$.

By \cite[Prop. 4.3.1]{berthelot1982theorie}, the Dieudonn\'{e} crystal of a finite flat group scheme killed by $p$, is also killed by $p$. Furthermore, by \cite[Lemme 4.3.5]{berthelot1982theorie}, the functor $i^*_{1 \CR}$ is fully faithful when restricted to crystals killed by $p$. We may thus think of the Dieudonn\'{e} functor as a functor to $(S/\bF_p)_{\CR}$ when restricted to the category of finite flat group schemes over $S$, killed by $p$. By \cite[Thm 4.1.1]{BerthelotMessing}, the Dieudonn\'{e} functor is fully faithful when $S$ is smooth.

We recall the following construction from \cite[$\mathsection$4.3]{berthelot1982theorie}: Let $(U,T,\delta) \in \CR(S/\bF_p)$. Since we are over $\bF_p$, the divided power structure implies that $\ccI^p =0$ where $\ccI$ is the ideal defining $U$ in $T$. We thus obtain a unique morphism of rings $\Pi: \ccO_U \to \ccO_T$, functorial in $(U,T,\delta)$, such that the following diagram commutes:

\[
\begin{tikzcd}
    \ccO_T \arrow[r] \arrow[d, "\Phi_T"' ] & \ccO_U \arrow[d, "\Phi_U" ] \arrow[ld, "\Pi", dotted]\\
    \ccO_T \arrow[r]  & \ccO_U. 
\end{tikzcd}
\]
Here $\Phi_T$ and $\Phi_U$ denote the respective absolute Frobenius. Given a sheaf $\ccF$ on $S$, we define a sheaf $\Pi^*\ccF \in (S/\bF_p)_{\CR}$ by setting for all $(U,T,\delta) \in \CR(S/\bF_p)$ with $f: U \to S$, the Zariski sheaf on $T$:
\[ (\Pi^* \ccF)_T = f^* \ccF \otimes_{\ccO_U, \Pi} \ccO_T.\] 
It follows that $\Pi^* \ccF$ is a crystal, and that if $\ccF$ is locally free over $\ccO_S$, then so is $\Pi^*\ccF$ over $\ccO_{S/\bF_p}$.

By \cite[Prop 4.3.6, Prop 4.3.10]{berthelot1982theorie}, if $G$ is a group scheme killed by Verschiebung, there is an isomorphism 
\[\bD(G) \simeq \Pi^* \textrm{Lie}(G^\vee),\]
and if $G$ is killed by Frobenius 
\[\bD(G) \simeq \Pi^* \omega_G.\]

We finish by recording the two following lemmas. Although stated in a different context, their proofs follow verbatim:

\begin{lem}[Proposition 7.1.3 \cite{2020arXiv200100530D}]\label{Raynaud crystal}
Let $S$ be a smooth scheme of characteristic $p$ containing $\bF_q$. Let $G$ be a partial Raynaud $\bF_q$-scheme attached to the partial Raynaud datum $(\ccL_\theta, s_\theta,t_\theta)_{\theta \in J}$. Then $\bD(G)$ is canonically isomorphic to $\Pi^* \ccL$ with $\Phi = \Pi^* s$ and $V = \Pi^* t$, where $\ccL = \bigoplus_{\theta \in J} \ccL_\theta$, $s = \bigoplus_{\theta \in J} s_\theta$ and $t = \bigoplus_{\theta \in J} t_\theta$.
\end{lem}

\begin{cor}\label{raynaud dieudonne dimension}
Let $H$ be as above, then at any geometric point $\overline{s}$ of $S$, we have
\[ 
\dim \bD(H_{\overline{s}})_{\phi \circ \theta} = 
\begin{cases}
1 & \textrm{ if } \theta \in J,\\
0 & \textrm{ if } \theta \notin J.
\end{cases}
\]
\end{cor}

We will often need to establish the equality of images of morphisms of de Rham cohomology sheaves. In general this equality is often deduced on geometric points by using the relation between Dieudonn\'{e} modules and de Rham cohomology. To upgrade these isomorphisms at geometric points to an isomorphism of sheaves we will make use of the following crystallization lemma:

\begin{lem}{\cite[Lemma 7.2.1]{2020arXiv200100530D}}\label{crystallization}
Let $X$ be a smooth scheme over $\bF$ and $A$, $B_1$ and $B_2$ abelian schemes over $X$ with an $\ccO_E$-action. Let $\alpha_i: A \to B_i$, for $i=1,2$, be $\ccO_E$-linear isogenies such that $\ker(\alpha_i) \cap A[p^\infty] \subset A[p]$. Suppose that $\tau \in \widehat{\Theta}_{E,\fp}$ is such that 
\[\alpha^*_{1,x}\mathbb{D}(B_{1,x}[p^\infty]) = \alpha^*_{2,x}\mathbb{D}(B_{2,x}[p^\infty]) \textrm{ for all } x \in X(\bF).\]
Then there is a unique isomorphism of coherent sheaves 
\[
\ccH^1_{\textrm{dR}}(B_1/X)_\tau \simeq \ccH^1_{\textrm{dR}}(B_2/X)_\tau 
\]
which is compatible with the isomorphism $\mathbb{D}(B_{1,x}[p^\infty]) \simeq (B_{2,x}[p^\infty])$ induced by $(\alpha^*_{2,x})^{-1} \alpha_{1,x}^*$ at all $x \in X(\bF)$. 
\end{lem}

\subsubsection{Revealing subgroups of abelian varieties}
In section \ref{subgroups of kernel}, we showed that the reduced kernel of the universal isogeny was Raynaud and deduced from that the existence of certain partial Raynaud subgroups. In this section we will provide, in some sense, the inverse operation. That is, given the special fiber $S$ of some Unitary Shimura variety, and $X = \prod_{\beta \in R} \mathbb{P}_S(\ccH^1_{\tilde{\beta}})$, for some $R$, we will construct an abstract partial Raynaud group $H$ over $X$ and show that this abstract group embeds inside the universal abelian scheme $A/X$. This is the key ingredient, used in section \ref{splitting}, that allows us to define a morphism from $X$ to some auxiliary Iwahori level Unitary Shimura variety.

Let $\Sigma$ and $U^\prime \subset G^\prime_{\Sigma}(\bA_f)$ be as usual. Let $S = Y_{U^\prime}(G^\prime_\Sigma)_\bF$, $\fp$, let $X/\bF$ be smooth, locally of finite type, and let $A/X$ be the pullback of the universal abelian scheme over $S$, via some fixed morphism $f:X \to S$.

\begin{lem}\label{local check}
    Suppose that $Y/\bF_p$ is locally of finite type and smooth, and let $\ccF, \mathcal{G} \in (Y/\bF_p)_{\CR}$ be crystals and $f: \ccF \to \mathcal{G}$ a morphism. Let $P$ be a property of $\ccF$ or $f$ which is local with respect to the topology of $\CR(Y/\bF_p)$, and  preserved by pullbacks and isomorphisms. Then to verify $P$, it suffices to verify it over $Y$.
\end{lem}

\begin{proof}
    To verify $P$, it suffices to verify it for $\ccF_T$ (resp: $f_T: \ccF_T \to \mathcal{G}_T$) for all $(U,T,\delta) \in \CR(Y/\Sigma_1)$. Since the topology is the fppf topology, it suffices to check it for affine opens $V$ of $T$. Now, $\ccF_V$ is the sheaf associated to $(U \times_T V, V, \delta \vert_V) \in \CR(Y/\bF_p)$, and by definition of the crystalline site, and the smoothness of $Y/\bF_p$, we have a  diagram 
    \begin{center}
    \begin{tikzcd}
        U \times_T V \arrow[r,"u"] \arrow[d] & S \arrow[d] \\
        V \arrow[r]  \arrow[ru,dotted,"v"]& \bF_p.
    \end{tikzcd}
    \end{center}
    In particular, we obtain a morphism $(u,v): (U \times_T V, V, \delta \vert_V) \to (Y,Y,0)$. However, since $\ccF$ and $\mathcal{G}$ are crystals, we obtain a commutative diagram
    
    \begin{center}
    \begin{tikzcd}
                v^{*} \ccF_Y \arrow[r, "{\rho_{(u,v)}}", "{\sim}"'] \arrow[d, "{v^*f_S}"'] & \ccF_{V} \arrow[d, "{f_V}"] \\
        v^* \mathcal{G}_Y \arrow[r,"{ \rho_{(u,v)}}"',"{\sim}"]  & \mathcal{G}_V,
    \end{tikzcd}
    \end{center}
    where the horizontal arrows are the canonical comparison maps that are isomorphisms since $\ccF$ and $\mathcal{G}$ are crystals. The result follows.
\end{proof}
Consider the Dieudonn\'{e} crystal $\bD = \bD(A[p])$. By additivity of the Dieudonn\'{e} functor, the left action of $\ccO_D$ on $A$ yields a right action of $\ccO_D$ on $\bD$. Since we are working over $\bF_p$, we have the usual decomposition 
\[\ccO_{X/\bF_p} \otimes_\bZ \ccO_E \simeq \bigoplus_{\tau \in \widehat{\Theta}_E} \ccO_{X/\bF_p}[u]/u^{e_\fp},\]
which yields the decomposition 
\[\mathbb{D} = \bigoplus_{\tau \in \widehat{\Theta}_E} \mathbb{D}_\tau.\]
By \cite[3.3.7.3]{berthelot1982theorie}, there is a canonical isomorphism $\bD_X \simeq \ccH^1_{\dR}(A/X)$, compatibly with the action of $\ccO_D$. We deduce then by lemma \ref{local check} that $\bD$ is locally free of rank two over $\ccO_{X/\bF_p} \otimes \ccO_D$.  

Furthermore, by \cite[4.3.7.1]{berthelot1982theorie}, the above isomorphism induces a canonical isomorphism $\bD^{(p)} \simeq \Pi^*(\bD_X) \simeq \Pi^*\ccH^1_{\dR}(A/X)$. \cite[4.3.10.2]{berthelot1982theorie} then shows that the pullback by $\Pi$ of the inclusion $\omega \hookrightarrow \ccH^1_{\dR}(A/X)$ induces a sequence:
\begin{center}
\begin{tikzcd}
    \Pi^* \omega \arrow[r] \arrow[rr, bend right, "0"']&  \bD^{(p)} \arrow[r,"\Phi"] & \mathbb{D} \arrow[r]  \arrow[rr, bend right, "V"'] &  \Pi^* \omega \arrow[r] & \bD^{(p)}, 
\end{tikzcd}
\end{center}
where $\Pi^* \omega \simeq \mathbb{D}(\ker F)$, whose base change to $X$ and restriction to $\tau$-components is given by 
\[
 \omega^{(p)}_{\phi^{-1}(\tau)} \to  \ccH^1_{\dR}(A/X)^{(p)}_{\phi^{-1}(\tau)} \xrightarrow{F} \ccH^1_{\dR}(A/X)_\tau \xrightarrow{V} \omega^{(p)}_{\phi^{-1}(\tau)} \to \ccH^1_{\dR}(A/X)^{(p)}_{\phi^{-1}(\tau)}.
\]

After fixing a suitable $\ccO_X \otimes_\bZ \ccO_D$ equivariant isomorphism $\omega \simeq \omega^0 \oplus \omega^0$, we can thus make sense of the crystals 
\[\Pi^*\omega_{\phi^{-1}(\tau)}(i) = \Pi^*(\omega^0_{\phi^{-1}(\tau)}(i) \oplus \omega^0_{\phi^{-1}(\tau)}(i)) \subset \bD^{(p)}_{\phi^{-1}(\tau)},\]
which are locally free over $\ccO_{X/\bF_p}$ of rank, $ 2 \cdot \textrm{rk}_{\ccO_X} \omega^0_{\phi^{-1}(\tau)}(i) $. It follows that we can we define morphisms 
\[V_{\es,\tau^1}: [\varpi_\fp]^{e_\fp-1}\mathbb{D}_\tau \to \Pi^* ([\varpi_\fp]^{-1} \omega_{\phi^{-1}(\tau)}(e_\fp-1) / \omega_{\phi^{-1}(\tau)}(e_\fp-1)) \]
which restrict over $X$ to the (direct sum of the) genuine $V_{\es,\tau^1}$: If $s_{\phi^{-1}(\tau)^{e_\fp}} \neq 0$, then $V_{\es, \tau^1}$ is just the usual (very abusively written) $V_{\es,\tau^1}([\varpi]^{e_\fp-1}s) = V(s)$. If $s_{\phi^{-1}(\tau)^{e_\fp}} = 0$ however, then we see that, at the level of crystals, $\ker \Phi_\tau = \Pi^* \omega \subset [\varpi_\fp]\bD^{(p)}_{\phi^{-1}(\tau)} $ since $\omega_\tau \subset [\varpi_\fp]\ccH^1_{\dR}(A/S)_\tau$, and since  both $\bD^{(p)}_{\phi^{-1}(\tau)}$ and $\bD_\tau$ are locally free of rank $4$ over $\ccO_{X/\bF_p}[u]/u^{e_\fp}$, then the image of $\Phi$ must actually contain $[\varpi_\fp]^{e_\fp-1} \bD_\tau$. We can then define $V_{\es,\tau^1}$ by $\Phi^{-1}$. By compatibility over $X$, we deduce by lemma \ref{local check}, that $V_{\es,\tau^1}$ is an isomorphism if $s_{\phi^{-1}(\tau)^{e_\fp}} \neq 1$.

We can now state and prove the main result of this section, we only state a very precise version, tailored to our future use. Consider the following setup: Let $\fp$ be a prime above $p$, $J \subset \widehat{\Theta}_{F,\fp}$, $(\ccL_{\tilde{\theta}},s_{\tilde{\theta}},t_{\tilde{\theta}})_{\theta \in J}$ a partial Raynaud datum for $\ccO_E/\fq$. Let $H^\prime$ be the associated group scheme and let $H = H \otimes_{\ccO_E} \ccO_E^2$ with its natural $\ccO_D$-action. Suppose that $\ccL_{\tilde{\theta}} = \omega^0_{\tilde{\theta}^{e_\fp}}$ if $s_{\tilde{\theta}^{e_\fp}}=1$, and we are given a collection of surjections of the form
\begin{equation}\label{very precise surjection}
\ccH^1_{\dR}(A/X)^0_{\phi \circ \tilde{\theta}} \xrightarrow{[\varpi_\fp]^{e_\fp-1}} \ccH^1_{(\phi \circ \tilde{\theta})^1} \xrightarrow[\sim]{V_{\es,(\phi \circ \tilde{\theta})^1}} \ccH^{1 \, (p)}_{\tilde{\theta}^{e_\fp}} \twoheadrightarrow \ccL^{(p)}_{\tilde{\theta}}
\end{equation}
otherwise. Write $\ccL = \bigoplus \ccL_{\tilde{\theta}}$, $s = \bigoplus s_{\tilde{\theta}}$, and $t = \bigoplus t_{\tilde{\theta}}$. Suppose finally that the above surjections are such that the diagrams
\begin{equation}\label{subgroup sesame diagram}
\begin{tikzcd}
    \ccH^1_{\dR}(A/X)^0 \arrow[r] \arrow[d,"V"'] & \Phi^* \ccL \arrow[d,"{\Phi^*(t)}"]      & & 
    \ccH^1_{\dR} (A^{(p)}/X)^0 \arrow[r] \arrow[d,"F"'] & \Phi^* \ccL^p \arrow[d,"{\Phi^*(s)}"]\\
    \ccH^1_{\dR}(A^{(p)}/X)^0 \arrow[r] & \Phi^* \ccL^p & &
    \ccH^1_{\dR}(A/X)^0 \arrow[r] & \Phi^* \ccL,
    \end{tikzcd}
\end{equation}
where $\ccH^1_{\dR}(A/X)^0_{\phi \circ \tilde{\theta}} \to \ccL^p_{\tilde{\theta}} = (\omega^0_{\tilde{\theta}^{e_\fp}})^p$ if $s_{\tilde{\theta}^{e_\fp}}=1$ is just Verschiebung, commutes.

\begin{lem}\label{subgroup sesame}
    In the above setup, there is a natural $\ccO_D$-linear embedding $H \hookrightarrow A[\fq]$. Furthermore, the same result is true when replacing $\tilde{\theta}$ for $\tilde{\theta}^c$ and $\fq $ for $\fq^c$.
\end{lem}
\begin{proof}
By lemma \ref{Raynaud crystal}, the Dieudonn\'{e} crystal of $H$ is given $(\Pi^*\ccL)^{\oplus 2}$ with $\Phi$ and $V$ given by $(\Pi^* s)^{\oplus 2}$ and $(\Pi^* s)^{\oplus 2}$. By the discussion above, we can thus define a morphism of crystals $\bD \to (\Pi^*\ccL)^{\oplus 2}$ whose base change to $X$ is given on components by (\ref{very precise surjection}), or just Verschiebung. It follows by lemma \ref{local check} that this is a surjection. Furthermore, it also follows by lemma \ref{local check} and the commutativity of the diagrams \ref{subgroup sesame diagram}, that this surjection is compatible with $\Phi$ and $V$. We therefore have a surjection of Dieudonn\'{e} crystals which yields a closed immersion $H \hookrightarrow A[p]$ by the full faithfulness of the Dieudonn\'{e} functor since $X$ is smooth. It is straightforward to see that $H$ lands in $A[\fq]$ and that this morphism is $\ccO_D$-equivariant. The conjugate case is entirely analogous. 
\end{proof}
\subsection{The Essential Frobenius isogeny at \texorpdfstring{$\fp$}{fp} and the relation between the Iwahori stratification and the Goren-Oort stratification}
Recall that in the case of Hilbert modular varieties for $p$ unramified in $F$, we have a natural section of the forgetful morphism $\pi_1 : \overline{Y}_0(p) \to \overline{Y}$, that sends $A$ to $(A,H)$ where $H$ is given by the kernel of Frobenius. In particular, this restricts to an isomorphism $\overline{Y}_0(p)_{\Theta_F, \emptyset} \to \overline{Y}$ which also induces isomorphisms $\overline{Y}_0(p)_{\Theta_F,T} \xrightarrow{\sim} \overline{Y}_T$.

In this section, we generalize this result to our setting by constructing over $\overline{Y}$ a canonical isogeny which we dub Essential Frobenius at $\fp$. We call it so because in some sense (made more precise in proposition \ref{essential name justification}) it induces  the Essential Frobenius morphisms defined in section \ref{Fesves}. We note that this construction, in the case of Hilbert modular varieties, was given in \cite[$\mathsection$ 6.1]{Diamond_2023}

\subsubsection{The Essential Frobenius isogeny at \texorpdfstring{$\fp$}{fp}}\label{partial frobenius}

Let $\Sigma$, $U^\prime \subset G^\prime_\Sigma(\bA_f)$ be as usual and write $S^\prime = \widetilde{Y}_{U^\prime}(G^\prime_\Sigma)_\bF$. In this section we define, for each prime $\fp$, endomorphisms $\Phi_{\fp}: S^\prime \to S^\prime$ which we will call the Essential Frobenius morphism at $\fp$. Of course, these morphisms depend on the underlying variety $S^\prime$. We stress that despite the name, $\Phi_\fp$ does not factor the genuine Frobenius isogeny for a general $\Sigma$.
Let $A/S$ be the universal abelian scheme. We fix a prime $\fp \vert p$ of $F$. Our strategy is to use lemma \ref{subgroup sesame}. Since this can be seen as a special instance of the splitting construction from section \ref{splitting}, we will simply state the relevant results and defer their proofs to section \ref{constructing A2} and \ref{filtering A2}. In particular, this corresponds to the case $I = \Theta_{F,\fp} \setminus \Sigma_{\infty,\fp}$ and $J = \emptyset$. In particular $T = R = \emptyset$.

We construct a full Raynaud datum as follows: Let $\theta \in \widehat{\Theta}_{F,\fp}$ and write $\tau = \tilde{\theta}$. If $\theta^{e_\fp} \notin \Sigma_\infty$, set $\mathcal{L}_\theta = \omega^0_{\tau^{e_\fp}}$. Otherwise, if $\theta^{e_\fp} \in \Sigma_\infty$ let $\beta = (\phi^\prime)^{-1}(\theta^{e_\fp})$, and set $\mathcal{L}_\theta = \ccH^1_{\tau^{e_\fp}}/ F_{\textrm{es},\tilde{\beta}}^{\tau^{e_\fp}}((\omega^0 _{\tilde{\beta}})^{p^{(n)}})$ for the appropriate $n$, where we recall that  $\phi^\prime$ is the cycle structure with respect to $\Theta_F \setminus \Sigma_\infty$ and $F_{\textrm{es},\tilde{\beta}}^{\tau^{e_\fp}} : \ccH^{1 \, (p^n)}_{\tilde{\beta}} \to \ccH^1_{\tau^{e_\fp}}$, defined in section \ref{Fesves}, is an isomorphism. These line bundles come with a natural action of $\ccO_E$ which factors through $\ccO_E/\fq$.
If $\theta^{e_\fp} \notin \Sigma_\infty$, we have the surjection
\[\ccH^1_{\textrm{dR}}(A/S)^0_{\phi \circ \tau} \xrightarrow{[\varpi_\fp]^{e_\fp-1}} \ccH^1_{(\phi \circ \tau)^1} \xrightarrow{V_{\textrm{es},(\phi \circ \tau)^1}} (\omega^0_{\tau^{e_\fp}})^{(p)}. \]
If $\theta^{e_\fp} \in \Sigma_\infty$, we have the surjection
\[\ccH^1_{\textrm{dR}}(A/S)^0_{\phi \circ \tau} \xrightarrow{[\varpi_\fp]^{e_\fp-1}} \ccH^1_{(\phi \circ \tau)^1} \xrightarrow[\sim]{V_{\textrm{es},(\phi \circ \tau)^1}} \ccH^{1 \, (p)}_{\tau^{e_\fp}} \twoheadrightarrow  (\ccH^1_{\tau^{e_\fp}}/ F_{\textrm{es},\tilde{\beta}}^{\tau^{e_\fp}}((\omega^0 _{\tilde{\beta}})^{p^{(n)}}))^{(p)}. \]

\noindent Setting $\ccL= \bigoplus_{\theta \in \widehat{\Theta}_{F,\fp}} \ccL_\theta$, we obtain a surjection of $\ccO_S \otimes \ccO_E$-modules \[\ccH^1_{\textrm{dR}}(A/S)^0 \twoheadrightarrow \Phi^*(\ccL),\]
and we define $s_\theta: \ccL_\theta^{p} \to \ccL_{\phi \circ \theta}$, $s = \bigoplus s_\theta$ and $t_\theta: \ccL_{\phi \circ \theta} \to \ccL^p_{\theta}$, $t = \bigoplus t_\theta$ to be the morphisms  such that the diagrams 

\begin{center}
\begin{tikzcd}
\ccH^1_{\textrm{dR}}(A/S)^0  \arrow[r] & \Phi^*(\ccL ) & \ccH^1_{\textrm{dR}}(A/S)^{0 \, (p)}  \arrow[r] & \Phi^*(\ccL)^{(p)} \\

\ccH^1_{\textrm{dR}}(A/S)^{0 \, (p)} \arrow[r] \arrow[u,"F"] & \Phi^*(\ccL ) \arrow[u,"\Phi^*(s  )"'] & \ccH^1_{\textrm{dR}}(A/S)^0 \arrow[r] \arrow[u,"V"] & \Phi^*(\ccL ) \arrow[u,"\Phi^*(t )"']
\end{tikzcd}
\end{center}
commute. Again, we defer the proof that $s$ and $t$ are well defined to section \ref{constructing A2}. We are now in the setup of lemma \ref{subgroup sesame}. Let $H^{\prime \prime}$ be the group scheme associated to the Raynaud datum $(\ccL_\theta,s_\theta,t_\theta)$ and let $H^{\prime} = H^{\prime \prime} \otimes \ccO_E^2$. By lemma \ref{subgroup sesame}, we have an $\ccO_D$-linear embedding $H^\prime \hookrightarrow A[\fq]$. We now set
\[H = H^\prime \oplus \lambda^{-1}((A[\fq]/H^\prime)^\vee) \subset A[\fq] \oplus A[\fq^c] = A[\fp].\] 

\noindent $H$ is an $\ccO_D$-stable subgroup of order $p^{4f_\fp}$ which is totally isotropic with respect to the Weil pairing induced by $\lambda$. We, suggestively, write $F_{\es,\fp}: A \to A^\prime = A / H$. Then $A^\prime$ inherits a natural action of $\ccO_D$ and we define a prime to $p$ quasi-polarization $\lambda^\prime$, for each connected component and $n$ such that $n\lambda$ is a genuine prime to $p$ polarization, via the diagram:
\[
\begin{tikzcd}
    0 \arrow[r] & H \arrow[r] \arrow[d,"n\lambda"', "\sim" {anchor = north, rotate = 90}] & A \arrow[r] \arrow[d,"n\lambda \otimes \varpi_\fp"] & A^\prime \arrow[d,"\lambda^\prime"] \arrow[r] & 0 \\
    0 \arrow[r] & (A[\fp]/H)^\vee \arrow[r] & A^\vee \otimes \fp \arrow[r] & (A^\prime)^\vee \arrow[r] & 0.
\end{tikzcd}
\]
We also define a level structure for $A^\prime$ by setting for each connected component $S_i$ of $S$, $\eta^\prime_i = F_{\es, \fp} \circ \eta_i$ and $\epsilon^\prime_i = \varpi_\fp \epsilon_i$.

Equation (\ref{fixed a2 de rham relations}) will show that at every geometric point $\overline{s}$, the Dieudonn\'{e} module $\Delta^\prime = \mathbb{D}(A^\prime_{\overline{s}}[p^\infty])$ inside $\Delta = \mathbb{D}(A_{\overline{s}}[p^\infty])$ is given on reduced components by 
\[F^*_{\es,\fp}\Delta^{\prime \, 0}_{\tilde{\theta}} = [\varpi_{\fp}]^{1-e_\fp} F_{\es,\tilde{\beta}}^{\tau^1}(\ccH^1_{\tilde{\beta}}),\]
where by abuse of notation, this denotes the lift of the same subspace inside $\Delta^0_{\tilde{\theta}}/p\Delta^0_{\tilde{\theta}} = H^1_{\dR}(A_{\overline{s}})^0_{\tilde{\theta}}$. From this, we deduce that $\dim \omega^0_{A^\prime,\tilde{\theta}} = \dim \omega^0_{A,\tilde{\theta}}$ for all $\theta \in \widehat{\Theta}_{F,\fp}.$ For $\theta \notin \widehat{\Theta}_{F,\fp}$, $F_{\es,\fp}: \omega^0_{A^\prime,\tilde{\theta}} \to \dim \omega^0_{A,\tilde{\theta}} $ is an isomorphism. We thus define the filtration on $\omega^0_{A^\prime,\tilde{\theta}}$ by pulling back the one on $\omega^0_{A,\tilde{\theta}}$. If $\theta \in \widehat{\Theta}_{F,\fp}$ however, we will see in (\ref{filtering a2 definition}) that we can define a filtration on $\omega^0_{A^\prime,\tilde{\theta}}$ of type $\widetilde{\Sigma}_\infty$, uniquely determined by 
\[
\omega^0_{A^\prime,\tilde{\theta}}(i) = (F^*_{\es,\fp})^{-1}(\omega^0_{A,\tilde{\theta}}(i-1))
\]
for all $i$ such that $s_{\tilde{\theta}^i}=1$.

We now have a tuple $\underline{A^{(\fp)}} = \underline{A^\prime}$ over $S$, which defines a morphism 
\[\Phi_{\fp}: S \to S\]
which we call the Essential Frobenius morphism at $\fp$. 
Consider two distinct primes $\fp_1,\fp_2 \vert p$ of $F$. It is straightforward to check, for example by looking at geometric points, that the isogenies 
\[A \xrightarrow{F_{\es,\fp_1}} A^{(\fp_1)} \xrightarrow{\Phi_{\fp_1}^*(F_{\es,\fp_2})} (A^{(\fp_1)})^{(\fp_2)} \, \textrm{ and } \, A \xrightarrow{F_{\es,\fp_2}} A^{(\fp_2)} \xrightarrow{\Phi_{\fp_2}^*(F_{\es,\fp_1})} (A^{(\fp_2)})^{(\fp_1)}\]
are the same. That is, that there is a canonical isomorphism $(A^{(\fp_1)})^{(\fp_2)} \simeq (A^{(\fp_2)})^{(\fp_1)}$ such that $\Phi_{\fp_1}^*(F_{\es,\fp_2}) = \Phi_{\fp_2}^*(F_{\es,\fp_1})$. For $\Fp = \fp_1 \cdots \fp_n$ a product of distinct primes over $p$, we thus write $F_{\es,\Fp}: A \to A^{(\Fp)}$ for the composite, in any order of the $F_{\es,\fp_i}$, which we call the essential Frobenius isogeny at $\Fp$.

We similarly define $V_{\es,\Fp}: A^{(\Fp)} \to A \otimes \Fp^{-1}$ to be the unique isogeny such that $A \xrightarrow{F_{\es,\Fp}} A^{(\Fp)} \xrightarrow{V_{\es,\Fp}} A \otimes \Fp^{-1}$ is the canonical isogeny $\iota_\Fp$ induced by the inclusion $\ccO_E \hookrightarrow \Fp^{-1}$. We call $V_{\es,\Fp}$, the Essential Verschiebung at $\Fp$.

\noindent The following proposition justifies the names Essential Frobenius and Essential Verschiebung at $\fp$:

\begin{prop}\label{essential name justification}
There is a necessarily unique system of isomorphisms \[\alpha_\beta : \ccH^{1 \, (p^\delta)}_{A,\phi^{-1}(\beta)} \xrightarrow{\sim} \ccH^1_{A^{(\fp)},\beta},\]
where $\delta =1$ if $\beta = \tau^i$ with $i=1$ and 0 otherwise, and $\beta \in \Theta_{E,\fp}$, such that: If $s_{\phi^{-1}(\beta)} \neq 1$, then $\alpha_{\phi(\beta)} = F_{\es,\beta} \circ \alpha_\beta$. Otherwise, if $s_{\phi^{-1}(\beta)}=1$ then the following diagram
\[
\begin{tikzcd}
&  \ccH^1_{A^{(\fp)},\beta} \arrow[rd, "F_{\es,\fp}^*"]& \\
\ccH^{1 \, (p^\delta)}_{A,\phi^{-1}(\beta)} \arrow[ru,"\alpha_\beta"] \arrow[rr,"F_{\es, \beta}"]& & \ccH^1_{A,\beta}
\end{tikzcd}
\]
commutes. In particular, after fixing the isomorphism $\ccH^1_{A \otimes \fp^{-1},\beta} \simeq \ccH^1_{A ,\beta} \otimes \fp \simeq \ccH^1_{A ,\beta}$ afforded by $\varpi_\fp$, if $s_{\phi^{-1}(\beta)} =1$, then the following diagram commutes:
\[
\begin{tikzcd}
&  \ccH^1_{A^{(\fp)},\beta} \arrow[rd,"{\alpha^{-1}_{\beta}}"]& \\
 \ccH^1_{A,\beta}  \arrow[rr,"V_{\es, \beta}"'] \arrow[ru, "{V_{\es,\fp}^*}"]& & \ccH^{1 \, (p^\delta)}_{A,\phi^{-1}(\beta)}. 
\end{tikzcd}
\]
\end{prop}

\begin{proof}
It is clear that we only need to construct $\alpha_\beta$ for $\beta \in \Theta_E$ such that $s_{\phi^{-1}(\beta)}=1$. Write $\beta = \tau^i$; if $i > 1$, then $\omega^0_{A^{(\fp)}}(i) = (F^*_\fp)^{-1} \omega^0_{A,\tau}(i-1)$ by construction. We thus take $\alpha_\beta: \ccH^1_{A,\tau^{i-1}} \to \ccH^1_{A^{(\fp)},\tau^{i}}$ to be $(F_{\es,\fp}^*)^{-1}$.

If $i = 1$, consider $q: A \to (A / \ker F \cap H) =:B$ where $H$ is the kernel of $F_{\es,\fp}$. $q$ factors both $F_{\es,\fp}$ and the genuine Frobenius. 

Let $\overline{s}$ be any geometric point of $S$. Since $(\phi^{-1} \circ \theta)^{e_\fp} \notin \Sigma_\infty$ by assumption, we know by above that 
\[F_{\es,\fp}^* \bD(A_{\overline{s}}^{(\fp)}[p^\infty])_\tau = p^{-1}\Phi (W_{(\phi^{-1} \circ \tau)}(e_\fp-1)),\]
where we write $W_{(\phi^{-1} \circ \tau)}(e_\fp-1)$ for the lift of  $\omega_{(\phi^{-1} \circ \tau)}(e_\fp-1) \subset \bD(A_{\overline{s}}[p^\infty])_{\phi^{-1} \circ \tau} / p$.
The Dieudonn\'{e} module $\bD(H \cap \ker F)$ is given by the pushout $\mathbb{D}(H) \oplus_{\bD(A_{\overline{s}}[p^\infty])} \bD(\ker F)$. Furthermore, we may take this pushout in the category of Dieudonn\'{e} modules with an action of $\ccO_D$ so that we may compute it componentwise. In particular, since the kernel of $\bD(A_{\overline{s}}[p^\infty])_\tau \to \mathbb{D}(H)_\tau $ contains $\Phi\bD(A^{(p)}_{\overline{s}}[p^\infty])_{\tau}$ by above, we then conclude that $\bD(H \cap \ker F)_\tau = \mathbb{D}(H)_\tau$. Therefore, $\mathbb{D}(A_{\overline{s}}^{(\fp)}[p^\infty])_{\phi \circ \tau} \to \mathbb{D}(B_{\overline{s}}[p^\infty])_{\phi \circ \tau}$ is an isomorphism and that $\ker \bD(A^{(p)}[p])_\tau \to \bD(B[p])_\tau$ is $\omega_{(\phi^{-1} \circ \tau)}(e_\fp-1)$. Since $S$ is smooth we thus obtain the same results at the level of de Rham cohomology over all of $S$. We thus have the following commutative diagram:

\begin{center}
\begin{tikzcd}
    \ccH^1_{\dR}(A^{(\fp)}/S)_{\tau} \arrow[r,"\sim"] \arrow[rd, "{F^*_{\es,\fp}}"'] & \ccH^1_{\dR}(B/S)_{ \tau} \arrow[d] & \ccH^1_{\dR}(A/S)^{(p)}_{\phi^{-1} \circ\tau} \arrow[l] \arrow[ld, "F"] \\
    & \ccH^1_{\dR}(A/S)_{\tau}. &
\end{tikzcd}
\end{center}
We thus have a morphism $\alpha_\beta:\ccH^1_{\dR}(A/S)^{(p)}_{\tau} \to \ccH^1_{\dR}(A^{(\fp)}/S)_{\phi \circ \tau} $ whose kernel is $\omega^0_{(\phi^{-1} \circ \tau)}(e_\fp-1))$. It follows that $\alpha_\beta$ yields the desired isomorphism $\ccH^{1 \, (p)}_{A,(\phi^{-1} \circ \tau)^{e_\fp}} \to \ccH^1_{A^{(\Fp)},\tau^1}$.

The compatibility with essential Verschiebung at $\fp$ follow from above and the relation (after fixing the suitable isomorphism) $V_{\es,\fp} F_{\es,\fp} =[\varpi_\fp]$, $F_{\es,\fp} V_{\es,\fp} =[\varpi_\fp]$.
\end{proof}

\begin{cor}
The Goren-Oort stratum $S_T$ can be described as the closed subscheme given by the vanishing 
\[V( \{ V^*_{\fP,\tilde{\beta}} \, \vert  \, \beta \in T \}).\] 
\end{cor}

\subsubsection{The Relation to the Goren-Oort stratification}\label{Iwahori goren relation section}
In this section we make explicit the link between the Goren-Oort stratification on $Y^\prime$ and the stratification upstairs, on $Y^\prime_0(\Fp)$, where $\Fp$ is a product of distinct primes above $p$, as defined in the previous section.

Keep the notation from section \ref{partial frobenius} so $S = \widetilde{Y}^\prime_{U^\prime}(G^\prime_\Sigma)_\bF$. Consider the tuple $(\underline{A},\underline{A^{(\Fp)}},F_{\es,\Fp})$ over $S$ constructed in the previous section, and let 
\[\widetilde{\psi}_\Fp: S \to \widetilde{Y}^\prime_0(\Fp) = \widetilde{Y}^\prime_{U^\prime_0(\Fp)}(G^\prime_\Sigma)_\bF\]
be the induced morphism. It follows from the construction of $F_{\es,\Fp}$ that this morphism factors through $S_{\emptyset} = S_{\Theta_F \setminus \Sigma_\infty,\emptyset}$ and that it is a section of the projection $\widetilde{\pi}_1: \widetilde{Y}_0(\Fp) \to S$, $(\underline{A_1},\underline{A_2},f) \mapsto \underline{A_1}$. Furthermore, it follows from construction that $\widetilde{\psi}_\Fp$ is equivariant with respect to the action of $\ccO_{F,(p),+}^\times$. Namely, the construction of the kernel $A \to A^{(\Fp)}$ depends only on the $\ccO_{F,(p),+}^\times$-orbit of the polarization $\lambda$ of $A$. Therefore, $\widetilde{\psi}_\Fp$ descends to a morphism $\psi_\Fp: \overline{Y}^\prime \to \overline{Y}^\prime_0(\Fp)$.

\begin{prop}\label{Iwahori goren relation}
Let $\pi_1: \overline{Y}^\prime_0(\Fp) \to \overline{Y}^\prime$ be the map induced by the projection $(\underline{A_1},\underline{A_2},f) \mapsto \underline{A_1}$. Then for any $I,J$ such that $I \cup J = \Theta_{F,\Fp} \setminus \Sigma_{\infty,\Fp}$, \[ \pi_1\left(\overline{Y}^\prime_0(\Fp)_{\phi^\prime(I),J}\right) \subset \overline{Y}^\prime_{\phi^\prime(I) \cap J}.\] In particular for every $J \subset \Theta_{F,\Fp} \setminus \Sigma_{\infty,\Fp}$, $\pi_1: \overline{Y}^\prime_0(\Fp)_{(\Theta_{F,\Fp} \setminus \Sigma_{\infty,\Fp}),J} \xrightarrow{\sim} \overline{Y}^\prime_J$ is an isomorphism with inverse $\psi_\Fp$.
\end{prop}

\begin{proof}
It suffices to prove the analogous results upstairs. The first part is standard: Consider the diagram 
\begin{center}
\begin{tikzcd}
{\ccH^1_{1,\beta}\otimes \fp} \arrow[d, "h_\beta"] \arrow[r, "g"] & {\ccH^1_{2,\beta}} \arrow[d, "h_\beta"] \arrow[r, "f"] & {\ccH^1_{1,\beta}} \arrow[d, "h_\beta"] \\
\left(\ccH^1_{1,{\phi'}^{-1}(\beta)}\right)^{(p^{n_\beta})} \otimes \fp \arrow[r, "g"]        & 
\left(\ccH^1_{2,{\phi'}^{-1}(\beta)}\right)^{(p^{n_\beta})} \arrow[r, "f"]        & 
\left(\ccH^1_{1,{\phi'}^{-1}(\beta)}\right)^{(p^{n_\beta})}       
\end{tikzcd}
\end{center}

which is commutative by functoriality of the Hasse invariants. We immediately see that \[ \ker h_\beta = \textrm{Im}(f) = \ker (g \otimes \fp^{-1}) = \omega_{1,\beta}. \] 

To prove the second result, it suffices by above to show that $\widetilde{\pi}_1: \widetilde{Y}^\prime_0(\Fp)_{(\Theta_{F,\Fp} \setminus \Sigma_{\infty,\Fp}),\emptyset} \to S$ has inverse $ \widetilde{\psi}_\Fp$. Note that $\widetilde{\pi}_1 \circ \widetilde{\psi}_\Fp$ is the identity; it thus suffices to show that $\widetilde{\pi}_1$ is injective on geometric points.\smallskip

Consider a geometric point $\overline{s}$ of $\widetilde{Y}^\prime_0(\Fp)_{(\Theta_{F,\Fp} \setminus \Sigma_{\infty,\Fp}),\emptyset}$ corresponding to a tuple $(\underline{A_1}, \underline{A_2},f)$ and let $\Delta_i = \mathbb{D}(A_i[p^\infty])$ be the Dieudonn\'{e} modules of $A_i$. We determine the image of $f^*\Delta_2 \subset \Delta_1$ component by component: Since $\ker f \subset A_1[\Fp]$, it follows that for any $\tau \notin \widehat{\Theta}_{E,\Fp}$ we have $f^*\Delta_{2,\tau} = \Delta_{1,\tau}$. If $\tau \in \widehat{\Theta}_{E,\Fp}$, then the compatibility of $f$ with the polarizations shows that $f^*\Delta_{2,\tau^c}$ is uniquely determined by  $f^*\Delta_{2,\tau}$. More precisely, for any lattice $\Lambda_\tau \subset \Delta_{1,\tau} \otimes \bQ$, write $\Lambda_{\tau}^c \subset \Delta_{\tau^c}$ for its dual lattice with respect to the pairing induced by $\lambda_1$. Then 
\[f^*\Delta_{2,\tau^c} = [\varpi_\fp]\left( f^*\Delta_{2,\tau} \right)^c.\]

In particular, to determine $f^*\Delta_2$, it suffices to determine the images $f^*\Delta_{2,\tilde{\theta}}$ for $\theta \in \widehat{\Theta}_{F,\fp}$, and to do that it suffices to determine the images $f^* \Delta^0_{2,\tau} \subset \Delta_{1,\tau}^0$. Let $\tau = \tilde{\theta}$ be such that there exists $i \geq 1$ such that $s_{\tau^i}=1$ and let $i \leq e_\fp$ be the maximal such $i$. Since $(\underline{A_1}, \underline{A_2},f)$ lies in $S_{ \emptyset}$, then $f^*\omega^0_{2,\tau}(i) = \omega^0_{1,\tau}(i-1)$. By definition of $i$, we obtain 
\[f^*\ccH^1_{2,(\phi \circ \tau)^1} = F_{\textrm{es},\tau^i}^{(\phi \circ \tau)^1}(\ccH^1_{1,\tau^i}).\]

\noindent Furthermore, this yields $f^*H^1_{\textrm{dR}}(A_2)^0_{\phi \circ \tau} = [\varpi_\fp]^{1-e_\fp}F_{\textrm{es},\tau^i}^{(\phi \circ \tau)^1}(\ccH^1_{1,\tau^i})$ and 
\[f^*\Delta^0_{2,\phi \circ \tau} = [\varpi_\fp]^{1-e_\fp}F_{\textrm{es},\tau^i}^{(\phi \circ \tau)^1}(\ccH^1_{1,\tau^i}),\]
via the  isomorphisms $\Delta_i / p \Delta_i \xrightarrow{\sim} H^1_{\textrm{dR}}(A_i)$, and the right hand side, is by abuse of notation, the lift of the corresponding space in $H^1_{\textrm{dR}}(A_1)^0_{\phi \circ \tau}$ to $\Delta_{1,\phi \circ \tau}$.

Suppose now that $\theta^i \in \Sigma_\infty$ for all $1 \leq i \leq e_\fp$ and write $\beta = (\phi^\prime)^{-1}((\phi \circ \theta)^1)$. Then 
\[f^* \ccH^1_{2,(\phi \circ \tau)^1} = f^*(F_{\textrm{es},\phi(\tilde{\beta})}^{(\phi \circ \tau)^1}(\ccH^1_{2,\phi(\tilde{\beta})})) = F_{\textrm{es},\phi(\tilde{\beta})}^{(\phi \circ \tau)^1}(f^*(\ccH^1_{2,\phi(\tilde{\beta})})) = F_{\textrm{es},\tilde{\beta}}^{(\phi \circ \tau)^1}(\ccH^1_{2,\tilde{\beta}}).\]
We conclude as before that 
\[f^*\Delta^0_{2,\phi \circ \tau} = [\varpi_\fp]^{1-e_\fp}F_{\textrm{es},\tilde{\beta}}^{(\phi \circ \tau)^1}(\ccH^1_{1,\tilde{\beta}}).\]
Hence $f^*\Delta_2$ is uniquely determined by $\Delta_1$ and the filtration on $\omega^0_2$ is also uniquely determined by the filtration on $\omega^0_1$, for any $s_{\tau^i}=1$:
\[
\omega^0_{2,\tau}(i) = (f^*)^{-1}\omega^0_{1,\tau}(i-1).
\]
In conclusion, after unravelling the definitions of the extra structures, $\underline{A_2}$ is uniquely determined by $\underline{A_1}$ and so $\widetilde{\pi}_1$ is injective on geometric points. This concludes the proof.

\end{proof}

\begin{cor}\label{iwahori strata non empty}
The strata $\overline{Y}^\prime_0(\Fp)_{\phi^\prime(I),J}$ are non empty.
\end{cor}
\begin{proof}
The stratum $\overline{Y}^\prime_0(\Fp)_{(\Theta_{F,\Fp} \setminus \Sigma_{\infty,\Fp}),(\Theta_{F,\Fp} \setminus \Sigma_{\infty,\Fp})}$ lies in every stratum $\overline{Y}^\prime_0(\Fp)_{\phi^\prime(I),J}$ and The Goren-Oort stratum $\overline{Y}_{\Theta_{F,\Fp} \setminus \Sigma_{\infty,\Fp}}$ is also non-empty by corollary \ref{non empty goren strata}. Therefore the stratum $\overline{Y}^\prime_0(\Fp)_{(\Theta_{F,\Fp} \setminus \Sigma_{\infty,\Fp}),(\Theta_{F,\Fp} \setminus \Sigma_{\infty,\Fp})}$ is non-empty and thus so is every stratum $\overline{Y}^\prime_0(\Fp)_{\phi^\prime(I),J}$.
\end{proof}

\begin{cor}
The Jacquet-Langlands relation at Iwahori level implies a relation for the Goren-Oort strata.
\end{cor}

\begin{rem}
The content of the above corollary in the case that $p$ is unramified is equivalent to \cite[Theorem 5.8]{tian_xiao_2016}. In fact we will obtain the same target as in loc cit.
\end{rem}

It follows from the construction of the Hecke actions that the projection $\pi_1: \overline{Y}^\prime_0(\Fp) \to \overline{Y}^\prime$ is Hecke equivariant. In particular, we deduce that $\psi_\Fp: \overline{Y}^\prime_J \xrightarrow{\sim} \overline{Y}^\prime_0(\Fp)_{(\Theta_{F,\Fp} \setminus \Sigma_{\infty,\Fp}),J} $ is Hecke equivariant in the sense that, for any sufficiently small open compact subgroups $U^\prime_1,U^\prime_2 \subset G^\prime_\Sigma(\bA_f^{(p)})$ and $g$ such that $g^{-1}U^\prime_1g \subset U^\prime_2$, the following diagram commutes:

\begin{equation}\label{Hecke equivariance psi}
    \begin{tikzcd}
        \overline{Y}^\prime_{1,J} \arrow[r,"{\psi_{1,\Fp}}"] \arrow[d,"{\rho_g}"'] & \overline{Y}^\prime_{1,0}(\Fp)_{(\Theta_{F,\Fp} \setminus \Sigma_{\infty,\Fp}),J} \arrow[d,"{\rho_g}"] \\
        \overline{Y}^\prime_{2,J} \arrow[r,"{\psi_{2,\Fp}}"] & \overline{Y}^\prime_{2,0}(\Fp)_{(\Theta_{F,\Fp} \setminus \Sigma_{\infty,\Fp}),J}.
    \end{tikzcd}
\end{equation}

Let $\Sigma, \Fp, I,J$ be as above and write $\overline{Y} = Y_{U}(G_\Sigma)_{\overline{\bF}_p}$, $\overline{Y}_J$, $\overline{Y}_0(\Fp) = Y_{U_0(\Fp)}(G_\Sigma)_{\overline{\bF}_p}$, and $\overline{Y}_0(\Fp)_{\phi^\prime(I),J}$ for the quaternionic models and their strata as defined in sections \ref{quaternionic models}, \ref{quaternionic goren oort}, and \ref{quaternionic iwahori}. It follows from lemma \ref{compatibility components unitary quaternionic}, which compatibly identifies each of the above schemes as an open and closed subscheme of its unitary counterpart, that we obtain an isomorphism $\psi_{\Fp} : \overline{Y}_J \xrightarrow{\sim} \overline{Y}_0(\Fp)_{(\Theta_{F,\Fp} \setminus \Sigma_{\infty,\Fp}),J} $ which is Hecke equivariant in the sense that the analogous diagram to (\ref{Hecke equivariance psi}) commutes.

Finally, we recall that a similar morphism to $\psi_\Fp$ was constructed in the Hilbert setting \cite[$\mathsection$ 6.1]{Diamond_2023}. One can check that this construction is compatible with $\psi_\Fp$ under the isomorphisms $Y_U(G) \xrightarrow{\sim} Y_{U_\emptyset}(G_\emptyset)$ and $Y_{U_0(\Fp)}(G) \xrightarrow{\sim} Y_{U_{\emptyset,0}(\Fp)}(G_\emptyset)$ from section \ref{model comparison} and proposition \ref{quaternion hilb iwahori}. We end this chapter by giving a brief discussion of Iwahori level $U_0(\fp)$ for a prime $\fp$ such that $\Theta_{F,\fp} \subset \Sigma_\infty$. 

It is straightforward to check, say by Grothendieck-Messing, that $\widetilde{Y}_0(\fp)$ is \'{e}tale over $Y$. Pick a geometric point $x \in Y(\overline{\bF}_p)$ corresponding to $\underline{A}$ with reduced Dieudonn\'{e} module $\Delta$. The data of a point of in $Y_0(\fp)$ above $x$ then corresponds to a line in $\Delta_\tau/[\varpi_\fp]_\tau\Delta_{\tau}$, for any $\tau \in \widehat{\Theta}_{E,\fp}$ which is stable under the total essential Frobenius $F_{\es}: \Delta_\tau/[\varpi_\fp]_\tau\Delta_{\tau} \to \Delta_\tau/[\varpi_\fp]_\tau\Delta_{\tau}$, composition of every single essential Frobenius. It is an isomorphism. In particular, because $\Delta_\tau/[\varpi_\fp]_\tau\Delta_{\tau}$ is a two dimensional vector space over $\overline{\bF}_p$ and $F_{\es}$ is $\phi^{f_\fp}$-semilinear, we may pick a basis $(e_1,e_2)$ such that $F_{\es}(e_i)=e_i$. Paying attention to the fact that $F_{\es}$ is not linear but $\phi^{f_\fp}$-semilinear, the only possible lines are the ones generated by $e_1$, $e_2$ and $e_1+\mu e_2$ for $\mu \in \bF_{p^{f_\fp}}^\times$. Now, any two subgroups of $A[\fp]$ defined by these lines can be mapped isomorphically into each other by some automorphism of $A[\fp]$, however, since $\End(A)\otimes \bQ =D$, there is no such automorphism of $A$. Therefore $Y_0(\fp)$ is finite \'{e}tale over $Y$ of degree $p^{f_\fp}+1$.

\section{Geometric Jacquet-Langlands Relations mod \texorpdfstring{$p$}{p}}
In this section, we state and prove the main result of the thesis, that is that the strata $\overline{Y}_0(\Fp)_{\phi^\prime(I),J}$ are isomorphic to products of $\mathbb{P}^1$-bundles over auxiliary Quaternionic Shimrua varieties. To do so, we expand the technique of "splicing" introduced by \cite{2020arXiv200100530D} to this more general setting. In fact, using the techniques developed in the previous section, we introduce an explicit construction of the inverse operation that we call splitting. We keep the notation from the previous sections. In order to to lighten notation, we will however only describe the case $\Fp = \fp$. The case for a general $\Fp = \fp_1 \cdots \fp_n$, is dealt with by applying the upcoming recipes independently at each prime $\fp_i \vert \Fp$. See remark \ref{how to do general case} for a more detailed explanation.

\subsection{Some auxiliary sets} \label{sets}
We start by defining some sets that we will need from now on. Recall from section \ref{unitary hasse} that we have a shift operation $\phi$ on the set $\Theta_F$ given as follows: For any $\theta \in \widehat{\Theta}_{F,\fp}$ and embedding $\theta^i$, set $\phi(\theta^i)=\theta^{i+1}$ if $i < e_\fp$ and $\phi(\theta^{e_\fp}) =(\phi \circ \theta)^1$ where $\phi \circ \theta$ is given by the genuine Frobenius. This induces a shift $\phi'$ on $\Theta_F \setminus \Sigma_\infty$ by setting $\phi'(\beta) := \phi^n(\beta)$ where $n \geq 1$ is the smallest integer such that $\phi^n(\beta) \notin \Sigma_\infty$. We will often extend $\phi^\prime$ to all of $\Theta_F$.

Let $I,J \subset \Theta_{F,\fp} \setminus \Sigma_{\infty,\fp}$ be subsets such that $I \cup J = \Theta_{F,\fp} \setminus \Sigma_{\infty,\fp}$. For such a pair, under the added condition that $(I,J) \neq (\Theta_{F,\fp} \setminus \Sigma_{\infty,\fp},\Theta_{F,\fp} \setminus \Sigma_{\infty,\fp})$ we define an auxiliary set $T$ as follows:\\
Start by setting $T_0 =I^c \subset J$. Now, write $I \cap J$ as a union of maximal disjoint chains with respect to the cycle structure $\phi'$ on $\Theta_{F,\fp} \setminus \Sigma_{\infty,\fp}$. To define $T$, we append to $T_0$, for each such chain $C= \{{\phi'}^{-n}(\beta), \cdots, \beta \}$, the elements
\[\{ {\phi'}^{-2k}(\beta) \, \vert \, 0 \leq k \leq \lfloor n/2 \rfloor \} \textrm{ if } {\phi'}(\beta) \notin J \]
and
\[\{ {\phi'}^{-2k-1}(\beta) \, \vert \, 0 \leq k \leq \lfloor (n-1)/2 \rfloor \} \textrm{ if } {\phi'}(\beta) \notin I. \]
It follows from construction that $I^c \subset T \subset J$.

For example, consider the case that $[F:\bQ]=12$ and there is only one prime $\fp$ above $p$ (once we have fixed an ordering on ramified embeddings, our formula is ramification agnostic), $\Sigma = \emptyset$ and $J=\{1,2,4,5,6,7,9,10,11,12\}$, $I= \{1,2,3,5,6,8,9,10,12\}$ where we identify $i$ with the embedding $\theta^i$. Then $I \cap J = \{ 5,6 \} \sqcup \{9,10\} \sqcup \{12,1,2\}$ and $T = \{2,4,5,7,9,11,12\}$. 

Note that $T$ is a subset of $\Theta_{F,\fp} \setminus \Sigma_{\infty, \fp}$ only, we extend $T$ to a full subset $T^\prime \subset \Theta_{F,\fp}$ as follows: Write $\Sigma_{\infty,\fp}$ as a union of maximal disjoint chains 
\[\Sigma_{\infty,\fp} = \bigsqcup_i C_i \sqcup \bigsqcup_j C^\prime_j, \]

\noindent where each $C_i$ is of the form $C_i= \{\phi^{-n}(\beta), \cdots, \beta \} \subset \Sigma_{\infty,\fp}$ with $\phi(\beta) \in T$ and each $C^\prime_j$ is of the form $C^\prime_j= \{\phi^{-n}(\beta), \cdots, \beta \} \subset \Sigma_{\infty,\fp}$ with $\phi(\beta) \notin T$. Define $T' \subset \Theta_F$ \label{T'} to be the set obtained by appending to $T$ the chains $C_i$. That is 
 \[T^\prime = T \sqcup \bigsqcup_i C_i.\]
 
 Finally we let $T^1 \subset \widehat{\Theta}_{F,\fp}$ denote the set \[T^1 := \{ \theta \in \widehat{\Theta}_{F,\fp} \, \vert \, \theta^1 \in T' \}. \]

Consider the following subset of $\Sigma^+_{IJ,\infty} \subset \Theta_{F} \setminus \Sigma_\infty$ given by:
\[\Sigma^+_{IJ,\infty} = \left\{\beta \notin T  \, \vert \,  \phi'(\beta) \in T \right\} \cup \left\{\beta \in T  \, \vert \,  \phi'(\beta) \notin T \right\}. \]
\noindent It is straightforward to check that the cardinality of $\Sigma^+_{IJ,\infty}$ is even and that it contains $I \cap J$ by the alternating construction of $T$. We then set \[ \Sigma_{IJ} = \Sigma \sqcup \Sigma^+_{IJ,\infty} \quad \textrm{and} \quad \widetilde{\Sigma}_{IJ,\infty} = \widetilde{\Sigma}_\infty \sqcup \widetilde{\Sigma}^+_{IJ,\infty},\]
\noindent where $\widetilde{\Sigma}^+_{IJ,\infty}$ is given by \[\widetilde{\Sigma}^+_{IJ,\infty} = \{ \tilde{\beta} \, \vert \, \beta \notin T \} \cup \{ \tilde{\beta}^c \, \vert \, \beta \in T \}. \]
Finally, we set $R = \Sigma^+_{IJ,\infty} \setminus I \cap J$.
\begin{rem}
Note that, by construction, we have $I \cap J \subset \Sigma_{IJ}^+$ and $\widetilde{\Sigma}_\infty \subset \widetilde{\Sigma}_{IJ}$.
\end{rem}

\subsection{The idea behind splicing}\label{motivating splices}
Before we move on, let us motivate the combinatorics behind the definitions from the previous section and give the main idea behind the upcoming constructions. 

Consider first the case that $p$ is unramified in $F$, $\Sigma = \emptyset$ and let $S^\prime = Y_{U_0^\prime(p)}(G^\prime_\emptyset)_\bF$ be the associated Unitary Shimura variety. To define a morphism from a stratum $S^\prime_{\phi(I),J}$ to an auxiliary Unitary Shimura variety, we need to define a new abelian scheme $A_{IJ}$ which is (quasi-)isogeneous to the $A_i$. By lemma \ref{C_T}, we can construct a subgroup $C_{T} \subset A_1[\fq]$ for any $I^c \subset T \subset J$ (Note here that $T =T^1$) and thus an $A_{IJ}$. Having fixed a choice of $T$ and thus after some more work, an $A_{IJ}$, proposition \ref{aij dieudonne calculations} will show that the Dieudonn\'{e} module $\mathbb{D}_{IJ}$ of $A_{IJ}$, at any geometric point $\overline{s}$ of $S^\prime_{\phi(I),J}$, is given on $\tilde{\theta}$ components by 
\begin{equation}\label{explanation AIJ components}
\mathbb{D}_{IJ,\tilde{\theta}} = 
\begin{cases}
\mathbb{D}_{1,\tilde{\theta}} & \textrm{ if } \theta \notin T,\\
f^*\mathbb{D}_{2,\tilde{\theta}} & \textrm{ if } \theta \in T,
\end{cases}
\end{equation}
where the equality is taken inside $\mathbb{D}_1 \otimes \bQ$.

In particular, after fixing the appropriate structures, the Unitary Shimura variety this abelian scheme comes from has underlying quaternion algebra given by $\Sigma_{IJ} = \Sigma_T = \{ \theta \in T \, \vert \, \phi \circ \theta \notin T \} \cup \{ \theta \notin T \, \vert \, \phi \circ \theta \in T \}$.
Our strategy now is to maximize the size of $\Sigma_T$, and thus minimize the dimension of the target, by choosing the appropriate $T$. This desire immediately leads to our specific choice of $T$ from section \ref{sets}, and its alternating construction. 

In order to obtain an isomorphism, we need to be able to reconstruct the $A_i$ from this new abelian scheme $A_{IJ}$ and potentially the data from some projective bundles. To do so, at least pointwise, it suffices to reconstruct their Dieudonn\'{e} modules $\bD_1$ and $\bD_2$ from $\mathbb{D}_{IJ}$; in particular it suffices to do so for each $\tilde{\theta}$ component. By formula \ref{explanation AIJ components}, we are halfway there. At the level of Dieudonn\'{e} modules, the conditions to lie in the stratum $S^\prime_{\phi(I),J}$ translate into the equalities inside $\mathbb{D}_1 \otimes \bQ$
\begin{equation}\label{explanation relations}
\begin{cases}
f^*\mathbb{D}_{2,\phi \circ \tilde{\theta}} = \Phi \mathbb{D}_{1, \tilde{\theta}} & \textrm{ if } \theta \in I,\\
\mathbb{D}_{1,\phi \circ \tilde{\theta}} = p^{-1}\Phi f^*\mathbb{D}_{2, \tilde{\theta}} & \textrm{ if } \theta \in J.
\end{cases}
\end{equation}
Therefore we can retrieve, say $\mathbb{D}_{2,\phi \circ \tilde{\theta}}$, provided that $\theta \in I$ and we already have $\Phi \mathbb{D}_{1, \tilde{\theta}}$. This is the case, for example, by formula \ref{explanation AIJ components}, if both $\theta, \phi \circ \theta \notin T$. 
To visualize the situation, one can consider the following type of diagram, as applied to the previous example of section \ref{sets}: Let $p\ccO_F = \fp$ and write $\Theta_F = \Theta_{F,\fp} = \{\theta_1,\cdots,\theta_{12}\}$, where $\phi(\theta_i)=\theta_{i+1}$, and identify $\theta_i$ with $i$. Take $J=\{1,2,4,5,6,7,9,10,11,12\}$, $I= \{1,2,3,5,6,8,9,10,12\}$ and $T =\{2,4,5,7,9,11,12\}$: 
\[\adjustbox{scale = 0.75, center}{
\begin{tikzcd}
& 1        & 2        &        & 4      & 5        & 6        & 7      &   & 9 & 10     & 11      & 12       \\
{} &\bigcirc &         & \times &        &        & \bigcirc  &       & \times & & \bigcirc&        &        \\
   &     & \bigcirc &      & \times  & \bigcirc &        & \times  &      & \bigcirc&  & \times  & \bigcirc  & {} \\
& 1        & 2        & 3      &        & 5        & 6        &        & 8     & 9 & 10 &        & 12      
\end{tikzcd}
}\]
The numbers on the top row represent the elements of $J$ and the numbers on the bottom row the elements of $I$; we add them for easier bookkeeping. The first inner row corresponds to $\mathbb{D}_1$ and the second to $\mathbb{D}_2$. We write the entries in coordinates $(1,i)$ and $(2,i)$. Filling in the the entry $(j,i)$ corresponds to knowing $\bD_{j,\tilde{\theta}_i}$ from $\bD_{IJ}$. Via \eqref{explanation AIJ components}, we can immediately fill in $(1,i)$ for every $\theta_i \notin T$ and $(2,i)$ for every $\theta_i \in T$. The crosses represent the entries which are forced by the condition $I^c \subset T \subset J$ and the circles our specific choice of $T$. 

Our goal now is to complete the diagram using the relations from \eqref{explanation relations}. 
This translates to re-obtaining all of $\mathbb{D}_1$ and $\mathbb{D}_2$ from $\mathbb{D}_{IJ}$. 
We visualize this as an arrow: Suppose we want to fill in the entry $(1,5)$, i.e. obtain $\mathbb{D}_{1,\tilde{\theta}_5}$.
We know that $4$ is in the top row, i.e. $\theta_4 \in J$, and $(2,4)$ is filled in, i.e. we already know $\mathbb{D}_{2,\tilde{\theta}_4}$.
Then \eqref{explanation relations} gives us $\mathbb{D}_{1,\tilde{\theta}_5} = p^{-1}\Phi f^*\mathbb{D}_{2, \tilde{\theta}_4}$ and we can add a square to $(1,5)$ and an arrow from $(2,4)$ to $(1,5)$. Now that we have filled in $(1,5)$, since $\theta_5 \in I$ we can use the same logic as above and fill in the entry $(2,6)$ from the relation $ f^*\mathbb{D}_{2, \tilde{\theta}_6} = \Phi\mathbb{D}_{1,\tilde{\theta}_5}$. Repeating this process to fill in as many missing entries as possible yields

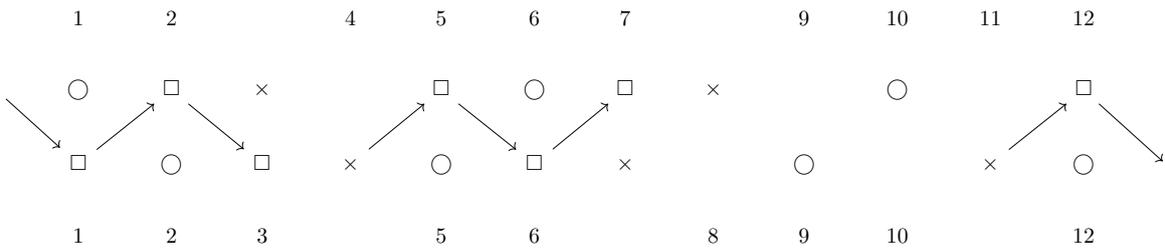
\[\adjustbox{scale = 0.75, center}{
\begin{tikzcd}
& 1        & 2        &        & 4      & 5        & 6        & 7      &   & 9 & 10     & 11      & 12       \\
{} \arrow[rd] &\bigcirc &     \Box \arrow[rd]     & \times &        &     \Box  \arrow[rd]   & \bigcirc  &   \Box     & \times & & \bigcirc&        &      \Box  \arrow[rd]  \\
   & \Box  \arrow[ru]    & \bigcirc &   \Box     & \times \arrow[ru] & \bigcirc &  \Box \arrow[ru]       & \times  &      & \bigcirc&  & \times \arrow[ru] & \bigcirc  & {} \\
& 1        & 2        & 3      &        & 5        & 6        &        & 8     & 9 & 10 &        & 12      
\end{tikzcd}
}\]

Note the zig-zags which appear in the chains of $I \cap J$. We are now missing certain entries that cannot be derived from our rules. In terms of the morphism we will construct, these then correspond to the set $R = \Sigma_T \setminus I \cap J$ of projective bundles. In particular, we will be adding back the data of the Hodge filtration $0 \to \omega^0_{j,\tilde{\theta}_i} \subset \ccH^1_{j,\tilde{\theta}_i} \to v^0_{j,\tilde{\theta}_i} \to 0$, where $j$ depends on whether $\theta_i$ is in $T$ or not. At the level of Dieudonn\'{e} modules, this corresponds to $V\mathbb{D}_{j,\tilde{\theta}_{i+1}} \subset \bD_{IJ,\tilde{\theta}_i}$ and thus gives us back $\mathbb{D}_{j,\tilde{\theta}_{i+1}}$.   
At the level of diagrams, we now add for each $\theta_i \in R$, a $+$ sign in the entry $(j,i+1)$. Continuing the zig-zag process we thus obtain a completed table: 

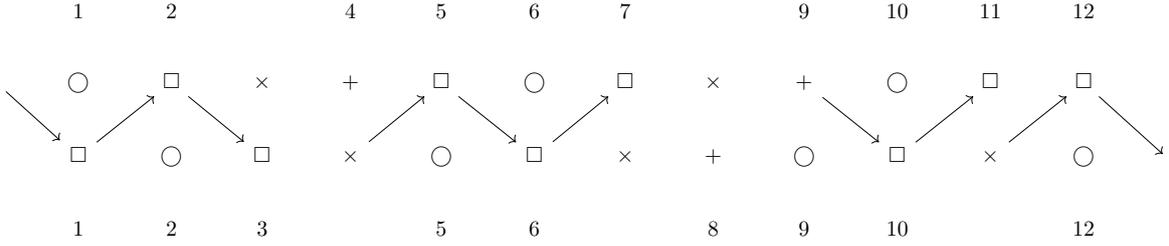
\[\adjustbox{scale = 0.75, center}{
\begin{tikzcd}
& 1        & 2        &        & 4      & 5        & 6        & 7      &   & 9 & 10     & 11      & 12       \\
{} \arrow[rd] &\bigcirc &     \Box \arrow[rd]     & \times &    +    &     \Box  \arrow[rd]   & \bigcirc  &   \Box     & \times & + \arrow[rd] & \bigcirc&    \Box    &      \Box  \arrow[rd]  \\
   & \Box  \arrow[ru]    & \bigcirc &   \Box     & \times \arrow[ru] & \bigcirc &  \Box \arrow[ru]       & \times  &  +    & \bigcirc& \Box \arrow[ru] & \times \arrow[ru] & \bigcirc  & {} \\
& 1        & 2        & 3      &        & 5        & 6        &        & 8     & 9 & 10 &        & 12      
\end{tikzcd}
}\]

We encourage the reader who wants to gain intuition behind the choices in our construction to play with these diagrams for various $I$ and $J$, even trying different choices of $I^c \subset T \subset J$ and $R \subset \Sigma_T$ to see why our specific choices work. As a non-example, consider for example $p$ inert, $f_\fp=4$, $I = \{2,3,4\}$, $J = \{1,2,3\}$ and $T = \{1,2\}$ (the correct choice of $T$ is $\{1,3\})$. Then $\Sigma_T = \{2,4\}$ and the splice $A_{IJ}$ lives over a two dimensional Unitary Shimura variety $\overline{Y}^\prime_{\Sigma_T}$, the same dimension as $S^\prime_{\phi(I),J}$. The induced morphism $S^\prime_{\phi(I),J} \to \overline{Y}^\prime_{\Sigma_T}$ would then have to be an isomorphism. However, this cannot be since we cannot complete the diagram:

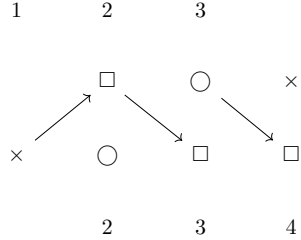
\[\adjustbox{scale = 0.75, center}{
\begin{tikzcd}
   1    & 2        &    3    &    {}   \\
{} &    \Box   \arrow[rd]   & \bigcirc  \arrow[rd] &    \times           \\
 \times \arrow[ur]    & \bigcirc &    \Box    & \Box  \\
 {}      & 2        & 3      &   4        
\end{tikzcd}
}\]

On the opposite end, suppose now that $p$ is totally ramified in $F$ and again $\Sigma = \emptyset$. In this situation $A_{IJ}$ can only be either $A_1$ or $A_2$ (up to a twist by an ideal of $\ccO_E$) therefore our only recourse is to modify the filtrations, essentially by forgetting parts of it. More specifically, suppose that we know that $[\varpi_\fp]\omega^0_\tau(i+1) = \omega^0_\tau(i-1)$ for some $1 \leq i < e_\fp$. Then we can define a new Pappas-Rapoport type filtration $\underline{\tilde{\omega}^0_\tau}$ on $\omega^0_\tau$ such that $\tilde{\omega}^0_\tau(j) = \omega^0_\tau(j)$ for $j \neq i$ and $\tilde{\omega}^0_\tau(i) = \omega^0_\tau(i+1)$.

Recall that the map $f^* \circ g^*: \omega^0_{1,\tau} \otimes \fp \to \omega^0_{1,\tau}$ is given by multiplication. Therefore, if $\theta^i \in I$ and $\theta^{i+1} \in J$ for some $i < e_\fp$ then 
\[[\varpi_\fp]\omega^0_{1,\tau}(i+1) = f^*\circ g^* (\omega^0_{1,\tau}(i+1)\otimes \fp) = f^*\omega^0_{2,\tau}(i) = \omega^0_{1,\tau}(i-1).\] Similarly, if $\theta^i \in J$ and $\theta^{i+1} \in I$, then $[\varpi_\fp]\omega^0_{2,\tau}(i+1) = \omega^0_{2,\tau}(i-1).$ This suggests picking an $A_{IJ}=A_i$ and a subset $I^c \subset T \subset J$ such that we forget the filtered pieces $\omega^0_{1,\tau}(j)$ for $j$ such that $\theta^j \notin T$ and $\theta^{j+1} \in T$ if $i=1$, and $\omega^0_{2,\tau}(j)$ for $j$ such that $\theta^j \in T$ and $\theta^{j+1} \notin T$ if $i=2$. However, this breaks down precisely at $j = e_\fp$ where the above formula, $[\varpi_\fp]\omega^0_{i,\tau}(e_\fp+1) = \omega^0_{i,\tau}(e_\fp-1)$, does not make sense. In order to avoid this situation we need to ensure, provided that we have picked a suitable $T$, that we can take $A_{IJ} = A_1$ if $\theta^1 \notin T$ and $A_{IJ} =A_2$ if $\theta^1 \in T$. This is precisely the content of lemma \ref{C_T} and the motivation behind the definition of the set $T^1$ in general.

The combinatorics behind the specific choice of $T$ then follows from the same diagram reasoning as above, this time using the relations between both filtrations as dictated by the stratum. Once we have retrieved both $A_1$ and $A_2$ from $A_{IJ}$, we re-obtain the missing pieces of the filtrations by taking images and preimages via $f^*$ and $g^*$.

\subsection{Splicing}\label{splicing}
In this section, we describe one of the main constructions in this thesis, that is for $I,J \subset \Theta_{F,\fp} \setminus \Sigma_{\infty,\fp}$ such that $I \cup J = \Theta_{F,\fp} \setminus \Sigma_{\infty,\fp}$, the so called "splice" of the universal pair $(\underline{A_1},\underline{A_2},f)$ over the stratum $S^\prime_{\phi^\prime(I),J}$. This splice is a new abelian scheme whose Dieudonn\'{e} module at each geometric point is obtained, in some sense, by splicing together the Dieudonn\'{e} modules of $A_1$ and $A_2$. This generalizes the construction of the same name introduced in \cite[section 5.1]{2020arXiv200100530D} in the case of $p$ unramified in $F$, $\Sigma = \emptyset$ and maximal strata $I = J^c$. Our construction will be somewhat more direct and will of course, yield the same abelian variety obtained in \cite{2020arXiv200100530D} when restricted to those specific cases.
\subsubsection{Constructing universal splices}

Fix $\Sigma$, a prime $\fp$ of $F$, and write $S^\prime = Y_{U^\prime_0(\fp)}(G^\prime_\Sigma)_\bF$ for $U^\prime \subset G^\prime_\Sigma(\bA_f)$ as usual. Fix $I,J \subset \Theta_{F,\fp} \setminus \Sigma_{\infty,\fp}$ such that $I \cap J \neq I \cup J = \Theta_{F,\fp} \setminus \Sigma_{\infty,\fp}$ and consider the stratum $S = S^\prime_{\phi^\prime(I),J}$. Let $T$ and $T^1$ be as defined in section \ref{sets}, where we recall that $T$ is a certain set $I^c \subset T \subset J$ determined by $I$ and $J$, and $T^1 \subset \widehat{\Theta}_{F,\fp}$ is determined by $T$.

We first start by mentioning with a construction: Let $A$ be an abelian scheme with an action of $\ccO_D$ and a prime to $p$ quasi-polarization $\lambda: A \to A^\vee$ whose induced Rosati involution is compatible with the anti-involution on $\ccO_D$. For a choice of prime $\fp$ and the chosen uniformizer $\varpi_\fp$, we construct a prime to $p$ quasi polarization $\lambda^\prime : A \otimes \fp \to (A \otimes \fp)^\vee \simeq A^\vee \otimes \fp^{-1}$, satisfying the same properties as $\lambda$, as follows: It suffices to do so on each connected component so let $n$ be an integer such that $n\lambda$ is an honest polarization. Write $\varpi_\fp^{e_\fp} = p \alpha$ for $\alpha \in \ccO_{F,(p),+}$. Let $m$ be the smallest positive integer such that $m \alpha^{-1} \in \ccO_F$. Then $m \varpi_\fp^{-1} \in \fp^{-1}$ and $m \varpi_\fp^{-2}\fp \subset \fp^{-1}$. We define $nm\lambda^\prime$ as the map
\begin{equation}\label{polarize tensor}
    A \otimes \fp \ni x \otimes \beta \mapsto n\lambda(x) \otimes m \varpi_\fp^{-2}\beta \in A^\vee \otimes \fp^{-1} \simeq (A \otimes \fp)^\vee.
\end{equation}

By lemma \ref{C_T} for $A_2 \otimes \fp$, we have the $\ccO_E$-stable subgroup $C_{T^c} \subset e_0 \ker (g \otimes \fp) [\fq^c]$. Write $C^\prime_{T^c} = C_{T^c} \otimes_{\ccO_E} \ccO_{E}^2 \subset (A_2 \otimes \fp)[\fq^c]$.
Let $\lambda^\prime_2$ be the polarization of $A_2 \otimes \fp$ constructed above. We now set 
\[H^\prime_{T^c} =  (\lambda^\prime_2)^{-1}(((A_2 \otimes \fp)[\fq^c]/C'_{T^c})^\vee) \oplus C'_{T^c} \subset (A_2 \otimes \fp)[\fq] \oplus (A_2 \otimes \fp)[\fq^c] = (A_2 \otimes \fp)[\fp].\]
We thus obtain the abelian variety $A_{IJ} = A_2 \otimes \fp / H^\prime_{T^c}$; it inherits an action of $\ccO_D$ since $H^\prime_{T^c}$ is stable under the action of $\ccO_D$. We write $\pi_2: A_2 \otimes \fp \to A_{IJ}$ and $\psi_2: A_{IJ} \to A_2$ for the unique isogeny such that $\psi_2 \circ \pi_2=m_{2,\fp}$ is given by multiplication by $\fp$. These are $\ccO_D$-linear isogenies of degree $p^{4f_\fp}$. Furthermore, we obtain a prime to $p$ polarization $\lambda_{IJ} : A_{IJ} \to A^\vee_{IJ}$ defined via the commutative diagram
\[
\begin{tikzcd}
A_2 \otimes \fp \arrow[r,"\pi_2"] \arrow[d,"\lambda_2^\prime \otimes \varpi_\fp"'] & A_{IJ} \arrow[d,"\lambda_{IJ}"] \\
A^\vee_2 \arrow[r, "\psi^\vee_2"] & A^\vee_{IJ},
\end{tikzcd}
\]
where, by construction of $\lambda^\prime_2$, the isogeny $\lambda^\prime_2 \otimes \varpi_\fp$ is given by $x \otimes \beta \mapsto (\varpi_\fp^{-1}\beta)\lambda_2(x)$. It follows that the Rosati involution induced by $\lambda_{IJ}$ is compatible with the anti-involution on $\ccO_D$.

We can construct a similar abelian scheme $A^\prime_{IJ}$ using $A_1$. By lemma \ref{C_T}, we have the $\ccO_E$-stable subgroup $C_{T} \subset e_0 \ker (f) [\fq]$; we let $C^\prime_{T} = C_{T} \otimes_{\ccO_E} \ccO_{E}^2 \subset A_1 [\fq]$. We now set \[H^\prime_T = C'_T \oplus \lambda_1^{-1}((A_1[\fq]/C'_T)^\vee) \subset A_1[\fq] \oplus A_1[\fq^c] = A_1[\fp],\]
and $A^\prime_{IJ} = A_1 / H^\prime_T$. As before, $A^\prime_{IJ}$ comes with an action of $\ccO_D$ since $H^\prime_{T}$ is stable under the action of $\ccO_D$ and we write $\pi_1: A_1 \to A^\prime_{IJ}$ and $\psi_1: A^\prime_{IJ} \to A_1 \otimes \fp^{-1}$ for the unique isogeny such that $\psi_1 \circ \pi_1 = \iota_{1,\fp}$, the isogeny induced by the inclusion $\ccO_E \hookrightarrow \fp^{-1}\ccO_E$. These are $\ccO_D$-linear isogenies of degree $p^{4f_\fp}$. Furthermore, we obtain a prime to $p$ polarization $\lambda^\prime_{IJ} : A^\prime_{IJ} \to A^{\prime \vee}_{IJ}$ defined via the commutative diagram
\[
\begin{tikzcd}
A_1 \arrow[r,"\pi_1"] \arrow[d,"\lambda_1 \otimes \varpi_\fp"'] & A^\prime_{IJ} \arrow[d,"\lambda^\prime_{IJ}"] \\
A^\vee_1  \otimes \fp \arrow[r, "\psi^\vee_1"] & A^{\prime \vee}_{IJ}.
\end{tikzcd}
\]
It follows that the Rosati involution induced by $\lambda^\prime_{IJ}$ is compatible with the anti-involution on $\ccO_D$.

$A_{IJ}$ and $A^\prime_{IJ}$ are related as follows: Since $C'_{T^c} \subset \ker(g \otimes \fp)$, then $(g \otimes \fp) (C'_{T^c})=0$. Also, $G = e_0(g \otimes \fp) (\lambda_2^{-1}(((A_2 \otimes \fp)[\fq^c]/C'_{T^c})^\vee) \subset \ker f$ has degree $p^{2f_\fp - (f_\fp - \vert T^1 \vert) -f_\fp} = p^{\vert T^1 \vert} = \deg C_T$. Furthermore, by looking at crystals say, for any $\tau = \tilde{\theta} \in \widehat{\Theta}_{E,\fp}$ such that $\omega^0_{i,\tau} \neq 0$, then $\omega^0_{G,\tau}$ is a line bundle. It follows that $G = C_T$. Therefore $ (g \otimes \fp)(H^\prime_{T^c}) \subset H^\prime_T$ and the isogeny $\pi_1 \circ (g \otimes \fp)$ factors through $\pi_2$. Writing $q$ for the $\ccO_D$-linear isogeny $q:A_{IJ} \to A^\prime_{IJ}$ such that $\pi_1 \circ (g \otimes \fp) = q \circ \pi_2$, it is clear that this isogeny has degree $p^{\deg( g \otimes \fp) + \deg \, \pi_1 - \deg \, \pi_2} = p^{4 f_\fp} $ and has kernel contained in $A_{IJ}[\fp]$. In fact, by comparing degrees, we have that $\ker (\pi_1 \circ (g \otimes \fp))[\fq] = H^\prime_{T^c}[\fq]$ so that $\ker q = A_{IJ}[\fq^c]$.  

We may thus identify $A^\prime_{IJ}$ with $A_{IJ} \otimes (\fq^c)^{-1}$ such that we obtain a commutative diagram 
\begin{equation}\label{splice diagram}
\begin{tikzcd}
A_2 \otimes \fp \arrow[r,"\pi_2"] \arrow[d, "g \otimes \fp"]& A_{IJ} \arrow[r,"\psi_2"] \arrow[d,"\iota_{IJ,\fq^c}"] & A_2 \arrow[d,"g"]\\
A_1 \arrow[r, "\pi_1"]  & A_{IJ} \otimes (\fq^c)^{-1} \arrow[r,"\psi_1"]  & A_1  \otimes \fp^{-1} 
\end{tikzcd}
\end{equation}
where the middle arrow is the natural isogeny induced by inclusion. Furthermore, under the isomorphism $\iota_{IJ,\fq^c}:A_{IJ}[\mathfrak{Q}^\infty] \xrightarrow{\sim} A_{IJ}[\mathfrak{Q}^\infty]\otimes \fq^c$, where $\mathfrak{Q} = \prod_{\fp \vert p} \fq$, the composite 
\[A_1[\mathfrak{Q}^\infty] \xrightarrow{\pi_1} A_{IJ}[\mathfrak{Q}^\infty] \xrightarrow{\psi_2} A_2[\mathfrak{Q}^\infty]\] is simply given by $f$, by the commutativity of the above diagram. From now on, in order to lighten notation, we will always make the isomorphism $\iota_{\fq^c}$ implicit when appealing to the above compatibility.

\begin{rem}
In the case that $(I,J) = (\Theta_{F,\fp} \setminus \Sigma_{\infty,\fp}, \emptyset) $, we have $T = \emptyset$ hence $C'_{T^c} = \ker (g \otimes \fp)[\fq^c]$ and thus $H'_{T^c} = \ker (g \otimes \fp)$ since it is totally isotropic with respect to the weil product. Hence $A_{IJ} \simeq A_1$ compatibly with the other structures. On the other hand, if $(I,J) = (\emptyset, \Theta_{F,\fp} \setminus \Sigma_{\infty,\fp})$ we have $T = \Theta_{F,\fp} \setminus \Sigma_{\infty,\fp}$ so that $C_{T^c} = 0$ and $H'_{T^c} = (A_2 \otimes \fp)[\fq]$. Hence $A_{IJ} \simeq A_2 \otimes \fq^c$ compatibly with the other structures.

Furthermore, one can check that this construction agrees with the construction of \cite[5.1.2]{2020arXiv200100530D} in the case of the maximal strata of Iwahori level Hilbert modular varieties for $p$ unramified in $F$. In fact, in this situation, after fixing the isomorphism $A_2 \otimes p\ccO_F \simeq A_2$, the isogeny $\pi_2: A_2 \otimes p\ccO_F \to  A_{J^c J} \simeq A_J$ is the same as the composite $ A_2 \xrightarrow{} A^\prime_J \xrightarrow{\psi} A_J$ shown in diagram \cite[(10)]{2020arXiv200100530D}.
\end{rem}

Note that we have defined a prime to $p$ polarization on $A^\prime_{IJ} \simeq A_{IJ} \otimes (\fq^c)^{-1}$ but we have a natural prime to $p$ polarization on $A_{IJ} \otimes (\fq^c)^{-1}$ defined by the commutative diagram 
\[
\begin{tikzcd}
A_{IJ} \arrow[r,"\iota_{\fq^c}"] \arrow[d,"\lambda_{IJ} \otimes \varpi_\fp"'] & A_{IJ} \otimes (\fq^c)^{-1} \arrow[d,"\lambda_{IJ,(\fq^c)^{-1}}"] \\
A^\vee_{IJ} \otimes \fp \arrow[r, "\iota^\vee_\fq"] & A^\vee_{IJ} \otimes \fq^c.
\end{tikzcd}
\]
Chasing through the definitions and using the compatibility between $\lambda^\prime_2$ and $\lambda_1$, we see that $\lambda^\prime_{IJ}$ and $\lambda_{IJ,(\fq^c)^{-1}}$ are equal under the isomorphism $A^\prime_{IJ} \simeq A_{IJ} \otimes (\fq^c)^{-1}$.

\subsubsection{Filtering the Splice}

Fix a geometric point $\overline{s}$ of $S^\prime_{\phi^\prime(I),J}$ and consider the diagram of Dieudonn\'{e} modules induced by the diagram \ref{splice diagram} over $\overline{s}$.

\begin{center}
    \begin{tikzcd}
{\mathbb{D}((A_2 \otimes \fp)[p^\infty])}                       & {\mathbb{D}(A_{IJ} [p^\infty])} \arrow[l, "\pi_2^*"']                              & {\mathbb{D}(A_2 [p^\infty])} \arrow[l, "\psi_2^*"']                                  \\
{\mathbb{D}(A_1 [p^\infty])} \arrow[u, "(g\otimes \fp)^*"] & {\mathbb{D}(A_{IJ} \otimes (\fq^c)^{-1}[p^\infty])} \arrow[u,"\iota_{\fq^c}^*"] \arrow[l, "\pi_1^*"] & {\mathbb{D}(A_1 \otimes \fp^{-1}[p^\infty])} \arrow[l, "\psi_1^*"] \arrow[u, "g^*"']
\end{tikzcd}
\end{center}
We now describe the image of the lattice $\pi_2^*\bD(A_{IJ}[p^\infty])$ in $ \bD((A_2 \otimes \fp)[p^\infty])$.

\begin{prop}\label{aij dieudonne calculations}
Let $\theta \in \widehat{\Theta}_{F,\fp}$, then
    \[\pi_2^* \bD(A_{IJ}[p^\infty])_{\tilde{\theta}}= 
    \begin{cases}
        m_{\fp}^*\bD(A_2[p^\infty])_{\tilde{\theta}} & \textrm{ if } \theta \in T^1, \\
        (g \otimes \fp)^* \bD(A_1[p^\infty])_{\tilde{\theta}} &  \textrm{ if } \theta \notin T^1.
    \end{cases}\]
    \[\pi_2^* \bD(A_{IJ}[p^\infty])_{\tilde{\theta}^c}= 
    \begin{cases}
        \bD((A_2 \otimes \fp)[p^\infty])_{\tilde{\theta}^c} & \textrm{ if } \theta \in T^1, \\
        (g \otimes \fp)^* \bD(A_1[p^\infty])_{\tilde{\theta}^c} &  \textrm{ if } \theta \notin T^1.
    \end{cases}\]
    \[
    \dim \omega_{IJ,\tilde{\theta}}= 
    \begin{cases}
    2s_{\tilde{\theta}}(e_\fp) + 2 & \textrm{ if } \theta \in T^1, \phi \circ \theta \notin T^1,\\
    2s_{\tilde{\theta}}(e_\fp) - 2 & \textrm{ if } \theta \notin T^1, \phi \circ \theta \in T^1,\\
    s_{\tilde{\theta}}(e_\fp) & \textrm{ otherwise.}
    \end{cases}
    \]
    \[\dim \omega_{IJ,\tilde{\theta}^c}= 4e_\fp - \dim \omega_{IJ,\tilde{\theta}}.   \]
\end{prop}

\begin{proof}
    Recall that the composition $\psi_2 \circ \pi_2:A_2 \otimes \fp \to A_2$ is given by the natural map induced by multiplication and that the composition $\psi_1 \circ \pi_1:A_1 \to A_1 \otimes \fp^{-1}$ is the natural map induced by the inclusion $\ccO_E \hookrightarrow \fp^{-1}\ccO_E$. Therefore $\pi_2^* \circ \psi_2^*$ has cokernel isomorphic to $\mathbb{D}((A_2 \otimes \fp)[\fp])$ and $\pi_1^* \circ \psi_1^*$ has cokernel isomorphic to $\mathbb{D}(A_1[\fp])$. Furthermore, recall that $\pi_1^* \circ (q^*)^{-1} \circ \psi_2^* = f^*$. Finally, recall that, by lemma \ref{raynaud dieudonne dimension} that 
    \[ \dim \mathbb{D}(C^\prime_T)_{\tilde{\theta}} = 
    \begin{cases}
    2 & \textrm{ if } \theta \in T^1, \\
    0 & \textrm{ if } \theta \notin T^1,
    \end{cases}
    \]
    and it follows, by construction, that 
    \[ \dim \mathbb{D}(H^\prime_T)_{\tilde{\theta}^c} = 
    \begin{cases}
    2 & \textrm{ if } \theta \in T^1, \\
    4 & \textrm{ if } \theta \notin T^1.
    \end{cases}
    \]
    Similarly, we have
     \[ \dim \mathbb{D}(C^\prime_{T^c})_{\tilde{\theta}^c} = 
    \begin{cases}
    2 & \textrm{ if } \theta \notin T^1, \\
    0 & \textrm{ if } \theta \in T^1.
    \end{cases}
    \]
    and
    \[ \dim \mathbb{D}(H^\prime_{T^c})_{\tilde{\theta}^c} = 
    \begin{cases}
    2 & \textrm{ if } \theta \notin T^1, \\
    4 & \textrm{ if } \theta \in T^1.
    \end{cases}
    \]
    where $C'_T \subset A_1[\fq]$ and $C'_{T^c} \subset (A_2 \otimes \fp)[\fq^c]$.

Let $\tau = \tilde{\theta}$ for $\theta \in T^1$. 
Then $\pi_2^*\bD(A_{IJ}[p^\infty])_\tau = [\varpi_\fp] \bD((A_2 \otimes \fp)[p^\infty])_\tau = m_\fp^*\bD(A_2[p^\infty])_\tau$ since it contains
$[\varpi_\fp] \bD((A_2 \otimes \fp)[p^\infty])_\tau$. 
If $\tau = \tilde{\theta}^c$ for $\theta \in T^1$, 
then $\pi_2^*$ is an isomorphism and 
$\pi_2^*\bD(A_{IJ}[p^\infty])_\tau = \bD((A_2 \otimes \fp)[p^\infty])_\tau$.

Let $\tau = \tilde{\theta}$ for $\theta \notin T^1$. Then $\iota_{\fq^c}^* \bD((A_{IJ}\otimes (\fq^c)^{-1})[p^\infty])_\tau = \bD(A_{IJ}[p^\infty])_\tau$ and $\pi_1^*\bD((A_{IJ}\otimes (\fq^c)^{-1})[p^\infty])_\tau = \bD(A_1[p^\infty])_\tau$. Hence $\pi_2^*\bD(A_{IJ}[p^\infty])_\tau = (g \otimes \fp)^*\bD(A_1[p^\infty])_\tau$. If $\tau = \tilde{\theta}^c$ for $\theta \notin T^1$, then $ (\iota_{\fq^c})^{-1}\bD((A_{IJ})[p^\infty])_\tau = [\varpi_\fp]^{-1} \bD((A_{IJ}\otimes (\fq^c)^{-1})[p^\infty])_\tau$ and $\pi_1^*\bD((A_{IJ}\otimes (\fq^c)^{-1})[p^\infty])_\tau = [\varpi_\fp]\bD(A_1[p^\infty])_\tau$. Hence $\pi_2^*\bD(A_{IJ}[p^\infty])_\tau = (g \otimes \fp)^*\bD(A_1[p^\infty])_\tau$.

For any $\tau \in \widehat{\Theta}_E$ and $A=A_1,A_2$ or $A_{IJ}$, we have 
$$\dim H^0 \left(A,\Omega^1_{A/\overline{\mathbb{F}}_p} \right)_\tau =  \dim V\mathbb{D}(A[p^\infty])_{\phi \circ \tau} / p \mathbb{D}(A[p^\infty])_{\tau}.$$

Then, writing $\Delta = \bD((A_2 \otimes \fp)[p^\infty])$ and $\Delta_{IJ} = \bD(A_{IJ}[p^\infty])$: 
\[
\begin{split}
\dim H^0 \left(A,\Omega^1_{A_{IJ}/\overline{\mathbb{F}}_p} \right)_\tau & =
\dim V \pi_2^*\Delta_{IJ,\phi \circ \tau} / p \pi_2^*\Delta_{IJ,\tau} \\
& = \dim V \Delta_{\phi \circ \tau} / p \Delta_{IJ, \tau} + \dim p\Delta_\tau / p \pi_2^*\Delta_{IJ,\tau} - \dim V \Delta_{\phi \circ \tau} / V \pi_2^*\Delta_{IJ,\phi \circ \tau} \\
& = \dim H^0(A_2 \otimes \fp , \Omega^1_{A_2 \otimes \overline{\bF}_p})  + \dim  \Delta_\tau /  \pi_2^*\Delta_{IJ,\tau} - \dim \Delta_{\phi \circ \tau} / \pi_2^*\Delta_{IJ, \phi \circ \tau}.
\end{split}
\]

\noindent The result follows.
\end{proof}

\begin{rem}
We note that the dimension formulae always make sense. Indeed, suppose for example that $\omega_{A_i,\tilde{\theta}}=0$, then we cannot have $\theta \notin T^1$ but $\phi \circ \theta \in T^1$ for which the formula would yield $\dim H^0(A_{IJ},\Omega^1_{A_{IJ}/\overline{\bF}_p})_{\tilde{\theta}} = -2$. To see this, simply note that $\omega_{A_i,\tilde{\theta}}=0$ implies that $\theta^1,\cdots, \theta^e \in \Sigma_\infty$, therefore by construction of $T'$, either all of $\theta^1,\cdots, \theta^e, (\phi \circ \theta)^1 \in T'$ or none are so that both $\theta$ and $\phi \circ \theta \in T^1$ or neither are. The case that $\dim \omega_{A_i,\tilde{\theta}} = 4e_\fp$ follows similarly.
\end{rem}

\begin{cor}\label{sheaves iso}
   We have the following isomorphisms:
    \begin{itemize}
        \item If $\tau \in \widehat{\Theta}_E$ is such that $\tau \vert_F = \theta \notin \widehat{\Theta}_{F,\fp}$ then 
        \[\ccH^1_{\dR}(A_2/S)^0_\tau \xrightarrow[\sim]{\psi_2^*} \ccH^1_{\dR}(A_{IJ}/S)^0_\tau \xrightarrow[\sim]{\pi_1^*} \ccH^1_{\dR}(A_1/S)^0_\tau.\] 
        \item If $\theta \in T^1$, then 
        \[ \ccH^1_{\dR}(A_2/S)^0_{\tilde{\theta}} \xrightarrow[\sim]{\psi_2^*} \ccH^1_{\dR}(A_{IJ}/S)^0_{\tilde{\theta}}
       \, \textrm{ and }\, 
        \ccH^1_{\dR}(A_{IJ}/S)^0_{\tilde{\theta}^c} \xrightarrow[\sim]{\pi_2^*} \ccH^1_{\dR}(A_2\otimes \fp/S)^0_{\tilde{\theta^c}},\]
        and 
        \[(\psi_2^*)^{-1} \omega^0_{IJ,\tilde{\theta}} = 
        \begin{cases}
            \omega^0_{2,\tilde{\theta}} & \textrm{ if } \phi \circ \theta \in T^1,\\
            (f^*)^{-1} \omega^0_{1,\tilde{\theta}} & \textrm{ if } \phi \circ \theta \notin T^1. 
        \end{cases}\]
        \item If $\theta \notin T^1$, then 
        \[ \ccH^1_{\dR}(A_{IJ}/S)^0_{\tilde{\theta}} \xrightarrow[\sim]{\pi_1^*} \ccH^1_{\dR}(A_1/S)^0_{\tilde{\theta}} \, \textrm{ and }\, \ccH^1_{\dR}(A_{IJ}/S)^0_{\tilde{\theta}^c} \xrightarrow{\sim} \ccH^1_{\dR}(A_1/S)^0_{\tilde{\theta}^c},\]
        and 
        \[\pi_1^* \omega^0_{IJ,\tilde{\theta}} = 
        \begin{cases}
            \omega^0_{1,\tilde{\theta}} & \textrm{ if } \phi \circ \theta \notin T^1,\\
            f^* \omega^0_{1,\tilde{\theta}} & \textrm{ if } \phi \circ \theta \in T^1. 
        \end{cases}\]
        
    \end{itemize}
\end{cor}

\begin{proof}
    The isomorphisms on de Rham bundles follow from proposition \ref{aij dieudonne calculations}, where the isomorphism $\ccH^1_{\dR}(A_{IJ}/S)^0_{\tilde{\theta}^c} \xrightarrow{\sim} \ccH^1_{\dR}(A_1/S)^0_{\tilde{\theta}^c}$ follows from the crystallization lemma \ref{crystallization} applied to $A_1 \xleftarrow{g\otimes \fp} A_2 \otimes \fp \xrightarrow{\pi_2} A_{IJ}$. 
    
    The relations on Hodge bundles are clear if $\theta,\phi \circ \theta \in T^1$ or $\theta,\phi \circ \theta \in T^1$. Suppose that $\theta \notin T^1, \phi \circ \theta \in T^1$ and let $\tau = \tilde{\theta}$. By compatibility of $f$, $\psi_2$ and $\pi_1$, we have the following commutative diagram:
    \[
    \begin{tikzcd}
        \ccH^1_{\dR}(A_2/S)^0_{\phi \circ \tau} \arrow[rr,"\psi_2^*" ,"\sim"'] \arrow[d,"V"'] && \ccH^1_{\dR}(A_{IJ}/S)^0_{\phi \circ \tau} \arrow[d,"V"]\\
        \ccH^1_{\dR}(A_2/S))^{0 (p)}_\tau \arrow[r,"f^*"] & \ccH^1_{\dR}(A_1/S))^{0 (p)}_\tau & \ccH^1_{\dR}(A_{IJ}/S))^{0 (p)}_\tau. \arrow[l, "\pi_1^*"',"\sim"]
    \end{tikzcd}
    \]
    from which we obtain $\pi^*_1 \omega_{IJ,\tau} = f^*\omega^0_{2,\tau}$. If $\theta \in T^1, \phi \circ \theta \notin T^1$, writing again $\tau = \tilde{\theta}$, we have the commutative diagram:
     \[
    \begin{tikzcd}
        \ccH^1_{\dR}(A_1/S)^0_{\phi \circ \tau}  \arrow[d,"V"'] && \ccH^1_{\dR}(A_{IJ}/S)^0_{\phi \circ \tau} \arrow[ll,"\pi_1^*"' ,"\sim"] \arrow[d,"V"]\\
        \ccH^1_{\dR}(A_1/S))^{0 (p)}_\tau  & \ccH^1_{\dR}(A_2/S))^{0 (p)}_\tau  \arrow[l,"f^*"']& \ccH^1_{\dR}(A_{IJ}/S))^{0 (p)}_\tau. \arrow[l, "(\psi_2^*)^{-1}"',"\sim"]
    \end{tikzcd}
    \]
    We deduce that, since $\theta \in T^1$, $\phi \circ \theta \notin T^1$, $\ker f^* \subset \omega^0_{2,\tau}$ and so $(\psi_2^*)^{-1}\omega^0_{IJ,\tau} = (f^*)^{-1}\omega^0_{1,\tau}$.
\end{proof}

We now define a Pappas-Rapoport filtration of type $\widetilde{\Sigma}_{IJ,\infty}$ on $\omega^0_{IJ,\tau}$ for each $\theta \in \widehat{\Theta}_F$, that is a Pappas-Rapoport filtration where the dimension of each graded piece is dictated by the number $s^\prime_\beta$ taken with respect to $\widetilde{\Sigma}_{IJ,\infty}$, defined in section \ref{sets}, as follows:

\noindent In the following, write $\tau = \tilde{\theta}$ for $\theta \in \widehat{\Theta}_F$ and fix the identification $\omega^0_{IJ,\tau} := \omega^0_{A_{IJ},\tau} = \omega^0_{A_{IJ} \otimes (\fq^c)^{-1},\tau}$ given by $\iota_{\fq^c}$. If $\theta \notin \widehat{\Theta}_{F,\fp}$, then we have isomorphisms \[\omega^0_{2,\tau} \xrightarrow{\psi_2^*} \omega^0_{IJ,\tau} \xrightarrow{\pi_1^*} \omega^0_{1,\tau}\]
where, by above, the composite $\pi_1^* \circ \psi_2^*$ is given by $f^*$ and thus compatible with the filtrations on the left and right hand side by definition of the moduli problem for $S^\prime_{\phi^\prime(I),J}$. We thus set for each $i$, 
\[\omega^0_{IJ,\tau}(i) =  \psi_2^* (\omega^0_{2,\tau}(i)) = (\pi_1^*)^{-1} (\omega^0_{1,\tau}(i)).\]

\noindent If $\theta \in \widehat{\Theta}_{F,\fp}$, we set for each $1 \leq i \leq e_\fp$
\begin{equation}\label{Aij filtering}
\omega^0_{IJ,\tau} (i) =
\begin{cases}
    \psi_2^* (\omega^0_{2,\tau}(i)) & \textrm{ if } \phi^\prime(\theta^i) \in T, \\
    (\pi_1^*)^{-1} (\omega^0_{1,\tau}(i)) & \textrm{ if } \phi^\prime(\theta^i) \notin T.
\end{cases}
\end{equation}

We now show that this yields a well defined filtration:

\noindent If $\theta \in T^1$, then $\psi_2^*: \ccH^1_{\dR}(A_2/S)^0_\tau \xrightarrow{\sim} \ccH^1_{\dR}(A_{IJ}/S)^0_\tau$ is an isomorphism by corollary \ref{sheaves iso} and we can thus identify $\pi_1^*$ with $f^*$. 
Therefore, the subsheaves $\psi_2^* (\omega^0_{2,\tau}(i))$ are given, in $\ccH^1_{\dR}(A_2/S)^0_\tau$, by $\omega^0_{2,\tau}(i)$, which are locally of dimension $s_\tau(i) = s^\prime_\tau(i)$, where $s^\prime_\tau$ is the dimension function attached to $\widetilde{\Sigma}_{IJ,\infty}$.

The subsheaves $(\pi_1^*)^{-1} (\omega^0_{1,\tau}(i))$ are given, in $\ccH^1_{\dR}(A_2/S)^0_\tau$, by $(f^*)^{-1}(\omega^0_{1,\tau}(i))$, which are locally free of dimension $s_\tau(i)+1 = s'_\tau(i)$. Consider an $i$ such that $\phi^\prime(\theta^i) \notin T$. If $\phi^\prime(\theta^i)=\theta^{j}$ for $j \geq i+1$, then this means that $(f^*)^{-1}(\omega^0_{1,\tau}(i)) \subset \omega^0_{2,\tau}(j) \subset \omega^0_{IJ,\tau}$. If however $\phi^\prime(\theta^i) \neq \theta^j$ for any $j \leq e_\fp$, this means that $\phi \circ \theta \notin T^1$ and $\omega^0_{IJ,\tau} = (f^*)^{-1}(\omega^0_{1,\tau}) $ by corollary \ref{sheaves iso}. Therefore the subsheaves $(f^*)^{-1}(\omega^0_{1,\tau}(i))$ are actually subsheaves of $\omega^0_{IJ,\tau}$. If $e_\fp = 1$, we are now done.

\noindent Suppose now that $e_\fp > 1$ and $1 \leq i < e_\fp $ is such that $\phi^\prime(\theta^i) \in T$ and $\phi^\prime(\theta^{i+1}) \notin T$. Then in fact $\theta^{i+1} \in J$ and 
\[[\varpi_\fp] \omega^0_{IJ,\tau}(i+1) = [\varpi_\fp] (f^*)^{-1}(\omega^0_{1,\tau})(i+1) = g^*(\omega^0_{A_1\otimes \fp^{-1},\tau}(i+1)) = \omega^0_{2,\tau}(i) = \omega^0_{IJ,\tau}(i).\]

If however $i$ is such that $\phi^\prime(\theta^i) \notin T$ and $\phi^\prime(\theta^{i+1}) \in T$. Then in fact $\theta^{i+1} \in I$ and
\[\omega^0_{IJ,\tau}(i+1) = \omega^0_{2,\tau}(i+1) = (f^*)^{-1}(\omega^0_{1,\tau})(i)=\omega^0_{IJ,\tau}(i).\]

The other inclusions are clear and it follows that the above construction yields a well defined filtration. The case that $\theta \notin T^1$ is entirely analogous.

By definition of the filtrations we have:
\begin{cor}\label{local de rham iso}
Let $\beta = \theta^i \in \widehat{\Theta}_{F,\fp}$. Let $\ccH^1_{IJ,\tilde{\beta}} = [\varpi_\fp]^{-1}\omega^0_{IJ,\tau}(i-1) / \omega^0_{IJ,\tau}(i-1)$ be the de Rham cohomology sheaf of $A_{IJ}$ at $\tilde{\beta}$ with respect to $\widetilde{\Sigma}_{IJ,\infty}$. If $\beta \in T^\prime$, then we have an isomorphism
\[\ccH^1_{2,\tilde{\beta}} \xrightarrow[\sim]{\psi_2^*} \ccH^1_{IJ,\tilde{\beta}}.\]
If $\beta \notin T^\prime$, then we have an isomorphism
\[\ccH^1_{IJ,\tilde{\beta}} \xrightarrow[\sim]{\pi_1^*} \ccH^1_{1,\tilde{\beta}}.\]
\end{cor}
We recall that $\beta \in T^\prime$ means either $\beta \in T$ or $\beta \in \Sigma_\infty$ and $\phi^\prime(\beta) \in T$.

\subsubsection{The morphism to a projective bundle}\label{the splice morphism}
We now define, as in \cite[section 5.2.1]{2020arXiv200100530D} the extra algebraic data needed to define the Shimura variety we will relate $S^\prime_{\phi^\prime(I),J}$ to and the morphism to it.

Let $\Sigma_{IJ} = \Sigma \sqcup \Sigma_{IJ}^+$ and $\widetilde{\Sigma}_{IJ,\infty} = \widetilde{\Sigma}_\infty \sqcup \widetilde{\Sigma}_{IJ}^+$ where 
\[\widetilde{\Sigma}_{IJ}^+ = \{ \tilde{\beta}^c \, \vert \, \beta \in T, \phi^\prime(\beta) \notin T \} \cup \{ \tilde{\beta} \, \vert \, \beta \notin T, \phi^\prime(\beta) \in T \}.\]

Write $B_{IJ} = B_{\Sigma_{IJ}}$ for the quaternion algebra over $F$ ramified exactly at the places in $\Sigma_{IJ}$, and let $G_{IJ}= G_{\Sigma_{IJ}}/\bQ$ be the algebraic group given by $G_{IJ}(R) = (B_{IJ} \otimes_\bQ R)^\times$, $G^\prime_{IJ} = G^\prime_{\Sigma_{IJ}} = (G_{IJ} \times T_E)/ T_F$, and $D_{IJ} = D_{\Sigma_{IJ}} = B_{IJ} \otimes_F E$.

Choose an $\mathbb{A}_{F,f}$-algebra isomorphism 
\[ \xi: B_{IJ} \otimes \widehat{\bZ} \xrightarrow{\sim} B \otimes \widehat{\bZ} \supset \ccO_B,\]
where $\ccO_B$ is the maximal order defined in section \ref{quaternionic models}, and set $\ccO_{B_{IJ}}$ to be the corresponding maximal order in $B_{IJ}$. We thus set $\ccO_{D_{IJ}} = \ccO_{B_{IJ}} \otimes_{\ccO_F} \ccO_E$, which is a maximal order in $D_{IJ}$, and fix an $\ccO_E$-algebra isomorphism 
\[ \vartheta: \ccO_{D_{IJ}} \xrightarrow{\sim} \ccO_D.\]
By lemma 5.4 of \cite{tian_xiao_2016}, we may choose an element $\delta_{IJ} \in D^\times_{IJ}$ satisfying conditions of \ref{unitary moduli} for $D_{IJ}$ and choice of lift $\widetilde{\Sigma}_{IJ,\infty}$, such that the resulting anti-involution $*$ on $D_{IJ}$ given by $\delta^{-1}_{IJ} \overline{\vartheta(\alpha)}\delta_{IJ}$ is compatible with $\vartheta$. The isomorphisms $\xi$ and $\vartheta$ induce two $\widehat{\ccO}_E$-algebra isomorphisms
\[ \xi_E, \hat{\vartheta}: \widehat{\ccO}_{D_{IJ}} \xrightarrow{\sim} \widehat{\ccO}_D,\]
so that there exists $h \in \widehat{\ccO}_D^\times$ such that $h \xi_E h^{-1} = \hat{\vartheta}$ and given that $\xi_E(\overline{u})  = \overline{\xi_E(u)}$ for all $u$, one can check that
\[\varepsilon := \overline{h} \delta h \xi_E(\delta_{IJ}^{-1}) \in \textrm{GL}_2(\mathbb{A}_{E,f})\]
is central and satisfies $\overline{\varepsilon} = \varepsilon$. Therefore $\varepsilon \in \mathbb{A}^\times_{F,f}$ and $\varepsilon_p \in \ccO_{F,p}^\times$.

For an open compact subgroup $U^\prime \subset G^\prime_{\Sigma}(\mathbb{A}_f)$, we let $U^\prime_{IJ}$ be the corresponding open compact subgroup of $G^\prime_{IJ}(\mathbb{A}_f)$ under the isomorphism induced by $\xi$. We suppose that $U^\prime$ is sufficiently small such that $U^\prime$ and $U^\prime_{IJ}$ are sufficiently small in the usual sense. Since $U^\prime$ is of level prime to $p$, so is $U^\prime_{IJ}$.

Consider the universal $IJ$-splice $A_{IJ}$ defined over the stratum $S^\prime_{\phi^\prime(I),J}$ of $\widetilde{Y}^\prime_{U^\prime_0(\fp)}(G^\prime_\Sigma)$. We now endow it with the extra structures such that we obtain a morphism \[\widetilde{\psi}^\prime_{IJ}: S^\prime_{\phi(I),J} \to \widetilde{Y}^\prime_{U^\prime_{IJ}}(G^\prime_{IJ}).\]
Let $\iota_{IJ}$ be the action of $\ccO_{D_{IJ}}$ obtained by composing the action of $\ccO_D$ on $A_{IJ}$ with $\vartheta$. Recall that we fixed an isomorphism $\alpha:\ccO_D \otimes \mathbb{F}_p \xrightarrow{\sim} M_2(\ccO_E/p\ccO_E)$ and considered the element $e_0 = \left(\begin{smallmatrix}1 & 0 \\0 &0 \end{smallmatrix}\right) \in \ccO_D \otimes \mathbb{F}_p$. We may extend this via $\xi_E \otimes \bF_p$ to an isomorphism $\ccO_{D_{IJ}} \otimes \mathbb{F}_p = M_2(\ccO_E/p\ccO_E)$. Then we can consider the reduced de Rham cohomology sheaf $\ccH^1_{\dR}(A_{IJ}/S^{\prime}_{\phi^\prime(I),J})^0$ of $A_{IJ}$ with respect to $\Sigma$ or with respect to $\Sigma_{IJ}$. Recall that, in the case of $\Sigma$, this is defined as $\ccH^1_{\dR}(A_{IJ}/S^{\prime}_{\phi^\prime(I),J}) \cdot \alpha^{-1}(e_0)$. Similarly, with respect to $\Sigma_{IJ}$, this is given via the above isomorphisms: 
\[\begin{split}
\ccH^1_{\dR}(A_{IJ}/S^{\prime}_{\phi^\prime(I),J})\cdot \xi^{-1}(\alpha^{-1}(e_0)) & = \ccH^1_{\dR}(A_{IJ}/S^{\prime}_{\phi^\prime(I),J}) \cdot \vartheta(\xi^{-1}(\alpha^{-1}(e_0)))\\
& = \ccH^1_{\dR}(A_{IJ}/S^{\prime}_{\phi^\prime(I),J}) \cdot h_p\alpha^{-1}(e_0)h^{-1}_p \\
& = \ccH^1_{\dR}(A_{IJ}/S^{\prime}_{\phi^\prime(I),J}) \cdot \alpha^{-1}(e_0)h^{-1}_p.
\end{split}\]
In particular, we obtain $\ccO_E/p\ccO_E$-linear isomorphisms $\cdot h_p^{-1}$
\[ \omega^0_{IJ, \Sigma,\tau} \xrightarrow{\sim} \omega^0_{IJ, \Sigma_{IJ},\tau},\]
where the ramification sets are added to make clear which reduced sheaf we're taking. Then we give $\omega^0_{IJ}$ the Pappas-Rapoport filtration of type $\widetilde{\Sigma}_{IJ,\infty}$, constructed after corollary \ref{sheaves iso}, under the above isomorphism. The Rosati involution induced by $\lambda_{IJ}$ is compatible with the anti-involution $*$ of $\ccO_{D_{IJ}}$ since $\vartheta$ is compatible with the anti-involutions. Finally, we define $\eta_{IJ}$ as the composite 
\[ \widehat{\ccO}^{(p)}_{D_{IJ}} \xrightarrow{\hat{\xi}_E^{(p)}} \widehat{\ccO}^{(p)}_D \xrightarrow{h^{(p)} \cdot} \widehat{\ccO}^{(p)}_D \xrightarrow{\eta_2, \overline{s}_i} \widehat{T}^{(p)}(A_{2,\overline{s}_i}) \xrightarrow{\sim} \widehat{T}^{(p)}(A_{2,\overline{s}_i} \otimes \fp) \xrightarrow{\pi_2} \widehat{T}^{(p)}(A_{IJ,\overline{s}_i})\]
for each $\overline{s}_i$ and set $\epsilon_{IJ} = \varepsilon^{(p)}\epsilon_2$. The fact that $(\eta_{IJ}, \epsilon_{IJ})$ is a $U^\prime_{IJ}$-level structure on $(A_{IJ},\iota_{IJ},\lambda_{IJ})$ follows from the relations between $\vartheta,\xi,h$ and $\varepsilon$. Therefore the tuple $(A_{IJ},\iota_{IJ},\lambda_{IJ}, (\eta_{IJ},\epsilon_{IJ}), \underline{\omega^0_{IJ}})$ defines a morphism $\widetilde{\psi}^\prime_{IJ}: S^\prime_{\phi(I),J} \to \widetilde{Y}_{U^\prime_{IJ}}(G^\prime_{IJ})_{\bF}$.

Finally, writing $S^\prime_{IJ} = \widetilde{Y}_{U^\prime_{IJ}}(G^\prime_{IJ})_\bF$, We define for each $\beta \in R$, morphisms $S^{\prime}_{\phi^\prime(I),J} \to \mathbb{P}^1_{S^\prime_{IJ}}(\mathcal{V}_{\tilde{\beta}})$ where the rank two vector bundles $\mathcal{V}_{\tilde{\beta}}$ are given by $\ccH^1_{A,\tilde{\beta}}$, where $\underline{A}$ is the universal abelian variety on $S^\prime_{IJ}$. Defining such a morphism is equivalent to giving a surjection from $\widetilde{\psi}^{  \prime *}_{IJ} \ccH^1_{\tilde{\beta}}$ to a line bundle on $S^\prime_{\phi^\prime(I),J}$. However this sheaf is canonically isomorphic to  $\ccH^1_{IJ,\tilde{\beta}}$. 

By corollary \ref{local de rham iso}, if $\beta \in R \cap T$, we have the isomorphism $\psi_2^* \circ h^{-1}_p: \ccH^1_{2,\tilde{\beta}} \xrightarrow{\sim} \ccH^1_{IJ,\tilde{\beta}}$. Under this isomorphism, we consider the morphism $S^{\prime}_{\phi^\prime(I),J} \to \mathbb{P}^1_{S^\prime_{IJ}}(\mathcal{V}_{\tilde{\beta}})$ given by the short exact sequence:
\[0 \to \omega^0_{2,\tilde{\beta}} \to \ccH^1_{2,\tilde{\beta}} \to v^0_{2,\tilde{\beta}} \to 0.\]

If $\beta \in R$, $\beta \notin T$, we have the isomorphism $\pi_1^* \circ h^{-1}_p: \ccH^1_{IJ,\tilde{\beta}} \xrightarrow{\sim} \ccH^1_{1,\tilde{\beta}} $ and so we consider, under this isomorphism, the morphism $S^{\prime}_{\phi^\prime(I),J} \to \mathbb{P}^1_{S^\prime_{IJ}}(\mathcal{V}_{\tilde{\beta}})$ given by the short exact sequence:
\[0 \to \omega^0_{1,\tilde{\beta}} \to \ccH^1_{1,\tilde{\beta}} \to v^0_{1,\tilde{\beta}} \to 0.\]

Taking the fiber product over $S^\prime_{IJ}$ of all these morphisms we just defined, we obtain a morphism
\begin{equation}\label{Morphism Psi} \widetilde{\Psi}^\prime_{IJ}: S^\prime_{\phi^\prime(I),J} \to \prod_{\beta \in R} \mathbb{P}^1_{S^\prime_{IJ}}(\mathcal{V}_{\tilde{\beta}}).\end{equation}

Let us finally give an example: Let $[F:\bQ]=12$ and $p$ be a prime such that $p\ccO_F = \fp^4$ (hence $f_\fp=3,e_\fp=4$). Write $\{\theta_1,\theta_2,\theta_3 \} = \widehat{\Theta}_F$, $\tau_i = \tilde{\theta}_i$ and let $\Sigma$ be an even set of places such that $\widetilde{\Sigma}_\infty = \{(\tau_1^1)^c,(\tau_2^2)^c,\tau_3^3\} $. Consider the subset $J= \{ \theta_1^2, \theta^1_2, \theta^4_2, \theta_3^4\}$ of $\Theta_F \setminus \Sigma_\infty$ and put $I= J^c$.

This gives us $T= I^c =J$ and $T' = T \sqcup \{\theta_1^1, \theta_3^3 \}$ so $T^1 = \{\theta_1,\theta_2\}$. We also have $\widetilde{\Sigma}^+_{IJ} = \{(\tau_1^2)^c,\tau_1^4, (\tau_2^1)^c, \tau_2^3 ,(\tau_2^4)^c,\tau_3^2 \}$. The full filtrations of $\omega^0_{IJ}$ are then given by 
\begin{align*}
0 \subset \omega^0_{2,\tau_1}(1) \subset (f^*)^{-1}\omega^0_{1,\tau_1}(2) \subset (f^*)^{-1}\omega^0_{1,\tau_1}(3) \subset \omega^0_{2,\tau_1}(4) = \omega^0_{IJ,\tau_1}, \\
0 \subset (f^*)^{-1}\omega^0_{1,\tau_1}(1) \subset (f^*)^{-1}\omega^0_{1,\tau_1}(2) \subset \omega^0_{2,\tau_2}(3) \subset (f^*)^{-1}\omega^0_{1,\tau_2} = \omega^0_{IJ,\tau_2},\\
0 \subset \omega^0_{1,\tau_3}(1) \subset f^*\omega^0_{2,\tau_3}(2) \subset f^*\omega^0_{2,\tau_3}(3) \subset f^*\omega^0_{2,\tau_3}(4) = \omega^0_{IJ,\tau_3}. 
\end{align*}

\subsection{Splitting}\label{splitting}
In this section, we expand on the technique applied in section \ref{partial frobenius}, to construct the Essential Frobenius isogeny, to construct an explicit inverse to the morphism $\widetilde{\Psi}^\prime_{IJ}$ (\ref{Morphism Psi}). We split this into two parts, we first construct an abelian scheme $A_2$, then an abelian scheme $A_1$, and we will show that there is a natural isogeny $f: A_1 \to A_2$. The details of the construction of $A_1$ will be extremely similar to that of $A_2$; we will therefore be more brief in describing its construction. We will make heavy use of essential Frobenius and essential Verschiebung, we thus refer to sections \ref{Fesves} and \ref{unitary hasse} for notation, particularly for the morphisms $F_{\es, \tilde{\beta}}^{\phi^{n}(\tilde{\beta})}$ and $V_{\es,\tilde{\beta}}^{\phi^{-n}(\tilde{\beta})}$; from now, when writing a power $(p^n)$ on the source or target of these morphisms, $n$ will always be the only choice that makes sense.

\subsubsection{Constructing $A_2$}\label{constructing A2}
Keep the notation from the previous section. Let $Y = \widetilde{Y}_{U'}(G'_{\Sigma_{IJ}})_\bF$ and $X= \prod_{\beta \in R} \mathbb{P}^1_Y(\ccH^1_{\tilde{\beta}})$. For any $\beta \in R$ and $n \in \bZ$, we have the natural projection $\pi_\beta: X \to \mathbb{P}^1_Y(\ccH^1_{\tilde{\beta}})$ and we write $\ccO(n)_\beta := \pi_\beta^*\ccO(n)$. For any $\beta \in R$, write $L_{\tilde{\beta}}$ for the kernel of the projection $\ccH^1_{\tilde{\beta}} \to \ccO(1)_\beta$. 

We define $\phi^{\prime \prime}$ to be the shift structure on $\Theta_{F,\fp}$ induced by $\Sigma_{\infty} \sqcup I \cap J$. That is, for any $\beta \in \Theta_{F,\fp}$, $\phi^{\prime \prime}(\beta) = \phi^n(\beta)$, where $n $ is chosen smallest such that $\phi^n(\beta) \notin \Sigma_\infty \sqcup I \cap J$. We define $(\phi^{\prime \prime})^{-1}$ similarly.

Let $A$ be the universal abelian variety over $X$ with its $\ccO_{D_{IJ}}$-action and from now on, for any $\theta \in \widehat{\Theta}_F$, write $\tau = \tilde{\theta}$ and $s^\prime_\beta$ for the counting function with respect to $\widetilde{\Sigma}_\infty$. We now construct a subgroup $M$ of $A[\fq]$. We start by defining a partial Raynaud datum; for each $\theta \in \widehat{\Theta}_{F,\fp}$ we define a sheaf $\mathcal{M}_\theta$, locally free over $X$, as follows: 
\begin{itemize}

\item If $\theta$ is such that $\phi \circ \theta \in T^1$, then set $\mathcal{M}_\theta = 0$.
 
\item Suppose now that $\phi \circ \theta \notin T^1$ and $\beta := (\phi^{ \prime \prime })^{-1}((\phi \circ \theta)^{1}) \notin R$. 
If $\beta = \theta^{e_\fp}$, set $\mathcal{M}_\theta = \omega^0_{\tau^{e_\fp}}$. 

\item Otherwise, if $\beta := (\phi^{ \prime \prime })^{-1}((\phi \circ \theta)^{1}) \notin R$ and $\beta \neq \theta^{e_\fp}$, we have the line bundle $ F_{\textrm{es}, \tilde{\beta}}^{\tau^{e_\fp}}(\ccH^{1 \, \pn }_{\tilde{ \beta}}) \subset \ccH^1_{\tau^{e_\fp}}$. We then set $\mathcal{M}_\theta = \ccH^1_{\tau^{e_\fp}} / F_{\textrm{es}, \tilde{\beta}}^{\tau^{e_\fp}}(\ccH^{1 \, \pn}_{\tilde{ \beta}})$ which is also a line bundle. 

\item If $\beta = (\phi^{ \prime \prime })^{-1}((\phi \circ \theta)^{1}) \in R$, by our abuse of notation described above, we have the short exact sequence 
\[
    0 \to L_{\tilde{\beta}} \to \ccH^1_{\tilde{\beta}} \to \ccO(1)_\beta \to 0
\]

\noindent and the isomorphism $F_{\textrm{es},\tilde{\beta}}^{\tau^{e_\fp}}: \ccH^{1 \, \pn }_{\tilde{\beta}} \xrightarrow{\sim} \ccH^1_{\tau^{e_\fp}}$. We thus set 
\[\mathcal{M}_\theta = F_{\textrm{es},\tilde{\beta}}^{\tau^{e_\fp}} (\ccO(p^n)_\beta) = \ccH^1_{\tau^{e_\fp}} / F_{\textrm{es},\tilde{\beta}}^{\tau^{e_\fp}}(L^{p^n}_{\tilde{\beta}}).\]
\end{itemize}
For each $\theta \in \widehat{\Theta}_{F,\fp}$, define maps $s_\theta: \mathcal{M}_\theta^p \to \mathcal{M}_{\phi \circ \theta}$ as follows: 

\begin{itemize}
\item If either $\phi \circ \theta \in T^1$ or $\phi^2 \circ \theta \in T^1$, then either $\mathcal{M}_\theta = 0$ or $\mathcal{M}_{\phi \circ \theta} = 0$ respectively. We thus set $s_\theta = 0$. 

\item Suppose now that $\phi \circ \theta, \phi^2 \circ \theta \notin T^1$. If $\omega^0_{\phi \circ \tau} \neq 0$, we set $s_\theta =0$. 

\item If $\omega^0_{\phi \circ \tau} = 0$, then $(\phi \circ \theta)^1, \cdots , (\phi \circ \theta)^{e_\fp} \in \Sigma_\infty$ and $\mathcal{M}_{\phi \circ \theta} = \ccH^1_{(\phi \circ \tau)^1} /F_{\es,\tilde{\beta}}^{(\phi \circ \tau)^1}(L^{\prime \, p^{n+1}}_\beta)$ where $\beta = (\phi^{\prime \prime})^{-1}((\phi^2 \circ \theta)^1) = (\phi^{\prime \prime})^{-1}((\phi \circ \theta)^1) $ and $L^\prime_\beta$ depends on whether $\beta \in R$ or not. $F_{\es,(\phi \circ \tau)^1}^{(\phi \circ \tau)^{e_\fp}}: \ccH^1_{(\phi \circ \tau)^1} \to \ccH^1_{(\phi \circ \tau)^{e_\fp}} = \ccH^1_{(\phi \circ \tau)^1}$   is the identity.

\noindent If $\beta = \theta^{e_\fp} \notin R$, then $\mathcal{M}_\theta^{(p)} = \omega^{0 (p)}_{\tau^{e_\fp}} \simeq \ccH^{1}_{(\phi \circ \tau)^1}/ F_{\textrm{es}, (\phi \circ \tau)^1}(\ccH^{1 \, (p)}_{\tau^{e_\fp}}) = \mathcal{M}_{\phi \circ \theta}$  via the short exact sequence 
\[0 \to F_{\textrm{es}, (\phi \circ \tau)^1}(\ccH^{1 \, (p)}_{\tau^{e_\fp}}) \to \ccH^{1}_{(\phi \circ \tau)^1} \xrightarrow{V_{\textrm{es}, (\phi \circ \tau)^1}} \omega^{0 (p)}_{\tau^{e_\fp}} \to 0.\]
We thus set $s_\theta$ to be the inverse of this isomorphism.

\item Otherwise, 
\[\mathcal{M}_{\phi \circ \theta} = 
\ccH^1_{\phi \circ \tau^{e_\fp}} / F_{\textrm{es},\tilde{\beta}}^{(\phi \circ \tau)^{e_\fp}}(L^{ (p^{n+1})}_{\tilde{\beta}}) = 
F_{\textrm{es}, \tau^{e_\fp}}^{(\phi \circ \tau)^{e_\fp}}(\ccH^{1 (p)}_{\tau^{e_\fp}}/F_{\textrm{es},\tilde{\beta}}^{\tau^{e_\fp}}(L^{(p^n)}_{\tilde{\beta}})) = 
F_{\textrm{es}, \tau^{e_\fp}}^{(\phi \circ \tau)^{e_\fp}} (\mathcal{M}_\theta^{(p)}),\]
We thus set $s_\theta = F_{\textrm{es}, \tau^{e_\fp}}^{(\phi \circ \tau)^{e_\fp}}$.
\end{itemize}

\noindent For each $\theta \in \widehat{\Theta}_{F,\fp}$, define maps $t_\theta: \mathcal{M}_{\phi \circ \theta} \to \mathcal{M}_\theta^p$ as follows:

\begin{itemize}
\item  If either $\phi \circ \theta \in T^1$ or $\phi^2 \circ \theta \in T^1$, then either $\mathcal{M}_\theta = 0$ or $\mathcal{M}_{\phi \circ \theta} = 0$. We thus set $t_\theta = 0$. 

\item Suppose now that $\phi \circ \theta, \phi^2 \circ \theta \notin T^1$. If there exists an $i$ such that $s'_{(\phi \circ \tau)^i}=0$, set $t_\theta = 0 $. 

\item Otherwise, set $t_\theta$ to be the morphism induced by the composition $V_{\textrm{es},(\phi \circ \tau)^{e_\fp}}^{\tau^{e_\fp}}: \ccH^1_{(\phi \circ \tau)^{e_\fp}} \to \ccH^{1 (p)}_{\tau^{e_\fp}}$.
\end{itemize}
Note that $V_{\textrm{es},(\phi \circ \tau)^{e_\fp}}^{\tau^{e_\fp}}$, is given by the genuine Verschiebung (followed by the quotient map) if $s'_{\tau^{e_\fp}} \neq 0$, and the composition $\ccH^1_{(\phi \circ \tau)^{e_\fp}} \xrightarrow{\cdot [\varpi_\fp]^{e_\fp-1}} \ccH^1_{(\phi \circ \tau)^1} \xrightarrow{F^{-1}} \ccH^1_{\tau^{e_\fp}}$ otherwise. If $(\phi^{\prime \prime})^{-1}((\phi^2 \circ \theta)^1) = (\phi^{\prime \prime})^{-1}((\phi \circ \theta)^1) =\beta $, then similarly as before $\mathcal{M}_{\phi \circ \theta} = F_{\textrm{es},\tau^{e_\fp}}^{(\phi \circ \tau)^{e_\fp}}(\mathcal{M}_\theta^{(p)})$ so that $V_{\textrm{es},(\phi \circ \tau)^{e_\fp}}^{\tau^{e_\fp}} = (F_{\textrm{es},\tau^{e_\fp}}^{(\phi \circ \tau)^{e_\fp}})^{-1}: \mathcal{M}_{\phi \circ \theta} \to \mathcal{M}_\theta^{(p)}$ is well defined, and an isomorphism. Otherwise, $\mathcal{M}_{\phi \circ \theta} = \omega^0_{\phi \circ \tau} / L'_{\beta^\prime}$ for some $L'_{\beta^\prime} \subset [\varpi_\fp] \ccH^1_{\textrm{dR}}(A/X)^0_{\phi \circ \tau}$. It follows that $V_{\textrm{es},(\phi \circ \tau)^{e_\fp}}^{\tau^{e_\fp}}$ kills $L^\prime_{\beta^\prime}$ and so $t_\theta$ is well defined.

The sheaves $\mathcal{M}_\theta$ come with a natural action of $\ccO_E/\fq$ and it follows by the above constructions that $(\mathcal{M}_\theta,s_\theta,t_\theta)$ is a partial Raynaud datum for $\phi^{-1}(T^1)^c$. We can thus form the group scheme $M^{\prime \prime}/X$ attached to it, as defined in section \ref{partial raynaud}. We recall that $M^{\prime \prime} = \underline{\textrm{Spec}}_S (( \textrm{Sym}_{\ccO_S} \mathcal{M}) / \mathcal{I})$ where $\mathcal{M}= \bigoplus \mathcal{M}_\theta$, $\mathcal{I}$ is the ideal generated by $(s_\theta-1)\mathcal{M}_\theta^p$ and that, by lemma \ref{Raynaud crystal}, there is a canonical isomorphism of crystals $\mathbb{D}(M^{\prime \prime}) \xrightarrow{\sim} \Pi^*\mathcal{M}^{\prime \prime}$ where Frobenius is given by $\Pi^*(s)$ and Verschiebung by $\Pi^*(t)$.

We now show that we can realize $M^\prime = M^{\prime \prime} \otimes_{\ccO_E} \ccO_E^2$ as a subgroup of $A[\fq]$. Since $X$ is smooth, we need to show that we are in the situation of lemma \ref{subgroup sesame}.

Let $\theta \notin T^1$ and $\beta$ be as above. If $\beta$ is such that either $\beta \neq \theta^{e_\fp}$ or $\beta = \theta^{e_\fp} \in R$, then we have the surjection 
\[ \ccH^1_{\textrm{dR}} (A/X)^0_{\phi \circ \tau} \xrightarrow{[\varpi_\fp]^{e_\fp-1}} \ccH^1_{(\phi \circ \tau)^1} \xrightarrow[\sim]{V_{\textrm{es},(\phi \circ \tau)^1}} \ccH^{1 (p)}_{\tau^{e_\fp}} \to (\ccH^1_{\tau^{e_\fp}} / F_{\textrm{es},\tilde{\beta}}^{\tau^{e_\fp}}(L_\beta^{\prime (p^n)}))^{(p)} = \Phi^* \mathcal{M}_\theta, \]
where $L^\prime_\beta$ is $F_{\textrm{es},\phi \circ \tilde{\beta}}(\ccH^1_{\tilde{\beta}})$ if $\beta \notin R$ or $L_{\tilde{\beta}}$ if $\beta \in R$. Since the inverse of $V_{\textrm{es},(\phi \circ \tau)^1}$ is $F_{\textrm{es},(\phi \circ \tau)^1}$, we see that the kernel of this morphism is equal to $[\varpi_\fp]^{1-e_\fp} F_{\textrm{es},\tilde{\beta}}^{(\phi \circ \tau)^1}(L^{\prime (p^{n+1})}_\beta)$.
If $\beta = \theta^{e_\fp} \notin R$ however, then $\mathcal{M}_\theta = \omega^0_{\tau^{e_\fp}} = \omega^0_\tau / \omega^0_\tau(e_\fp-1)$ and we have the surjection 
\[ \ccH^1_{\textrm{dR}} (A/X)^0_{\phi \circ \tau} \xrightarrow{V} (\omega^0_\tau)^{(p)} \twoheadrightarrow (\omega^{0}_{\tau^{e_\fp}})^p = \Phi^*\mathcal{M}_\theta. \]
The kernel of this morphism is $V^{-1}(\omega^0_\tau(e_\fp-1))$, however since \[[\varpi_\fp]^{e_\fp-1}V(V^{-1}(\omega^0_\tau(e_\fp-1))) = [\varpi_\fp]^{e_\fp-1}\omega^0_\tau(e_\fp-1) = 0,\]
we have $V^{-1}(\omega^0_\tau(e_\fp-1)) \subset [\varpi_\fp]^{1-e_\fp}\ker V_{\textrm{es},(\phi \circ \tau)^1} = [\varpi_\fp]^{1-e_\fp} F_{\textrm{es},(\phi \circ \tau)^1}(\ccH^{1 \, (p)}_{\tau^{e_\fp}})$ and this inclusion is an equality by dimension count.

Therefore we obtain a short exact sequence of $\ccO_{D_{IJ}} \otimes \ccO_S$ modules,
\[0 \to \ker \to \ccH^1_{\textrm{dR}}(A/X)^0 \to \Phi^*\mathcal{M} \to 0.\]

\noindent Furthermore, on isotypic components, we have:

\begin{equation}\label{A_2 de rham components}
   \ker_{\phi \circ \tau} =
    \begin{cases}
    \ccH^1_{\textrm{dR}}(A/X)^0_{\phi \circ \tau} & \textrm{if } \phi \circ \tau \in T^1 \textrm{ or } \phi \circ \tau \notin \widehat{\Theta}_{E,\fp},\\
    \ker_{\phi \circ \tau} = [\varpi_\fp]^{1-e_\fp} F_{\textrm{es},\tilde{\beta}}^{(\phi \circ \tau)^1}(\ccH^1_{\tilde{\beta}}) & \textrm{if } \phi \circ \theta \notin T^1, \beta \notin R,\\
    \ker_{\phi \circ \tau} = [\varpi_\fp]^{1-e_\fp} F_{\textrm{es},\tilde{\beta}}^{(\phi \circ \tau)^1}(L_{\tilde{\beta}}) & \textrm{if } \phi \circ \theta \notin T^1, \beta \in R.
    \end{cases}
\end{equation}

We now show that this short exact sequence is compatible with the respective operators $\Phi$ and $V$. If either $\phi \circ \theta\in T^1$ or $\phi^2 \circ \theta \in T^1$, then $\mathcal{M}_\theta =0$ or $\mathcal{M}_{\phi \circ \theta}= 0$ and there is nothing to do. Let $\tau = \tilde{\theta} \in \widehat{\Theta}_{E,\fq}$ be such that $\theta, \phi \circ \theta \notin T^1$, consider the following diagrams:

\begin{center}
\begin{tikzcd}[column sep = large, row sep = large]
\ccH^{1 (p)}_{(\phi \circ \tau)^1} \arrow[r, "{ V_{\textrm{es},(\phi \circ \tau)^1}}"] \arrow[d, "F"'] & \Phi^* (\mathcal{M}^{(p)}_\theta) \arrow[d, "\Phi^*(s_\theta)"] &  & \ccH^{1}_{(\phi^2 \circ \tau)^1} \arrow[d, "V"'] \arrow[r, "{ V_{\textrm{es},(\phi^2 \circ \tau)^1}}"] & \Phi^* \mathcal{M}_{\phi \circ \theta} \arrow[d, "\Phi^*(t_\theta)"] \\
\ccH^{1}_{(\phi^2 \circ \tau)^1} \arrow[r, "{ V_{\textrm{es},(\phi^2 \circ \tau)^1}}"']       & \Phi^* \mathcal{M}_{\phi \circ \theta}                          &  & \ccH^{1 (p)}_{(\phi \circ \tau)^1} \arrow[r, "{ V_{\textrm{es},(\phi \circ \tau)^1}}"']                         & \Phi^* (\mathcal{M}^{(p)}_\theta).                                   
\end{tikzcd}
\end{center}

If $\omega^0_{\phi \circ \tau} \neq 0$, then either $\ccH^{1 (p)}_{(\phi \circ \tau)^1} \subset \omega^{0 (p)}_{\phi \circ \tau}$ so that $F$ kills it 
or there exists $i$, which we take maximal, such that $s'_{(\phi \circ \tau)^i} = 1$. If $i = e_\fp$, then $\mathcal{M}_{\phi \circ \theta} = \omega^0_{(\phi \circ \tau)^{e_\fp}}$, 
otherwise $\mathcal{M}_{\phi \circ \theta} = \ccH^1_{(\phi \circ \tau)^{e_\fp}} / F_{\textrm{es},(\phi \circ \tau)^i}^{(\phi \circ \tau)^{e_\fp}}(\ccH^1_{(\phi \circ \tau)^i})$. 
If $s'_{(\phi \circ \tau)^{e_\fp}} \neq 0$, the composition $\ccH^{1 (p)}_{(\phi \circ \tau)^1} \xrightarrow{F} \ccH^1_{(\phi^2 \circ \tau)^1} \xrightarrow{V_{\textrm{es},(\phi^2 \circ \tau)^1}} \ccH^1_{(\phi \circ \tau)^{e_\fp}} $ is zero. 
Otherwise, by our assumption on $\omega^0_{\phi \circ \tau}$, we may locally take an $\ccO_U[x]/x^{e_\fp}$-linear trivialization of $\ccH^1_{\dR}(A/U)^0_{\phi \circ \tau}$ such that $\omega^0_{\phi \circ \tau} = x^a e_1$ for the appropriate $a$. In this case
\[\begin{split}
V_{\es,(\phi^2 \circ \tau)^1}\circ F(\ccH^1_{(\phi \circ \tau)^1}) & = F^{-1}(F(x^{e_\fp-1}e_2)) / \omega^0_{\phi \circ \tau}\\
& = x^{e_\fp-1}e_2 + \omega^0_{\phi \circ \tau} / \omega^0_{\phi \circ \tau} \\
& = F_{\es,(\phi \circ \tau)^i}^{(\phi \circ \tau)^{e_\fp}}(\ccH^1_{(\phi \circ \tau)^i}).
\end{split}\]
Therefore, we conclude that if $\omega^0_{\phi \circ \tau} \neq 0$, the composition $\ccH^{1 (p)}_{(\phi \circ \tau)^1} \to \Phi^*\mathcal{M}_{\phi \circ \theta}$ is zero and the first diagram commutes, since we set $s_\theta = 0$. 
If $\omega^0_{\phi \circ \tau} = 0$ however, it is straightforward to see that the first diagram commutes by construction of the $s_\theta$, since both $F$ and $V_{\textrm{es},(\phi^2 \circ \tau)^1} = F^{-1}$ are isomorphisms.

If there exists an $i$ such that $s'_{(\phi \circ \tau)^i}=0$, then $\omega^0_{\phi \circ \tau} \subset [\varpi_\fp]\ccH^1_{\textrm{dR}}(A/X)^0_{\phi \circ \tau}$ and so $V(\ccH^1_{(\phi^2 \circ \tau)^1}) = 0$. Therefore, the second diagram commutes since we set $t_\theta = 0$. Suppose now that $s'_{(\phi \circ \tau)^i} \neq 0$ for all $1 \leq i \leq e_\fp$ and consider the diagram:

\begin{center}
\begin{tikzcd}[row sep = large, column sep = large]
    \ccH^{1}_{(\phi^2 \circ \tau)^1} \arrow[d, "V"'] \arrow[r, "{ V_{\textrm{es},(\phi^2 \circ \tau)^1}}"] & \ccH^1_{(\phi \circ \tau)^{e_\fp}} \arrow[r] \arrow[d, "V_{\textrm{es}, (\phi \circ \tau)^{e_\fp}}^{\tau^{e_\fp}}"'] & \mathcal{M}_{\phi \circ \theta} \arrow[d,"t_\theta"] \\
    \ccH^{1 (p)}_{(\phi \circ \tau)^1} \arrow[r, "{ V_{\textrm{es},(\phi \circ \tau)^1}}"'] & \ccH^{1 (p)}_{\tau^{e_\fp}} \arrow[r] & \mathcal{M}^{(p)}_\theta. 
\end{tikzcd}
\end{center}

The second square commutes by definition of $t_\theta$ and the first square commutes since $V_{\textrm{es},(\phi^2 \circ \tau)^1}^{(\phi \circ \tau)^1} = V$ by assumption. We are now precisely in the situation of lemma \ref{subgroup sesame} and we obtain an $\ccO_{D_{IJ}}$-linear embedding $M^\prime \hookrightarrow A[\fq]$.

We now define $A^\prime_2/X$ to be the splice of $A$ by $M^\prime$. That is, we define the $\ccO_{D_{IJ}}$-stable subgroup $M = M^\prime \oplus \lambda^{-1}((A[\fq]/M^\prime)^\vee) \subset A[\fp]$ and 
\[A_2 = A / M, \]
and write $q_2^\prime: A \to A_2$ for the canonical isogeny. Note that, by construction, the kernel $\ker q^\prime_2 \subset A[\fp]$ of $q_2^\prime$ is stable under the action of $\ccO_{D_{IJ}}$, therefore $A_2$ comes equipped with an action of $\ccO_{D_{IJ}}$ such that $q^\prime_2$ is compatible with the actions of $\ccO_{D_{IJ}}$ on both sides. We can therefore set $q_2: A_2 \otimes \fp \to A$ to be the unique isogeny such that $q^\prime_2 \circ q_2 = m_{2,\fp}$, the isogeny induced by multiplication. We now define a prime to $p$ quasi-polarization $\lambda_2 : A_2 \to A_2
^\vee$ via the following commutative diagram:
\begin{center}
\begin{tikzcd}
  0 \arrow[r] & M \oplus \lambda^{-1}((A[\fq]/M)^\vee) \arrow[r] \arrow[d, "n\lambda", "\rotatebox{90}{\(\sim\)}"'] & A \arrow[rr, "q^\prime_2"] \arrow[d,"n\lambda \otimes \varpi_\fp",] && A_2 \arrow[r] \arrow[d,"n\lambda_2"] & 0 \\
  0 \arrow[r] & (M \oplus \lambda^{-1}((A[\fq]/M)^\vee)))^\vee \arrow[r] & (A \otimes \fp^{-1})^\vee \arrow[rr, "(q_2 \otimes \fp^{-1})^\vee"'] && A^\vee_2 \arrow[r] & 0,
\end{tikzcd}
\end{center}
where $n$, prime to $p$ is chosen such that $n\lambda$ is a genuine polarization. It follows that the Rosati involution associated to $\lambda_2$ is compatible with the anti-involution on $\ccO_{D_{IJ}}$.

For any $\theta \in T^1$ we obtain isomorphisms 
\begin{equation}\label{A_2 A splitting de Rham at T1}
\ccH^1_{\textrm{dR}}(A_2/X)_{\tilde{\theta}} \xrightarrow[\sim]{q_2^*} \ccH^1_{\textrm{dR}}(A/X)_{\tilde{\theta}} \, \textrm{ and } \, \ccH^1_{\textrm{dR}}(A_2 \otimes \fp/X)_{\tilde{\theta}^c} \xleftarrow[\sim]{(q_2^\prime)^*} \ccH^1_{\textrm{dR}}(A/X)_{\tilde{\theta}}. 
\end{equation}  If $\theta \notin T^1$, then we have short exact sequences 
\begin{equation}\label{A_2 A splitting away from T1} 0 \to q_2^*\ccH^1_{\textrm{dR}}(A_2/X)_{\tilde{\theta}}^0 \to \ccH^1_{\textrm{dR}}(A/X)_{\tilde{\theta}}^0 \to \mathcal{M}_{\phi^{-1} \circ \tilde{\theta}}^{(p)} \to 0\end{equation}
and \[ 0 \to (q_2^\prime)^*\ccH^1_{\textrm{dR}}(A/X)_{\tilde{\theta}^c}^0 \to \ccH^1_{\textrm{dR}}(A_2 \otimes \fp/X)_{\tilde{\theta}^c}^0 \to \mathcal{M}_{\phi^{-1} \circ \tilde{\theta}^c}^{\vee (p)} \to 0.\]
so that, by (\ref{A_2 de rham components}), we have: 
\begin{equation}\label{fixed a2 de rham relations}
q_2^*\ccH^1_{\textrm{dR}}(A_2/X)_{\tilde{\theta}}^0 =
\begin{cases}
     \ccH^1_{\textrm{dR}}(A/X)_{\tilde{\theta}} & \textrm{ if } \theta \notin \widehat{\Theta}_{F,\fp} \textrm{ or } \theta \in T^1, \\
    [\varpi_\fp]^{1-e_\fp} F_{\textrm{es},\tilde{\beta}}^{\tau^1}(\ccH^1_\beta) & \textrm{ if }  \tau \in \widehat{\Theta}_{E,\fp},  \theta \notin T^1 \textrm{ and } \beta \in R,\\
    [\varpi_\fp]^{1-e_\fp} F_{\textrm{es},\tilde{\beta}}^{\tau^1}(L_{\tilde{\beta}}) & \textrm{ if }  \tau \in \widehat{\Theta}_{E,\fp}, \theta \notin T^1 \textrm{ and } \beta \in R.
\end{cases}
\end{equation}
\subsubsection{Filtering $A_2$}\label{filtering A2}
Fix a geometric point $\overline{s}$ of $S$ and consider the Dieudonn\'{e} modules $\Delta = \mathbb{D}(A_{\overline{s}}[p^\infty])$ and $\Delta_2 = \mathbb{D}(A_{2,\overline{s}}[p^\infty])$. Then, by above, we have for each $\tau = \widetilde{\theta} \in \widehat{\Theta}_{E}$ 
\[q_2^* \Delta^0_{2,\tau} = 
\begin{cases}
    \Delta^0_\tau & \textrm{ if } \theta \notin \widehat{\Theta}_{F,\fp} \textrm{ or } \theta \in T^1, \\
    [\varpi_\fp]^{1-e_\fp} F_{\textrm{es},\tilde{\beta}}^{\tau^1}(\ccH^1_\beta) & \textrm{ if }  \tau \in \widehat{\Theta}_{E,\fp},  \theta \notin T^1 \textrm{ and } \beta \in R,\\
    [\varpi_\fp]^{1-e_\fp} F_{\textrm{es},\tilde{\beta}}^{\tau^1}(L_{\tilde{\beta}}) & \textrm{ if }  \tau \in \widehat{\Theta}_{E,\fp}, \theta \notin T^1 \textrm{ and } \beta \in R.
\end{cases}
\]
Here, by abuse of notation, the two bottom subspaces denote the lift of the subspace of the same name in $H^1_{\textrm{dR}}(A_{\overline{s}})^0_\tau = \Delta^0_\tau / p\Delta^0_\tau$ to $\Delta^0_\tau$.

Via the isomorphism $\omega^{0 \, (p)}_{A_{2,\overline{s}},\tau} \simeq V\Delta_{2,\phi \circ \tau}^0 / p \Delta^0_{2,\tau}$ we have:
\[\begin{split}
\dim \omega^0_{A_{2,\overline{s},\tau}}  & = \dim V q_2^*\Delta^0_{2,\phi \circ \tau} / pq^*_2\Delta^0_{2,\tau}\\
& = \dim V \Delta^{0}_{\phi \circ \tau} / p\Delta^{0}_{\tau} + \dim p \Delta_{\tau} / pq^*_2\Delta^0_{2,\tau} - \dim V \Delta_{\phi \circ \tau} / Vq_2^*\Delta^0_{2,\phi \circ \tau} \\
& = \dim \omega^0_{A_{\overline{s},\tau}} + \dim \Delta_{\tau} / q^*_2\Delta^0_{2,\tau} - \dim \Delta_{\phi \circ \tau} / q_2^*\Delta^0_{2,\phi \circ \tau}
\end{split}\] 

Therefore 
\[\dim \omega^0_{A_{2,\overline{s},\tau}} = 
\begin{cases}
    \dim \omega^0_{A_{\overline{s},\tau}} + 1 & \textrm{ if } \tau \notin T^1, \phi \circ \tau \in T^1\\
    \dim \omega^0_{A_{\overline{s},\tau}} - 1 & \textrm{ if } \tau \in T^1, \phi \circ \tau \notin T^1,\\
    \dim \omega^0_{A_{\overline{s},\tau}} & \textrm{ otherwise}.
\end{cases}\]
Note that, by construction of $\widetilde{\Sigma}_{IJ}$, this is well defined as in the remark after proposition \ref{aij dieudonne calculations}. We deduce that if $\phi \circ \theta \in T^1$, then $q_2^*\omega^0_{2,\tau} = \omega^0_{A,\tau}$ and so, by dimension count, $\omega^0_{2,\tau} = (q_2^*)^{-1}\omega^0_{A,\tau}$. 
If $\phi \circ \theta \notin T^1$, suppose further that $\omega^0_{A,\tau} \neq 0$ ($\omega^0_{A,\tau} \neq 0$ and $\phi \circ \theta \notin T^1$ implies that $\omega^0_{2,\tau} = 0$).Let $i$ be the maximal index such that $s^\prime_{\tau^i} \neq 0$. 
Then in fact, the morphism 
\[\ccH^1_{\textrm{dR}}(A/X)^0_{\phi \circ \tau} \to \ccH^{1 \, (p)}_{\tau^{e_\fp}} \xrightarrow{\sim} \ccH^{1 \, (p)}_{\tau^i},\]
where the last isomorphism is given by the identity if $i = e_\fp$ and $V_{\textrm{es},\tau^{e_\fp}}^{\tau^i}$ otherwise, is in fact given by Verschiebung. 
Suppose that $s^\prime_{\tau^i} = 1$, then set $L^\prime_\tau = \omega^0_{A,\tau}(i-1)$, locally free of rank $s^\prime_\tau(i-1) = s^\prime_\tau(e_\fp)-1$. 
Otherwise, set $L^\prime_\tau$ to be the lift of $F_{\es,\tilde{\beta}}^{\tau^i}(\ccH^1_{\tilde{\beta}}) \subset \omega^0_{A,\tau}/[\varpi_{\fp}]\omega^0_{A,\tau}$ if $\beta \notin R$, 
and $F_{\es,\tilde{\beta}}^{\tau^i}(L_{\tilde{\beta}}) \subset \omega^0_{A,\tau}/[\varpi_{\fp}]\omega^0_{A,\tau}$ 
if $\beta \in R$, where $\beta$ and $L_{\tilde{\beta}}$ are as in section \ref{constructing A2}. 
It is in any case locally free of rank $s^\prime_\tau(i) -1 = s^\prime_{\tau}(e_\fp)-1$. 

Then, by above, ignoring $p$-twists, we have $q_2^* \ccH^1_{\dR}(A_2/X)^0_{\phi \circ \tau} = V^{-1}(L^\prime_\tau)$. 
Furthermore, a case by case analysis shows that $V(L^\prime_\tau) \subset L^\prime_{\phi^{-1}\circ \tau}$. For example, if $L^\prime_\tau = \omega^0_{A,\tau}(e_\fp-1)$ and $L^\prime_{\phi^{-1}\circ \tau} =\omega^0_{A,\phi^{-1}\circ \tau}(e_\fp-1)$, this follows from the fact that that $\omega^0_{A,\tau}(e_\fp-1) \subset [\varpi_\fp]\ccH^1_{\dR}(A/X)^0_{A,\tau}$, hence $V\omega^0_{A,\tau}(e_\fp-1) \subset [\varpi_\fp]\omega^0_{A,\phi^{-1}\circ \tau} \subset \omega^0_{A,\phi^{-1}\circ \tau}(e_\fp-1)$. In particular, we conclude that $(q_2^*)^{-1}(L^\prime_\tau) \subset \ccH^1_{\dR}(A_2/X)^0_{\tau}$ is a subbundle of the same rank as $\omega^0_{2,\tau}$. Since $q_2^* \ccH^1_{\dR}(A_2/X)^0_{\phi \circ \tau} = V^{-1}(L^\prime_\tau)$, it thus follows by dimension count that $q_2^*\omega^0_{2,\tau} = L^\prime_\tau$.

We can now define a filtration on $\omega^0_{2,\tau}$ for each  $\theta \in \widehat{\Theta}_F$ as follows: If $\tau \notin \widehat{\Theta}_{E,\fp}$, then $q_2^*: \omega^0_{2,\tau} \xrightarrow{\sim} \omega^0_\tau$. For each $1 \leq i \leq e_\fp$, we then set 
\[\omega^0_{2,\tau}(i) = (q_2^*)^{-1} \omega^0_\tau(i).\]

Suppose that $s_{\tau^i} \neq 1$ for all $1 \leq i \leq e_\fp$. If $\omega^0_{2,\tau} = 0$, then we are done. Otherwise, $\omega^0_{A,\tau} = [\varpi_\fp]^{e_\fp-d} \ccH^1_{\dR}(A/X)^0_\tau$ for some $d > 0$. Then 
\[q^*_2([\varpi_\fp]^{e_\fp-d} \ccH^1_{\dR}(A/X)^0_\tau) = [\varpi_\fp]^{1-d} F_{\es,\tilde{\beta}}^{\tau^1}(\ccH^{1 \, (p^n)}_{\tilde{\beta}}) = L^\prime_\tau, \]
from which we deduce that $\omega^0_{2,\tau} = [\varpi_\fp]^{e_\fp-d} \ccH^1_{\dR}(A/X)^0_\tau)$. The filtration is forced, and more importantly, well-defined.

Consider now $\tau = \tilde{\theta} \in \widehat{\Theta}_{E,\fp}$ such that $s_{\tau^i}=1$ for some $i$. For any such $i$, set
\begin{equation}\label{filtering a2 definition} \omega^0_{2,\tau}(i) = 
\begin{cases}
    (q_2^*)^{-1}\omega^0_\tau(i) & \textrm{ if } \phi^\prime(\theta^i) \in T, \\
    (q_2^*)^{-1}\omega^0_\tau(i-1) & \textrm{ if } \theta^i \notin T^\prime.
\end{cases}
\end{equation}

First note that $\omega^0_{2,\tau}(i) \subset \omega^0_{2,\tau}$ by the discussion above, and in particular $\omega^0_{2,\tau}(e_\fp) = \omega^0_{2,\tau}$ if $s_{\tau^{e_\fp}}=1$. Secondly, note that if $\theta \in T^1$, then $q_2^*$ is injective and $\omega^0_{2,\tau}(i)$ is locally free of dimension $s^\prime_\tau(i) = s_\tau(i)$ in the first case and dimension $s^\prime_\tau(i-1) = s_\tau(i)$ in the second. Similarly, if $\theta \notin T^1$, then $q_2^*$ has locally free kernel of dimension one so, as before, $\omega^0_\tau(i)$ is locally free of dimension $s_\tau(i)$. Thirdly, if $\theta^i$ is such that $\theta^i \notin T$ and $\phi^\prime(\theta^i) \in T$, then $s^\prime_{\tau^i} = 0$ and $\omega^0_{\tau}(i-1) = \omega^0_\tau(i)$ so that $\omega^0_{2,\tau}(i)$ is well defined.

\noindent Consider now the $\theta^i$ which are not covered by the above definition: That is $\theta^i \in T^\prime$ but $\phi^\prime(\theta^i) \notin T$. In this case, we have $s^\prime_{\tau^i}=2$ and by definition $\omega^0_{2,\tau}(i-1) = (q_2^*)^{-1}\omega^0_\tau(i-1)$. If $\theta^i \in R$ and we have the line $L_{\tau^i} \subset \ccH^1_{\tau^i}$ whose lift to $\omega^0_\tau$ we, by abuse of notation, we also denote $L_{\tau^i}$. We then set
\[\omega^0_{2,\tau}(i) = (q_2^*)^{-1}L^0_{\tau^i}.\]  
If $\theta^i \in \Sigma_\infty$, then define $\omega^0_{2,\tau}(i)$ appropriately in terms of $\omega^0_{2,\tau}(i-1)$. If $\theta^i \in I \cap J$ we inductively define 
\[\omega^0_{2,\tau}(i) = [\varpi_\fp]^{-1}\omega^0_{2,\tau}(i-2).\]

We now show that this yields a well defined filtration. Let $i$ be minimal such that $s_{\tau^i} = 1$. If $\theta^i \in T$, then $\theta \in T^1$, $q^*_2$ is an injection and $\omega^0_{2,\tau}(i-1) = (q_2^*)^{-1}\omega_{A,\tau}(i-1) = [\varpi_\fp]^{e_\fp-d}\ccH^1_{\dR}(A_2/X)^0_\tau$ for the appropriate $d$. The filtration is thus well defined for $j \leq i$ and we see that in any case $\omega^0_{2,\tau}(i-1) \subset \omega^0_{2,\tau}(i)$ with quotient a line bundle. Similarly, if $\theta^i \notin T$, then $q_2^*$ has one dimensional kernel and  $[\varpi_\fp]^{e_\fp-d}\ccH^1_{\dR}(A_2/X)^0_\tau \subset (q_2^*)^{-1}\omega_{A,\tau}(i-1) = \omega^0_{2,\tau}(i)$ with quotient a line bundle, the filtration is well defined for $j \leq i$. A similar analysis shows that the filtration is well defined for $j \geq i$ where we now take $i$ to be maximal such that $s_{\tau^i}=1$. 

Suppose now that $i$ is such that $s_{\tau^i}=1$, $1 \leq j < i$ is the largest such that $s_{\tau^j}=1$ and the filtration is well defined for $k \leq j$. If both $\phi^\prime(\theta^j) = \theta^i$ and $\phi^\prime(\theta^i)$ are in $T$, then it is clear that the filtration is well defined for $j \leq k \leq i$. If $\theta^i \in T$ but $\phi^\prime(\theta^i) \notin T$, then $\omega^0_{2,\tau}(j) = (q_2^*)^{-1}\omega^0_{A,\tau}(j)$. If $\theta^i \in R$, then 
\[\omega^0_{2,\tau}(i) = (q^*_2)^{-1}L_{\tau^i} \subset (q^*_2)^{-1}\omega^0_{A,\tau}(i - 1) = [\varpi_\fp]^{-d}\omega^0_{2,\tau}(j), \]
for the appropriate $d$. We conclude as before. If $\theta^i \in I \cap J$ however, then $\omega_{2,\tau} = [\varpi_\fp]^{-d}\omega^0_{2,\tau}(j-1)$ and we conclude as before. If $\theta^i \notin T$ but $\theta^j \in T$, then 
\[ (q^*_2)^{-1}\omega^0_{A,\tau}(i - 1) = \omega^0_{\tau}(j-1) \subset \omega^0_{2,\tau}(j) \subset [\varpi_\fp]^{-1}\omega^0_{\tau}(j-1) \subset [\varpi_\fp]^{-d}\omega^0_{\tau}(j-1) = \omega^0_{\tau}(i).
\]
We conclude as before, Finally, if both $\theta^i,\phi^\prime(\theta^i) \notin T$, then 
\[q_2^*: \omega^0_{2,\tau}(i)/\omega^0_{2,\tau}(j) \xrightarrow{\sim} \omega^0_{2,\tau}(i-1)/\omega^0_{2,\tau}(j-1).\]
If we write $\omega^0_{A,\tau}(i-1) = [\varpi_\fp]^{-d}\omega^0_{A,\tau}(j)$, then $\omega^0_{A,\tau}(i-1)/[\varpi_\fp]^{-d}\omega^0_{A,\tau}(j-1)$ is also line bundle. We conclude that the filtrations are well defined.
\subsubsection{Constructing $A_1$}

 We now construct a subgroup $N \subset A$ similarly to how we constructed $M$. The construction will essentially be the (anti-) dual construction of $M$. From now on, for $\theta \in \widehat{\Theta}_{F,\fp}$, write $\tau = \tilde{\theta}^c$. We stress that, unlike usual, we pick the conjugate lift. 
 
 If $\theta$ is such that $\phi \circ \theta \notin T^1$, set $\mathcal{N}_\theta = 0$. Suppose then that $\phi \circ \tau \in T^1$ and let $\beta : = (\phi^{\prime \prime})^{-1}((\phi \circ \theta)^1)$. Suppose that $\beta \notin R$. If $\beta = \theta^{e_\fp}$, set $\mathcal{N}_\theta = \omega^0_{\tau^{e_\fp}}$. Otherwise, set $\mathcal{N}_\theta = \ccH^1_{\tau^{e_\fp}} / F_{\textrm{es},\tilde{\beta}^c}^{\tau^{e_\fp}}(\ccH^{1 (p^n)}_{\tilde{\beta}^c})$, where $n$ is suitably chosen. Suppose now that $\beta \in R$. Then as before we have the line bundle $L_{\tilde{\beta}} \subset \ccH^1_{\tilde{\beta}}$ and we set $L_{\tilde{\beta}^c} \subset \ccH^1_{\tilde{\beta}^c}$ to be the orthogonal complement of $L_{\tilde{\beta}}$ under the perfect pairing induced by $\lambda$. Set
 \[\mathcal{N}_\theta = \ccH^1_{\tau^{e_\fp}}/F_{\textrm{es},\tilde{\beta}^c}^{\tau^{e_\fp}}(L_{\tilde{\beta}^c}^{p^n}),\]

 \noindent where again $n$ is suitably chosen. 

 We can then define morphisms $s_\theta: \mathcal{N}^{(p)}_\theta \to \mathcal{N}_{\phi \circ \theta}$ and $t_\theta: \mathcal{N}_{\phi} \to \mathcal{N}^{(p)}_\theta$ entirely analogously to the case of $M$ and thus obtain a partial Raynaud datum $(\mathcal{N}_\theta,s_\theta,t_\theta)$. We thus let $N^{\prime \prime}$ be the associated group scheme and set $\mathcal{N} = \bigoplus_\theta \mathcal{N}_\theta$ and obtain a surjection
 \[0 \to \ker \to \ccH^1_{\dR}(A/S)^0 \to \Phi^*\mathcal{N} \to 0\]
 which is compatible with the $\ccO_{D_{IJ}}$-action and Frobenius and Verschiebung. We conclude again by lemma \ref{subgroup sesame} that we have an $\ccO_{D_{IJ}}$-linear embedding $N^\prime = N^{\prime \prime} \otimes \ccO_E^2 \hookrightarrow A[\fq^c]$. We thus set
 \[N = \lambda^{-1}((A[\fq^c]/N^\prime)^\vee)) \oplus [\varpi_{\fq^c}]^{-1}N^\prime \subset A[\fq] \oplus A[(\fq^c)^2].\]

 Note that $N$ does not depend on the choice of uniformizer $\varpi_{\fq^c}$. It is clear that $M[\fq^c] \subset A[\fq^c] \subset N[\fq^c]$, where $M$ denotes the kernel of $q_2: A \to A_2$, and at the level of crystals $\mathbb{D}(N^\prime)_{\tilde{\theta}^c} = 0$ if $\theta \notin T^1$ and $\mathbb{D}(M[\fq])_{\tilde{\theta}} = 0$ if $\theta \in T^1$ so that we have a surjection of crystals 
 \[\mathbb{D}(N[\fq]) \twoheadrightarrow \mathbb{D}(M[\fq])\]
 and thus an inclusion $M \subset N$. Writing $ A_1^\prime := A/N$ we obtain the commutative diagram:

\begin{center}
\begin{tikzcd}
 A  \arrow[r,"q_2"] \arrow[d,"\iota_{\fq^c}"] & A_2 \arrow[d,"g"] \\
 A \otimes (\fq^c)^{-1} \arrow[r,"q_1^\prime"] & A_1^\prime.
\end{tikzcd}
\end{center}

Every isogeny in this diagram has kernel killed by $[\varpi_\fp]$ and 
\[\textrm{rk} (\ker g )[\fq] = \textrm{rk} (\ker g)[\fq^c] = p^{2f_\fp}. \]

\noindent We thus set $A_1 = A^\prime_1 \otimes \fp$, so that $A^\prime_1 = A_1 \otimes \fp^{-1}$ and we let $f:A_1 \to A_2$ be the unique isogeny such that $g \circ f = \iota_{1, \fp}: A_1 \to A_1 \otimes \fp^{-1}.$ Similarly, we write $q_1: A_1 \to A \otimes (\fq^c)^{-1}$ for the unique isogeny such that $q^\prime_1 \circ q_1 = \iota_{1,\fp}$. We thus obtain the following commutative diagram:

\begin{equation}\label{splitting commutative diagram}
\begin{tikzcd}
  A_2 \otimes \fp \arrow[r, "q^\prime_2"] \arrow[d, "g \otimes \fp"'] & A \arrow[r, "q_2"] \arrow[d,"\iota_{\fq^c}"] & A_2 \arrow[d,"g"] \\
  A_1 \arrow[r, "q_1" ] & A \otimes (\fq^c)^{-1} \arrow[r,"q^\prime_1"] & A_1 \otimes \fp^{-1}.
\end{tikzcd}
\end{equation}

\noindent We note that under the isomorphism $\iota_{\fq^c}:A[\mathfrak{Q}^\infty] \xrightarrow{\sim} A[\mathfrak{Q}^\infty]\otimes \fq^c$, where $\mathfrak{Q} = \prod_{\fp \vert p} \fq$, the composite 
\[A_1[\mathfrak{Q}^\infty] \xrightarrow{q_1} A[\mathfrak{Q}^\infty] \xrightarrow{q_2} A_2[\mathfrak{Q}^\infty]\] is simply given by $f$, by the commutativity of the above diagram. 

Similarly, we have the commutative diagram:
\[\begin{tikzcd}
  A_1 \arrow[r,"f"] \arrow[rd,"\iota_{1,\fp}"'] & A_2 \arrow[d,"g"] \arrow[r,"q^\prime_2 \otimes \fp^{-1}"] & A \otimes \fp^{-1}, \\
& A_1 \otimes \fp^{-1} \arrow[ru,"q_1 \otimes \fp^{-1}"'] & 
\end{tikzcd}\]

from which we deduce that $A_1[\fp] \subset \ker \, ((q^\prime_2 \otimes \fp^{-1}) \circ f)$. Therefore 
\[f^\vee \lambda_2 ((\ker g)[\fq]) = f^\vee \lambda_2 ( q_2 (M[\fq])) = f^\vee (q^\prime_2 \otimes \fp^{-1})^\vee ((A[\fq^c]/M^\prime)^\vee \otimes \fp) = 0,\]
where the second equality comes from the compatibility of polarizations $\lambda$ and $\lambda_2$ and the definition of $M$, and the last equality comes from the above diagram. In particular, $\ker g$ is totally isotropic with respect to the Weil pairing and we can define a prime to $p$ quasi-polarization $\lambda^\prime_1$ on $A_1 \otimes \fp^{-1}$ via the diagram:

\[
\begin{tikzcd}
    0 \arrow[r] & \ker g \arrow[r] \arrow[d, "n\lambda_2 \otimes \varpi_\fp", "\sim"'{anchor = south, rotate =90}] & A_2 \arrow[r,"g"] \arrow[d, "n\lambda_2 \otimes \varpi_\fp"]& A_1 \otimes \fp^{-1} \arrow[r] \arrow[d,"n\lambda^\prime_1"]& 0 \\
    0 \arrow[r] & (\ker f)^\vee \otimes \fp \arrow[r] & A_2^\vee \otimes \fp  \arrow[r, "(f \otimes \fp^{-1})^\vee"] & A^\vee_1 \otimes \fp \arrow[r] & 0,
\end{tikzcd}
\]
where $n$ is a prime to $p$ integer such that $n\lambda_2$ is a genuine polarization. To obtain a polarization on $A_1$, write $\varpi_\fp^{e_\fp} = p \alpha$ for $\alpha \in \ccO_F$ and let $m$ be the minimal positive integer such that $m \alpha^{-1} \in \ccO_F$. Then $m$ is prime to $p$, $m\varpi_\fp^{-1} \in \fp^{-1}$ and $m\varpi_\fp^{-2}\fp \subset \fp^{-1}$. We define the polarization $mn\lambda_1$ via 
\[A_1 = A^\prime_1 \otimes \fp \xrightarrow{n\lambda^\prime_1 \otimes m\varpi_\fp^{-2} \cdot} (A^{\prime \vee}_1 \otimes_{\ccO_E} \fp^{-1}) \xrightarrow{\sim} A^\vee_1,\]
where $m\varpi_\fp^{-2} \cdot$ denotes the multiplication by $m\varpi_\fp^{-2}$ map. Since the middle map has degree prime to $p$ and $m$ and $n$ are also prime to $p$, then $\lambda_1$ is a prime to $p$ $\ccO_D$-antilinear quasi-polarization. Consider now the the diagram

\begin{center}
\begin{tikzcd}[row sep = large]
A^\prime_1 \otimes \fp \arrow[r,"\sim"] & A_1 \arrow[r,"f"] \arrow[d,"mn \lambda_1 \otimes \varpi_\fp"] & A_2 \arrow[r, "g"] \arrow[d, "mn\lambda_2"] & A_1^\prime \arrow[d, "n \lambda^\prime_1 \otimes m\varpi_\fp^{-1}"] \\
A^{\prime \vee}_1 \arrow[r,"\sim"] & A_1^\vee \otimes \fp \arrow[r, "g^\vee"] & A^\vee_2 \arrow[r, "f^\vee"] & A_1^{\prime \vee} \otimes \fp^{-1}.
\end{tikzcd}
\end{center}
Then the right hand square commutes by construction of $\lambda_1^\prime$ and the outside rectangle commutes by construction of $\lambda_1$. Therefore the left hand square commutes.

Finally, arguing as in section \ref{constructing A2} and using the compatibility of Essential Frobenius and Essential Verschiebung under duality, we have for each $\tau = \tilde{\theta}$: 

\begin{equation}\label{fixed A1 inside A}
\iota_{\fq^c}^* q_1^{\prime *}\ccH^1_{\dR}(A_1/X)^0_\tau = 
\begin{cases}
    [\varpi_\fp]\ccH^1_{\dR}(A/X)^0_\tau & \textrm{ if } \theta \notin \widehat{\Theta}_{F,\fp} \textrm{ or } \theta \notin T^1, \\
    [\varpi_\fp]^{1-e_\fp} F_{\textrm{es},\tilde{\beta}}^{\tau^1}(\ccH^1_{\tilde{\beta}}) & \textrm{ if }  \theta \in T^1 \textrm{ and } \beta \notin R,\\
    [\varpi_\fp]^{1-e_\fp} F_{\textrm{es},\tilde{\beta}}^{\tau^1}(L^{p^n}_{\tilde{\beta}}) & \textrm{ if }  \theta \in T^1 \textrm{ and } \beta \in R.
\end{cases}
\end{equation}

\subsubsection{Filtering $A_1$}\label{filtering A1}
Fix a geometric point $\overline{s}$ of $S$ and consider the Dieudonn\'{e} modules $\Delta = \mathbb{D}(A_{\overline{s}}[p^\infty])$ and $\Delta^\prime_1 = \mathbb{D}(A^\prime_{1,\overline{s}}[p^\infty])$. For each $\tau =\tilde{\theta}$, we obtain by above: 
\begin{equation}\label{splitting A1 inside A}
\iota_{\fq^c}^* q_1^{\prime *} \Delta^{\prime 0}_{1,\tau} = 
\begin{cases}
    [\varpi_\fp]\Delta^0_\tau & \textrm{ if } \theta \notin \widehat{\Theta}_{F,\fp} \textrm{ or } \theta \notin T^1, \\
    [\varpi_\fp]^{1-e_\fp} F_{\textrm{es},\tilde{\beta}}^{\tau^1}(\ccH^1_{\tilde{\beta}}) & \textrm{ if }  \theta \in T^1 \textrm{ and } \beta \notin R,\\
    [\varpi_\fp]^{1-e_\fp} F_{\textrm{es},\tilde{\beta}}^{\tau^1}(L^{p^n}_{\tilde{\beta}}) & \textrm{ if }  \theta \in T^1 \textrm{ and } \beta \in R.
\end{cases}
\end{equation}
Here, by abuse of notation, the last two subspaces denote the lift of the subspace of the same name in $H^1_{\textrm{dR}}(A_{\overline{s}})^0_\tau = \Delta^0_\tau / p\Delta^0_\tau$ to $\Delta^0_\tau$.

Proceeding as in the case of $A_2$, we obtain 
\[ \dim \omega^0_{1,\tau} = \dim \omega^0_{A^\prime_1,\tau} = \dim \omega^0_{2,\tau}\]
for each $\tau \in \widehat{\Theta}_{E}$.

A case by case analysis now shows that for each  $\tau = \tilde{\theta} \in \widehat{\Theta}_{E,\fp}$ and $\beta = (\phi^\prime)^{-1}((\phi \circ \theta)^1)$, 
\begin{equation}\label{A1 A2 spliting de rham compatibilities}
g^*\ccH^1_\dR(A'_1/X)^0_{\phi \circ \tau} =
\begin{cases}
   [\varpi_\fp]^{1-e_\fp} F_{\textrm{es},\tilde{\beta}}^{(\phi \circ \tau)^1}(\ccH^1_{2,\tilde{\beta}}) & \textrm{ if } \beta \in J,\\
    [\varpi_\fp]^{1-e_\fp} \ker q_2^* \vert \, \ccH^1_{2,(\phi \circ \tau)^1} & \textrm{ if } \beta \in I, \beta \notin R, \\
    [\varpi_\fp]^{1-e_\fp} F_{\textrm{es},\tilde{\beta}}^{(\phi \circ \tau)^1}(L^{p^n}_{\tilde{\beta}})  & \textrm{ if } \beta \in I, \beta \in R,
\end{cases}
\end{equation}

\noindent where $L_{\tilde{\beta}} \subset \ccH^1_{2,\tilde{\beta}}$ via the isomorphism $q_2^*: \ccH^1_{2,\tilde{\beta}} \xrightarrow{\sim} \ccH^1_{\tilde{\beta}}$.

\noindent We now filter $\omega^0_{1,\tau}$.

Let $\tau = \tilde{\theta}$ for $\theta \in \widehat{\Theta}_{F,\fp}$ and suppose that $\omega^0_{1,\tau} \neq 0$ (so $\omega^0_{2,\tau} \neq 0$). If there exists an $i$ such that $s_\tau^i = 2$, then $\ccH^1_{2,\tau^1} \subset \omega^0_{2,\tau}(j)$ for any $j \geq i$ and both $\ker f^* \subset \omega^0_{2,\tau}(j)$ and $\ker g^* \subset \omega^0_{A_1 \otimes \fp^{-1},\tau}$. Therefore, if $\theta^1, \cdots \theta^{e_\fp} \in \Sigma_\infty$ and $\omega^0_{2,\tau} = [\varpi_\fp]^{e_\fp-d} \ccH^1_{\dR}(A/S)^0_\tau$, we have, removing the powers of $\fp$ for clarity, 
\[ [\varpi_\fp]^d \omega^0_{1,\tau} = f^*([\varpi_\fp]^{d-1}g^*\omega^0_{1,\tau}) = f^* \ker f^* = 0.
\]
Suppose that there is an $i$ such that $\theta^i \in I$ and pick it to be minimal. If $\theta \notin T^1$, then $\ker f^* = \ker q^*_2 \subset \omega^0_{2,\tau}(i)$ by construction; it is a line bundle. Suppose then that $\theta \in T^1$ and, by above, suppose that also that $s_{\tau^j} \neq 2$ for all $j \leq i$. If there is no $j < i $ with $s_{\tau^j} = 1$ then, by construction, we have $\omega^0_{2,\tau}(i) = \ker f^* = g^*\ccH^1_{A_1 \otimes \fp^{-1},\tau^1}$. Suppose then that there exists a maximal $j <i$ with $s_{\tau^j} = 1$, then $\theta^j \in T$. If $\theta^i \notin T$, then $\ccH^1_{2,\tau^1} \subset \omega^0_{2,\tau}(i)$ and we conclude. So finally suppose that $\theta^i \in T$, and so $\theta^i \in I \cap J$ and $\theta^{i-1} \in R$ so $\ccH^1_{2,\tau^1} \subset \omega^0_{2,\tau}(i)$ again by construction.

In conclusion, we see that if $\theta^i \in I$, $f^* \omega^0_{2,\tau}(i)$ has dimension $s_{\tau}(i) - 1 = s_\tau(i-1)$.

Suppose that there is an $i$ such that $\theta^i \in J$. Then $\omega^0_{A_1 \otimes \fp^{-1},\tau} \neq 0$ and $\omega^0_{2,\tau} \neq 0$. Furthermore, using the description of $g^*\ccH^1_{\textrm{dR}}(A_1 \otimes \fp^{-1}/X)^0_\tau$, we see that $g^*\omega^0_{A_1 \otimes \fp^{-1},\tau}$ has codimension one in $\omega^0_{2,\tau}$. In particular, if $\theta^i \in J$ where $i$ is maximal among the $j$ such that $s_{\tau^j} = 1$, then $g^*\omega^0_{A_1 \otimes \fp^{-1},\tau} = [\varpi_\fp]^{-d}\omega^0_{2,\tau}(i-1)$ for the appropriate $d$.

We now proceed to filter $\omega^0_{A_1}$. If $\tau = \widetilde{\theta}$ for $\theta \notin \widehat{\Theta}_{F,\fp}$, then $q_1^*$ is an isomorphism and we set 
\[ \omega^0_{1,\tau}(i) = q_1^*\omega^0_\tau(i) = q_1^*q_2^*\omega^0_{2,\tau}(i) = f^*\omega^0_{2,\tau}(i).\]

If $\tau = \widetilde{\theta}$ for $\theta \in \widehat{\Theta}_{F,\fp}$ and $i$ such that $s_{\tau^i} = 1$, set :
\[\omega^0_{1,\tau}(i) = 
\begin{cases}
f^*\omega^0_{2,\tau}(i+1) & \textrm{ if } i < e_\fp, \theta^{i+1} \in I, \\
(g^*)^{-1}\omega^0_{2,\tau}(i-1) & \textrm{ if }  \theta^{i} \in J.
\end{cases}
\]
By the discussion above, these subsheaves are locally free of rank $s_{\tau}(i)$. Suppose that $\omega^0_{1,\tau}(i)$ is defined twice. That is, $\theta^i \in J$ and $\theta^{i+1} \in I$. Then these definitions are equal if and only if \[g^*f^*\omega^0_{2,\tau}(i+1) = [\varpi_\fp]\omega^0_{2,\tau}(i+1) = \omega^0_{2,\tau}(i-1).\]
A case by case analysis, on the inclusions of $\theta^i$ and $\theta^{i+1}$ in $T$ shows that this is the case. Therefore $\omega^0_{1,\tau}(i)$ is well defined. 

Consider now the $i$ such that $s_{\tau^i} = 1$ but for which $\omega^0_{1,\tau}(i)$ has not been defined. The first class is the case that $i$ is maximal such that $s_{\tau^j} = 1$ but $\theta^i \notin J$. Then we set, by abuse of notation, 
\[\omega^0_{1,\tau}(i) = V_{\es,\tau^{e_\fp}}^{\tau^i}( \omega^0_{1,\tau}).\] 
The second class is the $i$ such that $\theta^i \notin J$, $\theta^{i+1} \in \Sigma_\infty$ but $\phi^\prime(\theta^i) = \theta^j \in I$. We set, again by abuse of notation, 
\[ \omega^0_{1,\tau}(i) = V_{\es,\tau^{j}}^{\tau^i} (f^*\omega^0_{2,\tau}(j)).\]
The third and last case are the $i$ such that $\theta^i \notin J$ and $\phi^\prime(\theta^i) = \theta^j \notin I$. Then such a $\theta^i$ satisfies $\theta^i \in R \cap I$ and we have the line bundle 
\[ L_{\tau^i} \subset \ccH^1_{\tau^i} \xleftarrow[\sim]{q_2^*} \ccH^1_{2,\tau^{i+1}}.\]
We thus set 
\[\omega^0_{1,\tau}(i) = f^*( (q^*_2)^{-1} L_{\tau^i}). \]

One can check that this extends to a full filtration of type $\widetilde{\Sigma}_\infty$ on $\omega^0_1$ and that, by construction $f$ respects the filtrations, hence so does $g$ by lemma \ref{respect filtrations}, and 
\[ f^* \omega^0_{2,\tau^i} = 0 \textrm{ if } \theta^i \in I \quad \textrm{and} \quad g^* \omega^0_{A_1 \otimes \fp^{-1},\tau^i} = 0 \textrm{ if } \theta^i \in J. \]
\subsubsection{The morphism to Iwahori level}
Recall the notation from section \ref{the splice morphism} For an open compact subgroup $U^\prime_{IJ} \subset G^\prime_{IJ}(\mathbb{A}_f)$, we let $U^\prime$ be the corresponding open compact subgroup of $G^\prime_\Sigma(\mathbb{A}_f)$ under the isomorphism induced by $\xi^{-1}$. We suppose that $U^\prime_{IJ}$ is sufficiently small such that $U^\prime_{IJ}$ and $U^\prime$ are sufficiently small in the usual sense. Since $U^\prime_{IJ}$ is of level prime to $p$, so is $U^\prime$.

We now define the extra algebraic structures such that the $\underline{A_i}$ define points of $\widetilde{Y}^\prime_{U^\prime}(G^\prime_\Sigma)_\bF$. Let $\iota_i$ be the action of $\ccO_D$ on $A_i$ obtained by composing the action of $\ccO_{D_{IJ}}$ with $\vartheta^{-1}$. Recall that we had set the isomorphism $\ccO_{D_{IJ}} \otimes \bF_p \xrightarrow{\sim} M_2(\ccO_E/p\ccO_E)$ to be $\xi \circ \alpha$, where $\alpha: \ccO_D \otimes \bF_p \xrightarrow{\sim} M_2(\ccO_E/p\ccO_E)$. Similarly as before, we thus obtain $\ccO_E/p\ccO_E$-linear isomorphisms, by composing with $\xi^{-1}(h_p)$,
\[\omega^0_{i, \Sigma_{IJ}} \xrightarrow{\sim} \omega^0_{i, \Sigma},\]
where we recall that the sheaf on the left hand side is the reduction obtained considering the $\ccO_{D_{IJ}}$-action and the right hand side the reduction obtained by considering the $\ccO_D$-action. By sections \ref{filtering A2} and \ref{filtering A1},
we can thus equip $\omega^0_i$ with a Pappas-Rapoport filtration of type $\widetilde{\Sigma}_\infty$. It also follows, via the compatibility between $\vartheta^{-1}$ and the anti-involutions on $\ccO_{D}$ and $\ccO_{D_{IJ}}$, that the Rosati involution induced by $\lambda_i$ is compatible with the anti-involution $*$ on $\ccO_D$. If $i = 2$, we define $\eta_{2,j}$ as the composition 
\[ \widehat{\ccO}^{(p)}_D \xrightarrow{(h^{(p)} \cdot)^{-1}} \widehat{\ccO}^{(p)}_D \xrightarrow{(\hat{\xi}_E^{(p)})^{-1}} \widehat{\ccO}^{(p)}_{D_{IJ}}  \xrightarrow{\eta_2, \overline{s}_i} \widehat{T}^{(p)}(A_{\overline{s}_j}) \xrightarrow{q_2} \widehat{T}^{(p)}(A_{2,\overline{s}_j})\]
for each $\overline{s}_j$ and set $\epsilon_2 = (\varepsilon^{(p)})^{-1}\epsilon$. The fact that $(\eta_2, \epsilon_2)$ is a $U^\prime$-level structure on $(A_2,\iota_2,\lambda_2)$ follows from the relations between $\vartheta,\xi,h$ and $\varepsilon$.
Similarly, we set the level structure on $A_1$ by taking $\eta_{1,i}$ to be the composite $f^{-1} \circ \eta_{2,i}$ and $\epsilon_{1,i} = \varpi_\fp^{-1}\epsilon_{2,i}$. 

In conclusion, the tuples $(A_i, \iota_i, \lambda_i, (\eta_i,\epsilon_i), \underline{\omega^0_i})$ define points of $\widetilde{Y}^\prime_{U^\prime}(G^\prime_\Sigma)$ and the tuple $(\underline{A_1},\underline{A_2},f)$ defines a point of $\widetilde{Y}^\prime_{U^\prime_0(\fp)}(G^\prime_\Sigma)$. The compatibility of $f$ with the filtrations follows from the compatibility on $\omega^0_{i, \Sigma_{IJ}}$ and in particular, we find that $(\underline{A_1},\underline{A_2},f)$ defines a point of the stratum $S^\prime_{\phi^\prime(I),J}$. Writing $S^\prime_{IJ} = \widetilde{Y}^\prime_{U^\prime_{IJ}}(G^\prime_{IJ})_\bF$ and $\mathcal{V}_{\tilde{\beta}} = \ccH^1_{\tilde{\beta}}$, we thus obtain a morphism 
\[\widetilde{\Xi}^\prime_{IJ} : \prod_{\beta \in R} \mathbb{P}_{S^\prime_{IJ}}(\mathcal{V}_{\tilde{\beta}}) \to S^{\prime}_{\phi^\prime(I),J}.\]

\subsection{Isomorphisms}

We now state and prove the main result of this thesis. Keep the notation from sections \ref{splicing} and \ref{splitting}.

\subsubsection{The Unitary case}
\begin{thm}\label{unitary iwahori unquotiented thm}
$\widetilde{\Psi}^\prime_{IJ}$ and $\widetilde{\Xi}^\prime_{IJ}$ are inverses of each other and are therefore isomorphisms. 
\end{thm}

\begin{proof}
From now on, write $X= S^\prime_{\phi^\prime(I),J}$ and $Y = \prod_{\beta \in R} \mathbb{P}^1_{S_{IJ}}(\mathcal{V}_{\tilde{\beta}})$. We start by showing that $\widetilde{\Psi}^\prime_{IJ} \circ \widetilde{\Xi}^\prime_{IJ}$ is the identity. Let $(\underline{A},\{ L_{\tilde{\beta}} \}_{\beta \in R})$ be the universal tuple over $Y$ and $(\underline{A_1}, \underline{A_2},f)$ denote its splitting. Finally, write $(\underline{A^\prime},\{ \ccL_\beta \}_{\beta \in R})$ for the splice of $(\underline{A_1}, \underline{A_2},f)$. We recall that $L_{\tilde{\beta}}$ denotes the kernel of the projection to $\ccO(1)_\beta$ and $\ccL_{\tilde{\beta}} \subset \ccH^1_{A^\prime,\tilde{\beta}} \simeq \ccH^1_{j,\tilde{\beta}}$, for the appropriate $j$, is given by $\omega^0_{j,\tilde{\beta}}$. We want to show that $(\underline{A},\{ L_{\tilde{\beta}} \}_{\beta \in R})$ and $(\underline{A^\prime},\{ \ccL_\beta \}_{\beta \in R})$ are isomorphic.  Recall from diagrams \ref{splice diagram} and \ref{splitting commutative diagram} that we have the commutative diagrams 
\begin{equation}\label{big proof diagram 1}
\begin{tikzcd}
    A_2 \otimes \fp \arrow[r,"q^\prime_2"] \arrow[d,"g\otimes \fp"'] & A \arrow[r,"q_2"] \arrow[d,"\iota_{\fq^c}"] & A_2 \arrow[d,"g"] \\
    A_1 \arrow[r,"q_1"] & A \otimes (\fq^c)^{-1} \arrow[r,"q^\prime_1"] & A_1 \otimes \fp^{-1}
\end{tikzcd}
\end{equation}

\noindent and 
\begin{equation}\label{big proof diagram 2}
\begin{tikzcd}
    A_2 \otimes \fp \arrow[r,"\pi_2"] \arrow[d, "g \otimes \fp"'] & A^\prime \arrow[r,"\psi_2"] \arrow[d,"\iota_{\fq^c}"] & A_2 \arrow[d,"g"] \\
    A_1 \arrow[r,"\pi_1"] & A^\prime \otimes (\fq^c)^{-1} \arrow[r,"\psi_1"] & A_1 \otimes \fp^{-1},
\end{tikzcd}
\end{equation}
where the first diagram is a diagram of abelian schemes with $\ccO_{D_{IJ}}$-action and the second of abelian schemes with $\ccO_D$-action. The action of $\ccO_D$ on $A_i$ is given by composing the natural $\ccO_{D_{IJ}}$-action on $A_i$ induced by $A$ with the isomorphism $\vartheta^{-1}$, and similarly the $\ccO_{D_{IJ}}$-action on $A^\prime$ is given by composing the natural $\ccO_D$-action induced by $A_2 \otimes \fp$ with the isomorphism $\vartheta$. Therefore the quasi-isogeny \[A \xleftarrow{q^\prime_2} A_2 \otimes \fp \xrightarrow{ \pi_2} A^\prime\]
is compatible with the $\ccO_{D_{IJ}}$-actions. Furthermore, considering the reductions of the de Rham cohomology of $A_2 \otimes \fp$ and $A^\prime$ with respect to both the $\ccO_D$-module structure and $\ccO_{D_{IJ}}$-module structure, we have the composition:
\[
\ccH^1_{\dR}(A^\prime / Y)^{0_{\Sigma_{IJ}}} \xleftarrow[\sim]{\cdot h_p^{-1}} \ccH^1_{\dR}(A^\prime / Y)^{0_{\Sigma}} \xrightarrow{\pi_2^*} \ccH^1_{\dR}(A_2 \otimes \fp / Y)^{0_\Sigma} \xleftarrow[\sim]{\xi^{-1}(h_p)}  \ccH^1_{\dR}(A_2 \otimes \fp / Y)^{0_{\Sigma_{IJ}}}.
\]
Therefore, we can restrict the above morphisms to reduced components (with respect to $\ccO_{D_{IJ}}$)
\[ \ccH^1_{\dR}(A / Y)^0 \xrightarrow{q^{\prime *}_2} \ccH^1_{\dR}(A_2 \otimes \fp / Y)^{0} \xleftarrow{\pi_2^*} \ccH^1_{\dR}(A^\prime / Y)^0, \]
compatibly with the filtrations. The case for $A_1$ is similar and we thus from now on consider all our abelian schemes as coming with an $\ccO_{D_{IJ}}$-action.

We first start by showing that $\underline{A}$ and $\underline{A^\prime}$ are isomorphic. Since both $A$ and $A^\prime$ admit isogenies, $q^\prime_2$ and $\pi_2$ respectively, from $A_2 \otimes \fp$, they are thus isomorphic if we can show that $\ker q^\prime_2 = \ker \pi_2$. Since both are contained in $A[\fp]$ and are totally isotropic with respect to the Weil pairing induced by $\lambda_2^\prime$, it suffices to show that we have we have equality on $\fq$-components. Furthermore, since $Y$ is smooth, this is equivalent to, by the full faithfulness of the crystalline Dieudonn\'{e} functor, showing that the morphisms of crystals 
\[\mathbb{D}(A[\fq]) \xrightarrow{q^{\prime *}_2} \mathbb{D}((A_2 \otimes \fp)[\fq]) \xleftarrow{ \pi^*_2} \mathbb{D}(A^\prime[\fq])\]
have the same image. In particular, it suffices to show that this holds over $Y$. Recall that for any abelian scheme $B/Y$, we have a canonical isomorphism $\mathbb{D}(B[p])_Y = \ccH^1_{\dR}(B/Y)$, compatibly with endomorphisms and polarizations. Therefore, we are reduced to showing that for each $\tau = \tilde{\theta} \in \widehat{\Theta}_{E,\fp}$, the morphisms 
\[ \ccH^1_{\dR}(A/Y)^0_\tau \xrightarrow{q^{\prime *}_2} \ccH^1_{\dR}(A_2 \otimes \fp/Y)^0_\tau \xleftarrow{ \pi^*_2} \ccH^1_{\dR}(A^\prime/Y)^0_\tau\]
have the same image.

If $\theta \in T^1$, then by corollary \ref{sheaves iso} and by \ref{fixed a2 de rham relations}, both of the above morphisms have image $[\varpi_\fp]\ccH^1_{\dR}(A_2 \otimes \fp/X)^0_\tau$. 

If $\theta \notin T^1$ however, then $\pi_2^*$ has, again by corollary \ref{sheaves iso}, image $(g \otimes \fp)^*\ccH^1_{\dR}(A_1/Y)^0_\tau$. Similarly, by equation \ref{fixed A1 inside A}, $\iota_{\fq^c}^* q^{\prime *}_1\ccH^1_{\dR}(A_1 \otimes \fp^{-1}/Y)^0_\tau = [\varpi_\fp]\ccH^1_{\dR}(A/Y)^0_\tau$; the commutativity of diagram \ref{big proof diagram 1} then shows that $q^{\prime *}_2$ also has image $(g \otimes \fp)^*\ccH^1_{\dR}(A_1/Y)^0_\tau$. Therefore, we obtain an isomorphism  
\[q : A \to A^\prime,\]
which, after unravelling definitions, we see is compatible with polarizations and level structures. $q$ is also compatible with the diagrams. We now show that $\omega^0_{A}$ and $\omega^0_{A^\prime}$ have the same filtrations.

Fix $\theta \in \widehat{\Theta}_F$ and let $\tau = \tilde{\theta}$. If $\theta \notin \widehat{\Theta}_{F,\fp}$, then $q^*$, along with with every other isogeny in the above diagrams, induces an isomorphism on the $\tau$-components of the Hodge bundles and $q^*$ respects the filtrations by definition. Suppose now that $\theta \in \widehat{\Theta}_{F,\fp}$. If $\theta^1, \cdots, \theta^{e_\fp} \in \Sigma_{IJ,\infty}$, we have nothing to do. Therefore, let $i$ be such that $\theta^i \notin \Sigma_{IJ,\infty}$. If $\theta^i \in T$, then, by construction of $\Sigma_{IJ}$, we must have $\phi^\prime(\theta^i) \in T$ and by definitions \ref{Aij filtering} and \ref{filtering a2 definition}:
\[q^*\omega^0_{A^\prime,\tau}(i) = q^*(\psi_2^*\omega^0_{2,\tau}(i)) = q_2^*\omega^0_{2,\tau}(i)= q_2^*((q_2^*)^{-1}\omega^0_{A,\tau}(i))=\omega^0_{A,\tau}(i).\]
Similarly, if $\theta^i \notin T$, then $\phi^\prime(\theta^i) \notin T$ and 
\[q^*\omega^0_{A^\prime,\tau}(i) = q^*(\pi_1^*)^{-1}\omega^0_{1,\tau}(i)) = q^*(\pi_1^*)^{-1}f^*\omega^0_{2,\tau}(i+1) = q_2^*\omega^0_{2,\tau}(i+1) = \omega^0_{A,\tau}(i).\]

Finally, we then see that for $\theta^i \in R$, $q^*: \ccH^1_{A^\prime,\tau^i} \xrightarrow{\sim} \ccH^1_{A,\tau^i}$ and 
\[q^* \ccL_{\tau^i} = q^*\psi_2^*\omega^0_{2,\tau^i} = q_2^* \omega^0_{2,\tau^i} = L_{\tau^i},\]
\noindent if $\theta^i \in T$, and 
\[q^* \ccL_{\tau^i} = q^*(\pi_1^*)^{-1}\omega^0_{1,\tau^i} = (q_1^*)^{-1} \omega^0_{1,\tau^i} = L_{\tau^i},\]
\noindent if $\theta^i \notin T$. Therefore the tuples $(\underline{A},\{ L_{\tilde{\beta}} \}_{\beta \in R})$ and $(\underline{A^\prime},\{ \ccL_{\tilde{\beta}} \}_{\beta \in R})$ are isomorphic and $\widetilde{\Psi}^\prime_{IJ} \circ \widetilde{\Xi}^\prime_{IJ}$ is the identity.

We now show that $ \widetilde{\Xi}^\prime_{IJ} \circ \widetilde{\Psi}^\prime_{IJ}$ is the identity; this direction is slightly trickier. Write $(A_1,A_2,f,g)$ for the universal tuple over $X$, $(\underline{A},\{ \ccL_\beta \}_{\beta \in R})$ its splice and $(A^\prime_1,A^\prime_2,f^\prime,g^\prime)$ its splitting. Again, we want to show that the tuples $(A_1,A_2,f,g)$ and $(A^\prime_1,A^\prime_2,f^\prime,g^\prime)$ are isomorphic. We have this time the commutative diagrams
\begin{equation}\label{final diagram 1}
\begin{tikzcd}
    A_2 \otimes \fp \arrow[r,"\pi_2"] \arrow[d, "g \otimes \fp"'] & A \arrow[r,"\psi_2"] \arrow[d,"\iota_{\fq^c}"] & A_2 \arrow[d,"g"] \\
    A_1 \arrow[r,"\pi_1"] & A \otimes (\fq^c)^{-1} \arrow[r,"\psi_1"] & A_1 \otimes \fp^{-1}
\end{tikzcd}
\end{equation}
\noindent and 
\begin{equation}\label{final diagram 2}
\begin{tikzcd}
    A^\prime_2 \otimes \fp \arrow[r,"q^\prime_2"] \arrow[d,"g^\prime\otimes \fp"'] & A \arrow[r,"q_2"] \arrow[d,"\iota_{\fq^c}"] & A^\prime_2 \arrow[d,"g^\prime"] \\
    A^\prime_1 \arrow[r,"q_1"] & A \otimes (\fq^c)^{-1} \arrow[r,"q^\prime_1"] & A^\prime_1 \otimes \fp^{-1}.
\end{tikzcd}
\end{equation}

\noindent As before, the first diagram is a diagram of abelian schemes with $\ccO_D$-action and the second is a diagram of abelian schemes with $\ccO_{D_{IJ}}$-action. The compatibility between $\ccO_{D_{IJ}}$ and $\ccO_D$-action is dealt with as above; we thus consider every abelian scheme here as coming with an $\ccO_D$-action.

Before we continue let us discuss certain compatibilities between vector bundles \label{essential communication}:

Suppose that $\theta^i \in T$ and $\phi^\prime(\theta^i) \notin T$. Then $s^\prime_{\tau^i}=2$ and we have isomorphisms:
\[\ccH^1_{2,\tau^i} \xrightarrow[\sim]{\psi_2^*} \ccH^1_{\tau^i} \xrightarrow[\sim]{F_{\textrm{es},\phi(\tau^i)}} \ccH^1_{\phi(\tau^i)} \xrightarrow[\sim]{\pi_1^*} \ccH^1_{1,\phi(\tau^i)}.\]
Suppose further that $i < e_\fp$, then $\phi(\tau^i) = \tau^{i+1}$ and $F_{\textrm{es},\phi(\tau^i)}$ is given by $[\varpi_\fp]^{-1}$. Since $\tau^{i+1} \in I$ and $\tau^i \in J$, $[\varpi_\fp] \omega^0_{2,\tau}(i+1) = \omega^0_{2,\tau}(i-1)$ and we have an isomorphism $[\varpi_\fp]^{-1}: \ccH^1_{2,\tau^i} \xrightarrow{\sim} \ccH^1_{2,\tau^{i+2}}$. We note that if $i = e_\fp-1$ we set $\ccH^1_{2,\tau^{e_\fp+1}}:= [\varpi_\fp]^{-1}\omega^0_{2,\tau} / \omega^0_{2,\tau}$. Indeed, in this situation, $\omega^0_{2,\tau}$ is killed by $[\varpi_\fp]^{e_\fp-1}$ and we can apply lemma \ref{filtration torsion} so that $\ccH^1_{2,\tau^{e_\fp+1}}$ is locally free of rank 2. Recalling that $\pi_1^* \circ \psi_2^* = f^*$, we thus obtain the following commutative diagram:

\begin{center}
\begin{tikzcd}
{\ccH^1_{IJ,\tau^i}} \arrow[rr, "{[\varpi_\fp]^{-1}}","\sim"']  &                                           & {\ccH^1_{IJ,\tau^{i+1}}} \arrow[d, "\pi_1^*","\sim"' {rotate=90, anchor=south}] \\
{\ccH^1_{2,\tau^i}} \arrow[dr, "{[\varpi_\fp]^{-1}}"',"\sim"{rotate=-45, anchor=south}] \arrow[u, "\psi_2^*","\sim"' {rotate=90, anchor=north}]                        & & {\ccH^1_{1,\tau^{i+1}}} \\
& {\ccH^1_{2,\tau^{i+2}}} \arrow[ur, "f^*"',"\sim" {rotate=45, anchor=south}] 
\end{tikzcd}
\end{center}

If $i = e_\fp$, then $\phi(\tau^{e_\fp}) = (\phi \circ \tau)^1$ and $F_{\textrm{es},\phi(\tau^i)}$ is given by $(V')^{-1}$ as defined in section \ref{Fesves}. Since $\theta^{e_\fp} \in J$, $g^*\omega^0_{A_1\otimes \fp^{-1},\tau^{e_\fp}}=0$ and $g^*\ccH^1_{A_1\otimes \fp^{-1},(\phi \circ \tau)^1} = F_{\textrm{es},(\phi \circ \tau)^1}(\ccH^{1 \, (p)}_{2,\tau^{e_\fp}})$ so that $g^* \ccH^1_{\dR}(A_1 \otimes \fp^{-1} /X)_{\phi \circ \tau} = F_\fp^* \ccH^1_{\dR}(A_2^{(\fp)} \otimes \fp^{-1} /X)_{\phi \circ \tau}$ where we recall that $F_{\es,\fp} : A_2 \to A^{(\fp)}_2$ is the partial Frobenius isogeny at $\fp$ defined in section \ref{partial frobenius}. We can thus apply the crystallization lemma \ref{crystallization} to obtain a commutative diagram 
\[ 
\begin{tikzcd}
{\ccH^{1 \, (p)}_{2,\tau^{e_\fp}}}  \arrow[rd, "F"'] \arrow[rr, "\sim"] &   & {\ccH^1_{1,(\phi \circ \tau)^1}} \arrow[ld, "g^*"]     \\
& {\ccH^1_{2,(\phi \circ \tau)^1}}. &
\end{tikzcd}
\]
\noindent Note that we have identified $\ccH^1_{1,(\phi \circ \tau)^1}$ with $\ccH^1_{A_1 \otimes \fp^{-1},(\phi \circ \tau)^1}$ in the diagram. Consider now the extension of the previous diagram:
\begin{center}
\begin{tikzcd}
{\ccH^{1 \, (p)}_{IJ,\tau^{e_\fp}}} \arrow[rr, "(V')^{-1}"]    &   & {\ccH^1_{IJ,(\phi \circ \tau)^1}} \arrow[d, "\pi_1^*"] \\
{\ccH^{1 \, (p)}_{2,\tau^{e_\fp}}} \arrow[u, "\psi_2^{* (p)}"] \arrow[rd, "F"'] \arrow[rr, "\sim"] &   & {\ccH^1_{1,(\phi \circ \tau)^1}} \arrow[ld, "g^*"]     \\
& {\ccH^1_{2,(\phi \circ \tau)^1}}. &                                                       
\end{tikzcd}
\end{center}

\noindent The composition $ F \circ (\psi_2^*)^{-1} \circ V^\prime \circ (\pi_1^*)^{-1}$  has the same image and kernel as $g^*$, that is $F(\ccH^{1 \, (p)}_{2,\tau^{e_\fp}})$ and $f^*\ccH^1_{2,(\phi \circ \tau)^1}$ respectively, so we conclude that the outer diagram commutes (up to isomorphism) and thus so does the top rectangle.

Finally, suppose that $\theta^i$ is as above and we are now given 
\[\psi_2^*g^* \ccH^1_{A_1 \otimes \fp^{-1},\tau^i} \subset  \ccH^1_{IJ,\tau^i}  \xleftarrow[\sim]{\psi_2^*} \ccH^1_{2,\tau^i}.\]
We then conclude from the commutativity the above diagrams that 
\begin{equation}\label{communication lemma 1}
F_{\es,\phi(\tau^i)}(\psi_2^*g^* \ccH^{1\, (p^\delta)}_{A_1 \otimes \fp^{-1},\tau^i}) = (\pi_1^*)^{-1}(f^*\ccH^1_{2,\phi(\tau^i)}) \subset \ccH^1_{IJ, \phi(\tau^i)} \xrightarrow[\sim]{\pi_1^*} \ccH^1_{1, \phi(\tau^i)},
\end{equation}
\noindent where $\delta = 1$ if $i = e_\fp$ and 0 otherwise. Note that $F_{\es,\phi(\tau^i)}$ here is taken with respect to $\widetilde{\Sigma}_{IJ,\infty}$, not $\widetilde{\Sigma}_{\infty}$.

A similar analysis shows that if $\theta^i \notin T$, $\phi^\prime(\theta^i) \in T$ and we are now given 
\[
(\pi_1^*)^{-1}(f^* \ccH^1_{2,\tau^i}) \subset \ccH^1_{IJ,\tau^i} \xrightarrow[\sim]{\pi_1^*} \ccH^1_{1,\tau^i},\]
We then have
\begin{equation}\label{communication lemma 2}
F_{\es,\phi(\tau^i)}(\pi_1^*)^{-1}(f^* \ccH^{1\, (p^\delta)}_{2,\tau^i})) = \psi_2^* g^* \ccH^1_{A_1 \otimes \fp^{-1},\phi(\tau^i)} \subset \ccH^1_{IJ, \phi(\tau^i)} \xleftarrow[\sim]{\psi_2^*} \ccH^1_{1, \phi(\tau^i)},
\end{equation}
\noindent where $\delta = 1$ if $i = e_\fp$ and 0 otherwise.

Back to the proof now, we start by showing that $A_2$ and $A^\prime_2$ are isomorphic. To show this, we reduce, as before, to showing that for each $\tau = \tilde{\theta} \in \widehat{\Theta}_{E,\fp}$, the morphisms 
\[\ccH^1_{\dR}(A_2/X)^0_\tau \xrightarrow{\psi^*_2} \ccH^1_{\dR}(A/X)^0_\tau \xleftarrow{q^*_2} \ccH^1_{\dR}(A^\prime_2/X)^0_\tau\]
have the same image.

\noindent If $\theta \in T^1$, then by corollary \ref{sheaves iso} and \ref{A_2 A splitting de Rham at T1}, both maps are isomorphisms so have the same image. Suppose now that $\theta \notin T^1$. Then, again by corollary \ref{sheaves iso}, 

\[\pi_1^* \circ (\iota_{\fq^c}^*)^{-1}: \ccH^1_{\dR}(A/X)^0_\tau \xrightarrow{\sim} \ccH^1_{\dR}(A_1/X)^0_\tau\]

is an isomorphism, and under this isomorphism:
\begin{equation}\label{image of A2 in A} 
\psi_2^* \ccH^1_{\dR}(A_2/X)^0_\tau  = (\pi_1^*)^{-1}(f^*\ccH^1_{\dR}(A_2/X)^0_\tau).
\end{equation}

\noindent Furthermore, setting $\beta = (\phi^{\prime \prime})^{-1}(\theta^1)$, we also have, by equation \ref{fixed a2 de rham relations}:
\begin{equation}\label{image of A'_2 in A}
q_2^* \ccH^1_{\dR}(A^\prime_2/X)^0_\tau = 
\begin{cases}
[\varpi_\fp]^{1-e_\fp} F_{\es,\tilde{\beta}}^{\tau^1} (\ccH^1_{\tilde{\beta}}) & \textrm{ if } \beta \notin R \\
[\varpi_\fp]^{1-e_\fp} F_{\es,\tilde{\beta}}^{\tau^1} (\ccL_{\tilde{\beta}}) & \textrm{ if } \beta \in R.
\end{cases}
\end{equation}
where we have removed the powers of $p$-twists for legibility. Let us proceed case by case and again ignore writing down powers of $p$-twists: Suppose that $\beta \notin R$ and $\beta \in T$. Then by corollary \ref{sheaves iso}, we have isomorphisms $\psi_2^* : \ccH^1_{2,\tilde{\beta}} \xrightarrow{\sim} \ccH^1_{\tilde{\beta}}$ and $\psi_2^* : \ccH^1_{2,\phi(\tilde{\beta})} \xrightarrow{\sim} \ccH^1_{\phi(\tilde{\beta})}$, so that 
\[F_{\es,\phi(\tilde{\beta})}(\ccH^1_{\tilde{\beta}}) = \psi_2^*( F_{\es,\phi(\tilde{\beta})}(\ccH^1_{2,\tilde{\beta}})) = \psi_2^* g^* \ccH^1_{A_1 \otimes \fp^{-1} , \phi(\tilde{\beta})}, \]
since $\beta \in T$ implies $\beta \in J$. Similarly, if $\beta \notin R$ and $\beta \notin T$, then 
\[F_{\es,\phi(\tilde{\beta})}(\ccH^1_{\tilde{\beta}}) = (\pi_1^*)^{-1}( F_{\es,\phi(\tilde{\beta})}(\ccH^1_{1,\tilde{\beta}})) = (\pi_1^*)^{-1} (f^* \ccH^1_{2 , \phi(\tilde{\beta})}). \]
If $\beta \in R$, $\beta \in T$, then $\phi^\prime(\beta) \notin T$ and we have the line bundle 
\[\ccL_\beta = \psi_2^* \omega^0_{2,\tilde{\beta}} = \psi_2^* \ker f^* = \psi_2^* g^* \ccH^1_{A_1 \otimes \fp^{-1},\tilde{\beta}}.\]
By the previous discussion 
\[F_{\es,\phi(\tilde{\beta})}(\ccL_{\tilde{\beta}}) = (\pi_1^*)^{-1} (f^* \ccH^1_{2 , \phi(\tilde{\beta})}). \]
Similarly, if $\beta \in R$, $\beta \notin T$, then 
\[F_{\es,\phi(\tilde{\beta})}(\ccL_{\tilde{\beta}}) = \psi_2^* g^* \ccH^1_{A_1 \otimes \fp^{-1} , \phi(\tilde{\beta})}.\] 
If $\phi(\beta) =\theta^1$, then we deduce that from the equations \ref{image of A2 in A} and \ref{image of A'_2 in A} and the above that $\psi_2^*$ and $q_2^*$ have the same image.

If, however, $\phi(\beta) \in \Sigma_\infty$, then $F_{\es,\phi^2(\tilde{\beta})}$, defined for all three abelian varieties, is an isomorphism compatible with the isomorphisms $q_2^*$ and $\pi_1^*$ respectively. Applying it, we obtain the same formulae as above replacing $\phi(\beta)$ by $\phi^2(\beta)$. Finally, if $\phi(\beta) \in I \cap J$ so that $\phi ( \beta ) \in \Sigma_{IJ,\infty}$, by the equations \ref{communication lemma 1} and \ref{communication lemma 2} from the above discussion, we can similarly push the formulae to $\phi^2(\tilde{\beta})$ components. Repeatedly applying these operations until we reach $\theta^1$, we deduce that, writing $\ccF_\beta = \ccH^1_{\tilde{\beta}}$ if $\beta \notin R$, or $\ccF_\beta = \ccL_{\tilde{\beta}}$ if $\beta \in R$:
\[ q_2^* \ccH^1_{A_2^\prime,\tau^1} = F_{\es,\tilde{\beta}}^{\tau^1} (\ccF_\beta) = (\pi_1^*)^{-1} (f^*\ccH^1_{2,\tau^1}).\]
We thus obtain the desired equality of images: 
\[q_2^* \ccH^1_{\dR}(A^\prime_2/X)^0 = \psi_2^*\ccH^1_{\dR}(A_2/X)^0,\]
and there is an isomorphism $\alpha_2 : A_2 \to A^\prime_2$ which is compatible with $\psi_2$ and $q_2$. Furthermore, unravelling the definitions, we see that it is compatible with the polarizations and level structures. 

Before we show that $\alpha_2$ is compatible with the filtrations, we first show that $A_1 \otimes \fp^{-1}$ and $A^\prime_1 \otimes \fp^{-1}$ are isomorphic so that $A_1$ and $A^\prime_1$ are isomorphic as well. The argument for this is virtually identical to the argument for $A_2$. Using the descriptions of $\psi_1^*\ccH^1_{\dR}(A_1 \otimes \fp^{-1} / X )^0 $ given in corollary \ref{sheaves iso} and $q_1^{\prime *}\ccH^1_{\dR}(A^\prime_1 \otimes \fp^{-1} / X )^0 $ given in equation \ref{fixed A1 inside A}, we determine that
\[
\ccH^1_{\dR}(A_1 \otimes \fp^{-1} / X )^0 \xrightarrow{\psi_1^*} \ccH^1_{\dR}(A/ X )^0  \xleftarrow{q^{\prime *}_1} \ccH^1_{\dR}(A^\prime_1 \otimes \fp^{-1} / X )^0
\]
have the same image. Therefore we have an isomorphism $\alpha_1: A_1 \otimes \fp^{-1} \xrightarrow{\sim} A^\prime_1 \otimes \fp^{-1} $ which is compatible with $\psi_1$ and $q^\prime_1$. In particular $\alpha_1 \otimes \fp$ is compatible with $\pi_1$ and $q_1$, and $g = \alpha_1^{-1} \circ g^\prime \circ \alpha_2$, $f = \alpha_2^{-1} \circ f^\prime \circ (\alpha_1 \otimes \fp)$. We thus deduce that the diagrams \ref{final diagram 1} and \ref{final diagram 2} are actually isomorphic. In particular, $\alpha_1 \otimes \fp$ preserves the polarizations and the level structures.

We now show that $\alpha_2$ respects the filtrations:

\noindent Let $\tau = \tilde{\theta} \in \widehat{\Theta}_{E}$. If $\tau \notin \widehat{\Theta}_{E,\fp}$, then $q_2^*$ and $\psi_2^*$ are isomorphisms and, for any $i$:
\[
\alpha_2^*\omega^0_{A^\prime_2,\tau}(i)= \alpha_2^*(q_2^*)^{-1}\omega_\tau(i) = (\psi^*_2)^{-1}(\psi_2^* \omega^0_{2,\tau}(i)) = \omega^0_{2,\tau}(i).
\]
Suppose then that $\theta \in \widehat{\Theta}_{F,\fp}$. If $\theta^1, \cdots, \theta^{e_\fp} \in \Sigma_\infty$, we have nothing to do. Let $i$ be such that $\theta^i \notin \Sigma_{\infty}$. If $\phi^\prime(\theta^i) \in T$, then by definition \ref{Aij filtering} and definition \ref{filtering a2 definition}:
\[
\alpha^*_2 \omega^0_{A^\prime_2,\tau}(i) = \alpha^*_2 (q^*_2)^{-1} \omega^0_\tau(i) = (\psi^*_2)^{-1} \psi_2^* \omega^0_{2,\tau}(i) = \omega^0_{2,\tau}(i).\]
Similarly, if $\theta^i, \phi^\prime(\theta^i) \notin T$:
\[
\begin{split}
\alpha^*_2 \omega^0_{A^\prime_2,\tau}(i)  & =  \alpha^*_2 (q^*_2)^{-1} \omega^0_\tau(i-1)\\
& = \alpha^*_2 (q^*_2)^{-1} (\pi_1^*)^{-1} \omega^0_{1,\tau}(i-1) \\
& = (f^*)^{-1}\omega^0_{1,\tau}(i-1)\\
& = \omega^0_{2,\tau}(i),
\end{split}
\]
since $\theta^i \in I$. Suppose then that $\theta^i \in T$ but $\phi^\prime(\theta^i) \notin T$. If $\theta^i \notin I \cap J$, then $\theta^i \in R$ and we have the line bundle $\ccL_{\theta^i} = \psi_2^* \omega^0_{2,\tau^i} \subset \ccH^1_{\tau^i} \xleftarrow[\sim]{\psi_2^*}\ccH^1_{2,\tau^i}$ and 
\[\alpha^*_2 \omega^0_{A^\prime_2,\tau^i} = \alpha^*_2 (q_2^*)^{-1}(\psi_2^* \omega^0_{2,\tau^i}) = (\psi_2^*)^{-1}(\psi_2^* \omega^0_{2,\tau^i}) = \omega^0_{2,\tau^i}.\]
Finally, $\theta^i \in I \cap J$ and $\phi^\prime(\theta^i) \notin T$, write $\beta = (\phi^{\prime \prime})^{-1}(\theta^i)$. Then, by the above and the equations \ref{communication lemma 1} and \ref{communication lemma 2}, we find that 
\[\omega^0_{A^\prime_2,\tau^i} = (q_2^*)^{-1} (\psi_2^*g^* \ccH^1_{A_1 \otimes \fp^{-1},\tau^i}) = (q_2^*)^{-1}(\psi_2^*\omega^0_{2,\tau^i}),\]
since $\theta^i \in I$. Therefore 
\[\alpha^*_2\omega^0_{A^\prime_2,\tau^i} = \alpha_2^*(q_2^*)^{-1}(\psi_2^*\omega^0_{2,\tau^i}) = \omega^0_{2,\tau^i}.\]
We thus conclude $\alpha_2$ respects the filtrations and so extends to an isomorphism $\alpha_2: \underline{A_2} \xrightarrow{\sim} \underline{A^\prime_2}$.

We now show that $\alpha_1$ respects the filtrations: Let $\tau = \tilde{\theta} \in \widehat{\Theta}_{E}$. If $\tau \notin \widehat{\Theta}_{E,\fp}$, then $f^*$ and $f^{\prime *}$ are isomorphisms that respect filtrations and, for any $i$:
\[
(\alpha_1 \otimes \fp)^*\omega^0_{A^\prime_1,\tau}(i)= (\alpha_1 \otimes \fp)^*f^{\prime *}\omega_{A^\prime_2,\tau}(i) = f^* \omega^0_{2,\tau}(i) = \omega^0_{1,\tau}(i).
\]
Suppose then that $\theta \in \widehat{\Theta}_{F,\fp}$. If $\theta^1, \cdots, \theta^{e_\fp} \in \Sigma_\infty$, we have nothing to do. Let $i$ be such that $\theta^i \notin \Sigma_{\infty}$. If $i < e_\fp$ and $\theta^{i+1} \in I$ then, by definition, 
\[
(\alpha_1 \otimes \fp) \omega^0_{A^\prime_1,\tau}(i) = (\alpha_1 \otimes \fp) f^{\prime *} \omega^0_{A_2^\prime,\tau}(i+1) = f^{*} \omega^0_{2,\tau}(i+1) = \omega^0_{1,\tau}(i).
\]
Similarly, if $\theta^i \in J$, then 
\[
\alpha_1^* \omega^0_{A^\prime_1 \otimes \fp^{-1},\tau}(i) = \alpha_1^* (g^{\prime *})^{-1} \omega^0_{A_2^\prime,\tau}(i-1) = (g^{*})^{-1} \omega^0_{2,\tau}(i-1) = \omega^0_{A_1 \otimes \fp^{-1},\tau}(i).\]

If $i$ is maximal such that $s_{\tau^i}=1$, then $\omega^0_{1,\tau}(i) $ and $\omega^0_{A^\prime_1,\tau}(i)$ are uniquely determined so $\alpha^*\omega^0_{A^\prime_1,\tau}(i) = \omega^0_{1,\tau}(i)$. Finally, if $i$ is such that $\theta^i \notin J$ and $\phi^{\prime}(\theta^i) \notin I$, then $\theta^i \in R$ and we have the line bundle $\ccL_{\theta^i} = (\pi_1^*)^{-1} \omega^0_{1,\tau^i} \subset \ccH^1_{\tau^i}$. Furthermore, as established above, we have the isomorphism $\alpha_1^*: \ccH^1_{A^\prime_1,{\tau^i}} \xrightarrow{\sim} \ccH^1_{1,{\tau^i}}$ so that
\[\begin{split}
\alpha_1^*\omega^0_{A_1^\prime,\tau^i} & = 
\alpha_1^* f^{\prime *}((q_2^*)^{-1}(\ccL_{\theta^i}) \\
& =\alpha_1^* f^{\prime *}((q_2^*)^{-1}((\pi_1^*)^{-1} \omega^0_{1,\tau^i}))\\
& = f^* ((f^*)^{-1}\omega^0_{1,\tau^i}) \\
& = \omega^0_{1,\tau^i},
\end{split}\]
where we consider $f^*: \ccH^1_{2,\tau^{i+1}} \xrightarrow{\sim} \ccH^1_{1,\tau^{i}}$.

Therefore $\alpha_1 \otimes \fp$ preserves the filtrations and defines an isomorphism $\alpha_1 \otimes \fp: \underline{A_1} \xrightarrow{\sim} \underline{A^\prime_1}$. It follows that 
\[(\alpha_1 \otimes \fp, \alpha_2) :(A_1,A_2,f,g) \xrightarrow{\sim} (A^\prime_1,A^\prime_2,f^\prime,g^\prime) \]
is an isomorphism and $\widetilde{\Xi}^\prime_{IJ} \circ \widetilde{\Psi}^\prime_{IJ}$ is the identity. This finishes the proof.
\end{proof}

\begin{rem}\label{how to do general case}
    Recall that we had only considered the case $\Fp = \fp$ for notational simplicity. For an arbitrary product of distinct primes over $p$, $\Fp = \fp_1 \cdots \fp_n$, and $I,J \subset \Theta_{F,\Fp} \setminus \Sigma_{\infty,\Fp}$ such that $(I_{\fp_i},J_{\fp_i}) \neq (\Theta_{F,\fp_i} \setminus \Sigma_{\infty,\fp_i},\Theta_{F,\fp_i} \setminus \Sigma_{\infty,\fp_i})$, we proceed by simultaneously performing the constructions of sections \ref{sets}, \ref{splicing}, and \ref{splitting} independently at each prime $\fp_i$. That is, we set $T = \bigsqcup T_{\fp_i}$, $R = \bigsqcup R_{\fp_i}$, and $H^\prime = \bigoplus H^\prime_{\fp_i}$, where $H$ denotes any of the three subgroups required to define the splice or the splitting, where the subscript $\fp_i$ indicates the object obtained by applying the procedure for the pair $(I_{\fp_i},J_{\fp_i})$. We also filter $\omega^0[\fp_i]$ independently at each prime $\fp_i$. Arguing as in the previous sections, we thus obtain isomorphisms $ \widetilde{\Xi}^\prime_{IJ}$ and $\widetilde{\Psi}^\prime_{IJ}$.
\end{rem}
  
 In view of the above remark, we now write $\Fp = \fp_1 \cdots \fp_n$ for a non-empty product of distinct primes above $p$ and let $I,J \subset \Theta_{F,\Fp} \setminus \Sigma_{\infty,\Fp}$ be such that $(I_{\fp},J_{\fp}) \neq (\Theta_{F,\fp} \setminus \Sigma_{\infty,\fp},\Theta_{F,\fp} \setminus \Sigma_{\infty,\fp})$ for all $\fp \vert \Fp$. Recall from sections \ref{PR unitary} and \ref{iwahori unitary def} that the varieties $\overline{Y}^\prime_0(\Fp)_{\phi^\prime(I),J}$ and $\overline{Y}^\prime_{IJ}$ were defined as the quotients of $\widetilde{Y}_{U^\prime_0(\Fp)}(G^\prime_\Sigma)_{\phi^\prime(I),J,\bF}$ and $\widetilde{Y}_{U^\prime_{IJ}}(G^\prime_{IJ})_{\bF}$ by the action of $\ccO_{F,(p),+}^\times$ on quasi-polarizations $\lambda$ and multipliers $\epsilon$. 
 Furthermore, for sufficiently small $U^\prime_{IJ}$, the vector bundles $\ccH^1_{\tilde{\beta}}$ on $\widetilde{Y}_{U^\prime_{IJ}}(G^\prime_{IJ})_{\overline{\bF}}$ descend to vector bundles of the same name on the quotient $\overline{Y}^\prime_{IJ}$. Since the morphisms $\widetilde{\Psi}^\prime_{IJ}$ and $\widetilde{\Xi}^\prime_{IJ}$ are compatible with the action of $\ccO_{F,(p),+}^\times$ on both sides, they descend to isomorphisms:
 \[\Psi^\prime_{IJ}: \overline{Y}^\prime_0(\Fp)_{\phi^\prime(I),J} \xrightarrow{\sim} \prod_{\beta \in R} \mathbb{P}^1_{\overline{Y}^\prime_{IJ}}(\ccH^1_{\tilde{\beta}}) \, \textrm{ and } \, \Xi^\prime_{IJ}:  \prod_{\beta \in R} \mathbb{P}^1_{\overline{Y}^\prime_{IJ}}(\ccH^1_{\tilde{\beta}}) \xrightarrow{\sim} \overline{Y}^\prime_0(\Fp)_{\phi^\prime(I),J}.\]

 Furthermore, these isomorphisms, varying over $U^\prime$, are compatible with the Hecke action in the following sense: Recall that for two sufficiently small open compact subgroups $U^\prime_1,U^\prime_2 \subset G^\prime_{\Sigma}(\mathbb{A}_f)$ of level prime to $p$ and $g \in G^\prime_{\Sigma}(\mathbb{A}_f)$ such that $g^{-1} U^\prime_1 g \subset U^\prime_2$, 
 we defined finite \'{e}tale morphisms $\tilde{\rho}_g: \widetilde{Y}^\prime_{1,0}(\Fp) \to \widetilde{Y}^\prime_{2,0}(\Fp)$ which restrict to morphisms on the closed subschemes $S^\prime_{1,\phi^\prime(I),J} \to S^\prime_{2,\phi^\prime(I),J}$. 
 Setting $g_{IJ} = \xi_E^{-1}(g)$, we have $g_{IJ}^{-1} U^\prime_{1,IJ} g_{IJ} \subset U^\prime_{2,IJ}$ and so morphisms 
 \[ \tilde{\rho}_{g_{IJ}}: \widetilde{Y}_{U^\prime_{1,IJ}} \to \widetilde{Y}_{U^\prime_{2,IJ}} \, \textrm{ and } \, \pi^*_{g_{IJ}}: \tilde{\rho}_{g_{IJ}}^* \ccH^1_{2,\tilde{\beta}} \xrightarrow{\sim} \ccH^1_{1,\tilde{\beta}},\]
 where $\ccH^1_{i,\tilde{\beta}}$ on $S^\prime_i = \widetilde{Y}_{U^\prime_{i,IJ}}(G^\prime_{IJ})_{\bF}$, for $i =1,2$, is the respective de Rham cohomology sheaf at $\tilde{\beta}$. These morphisms induce a morphism:
 \[ \prod_{\beta \in R} \mathbb{P}^1_{S^\prime_1}(\ccH^1_{1,\tilde{\beta}}) \to \prod_{\beta \in R} \mathbb{P}^1_{S^\prime_2}(\ccH^1_{2,\tilde{\beta}}),\]
 which we also denote by $\tilde{\rho}_{g_{IJ}}$. One can check that the diagrams 
 \[
 \begin{tikzcd}
 S^\prime_{1,\phi^\prime(I),J} \arrow[d,"\tilde{\rho}_g"]  \arrow[r, " \widetilde{\Psi}^\prime_{1,IJ}"] & \prod_{\beta \in R} \mathbb{P}^1_{S^\prime_1}(\ccH^1_{1,\tilde{\beta}}) \arrow[d,"\tilde{\rho}_{g_{IJ}}"] &
 \prod_{\beta \in R} \mathbb{P}^1_{S^\prime_1}(\ccH^1_{1,\tilde{\beta}}) \arrow[r, " \widetilde{\Xi}^\prime_{1,IJ}"] \arrow[d,"\tilde{\rho}_{g_{IJ}}"] & S^\prime_{1,\phi^\prime(I),J} \arrow[d,"\tilde{\rho}_g"]
 \\
S^\prime_{2,\phi^\prime(I),J} \arrow[r, " \widetilde{\Psi}^\prime_{2,IJ}"] & \prod_{\beta \in R} \mathbb{P}^1_{S^\prime_2}(\ccH^1_{2,\tilde{\beta}}) &
 \prod_{\beta \in R} \mathbb{P}^1_{S^\prime_2}(\ccH^1_{2,\tilde{\beta}}) \arrow[r, " \widetilde{\Xi}^\prime_{2,IJ}"] & S^\prime_{2,\phi^\prime(I),J},
 \end{tikzcd}
 \]
 commute. Taking quotients by the action of $\ccO_{F,(p),+}^\times$ yields commutative diagrams:
 \begin{equation}\label{Isomorphism hecke compatibility unitary}
 \begin{tikzcd}
 \overline{Y}^\prime_{1,0}(\fp)_{\phi^\prime(I),J} \arrow[d,"\rho_g"]  \arrow[r, " \Psi^\prime_{1,IJ}"] & \prod_{\beta \in R} \mathbb{P}^1_{\overline{Y}^\prime_{1,IJ}}(\ccH^1_{1,\tilde{\beta}}) \arrow[d,"\rho_{g_{IJ}}"] &
 \prod_{\beta \in R} \mathbb{P}^1_{\overline{Y}^\prime_{1,IJ}}(\ccH^1_{1,\tilde{\beta}}) \arrow[r, " \Xi^\prime_{1,IJ}"] \arrow[d,"\rho_{g_{IJ}}"] & \overline{Y}^\prime_{1,0}(\fp)_{\phi^\prime(I),J} \arrow[d,"\rho_g"]
 \\
\overline{Y}^\prime_{2,0}(\fp)_{\phi^\prime(I),J} \arrow[r, " \Psi^\prime_{2,IJ}"] & \prod_{\beta \in R} \mathbb{P}^1_{\overline{Y}^\prime_{2,IJ}}(\ccH^1_{2,\tilde{\beta}}) &
 \prod_{\beta \in R} \mathbb{P}^1_{\overline{Y}^\prime_{2,IJ}}(\ccH^1_{2,\tilde{\beta}}) \arrow[r, " \Xi^\prime_{2,IJ}"] & \overline{Y}^\prime_{2,0}(\fp)_{\phi^\prime(I),J}.
 \end{tikzcd}
 \end{equation}
 
Writing $\mathcal{V}_{\tilde{\beta}}$ for the descent of $\ccH^1_{\tilde{\beta}}$ to $\overline{Y}^\prime_{IJ}$, we now have the analogue of the main theorem at the level of Unitary Shimura varieties:

\begin{thm}\label{unitary thm}
For each sufficiently small open compact subgroup $U^\prime \subset G_\Sigma^\prime(\mathbb{A}_f)$ of level prime to $p$ and $I,J \subset \Theta_{F,\Fp} \setminus \Sigma_{\infty,\Fp}$ such that $(I_{\fp},J_{\fp}) \neq (\Theta_{F,\fp} \setminus \Sigma_{\infty,\fp},\Theta_{F,\fp} \setminus \Sigma_{\infty,\fp})$ for all $\fp \vert \Fp$, there is an isomorphim 
 \[\Psi^\prime_{IJ}: \overline{Y}^\prime_0(\Fp)_{\phi^\prime(I),J} \xrightarrow{\sim} \prod_{\beta \in R} \mathbb{P}^1_{\overline{Y}^\prime_{IJ}}(\mathcal{V}_{\tilde{\beta}}).\]
 Furthermore, these isomorphisms are compatible with the Hecke action in the sense that if $g \in G^\prime_{\Sigma}(\mathbb{A}_f)$ is such that $g^{-1} U^\prime_1 g \subset U^\prime_2$, then the diagrams \ref{Isomorphism hecke compatibility unitary} commute.
\end{thm}
 
 \subsubsection{The Quaternionic and Hilbert case}

We now show how to prove the main theorem from theorem \ref{unitary thm}. Recall that for a quaternion algebra $B$ ramified at a set of places $\Sigma$, we had defined the the quaternionic Shimura variety $Y_{U_\Sigma}(G_\Sigma)$ as the fiber product $Y_{U_\Sigma}(G_\Sigma) = Y_{U^\prime_\Sigma}(G^\prime_\Sigma) \times_{C^\prime} C$, where $C = C_{\det U_\Sigma}$ indexes the set of components of $\overline{Y}_\Sigma = Y_{U_\Sigma}(G_\Sigma)_{\overline{\bF}_p}$ and $C^\prime = C^\prime_{\nu^\prime(U^\prime_\Sigma)}$ indexes the components of $\overline{Y}^\prime_\Sigma = Y_{U^\prime_\Sigma}(G^\prime_\Sigma)_{\overline{\bF}_p}$. Similarly, we defined $Y_{U_{\Sigma,0}(\fp)}(G_\Sigma) = Y_{U^\prime_{\Sigma,0}(\fp)}(G^\prime_\Sigma) \times_{C^\prime} C$

Keep the notation from Theorem \ref{unitary thm}. Instead of trying to show that the isomorphism of Theorem \ref{unitary thm} is compatible with the natural projections, we opt to use the strategy of \cite[$\mathsection$5.3]{2020arXiv200100530D}. Consider the diagram
\[
\begin{tikzcd}
{\overline{Y}^\prime_0(\Fp)_{\phi^\prime(I),J}} \arrow[d, "\Psi^\prime_{IJ}"'] \arrow[r]     & \overline{Y}^\prime_0(\fp) \arrow[r] & \overline{Y}^\prime \arrow[r]      & C^\prime           \\
\prod_{\beta \in R}\mathbb{P}^1_{\overline{Y}^\prime_{IJ}}(\ccH^1_{\tilde{\beta}}) \arrow[rr] &                                      & \overline{Y}^\prime_{IJ} \arrow[r] & C^\prime, \arrow[u]
\end{tikzcd}
\]

where the rightmost arrow is an automorphism of $C^\prime$ that makes the diagram commute. The same arguments as in \cite{2020arXiv200100530D} show that there exists an element $t \in T^\prime(\mathbb{A}_f^{(p)})$, not unique but independent of $U$, such that we can pick the rightmost arrow to be multiplication $t^{-1}$. Recall that $T^\prime = (T_F \times T_E)/T_F$ where $T_F$ is embedded via $x \mapsto (x^2,x^{-1})$. We may thus choose an element $u \in T_E(\mathbb{A}^{(p)}_f)$ such that $(1,u)t^{-1}$ is in the image of the embedding $T_F \to T^\prime$ given by $x \mapsto (x,1)$. Therefore we obtain the commutative diagram
\begin{equation}\label{deriving the quaternionic diagram}
\begin{tikzcd}
\overline{Y}^\prime_0(\fp)_{\phi^\prime(I),J} \arrow[r,"\rho_u"] \arrow[d] & \overline{Y}^\prime_0(\fp)_{\phi^\prime(I),J} \arrow[d] \arrow[r,"{\Psi^\prime_{IJ}}"] & \prod_{\beta \in R}\mathbb{P}^1_{\overline{Y}^\prime_{IJ}}(\ccH^1_{\tilde{\beta}}) \arrow[d] \\
C^\prime \arrow[r,"{\cdot (1,u)}"] & C^\prime \arrow[r,"\cdot t^{-1}"] & C^\prime,
\end{tikzcd}
\end{equation}
where the downwards arrows are the natural ones given by the rows of the previous diagram. 

Let $U \subset G_\Sigma(\mathbb{A}_f)$ be a sufficiently small open compact with level prime to $p$ and write $U_{IJ} = \xi^{-1}(U) \subset G_{IJ}(\mathbb{A}_f)$, which is also a sufficiently small open compact of level prime to $p$. Let $V_E \subset \mathbb{A}_{E,f}^\times$ be an open compact subgroup, of level prime to $p$, which is sufficiently small relative to $U$ and so sufficiently small relative to $U_{IJ}$. Note that if we set $U^\prime$ to be the image of $U \times V_E$ in $G_\Sigma^\prime(\mathbb{A}_f)$, then $U^\prime_{IJ} = \xi^{-1}(U^\prime) \subset G^\prime_{IJ}(\mathbb{A}_f)$ is also the image of $U_{IJ} \times V_E$, and $U^\prime$ and $U^\prime_{IJ}$ are sufficiently small. Recall that we had over $\overline{\bF}_p$, for $\overline{Y}_0(\Fp)_{\phi^\prime(I),J} = Y_{U_0(\Fp)}(G_\Sigma)_{\phi^\prime(I),J,\overline{\bF}_p}$, the Cartesian diagram \ref{compatibility components unitary quaternionic} 
\[
\begin{tikzcd}
   \overline{Y}_0(\Fp)_{\phi^\prime(I),J} \arrow[r,"i"] \arrow[d] & \overline{Y}^\prime_0(\Fp)_{\phi^\prime(I),J} \arrow[d] \\
   C_{\det(U)} \arrow[r] & C_{\nu^\prime(U^\prime)}.
\end{tikzcd}
\]
Similarly, for $\overline{Y}_{IJ} = Y_{U_{IJ}}(G_{IJ})_{\overline{\bF}_p}$, we defined in section \ref{quaternion bundles}, the automorphic vector bundles $\mathcal{V}_\beta = \ccH^1_\beta = i_{IJ}^* \ccH^1_{\tilde{\beta}}$ so we obtain a Cartesian diagram 
\[
\begin{tikzcd}
   \prod_{\beta \in R} \mathbb{P}_{\overline{Y}_{IJ}}(\mathcal{V}_\beta) \arrow[r,"i_{IJ}"] \arrow[d] &
   \prod_{\beta \in R} \mathbb{P}_{\overline{Y}^\prime_{IJ}}(\ccH^1_{\tilde{\beta}}) \arrow[d] \\
   C_{\det(U)} \arrow[r] & C_{\nu^\prime(U^\prime)}.
\end{tikzcd}
\]

The bottom arrow of diagram \ref{deriving the quaternionic diagram}, restricted to $C_{\det(U)}$ is given by multiplication by $v$ for some $v \in (\mathbb{A}^{(p)}_{F,f})^\times$. Taking the fiber product of the morphisms $\Psi^\prime_{IJ} \circ \rho_u: \overline{Y}^\prime_0(\Fp)_{\phi^\prime(I),J} \to \prod_{\beta \in R} \mathbb{P}_{\overline{Y}^\prime_{IJ}}(\ccH^1_{\tilde{\beta}})$ and $C_{\det(U)} \xrightarrow{\cdot v} C_{\det(U)}$ over $C_{\nu^\prime(U^\prime)} \xrightarrow{\cdot v} C_{\nu^\prime(U^\prime)}$ yields an isomorphism 
\[\Psi_{IJ}: \overline{Y}_0(\Fp)_{\phi^\prime(I),J} \to \prod_{\beta \in R} \mathbb{P}_{\overline{Y}_{IJ}}(\mathcal{V}_\beta).\]

Furthermore, if $U_1$ and $U_2$ are two sufficiently small open compact subgroups of $G(\mathbb{A}_f)$ of level prime to $p$, and $g \in G(\mathbb{A}_f^{(p)})$ is such that $g^{-1} U_1 g \subset U_2$, writing $g_{IJ} = \xi^{-1}(g) \subset G_{IJ}(\mathbb{A}_f^{(p)})$ and assuming that $V_{E,1} \subset V_{E,2}$, the morphisms
\[ \overline{Y}_{1,0}(\Fp)_{\phi^\prime(I),J} \xrightarrow{\rho_g} \overline{Y}_{2,0}(\Fp)_{\phi^\prime(I),J} \quad \textrm{and} \quad \prod_{\beta \in R} \mathbb{P}_{\overline{Y}_{1,IJ}}(\mathcal{V}_{1,\beta}) \xrightarrow{g_{IJ}} \prod_{\beta \in R} \mathbb{P}_{\overline{Y}_{2,IJ}}(\mathcal{V}_{2,\beta})\]
are the restrictions of the ones obtained by the images of $g $ in $G^\prime(\mathbb{A}_f)$ and of $g_{IJ}$ in $G^\prime_{IJ}(\mathbb{A}_f)$. Since $u$ is central in $g $ in $G^\prime(\mathbb{A}^{(p)}_f)$, the morphism $\rho_u$ commutes with $\rho_g$ and we obtain the commutative diagram:
\begin{equation}\label{finally}
\begin{tikzcd}
\overline{Y}_{1,0}(\Fp)_{\phi^\prime(I),J} \arrow[r,"\Psi_{1,IJ}"] \arrow[d,"\rho_g"] & \prod_{\beta \in R} \mathbb{P}_{\overline{Y}_{1,IJ}}(\mathcal{V}_{1,\beta}) \arrow[d,"\rho_{g_{IJ}}"]\\
 \overline{Y}_{2,0}(\Fp)_{\phi^\prime(I),J}  \arrow[r,"\Psi_{2,IJ}"] & \prod_{\beta \in R} \mathbb{P}_{\overline{Y}_{2,IJ}}(\mathcal{V}_{2,\beta}).
\end{tikzcd}
\end{equation}
In particular, setting $U_1=U_2$ and $g=1$, we see that $\Psi_{IJ}$ is independent of the choice $V_E$. $\Psi_{IJ}$ does however depend on all of the other auxiliary data we chose, including the choice of models. We obtain for any $I,J \subset \Theta_{F,\Fp} \setminus \Sigma_{\infty,\Fp}$ such that $I \cup J = \Theta_{F,\Fp} \setminus \Sigma_{\infty,\Fp}$ and $(I_\fp,J_\fp) \neq (\Theta_{F,\fp} \setminus \Sigma_{\infty,\fp},\Theta_{F,\fp} \setminus \Sigma_{\infty,\fp})$ for all primes $\fp \vert \Fp$:

 \begin{thm}\label{main theorem quat}
For every sufficiently small open compact subgroup $U \subset G_\Sigma(\mathbb{A}_f)$ of level prime to $p$, there is an isomorphism 
\[\Psi_{IJ}: \overline{Y_0}(\Fp)_{\phi^\prime(I),J} \xrightarrow{\sim} \prod_{\beta \in R} \mathbb{P}^1_{\overline{Y}_{IJ}}(\mathcal{V}_\beta).\]
Furthermore, these isomorphisms, ranging over $U$, are compatible with the Hecke action on both sides in the sense that if $g \in G_{\Sigma}(\mathbb{A}_f)$ is such that $g^{-1} U_1 g \subset U_2$, then diagram \ref{finally} commutes.
\end{thm}
 
Recall that by proposition \ref{Iwahori goren relation} and the following discussion, that there is a morphism $\psi_\Fp: \overline{Y} \to \overline{Y}_0(\Fp)$ which restricts to a Hecke equivariant isomorphism $\psi_\Fp: \overline{Y}_J \to \overline{Y}_{(\Theta_{F,\Fp} \setminus \Sigma_{\infty,\Fp}),J}$ for every $J \subset \Theta_{F,\Fp} \setminus \Sigma_{\infty,\Fp}$. For every $J$ such that $J_\fp \neq \Theta_{F,\fp} \setminus \Sigma_{\infty,\fp}$ at all primes $\fp \vert \Fp$, we write $T = T(J)$ and $R(J)$ for the sets obtained from the pair $(\Theta_{F,\Fp} \setminus \Sigma_{\infty,\Fp},J)$, and $\overline{Y}_{\Sigma_{T(J)}} = \overline{Y}_{IJ}$, not to be confused with a Goren-Oort stratum of $\overline{Y}$. We thus obtain a morphism $\Psi_J = \Psi_{IJ} \circ \psi_\Fp: \overline{Y}_J \to \prod_{\beta \in R(J)} \mathbb{P}^1_{\overline{Y}_{\Sigma_{T(J)}}}(\mathcal{V}_\beta)$ which is Hecke equivariant in the sense that for $g$, $U_1$ and $U_2$ as above, the following diagram commutes:
\begin{equation}\label{hecke tame level}
\begin{tikzcd}
    \overline{Y}_{1,J} \arrow[d,"{\rho_g}"] \arrow[r,"{\Psi_{1,J}}"] & \prod_{\beta \in R(J)} \mathbb{P}^1_{\overline{Y}_{1,\Sigma_{T(J)}}}(\mathcal{V}_\beta) \arrow[d,"{\rho_g}"] \\
    \overline{Y}_{2,J} \arrow[r,"{\Psi_{2,J}}"] & \prod_{\beta \in R(J)} \mathbb{P}^1_{\overline{Y}_{2,\Sigma_{T(J)}}}(\mathcal{V}_\beta).
\end{tikzcd}
\end{equation}

We now have our main theorem at tame level:

\begin{thm}\label{main thm quat tame}
For every sufficiently small open compact subgroup $U \subset G_\Sigma(\mathbb{A}_f)$ of level prime to $p$, there is an isomorphism 
\[\Psi_J: \overline{Y}_J \xrightarrow{\sim} \prod_{\beta \in R(J)} \mathbb{P}^1_{\overline{Y}_{\Sigma_{T(J)}}}(\mathcal{V}_\beta).\]
Furthermore, these isomorphisms, ranging over $U$, are compatible with the Hecke action on both sides in the sense that if $g \in G_{\Sigma}(\mathbb{A}^{(p)}_f)$ is such that $g^{-1} U_1 g \subset U_2$, then diagram \ref{hecke tame level} commutes.
\end{thm}

Finally, combining the above theorems with the compatibilities established in sections \ref{model comparison}, \ref{quaternion bundles}, \ref{quaternionic goren oort} and \ref{quaternion hilb iwahori}, we obtain the analogous result for Hilbert modular varieties: Let $G = \textrm{Res}_{F/\bQ} \, \textrm{GL}_2$ and $U \subset G(\mathbb{A}_f)$ be a sufficiently small open compact subgroup of level prime to $p$. Write $\overline{Y} = Y_U(G)_{\overline{\bF}_p}$, and $\overline{Y}_0(\Fp) = Y_{U_0(\Fp)}(G)_{\overline{\bF}_p}$. For any $I,J \subset \Theta_{F,\Fp}$ such that $I \cup J = \Theta_{F,\Fp} $ and $(I,J) \neq (\Theta_{F,\fp},\Theta_{F,\fp})$ at all primes $\fp \vert \Fp$, we have:

 \begin{thm}\label{hilb thm iwahori}
For every sufficiently small open compact subgroup $U \subset G(\mathbb{A}_f)$ of level prime to $p$, there is an isomorphism 
\[\Psi_{IJ}: \overline{Y_0}(\Fp)_{\phi(I),J} \xrightarrow{\sim} \prod_{\beta \in R} \mathbb{P}^1_{\overline{Y}_{IJ}}(\mathcal{V}_\beta).\]
Furthermore, these isomorphisms, ranging over $U$, are compatible with the Hecke action on both sides in the sense that if $g \in G_{\Sigma}(\mathbb{A}^{(p)}
_f)$ is such that $g^{-1} U_1 g \subset U_2$, the analogue of diagram \ref{finally} commutes.
\end{thm}
Similarly, for every subset $J \subset \Theta_{F,\Fp}$ such that $J_\fp \neq \Theta_{F,\fp}$ at all primes $\fp \vert \Fp$, we have:
\begin{thm}\label{hilb thm tame}
For every sufficiently small open compact subgroup $U \subset G(\mathbb{A}_f)$ of level prime to $p$, there is an isomorphism 
\[\Psi_J: \overline{Y}_J \xrightarrow{\sim} \prod_{\beta \in R(J)} \mathbb{P}^1_{\overline{Y}_{\Sigma_{T(J)}}}(\mathcal{V}_\beta).\]
Furthermore, these isomorphisms, ranging over $U$, are compatible with the Hecke action on both sides in the sense that if $g \in G_{\Sigma}(\mathbb{A}^{(p)}_f)$ is such that $g^{-1} U_1 g \subset U_2$, then the analogue of diagram \ref{hecke tame level} commutes.
\end{thm}
 
\subsubsection{The bottom strata}\label{bottom strata}
Consider now, $\Fp = \fp$ for simplicity, and the case $I \cap J = \Theta_{F,\fp} \setminus \Sigma_{\infty,\fp}$ of even cardinality $2k$. Arbitrarily pick a $\beta \notin \Sigma_{\infty,\fp}$ and set $T=\{\phi^{\prime 2n-1}(\beta) \, \vert \, 1 \leq n \leq \vert k \setminus \Sigma_{\infty,\fp}\}$. Then we can, as usual, form the $T$-splice $A_{1,IJ}$ of the tuple $(\underline{A_1},\underline{A_2},f,g)$. However, we can also define the $T$-splice $A_{2,IJ}$ of the tuple $(\underline{A_2},\underline{A_1 \otimes \fp^{-1}},g,f\otimes \fp^{-1})$, since $I = J = \Theta_{F,\fp} \setminus \Sigma_{\infty,\fp}$. Now, crucially, since $\Theta_{F,\fp} \setminus \Sigma_{\infty,\fp}$ is even, $\Sigma_{IJ} = \Sigma \sqcup \Theta_{F,\fp} \setminus \Sigma_{\infty,\fp}$ is of even cardinality, so we may speak of the quaternion algebra attached to it; this fails if the cardinality of $\Theta_{F,\fp} \setminus \Sigma_{\infty,\fp}$ is odd.
 $A_{1,IJ}$ and $A_{2,IJ}$ sit in the following commutative diagram:
 \[
 \begin{tikzcd}
     A_2 \otimes \fp \arrow[d,"g \otimes \fp"'] \arrow[r] & A_{1,IJ} \arrow[d, "f_{IJ}"] \arrow[r] & A_2 \arrow[d,"g"]\\
     A_1  \arrow[d,"f"'] \arrow[r] & A_{2,IJ} \arrow[d,"g_{IJ}"] \arrow[r] & A_1 \otimes \fp^{-1} \arrow[d,"f \otimes \fp^{-1}"]\\
     A_2 \arrow[r] & A_{1,IJ} \otimes \fp^{-1} \arrow[r] & A_2 \otimes \fp^{-1},
 \end{tikzcd}
 \]
\noindent We can thus define a morphism $\widetilde{Y}^\prime_0(\fp)_{\phi^\prime(I),J} \to \widetilde{Y}_{U^\prime_{IJ,0}(\fp)}(G^\prime_{IJ})_{\bF}$, where $G^\prime_{IJ}$ is the group attached to $\Sigma_{IJ}$, by mapping $(\underline{A_1},\underline{A_2},f,g)$ to $(\underline{A_{IJ,1}},\underline{A_{2,IJ}},f_{IJ},g_{IJ})$. One may check that this is an isomorphism by showing that it is a bijection between geometric points and injective on tangent spaces.
 
For the final case, $I \cap J = \Sigma_{\infty,\fp} \setminus \Sigma_\infty$ with $\Theta_{F,\fp} \setminus \Sigma_\infty$ of odd cardinality, we need to change the definition of our PEL datum. We take $E/F$ as usual, except that we require it to be inert above $\fp$ instead of split, with unique prime $\fp \ccO_E = \fq$ above it. In this situation, we define our Unitary Shimura variety similarly by taking our lattice $\ccO_D$ to satisfy $\ccO_{D,\fp} = M_2(\ccO_{E_\fq})$; there are no issues defining the Pappas-Rapoport filtrations, requiring it to be self dual as defined in section \ref{dual filtration}, and there is no issue defining the Iwahori level model, the only difference being that $H[\fp] = H[\fq]$ is totally isotropic of rank $p^{2f_\fq}=p^{4f_\fp}$. In this situation, we may also consider the quaternion algebra given by $\Sigma_{IJ} = \Sigma \sqcup  \Theta_{F,\fp} \setminus \Sigma_{\infty,\fp} \sqcup \{\fp\}$ which, by assumption, is even. To define the tame level Unitary Shimura variety attached to it, we must however consider, as in \cite{tian_xiao_2016}, the following modification: We take everything as usual, except we require $\lambda: A \to A^\vee$ to have nontrivial kernel of rank $p^{2f_\fp} \subset A[\fp]$, and that the cokernel $\lambda_*: \ccH^1_{\dR}(A^\vee/S) \to  \ccH^1_{\dR}(A/S)$ be locally free of rank 2 over $(\ccO_E/\fq) \otimes \ccO_S$. Note that there are no issues with defining a suitable duality condition for the Pappas-Rapoport filtrations at $\fp$ because $\Sigma_{IJ,\infty,\fp} = \Theta_{F,\fp}$ by assumption. Taking $A_{IJ} = A_1/C^\prime_T$ for a maximally spaced subset $T \subset \widehat{\Theta}_{E,\fq}$ with the polarization $\lambda_{IJ} = \psi_2^\vee \circ \lambda_2 \circ \psi_2$ where $\psi_2$ is the natural map $A_{IJ} \to A_2$ yields the desired isomorphism.

 \subsection{Comparing Automorphic Vector Bundles}
 We finish by relating the Raynaud bundles on $\overline{Y}_0(\Fp)_{\phi^\prime(I),J}$ to automorphic vector bundles and tautological line bundles on $\prod_{\beta \in R} \mathbb{P}_{\overline{Y}_{IJ}}(\ccH^1_{\beta})$
 under the isomorphism $\Psi_{IJ}$. We also provide similar isomorphisms at the tame level. In particular, we work over geometric special fibers.

 \subsubsection{The Iwahori level Unitary setting}\label{Iwahori unitary bundles}
 We start by showing the analogous relations at the Unitary level. Write $\Fp=\fp_1 \cdots \fp_k$ for a non-empty prod of distinct primes above $p$. We start at Iwahori level; fix subsets $I,J \subset \Theta_F \setminus \Sigma_\infty$ such that $I \cup J = \Theta_F \setminus \Sigma_\infty$ and $(I_\fp,J_\fp) \neq (\Theta_{F,\fp} \setminus \Sigma_{\infty,\fp}, \Theta_{F,\fp} \setminus \Sigma_{\infty,\fp})$ for all primes $\fp$ over $p$, and write $S^\prime_{IJ} = \widetilde{Y}^\prime_0(\Fp)_{\phi^\prime(I),J,\bF}$ and $S = \widetilde{Y}^\prime_{IJ,\bF}$. 
 
 By theorem \ref{unitary iwahori unquotiented thm}, we have an isomorphism $\widetilde{\Psi}^\prime_{IJ}: S^\prime_{IJ} \xrightarrow{\sim }\prod_{\beta \in R} \mathbb{P}_S(\ccH^1_\beta)$ which we recall fits in the following diagram:

\begin{center}
    \begin{tikzcd}
        S^\prime_{IJ} \arrow[r, "\widetilde{\Psi}^\prime_{IJ}"]  \arrow[rd, "\widetilde{\psi}^\prime_{IJ}"'] & \prod_{\beta \in R} \mathbb{P}_S(\ccH^1_\beta) \arrow[d] \\
        & S.
    \end{tikzcd}
\end{center}

 On $S^\prime_{IJ}$, we have for all $\beta \in \Theta_F$ and $j=1,2$ the rank two locally free sheaves $\ccH^1_{j,\tilde{\beta}}$, whose determinant we write $\delta_{j,\tilde{\beta}}$. If furthermore $\beta \notin \Sigma_\infty$, we have the line bundles $\omega^0_{j,\tilde{\beta}}$ and $v^0_{j,\tilde{\beta}}$ and a short exact sequence $0 \to \omega^0_{j,\tilde{\beta}} \to \ccH^1_{j,\tilde{\beta}} \to v^0_{j,\tilde{\beta}} \to 0$. Over $S$, we have similar bundles $\ccH^1_{\tilde{\beta}}$ and $\delta_{\tilde{\beta}}$, and line bundles $\omega^0_{\tilde{\beta}}$ and $v^0_{\tilde{\beta}}$ for all $\beta \notin \Sigma_{IJ,\infty} = \Sigma_{\infty} \sqcup R \sqcup (I \cap J)$. We use the same notation to denote their pullbacks to $\prod_{\beta \in R} \mathbb{P}_S(\ccH^1_\beta)$. Finally, over $\prod_{\beta \in R} \mathbb{P}_S(\ccH^1_\beta)$, we have for every $\beta$ in $R$, the line bundle $\ccO(1)_\beta = \pi_\beta^* \ccO(1)$ where $\pi_\beta$ denotes the projection to the $\beta$-component of the fiber product.

By corollary \ref{local de rham iso} and the definition of $\widetilde{\psi}^\prime_{IJ}$, we have isomorphisms $\ccH^1_{2,\tilde{\beta}} \xrightarrow{\sim} \widetilde{\psi}^{\prime *}_{IJ}\ccH^1_{\tilde{\beta}}$ for $\beta \in T^\prime$ and $\ccH^1_{1,\tilde{\beta}} \xrightarrow{\sim} \widetilde{\psi}^{\prime *}_{IJ}\ccH^1_{\tilde{\beta}}$ for $\beta \notin T^\prime$, where we recall that $\beta \in T^\prime $ if and only if $\beta \in T$ or $\beta \in \Sigma_\infty$ and $\phi^\prime(\beta) \in T$. Furthermore, this isomorphism is compatible with the Hodge filtration if $\beta \notin \Sigma_{IJ}$. We thus obtain for $\beta \notin \Sigma_{IJ}$ and the appropriate $j$,
\begin{equation}\label{comparison not in R}
    \omega^0_{j,\tilde{\beta}} \simeq \widetilde{\psi}^{\prime *}_{IJ} \omega^0_{\tilde{\beta}} \, \textrm{ and } \, v^0_{j,\tilde{\beta}} \simeq \widetilde{\psi}^{\prime *}_{IJ} v^0_{\tilde{\beta}}  \simeq \widetilde{\psi}^{\prime *}_{IJ} \delta_{\tilde{\beta}} \otimes \widetilde{\psi}^{\prime *}_{IJ} (\omega^0_{\tilde{\beta}})^{-1}.
\end{equation}

 If $\beta \in R$ however, the morphism $S^\prime_{IJ} \to \mathbb{P}_S(\ccH^1_{\beta})$ was defined in terms of the Hodge filtration of $\ccH^1_{j,\tilde{\beta}}$ for the appropriate $j$ and so induces isomorphisms
 \begin{equation}\label{comparison in R} \omega^0_{j,\tilde{\beta}} \simeq \widetilde{\psi}^{\prime *}_{IJ} \delta_{\tilde{\beta}} \otimes \widetilde{\Psi}^{\prime *}_{IJ} \ccO(-1)_\beta \, \textrm{ and } \, v^0_{j,\tilde{\beta}} \simeq \widetilde{\Psi}^{\prime *}_{IJ} \ccO(1)_\beta.
 \end{equation}

Let $H$ denote the reduced universal kernel restricted to $S^\prime_{IJ}$, recall that the Raynaud bundles $\ccL_\tau$, for $\tau \in \widehat{\Theta}_{E,\Fp}$ are given by the identification 
\[H = \underline{\textrm{Spec}}_{S^\prime_{IJ}} (  \textrm{Sym}_{\ccO_{S^\prime_{IJ}}}(\bigoplus_{\tau \in \widehat{\Theta}_{E,\Fp}} \ccL_\tau) \langle (s_\tau-1 )\ccL_\tau \rangle_{\tau \in \widehat{\Theta}_{E,\Fp}}). \]
Furthermore, we have $\omega_{H,\tau} = \ccL_{\tau}/ s_{\phi^{-1} \circ \tau}\ccL^p_{\phi^{-1}(\tau)}$ and $\omega_{H,\tau}$ is canonically isomorphic to $\omega^0_{1,\tau}/f^*\omega^0_{2,\tau}$. Let $\theta \in \widehat{\Theta}_F$ and write $\tau = \tilde{\theta}$. 
Suppose that there exists $i$, which we take maximal, such that $s_{\tau^i}=1$.
If $\theta^i \in I$, we deduce as in the proof of lemma \ref{raynaud vanishing} that 
\[\ccL_\tau \simeq \omega_{H,\tau} \simeq \omega^0_{1,\tau}/f^*\omega^0_{2,\tau} \simeq \omega_{1,\tau^i}.\]
If $\tau^i \in J$ however, recall that $\lambda_1$ induces an $\ccO_E$-antilinear isomorphism $H \xrightarrow{\sim} (\ker g)^\vee$,
where $e_0\ker g$ and its dual are also Raynaud $\ccO_E/\Fp$-module schemes, given by the Raynaud data $(\mathcal{N}_\tau,s^\prime_\tau,t^\prime_\tau)$
and $(\mathcal{N}^\vee_\tau,t^{\prime \vee}_\tau,s^{\prime \vee}_\tau)$ respectively.
We deduce as above that $\mathcal{N}_{\tau^c} \simeq \omega^0_{2,(\tau^c)^i}$ and we have an isomorphism 
\[ \ccL_\tau \xrightarrow[\sim]{\lambda^{-1}_1} \mathcal{N}^\vee_{\tau^c} \simeq (\omega^0_{2,(\tau^c)^i})^\vee \xrightarrow[\sim]{\lambda^{\vee}_2} v^0_{2,\tau^i}\]
which does not depend on the specific choice of polarizations $\lambda_1$ and $\lambda_2$ under the action of $\ccO_{F,(p),+}^\times$.

Suppose now that $\theta$ is such that $\theta^i \in \Sigma_\infty$ for all $1 \leq i \leq e_\fp$.
If $\omega^0_{i,\tau} \neq 0$, 
then $\omega^0_{i,\tau} = [\varpi_\fp]^d\ccH^1_{\dR}(A_i/S^\prime_{IJ})^0_\tau$ and we deduce that 
\[\ccL_\tau \simeq \omega^0_{1,\tau}/f^*\omega^0_{2,\tau}
\xrightarrow[\sim]{V_{\es}} \ccH^1_{1,\tau^1}/f^*\ccH^1_{2,\tau^1}.\]
Finally, if $\omega^0_{i,\tau}=0$, we similarly have $\mathcal{N}^\vee_{\tau^c} \simeq \ccH^1_{2,(\tau^c)^1} /g^*\ccH^1_{A_1 \otimes \Fp^{-1},\tau^c} \simeq f^*\ccH^1_{2,(\tau^c)^1}$ and an isomorphism
\[ \ccL_\tau \xrightarrow[\sim]{\lambda^{-1}_1} \mathcal{N}^\vee_{\tau^c} \simeq (f^*\ccH^1_{2,(\tau^c)^1})^\vee \xrightarrow[\sim]{\lambda^{\vee}_1} \ccH^1_{1,\tau^1}/f^*\ccH^1_{2,\tau^1}\]
which does not depend on the specific choice of polarizations. 
In either case, we obtain the isomorphism $\ccL_\tau \xrightarrow{\sim} \ccH^1_{1,\tau^1}/f^*\ccH^1_{2,\tau^1}$. 
Write  $\beta^\prime = (\phi^{\prime})^{-1}(\theta^1)$. Suppose that $\beta^\prime \in I$. 
Then $f^*\ccH^1_{2,\phi(\tilde{\beta}^\prime)} = F_{\es,\phi(\tilde{\beta}^\prime)}(\ccH^{1 \, (p^\delta)}_{1,\tilde{\beta}^\prime})$
and $V_{\es,\tau^1}^{\tilde{\beta}^\prime}$ induces an isomorphism 
$\ccH^1_{1,\tau^1}/f^*\ccH^1_{2,\tau^1} \simeq (\omega^0_{1,\tilde{\beta}^\prime})^{p^{n^\prime_\theta}}$ 
for the appropriate $n^\prime_\theta$. If $\beta^\prime \in J$, then 
\[\ccH^1_{1,\phi(\tilde{\beta}^\prime)}/f^*\ccH^1_{2,\phi(\tilde{\beta}^\prime)} \simeq 
(g\otimes \Fp)^*\ccH^1_{1,\phi(\tilde{\beta}^\prime)} = F_{\es,\phi(\tilde{\beta}^\prime)}(\ccH^{1 \, (p^\delta)}_{A_2 \otimes \Fp,\tilde{\beta}^\prime}) \simeq 
(v^0_{2,\phi(\tilde{\beta}^\prime)})^{p^\delta}\]
and thus $F_{\es,\tau^1}^{\tilde{\beta}^\prime}$ induces an isomorphism 
$\ccH^1_{1,\tau^1}/f^*\ccH^1_{2,\tau^1} \simeq (v^0_{2,\tilde{\beta}^\prime})^{p^{n^\prime_\theta}}$ 
for the appropriate $n^\prime_\theta$.

In summary, for $\tau = \tilde{\theta} \in \widehat{\Theta}_E$ and $\beta^\prime = (\phi^\prime)^{-1}((\phi \circ \theta)^1) = (\phi^{-n^\prime_\theta}\circ \theta)^i$, we obtain the isomorphism 
\begin{equation}\label{Raynaud bundle I J}
    \ccL_\tau \simeq 
    \begin{cases}
        (\omega^0_{1,\tilde{\beta^\prime}})^{p^{n^\prime_\theta}} & \textrm{ if } \beta^\prime \in I,\\
        (v^0_{2,\tilde{\beta^\prime}})^{p^{n^\prime_\theta}} & \textrm{ if } \beta^\prime \in J.
    \end{cases}
\end{equation}
Note that this implies that $\omega^0_{1,\tilde{\beta^\prime}} \simeq v^0_{2,\tilde{\beta^\prime}}$ if $\theta^i \in I \cap J$. This follows from the relations
\[v^0_{2,\tilde{\beta^\prime}} = \ccH^1_{2,\tilde{\beta^\prime}} / \omega^0_{2,\tilde{\beta^\prime}} \simeq f^* \ccH^1_{2,\tilde{\beta^\prime}} = \ker (g \otimes \Fp)^*_{\tilde{\beta^\prime}} = \omega^0_{1,\tilde{\beta^\prime}}.\]
In fact, we record the following isomorphism of line bundles for such $\beta \in I \cap J$:
Recall from the proof of proposition \ref{Iwahori goren relation} that if $\beta \in \Theta_F$ is such that $ \beta \in J$ and $(\phi^\prime)^{-1}(\beta) \in I$, then $V_{\es, \tilde{\beta}}^{(\phi^\prime)^{-1}(\tilde{\beta})}(\omega^0_{1,\tilde{\beta}})=0$. In particular, since the image of $V_{\es, \tilde{\beta}}^{(\phi^\prime)^{-1}(\tilde{\beta})}$ is exactly $(\omega^0_{1,(\phi^\prime)^{-1}(\tilde{\beta})})^{p^n}$ (for the appropriate $n$), we obtain the isomorphism \[v^0_{1,\tilde{\beta}} \simeq \ccH^1_{1,\tilde{\beta}}/ \omega^0_{1,\tilde{\beta}} \simeq (\omega^0_{1,(\phi^\prime)^{-1}(\tilde{\beta})})^{p^n}.\]
Similarly, $F_{\es, \tilde{\beta}}^{(\phi^\prime)^{-1}(\tilde{\beta})}$ then induces an isomorphism 
\[ \omega^0_{1,\tilde{\beta}} \simeq (v^0_{1,(\phi^\prime)^{-1}(\tilde{\beta})})^{p^n}.\]
The same argument shows that if $\beta \in \Theta_F$ is such that $ \beta  \in I$ and $(\phi^\prime)^{-1}(\beta)  \in J$, then we have isomorphisms 
\[v^0_{2,\tilde{\beta}} \simeq (\omega^0_{2,(\phi^\prime)^{-1}(\tilde{\beta})})^{p^n} \quad \textrm{and} \quad \omega^0_{2,\tilde{\beta}} \simeq (v^0_{2,(\phi^\prime)^{-1}(\tilde{\beta})})^{p^n}.\]

It follows if $\beta \in \Theta_F$ and $k \geq 1$ are such that $\beta, (\phi^\prime)^{-1}(\beta), \cdots, (\phi^\prime)^{-(k-1)}(\beta) \in I \cap J$ and $(\phi^\prime)^{-k}(\beta) \in I$, we can chain these isomorphisms and obtain, for the appropriate $n$ depending on $\beta$ and $k$,

\begin{equation}\label{IJ 1}
    \omega^0_{1,\tilde{\beta}} \simeq 
    \begin{cases}
        (\omega^0_{1,(\phi)^{-k}(\tilde{\beta})})^{p^n} & k \textrm{ even}, \\
        (v^0_{1,(\phi)^{-k}(\tilde{\beta})})^{p^n} & k \textrm{ odd}.
    \end{cases}
\end{equation}

Similarly, if $\beta \in \Theta_F$ and $k \geq 1$ are such that $\beta, (\phi^\prime)^{-1}(\beta), \cdots, (\phi^\prime)^{-(k-1)}(\beta) \in I \cap J$ and $(\phi^\prime)^{-k}(\beta) \in J$, we obtain, for the appropriate $n$ depending on $\beta$ and $k$,
\begin{equation}\label{IJ 2}
    v^0_{2,\tilde{\beta}} \simeq 
    \begin{cases}
        (v^0_{2,(\phi)^{-k}(\tilde{\beta})})^{p^n} & k \textrm{ even}, \\
        (\omega^0_{2,(\phi)^{-k}(\tilde{\beta})})^{p^n} & k \textrm{ odd}.
    \end{cases}
\end{equation}

We can now relate the Raynaud bundles $\ccL_\tau$ to line bundles on $\prod_{\beta \in R} \mathbb{P}_S(\ccH^1_\beta)$. Let $\theta$, $\tau$ and $\beta^\prime$ be as in \eqref{Raynaud bundle I J} and set $\beta = (\phi^{\prime \prime})^{-1}((\phi \circ \theta)^1)$ where we recall that $\phi^{\prime \prime}$ denotes the cycle structure with respect to $\Sigma_\infty \sqcup I \cap J$. Write $0 \leq n_\theta <e_\fp$ for the unique integer such that $\beta = (\phi^{-n_\theta} \circ \tau)^i$, for some $i$.

We consider four different cases corresponding to the membership of $\beta$ in $R$ and $\phi^\prime(\beta^\prime)$ in $T$. 
Consider the case that $\theta$ is such that both $\beta \notin R$ and $\phi^\prime(\beta^\prime) \notin T$. 
If $\beta^\prime \notin I \cap J$, then $\beta = \beta^\prime$ and $\beta \notin R$ implies that $\beta \notin T$ since $\phi^\prime(\beta^\prime) = \phi^\prime(\beta) \notin T$. In particular, $\beta \in I$ since $I^c \subset T$, and $\beta \notin \Sigma_{IJ}$. Combining the isomorphisms of \ref{Raynaud bundle I J} and \ref{comparison not in R}, we obtain the isomorphism $\ccL_\tau \simeq \widetilde{\Psi}^{\prime *}_{IJ}(\omega^0_{\tilde{\beta}})^{p^{n_\theta}}$ for the appropriate $n_\theta$.

If $\beta^\prime \in I \cap J$ however, write $\beta = (\phi^\prime)^{-k}(\beta^\prime)$. Recall from the construction of $T$ in section \ref{sets} that for any $\theta^i \in I \cap J$, we have $\theta^i \in T $ if and only if $\phi^\prime(\theta^i) \notin T$. In particular, since $\phi^\prime(\beta^\prime) \notin T$ it follows that $\beta^\prime \in T$. If we suppose furthermore $k$ is odd, then $(\phi^\prime)^{-(k-1)}(\beta^\prime) \in T$ and so $\beta \in T$ since $\beta \notin R$. 
By \ref{Raynaud bundle I J}, we have $\ccL_\tau \simeq (v^0_{2,\tilde{\beta}^\prime})^{p^m}$ (for the appropriate $m$) and, since $(\phi^\prime)^{-(k-1)}(\beta^\prime), \cdots, \beta^\prime \in I \cap J$ and $\beta \in T \subset J$, we also have by \ref{IJ 2} the isomorphism $(v^0_{2,\tilde{\beta^\prime}})^{p^m} \simeq (\omega^0_{2,\tilde{\beta}})^{p^{n_\theta}}$. By \ref{comparison not in R}, we finally obtain the isomorphism $\ccL_\tau \simeq \widetilde{\Psi}^{\prime *}_{IJ}(\omega^0_{\tilde{\beta}})^{p^{n_\theta}}$.
If $k$ is even however, we deduce as above that $\beta \notin T$ and by the isomorphisms \ref{Raynaud bundle I J} and \ref{IJ 1} that $\ccL_\tau \simeq (\omega^0_{1,\tilde{\beta}})^{p^{n_\theta}}$. Therefore by \ref{comparison not in R}, we obtain the isomorphism $\ccL_\tau \simeq \widetilde{\Psi}^{\prime *}_{IJ}(\omega^0_{\tilde{\beta}})^{p^{n_\theta}}$.

In conclusion, if $\theta$ is such that both $\beta \notin R$ and $\phi^\prime(\beta^\prime) \notin T$ then we have the isomorphism $\ccL_\tau \simeq \widetilde{\Psi}^{\prime *}_{IJ}(\omega^0_{\tilde{\beta}})^{p^{n_\theta}}$. The other cases are dealt with similarly. Considering that $\phi^\prime(\beta^\prime) \notin T$ is equivalent to $(\phi \circ \tau)^1 \notin T$ which, by definition, is equivalent to $\phi \circ \tau \notin T^1$, we obtain for any $\theta \in \widehat{\Theta}_F$ the isomorphism

\begin{equation}\label{Raynaud comparison unitary}
    \ccL_\tau \simeq \widetilde{\Psi}^{\prime *}_{IJ}\mathcal{M}_{\tau}^{p^{n_\theta}}, \textrm{ where } \, 
    \mathcal{M}_\tau =
    \begin{cases}
        \omega^0_{\tilde{\beta}} & \textrm{if } \beta \notin R, \phi \circ \theta \notin T^1,\\
        \delta_{\tilde{\beta}}(\omega^0_{\tilde{\beta}})^{-1} & \textrm{if } \beta \notin R, \phi \circ \theta \in T^1,\\
        \delta_{\tilde{\beta}}(-1)_\beta & \textrm{if } \beta \in R, \phi \circ \theta \in T^1,\\
        \ccO(1)_\beta & \textrm{if } \beta \in R, \phi \circ \theta \notin T^1.
    \end{cases}
\end{equation}

\begin{rem}
    Consider the simplest case: $p$ unramified in $F$, $\Sigma = \emptyset$ and $I \cap J = \emptyset$. We thus have $\widehat{\Theta}_F = \Theta_F$, $T = J$, $R = \Sigma_{IJ} = \Sigma_J$ and for each $\theta \in \Theta_F$, $\theta = \beta = \beta^\prime$, and $n_\theta = 0$. The conditions then become, in order, $\theta \notin J$ and $\theta \notin \Sigma_J$, $\theta \in J$ and $\theta \notin \Sigma_J$, $\theta \notin J$ and $\theta \in \Sigma_J$, and $\theta \in J$ and $\theta \in \Sigma_J$. We thus reobtain the isomorphisms of \cite[(17)]{2020arXiv200100530D}.
\end{rem}

 The above isomorphisms are Hecke equivariant in the following sense: Suppose that $U^\prime_1,U^\prime_2 \subset G^\prime_\Sigma(\mathbb{A}_f^{(p)})$ are sufficiently small open compact subgroups and $g \in G^\prime_\Sigma(\mathbb{A}_f^{(p)})$ is such that $g^{-1}U^\prime_1g \subset U^\prime_2$. Write $\tilde{\rho}_g: \widetilde{Y}^\prime_{1,0}(\Fp) \to \widetilde{Y}^\prime_{2,0}(\Fp)$ for the Hecke action and write $H_i$ for the respective reduced universal kernel. Then the canonical quasi-isogeny $\pi_g$ induces an isomorphism $H_1 \to \tilde{\rho}^*_g H_2$ and in particular induces an isomorphism $\tilde{\rho}^*_g \ccL_{2,\tau} \to \ccL_{1,\tau}$ by functoriality of the Raynaud construction. We thus obtain the following commutative diagram:
 \begin{equation}\label{Hecke equivariance raynaud}
 \begin{tikzcd}
     \tilde{\rho}_g^* \ccL_{2,\tau} \arrow[rr, "\tilde{\rho}_g^*(\sigma_2)"] \arrow[d,"\pi_g^*"]&& \tilde{\rho}_g^*  \widetilde{\Psi}_{2,IJ}^{\prime*} \mathcal{M}_{2,\tau}^{p^{n_\theta}} = \widetilde{\Psi}_{1,IJ}^{\prime*} \tilde{\rho}_{g_{IJ}}^* \mathcal{M}_{2,\tau}^{p^{n_\theta}} \arrow[d," \widetilde{\Psi}_{1,IJ}^{\prime *} (\pi^*_{g_{IJ}})"]\\
     \ccL_{1,\tau} \arrow[rr,"\sigma_1"] && \widetilde{\Psi}_{1,IJ}^{\prime*}\mathcal{M}_{1,\tau}^{p^{n_\theta}},
 \end{tikzcd}
 \end{equation}
 where $\sigma_i: \ccL_{i,\tau} \xrightarrow{\sim} \widetilde{\Psi}_{i,IJ}^* \mathcal{M}_{i,\tau}^{p^{n_\theta}}$ is the isomorphism given above and $\pi_{g_{IJ}}$ is the quasi-isogeny that defines the Hecke action.

 If we assume in addition that $U^\prime$ is $\Fp$-neat, 
 i.e. that $\alpha-1 \in \Fp$ for all $\alpha \in U^\prime \cap \ccO_F^\times$, 
 then the action of $( U^\prime \cap \ccO_F^\times)^2$ 
 is trivial on $H$ so that both $H$ and its associated Raynaud data descend to the quotient 
 $\overline{Y}^\prime_{0}(\Fp)_{\phi^\prime(I),J}$. Since the isomorphism $\Psi^\prime_{IJ}: \overline{Y}^\prime_0(\Fp)_{\phi^\prime(I),J} \to \prod_{\beta \in R} \mathbb{P}_{\overline{Y}^\prime_{IJ}}(\ccH^1_\beta)$ 
 from theorem \ref{unitary thm} is obtained by taking the quotient of $\widetilde{\Psi}^\prime_{IJ}$ with respect to the action of $\ccO_{F,(p),+}^\times$ on both sides and none of the isomorphisms above depend on polarizations, we thus obtain isomorphisms 
 $\sigma: \ccL_\tau \xrightarrow{\sim} \Psi^{\prime *}_{IJ}\mathcal{M}_{\tau}^{p^n}$ on $\overline{Y}^\prime_0(\Fp)_{\phi^\prime(I),J}$, 
 where $\mathcal{M}_{\tau}^{p^n}$ is defined using the same formulae as in \ref{Raynaud comparison unitary} but using the quotients, with respect to the action of $\ccO_{F,(p),+}^\times$, of the relevant bundles upstairs. Furthermore, these isomorphisms are Hecke equivariant in the sense that the analogous diagram to \ref{Hecke equivariance raynaud} commutes.

\subsubsection{The Tame level Unitary setting}\label{tame unitary bundles}
We now describe isomorphisms between automorphic vector bundles at Tame level.  Fix a subset $J \subset \Theta_F \setminus \Sigma_\infty$ such that $J_\fp \neq \Theta_{F,\fp} \setminus \Sigma_{\infty,\fp}$ for all primes $\fp$, and write $I = \Theta_F \setminus \Sigma_\infty$. 
Write $S^\prime_J = \widetilde{Y}^\prime_{J,\overline{\mathbb{F}}_p}$ for the geometric fiber of the Goren-Oort stratum as defined in section \ref{unitary hasse}, and let
$S^\prime_{IJ}$ and $S$ be as in the previous section.

By proposition \ref{Iwahori goren relation} and theorem \ref{unitary iwahori unquotiented thm}, we have an isomorphism $\widetilde{\Psi}^\prime_J: S^\prime_J \xrightarrow{\sim} \prod_{\beta \in R} \mathbb{P}_S(\ccH^1_\beta)$ given as the composite of isomorphisms

\begin{center}
    \begin{tikzcd}
        S^\prime_J \arrow[r,"\widetilde{\Psi}^\prime_J"] \arrow[d, "\widetilde{\psi}_\Fp"']  &\prod_{\beta \in R} \mathbb{P}_S(\ccH^1_\beta) \\
        S^\prime_{IJ} \arrow[ru,"\widetilde{\Psi}^\prime_{IJ}"'],
    \end{tikzcd}
\end{center}
where $\widetilde{\psi}_\Fp$, defined in section \ref{Iwahori goren relation section}, is the inverse of the natural projection $\pi_1: S^\prime_{IJ} \to S^\prime_J$ given explicitly by $\underline{A} \mapsto (\underline{A},\underline{A^{(\Fp)}},F_{\es,\Fp})$ as defined in section \ref{partial frobenius}. 

As in the previous section, we have the bundles $\omega^0_{\star,\tilde{\beta}} = \omega^0_{A_\star,\tilde{\beta}}$, $\ccH^1_{\star,\tilde{\beta}} = \ccH^1_{A_\star,\tilde{\beta}}$, $v^0_{\star,\tilde{\beta}} = v^0_{A_\star,\tilde{\beta}}$, and $\delta_{\star,\tilde{\beta}} = \delta_{A_\star,\tilde{\beta}}$ for $\star \in \{ \, \, , 1,2,J\}$, which live over $S^\prime_J$ for $\star =$ blank, $S^\prime_{IJ}$ for $\star = 1,2$ and $ \prod_{\beta \in R} \mathbb{P}_S(\ccH^1_\beta)$ for $\star = J$.

Recall from proposition \ref{essential name justification} that for each $\beta = \theta^i \in \Theta_F$ we have a canonical isomorphism $\ccH^1_{A^{(\Fp)},\tilde{\beta}} \simeq \ccH^{ 1 \, p^\delta}_{A, \phi^{-1}(\tilde{\beta})}$ where $\delta = 1$ if $i =1$ and 0 otherwise. It then follows that we have $\widetilde{\psi}_\Fp^*\ccH^1_{1,\tilde{\beta}} \simeq \ccH^1_{\tilde{\beta}}$ and $\widetilde{\psi}_\Fp^*\ccH^1_{2,\tilde{\beta}} \simeq \ccH^{1 \, p^\delta}_{\phi^{-1}(\tilde{\beta})}$ and we determine, as in the previous section, that we have the isomorphisms 
\begin{equation}
    \widetilde{\Psi}^{\prime *}_J \delta_{J,\tilde{\beta}} \simeq
    \begin{cases}
        \delta_{\tilde{\beta}} & \textrm{ if } \beta \notin T^\prime,\\
        \delta^{p^\delta}_{\phi^{-1}(\tilde{\beta})} & \textrm{ if } \beta \in T^\prime.
    \end{cases}
\end{equation}

We now describe each of the line bundles $\widetilde{\Psi}^*_{IJ,*} \omega^0_{\tilde{\beta}}$ for $\beta \notin \Sigma_\infty$. Write $\phi^{\prime \prime}$ for the cycle structure with respect to $\Sigma_\infty \sqcup I \cap J = \Sigma_\infty \sqcup J$, since $I = \Theta_F \setminus \Sigma_\infty$, and write $\beta^\prime = (\phi^{\prime \prime})^{-1}(\phi^\prime(\beta))$. In other words, $\beta^\prime = (\phi^\prime)^{-k}(\beta)$ where $k \geq 0$ is the smallest $n$ such that $(\phi^\prime)^{-n}(\beta) \notin J$. We apologize for the clash in notation with the previous section. We thus have, by \ref{IJ 1}, for the appropriate $n_\beta$, the isomorphisms
\[  
\widetilde{\psi}_{\Fp,*}\omega^0_{\tilde{\beta}} \simeq \omega^0_{1,\tilde{\beta}} \simeq 
\begin{cases}
    (\omega^0_{1,\tilde{\beta}^\prime})^{p^{n_\beta}} & \textrm{ if } k \textrm{ is even},\\
    (v^0_{1,\tilde{\beta}^\prime})^{p^{n_\beta}} & \textrm{ if } k \textrm{ is odd}.
\end{cases}
\]

Since $\widetilde{\Psi}^\prime_{J,*}\omega^0_{\tilde{\beta}} = 
\widetilde{\Psi}^\prime_{IJ,*} \widetilde{\psi}_{\Fp,*}\omega^0_{\tilde{\beta}}$, a case by case analysis as in the previous section and the isomorphisms \ref{comparison not in R} and \ref{comparison in R} then yield the isomorphisms

\begin{equation}\label{tame bundle comparison}  
\widetilde{\Psi}^\prime_{J,*}\omega^0_{\tilde{\beta}} \simeq
\mathcal{M}_{\tilde{\beta}}^{p^{n_\beta}}, \textrm{ where } 
\mathcal{M}_{\tilde{\beta}} =
\begin{cases}
        \omega^0_{J,\tilde{\beta}^\prime} & \textrm{if } \beta^\prime \notin R, \phi^\prime(\beta) \notin T,\\
        \delta_{J,\tilde{\beta}^\prime}(\omega^0_{J,\tilde{\beta}^\prime})^{-1} & \textrm{if } \beta^\prime \notin R, \phi^\prime(\beta) \in T,\\
        \delta_{J,\tilde{\beta}^\prime}(-1)_{\beta^\prime} & \textrm{if } \beta^\prime \in R, \phi^\prime(\beta) \in T,\\
        \ccO(1)_{\beta^\prime} & \textrm{if } \beta^\prime \in R, \phi^\prime(\beta) \notin T.
    \end{cases}
\end{equation}
\begin{rem}
    In the case that $J = \{\beta \}$, we obtain the isomorphisms
    $\widetilde{\Psi}^\prime_{J,*}\omega^0_{\tilde{\beta}^\prime} \simeq \omega^0_{J,\tilde{\beta}^\prime}$
    for $\beta^\prime \neq (\phi^\prime)^{-1}(\beta),\beta$,
    $\widetilde{\Psi}^\prime_{J,*}\omega^0_{(\phi^\prime)^{-1}(\tilde{\beta})} \simeq \delta_{J,(\phi^\prime)^{-1}(\tilde{\beta})}(-1)_{(\phi^\prime)^{-1}(\beta)}$,
    and $\widetilde{\Psi}^\prime_{J,*}\omega^0_{\tilde{\beta}} \simeq \ccO(p^{n_\beta})_{(\phi^\prime)^{-1}(\beta)}$.
\end{rem}

The above isomorphisms are Hecke equivariant in the sense that if, as usual, $U^\prime_1,U^\prime_2 \subset G^\prime_\Sigma(\mathbb{A}_f^{(p)})$ are sufficiently small open compact subgroups and $g \in G^\prime_\Sigma(\mathbb{A}_f^{(p)})$ is such that $g^{-1}U^\prime_1g \subset U^\prime_2$, then the following diagram commutes:

 \begin{equation}\label{Hecke equivariance bundles tame unitary}
 \begin{tikzcd}
     \tilde{\rho}_{g_{IJ}}^*\widetilde{\Psi}^\prime_{2,J,*} \omega^0_{2,\tilde{\beta}} = \widetilde{\Psi}^\prime_{1,J,*} \tilde{\rho}_{g}^*\omega^0_{2,\tilde{\beta}} \arrow[rr, "\tilde{\rho}^*_{g_{IJ}}(\sigma_2)"] \arrow[d,"\widetilde{\Psi}^\prime_{1,J,*}(\pi_g^*)"']&& \mathcal{M}_{2,\tilde{\beta}}^{p^{n_\beta}} \arrow[d," \pi^*_{g_{IJ}}"]\\
     \widetilde{\Psi}^\prime_{1,J,*} \omega^0_{1,\tilde{\beta}} \arrow[rr,"\sigma_1"] && \mathcal{M}_{1,\tilde{\beta}}^{p^{n_\beta}},
 \end{tikzcd}
 \end{equation}
 where $\sigma_i: \widetilde{\Psi}^\prime_{i,J,*}\omega^0_{i,\tilde{\beta}} \xrightarrow{\sim} \mathcal{M}_{i,\tilde{\beta}}^{p^{n_\beta}}$ is the isomorphism given above and $\pi_g$ is the quasi-isogeny that defines the Hecke action, similarly for $\pi_{g_{IJ}}$.

These isomorphisms are compatible with the action of $\ccO_{F,(p),+}^\times$, and thus descend (for sufficiently small $U^\prime$) to isomorphisms $\sigma: \Psi^\prime_{J,*}\omega^0_{\tilde{\beta}} \xrightarrow{\sim}
\mathcal{M}_\beta^{p^{n_\beta}}$, where both line bundles $\omega^0_{\tilde{\beta}}$ and $\mathcal{M}_\beta$ denote their quotient on $\overline{Y}^\prime_J$ and $\prod_{\beta \in R} \mathbb{P}^1_{\overline{Y}^\prime_{IJ}}(\ccH^1_{\tilde{\beta}})$ respectively, of the corresponding bundle upstairs. Furthermore, these isomorphisms are Hecke equivariant in the sense that the analogous diagram to \ref{Hecke equivariance bundles tame unitary} commutes.

 \subsubsection{The Quaternionic and Hilbert setting}
Let $\Sigma$, $\Fp$, $I$, and $J$ be as in section \ref{Iwahori unitary bundles}. 
Write $\overline{Y}_0(\Fp) = Y_{U_0(\Fp)}(G_\Sigma)_{\overline{\bF}_p}$ 
for a sufficiently small open compact subgroup 
$U \subset G_\Sigma(\mathbb{A}_f)$ of level prime to $p$ which we assume is $\Fp$-neat.
Write $\overline{Y}_{IJ} = Y_{U_{IJ}}(G_{IJ})_{\overline{\bF}_p}$.
Then we have have embeddings $i: \overline{Y}_0(\Fp) \to \overline{Y}^\prime_0(\Fp)$ and $i_{IJ}: \overline{Y}_{IJ} \to \overline{Y}^\prime_{IJ}$ which identify the former varieties as closed and open subschemes of the latter. As in section \ref{quaternion bundles}, we thus obtain the bundles $\ccH^1_\beta = i^*_{IJ}\ccH^1_{\tilde{\beta}}$ for all $\beta \in \Theta_F$, $\delta_\beta = \det \ccH^1_\beta$, and the line bundles $\omega_\beta = i^*_{IJ}\omega^0_{\tilde{\beta}}$ for all $\beta \notin \Sigma_\infty$. We also define for all $\theta \in \widehat{\Theta}_{F,\Fp}$, the line bundles $\ccL_\theta = i^* \ccL_{\tilde{\theta}}$ on $\overline{Y}_0(\Fp)$. We will also write $\delta_\beta$ and $\omega_\beta$ for the pullback of the corresponding bundles on $\prod_{\beta \in R} \mathbb{P}_{\overline{Y}_{IJ}}(\ccH^1_\beta)$, and $\ccO(1)_\beta$ for the pullback of the tautological bundle with respect to the projection to the $\beta$-component.

 Recall that the isomorphism $\Psi_{IJ}: \overline{Y}_0(\Fp)_{\phi^\prime(I),J} \to \prod_{\beta \in R} \mathbb{P}_{\overline{Y}_{IJ}}(\ccH^1_\beta)$ of theorem \ref{main theorem quat} is obtained as the restriction of the composite $\Psi_{IJ}^\prime \circ \rho_u$ for a suitably chosen $u \in (\mathbb{A}_{F,f}^{(p)})^\times$, which is independent of $U$ and $U^\prime$. The isomorphisms from (\ref{Raynaud comparison unitary}) then yield isomorphisms 

 \begin{equation}\label{Raynaud comparison quat}
    \ccL_\theta \simeq \Psi^*_{IJ}\mathcal{M}_{\theta}^{p^{n_\theta}}, \textrm{ where } \, 
    \mathcal{M}_\theta =
    \begin{cases}
        \omega_\beta & \textrm{if } \beta \notin R, \phi \circ \theta \notin T^1,\\
        \delta_\beta(\omega_\beta)^{-1} & \textrm{if } \beta \notin R, \phi \circ \theta \in T^1,\\
        \delta_\beta(-1)_\beta & \textrm{if } \beta \in R, \phi \circ \theta \in T^1,\\
        \ccO(1)_\beta & \textrm{if } \beta \in R, \phi \circ \theta \notin T^1,
    \end{cases}
\end{equation}
where we recall that $\beta = (\phi^{\prime \prime})^{-1}((\phi\circ \theta)^1)$ and $\phi^{\prime \prime}$ denotes the cycle structure with respect to $\Sigma_\infty \sqcup I \cap J$, by the restriction of the composite of isomorphisms 
\[\ccL_\tau \xrightarrow{(\pi_u^*)^{-1}} \rho^*_u \ccL_\tau \xrightarrow{\rho_u^* \sigma} \rho^*_u\Psi^{\prime *}_{IJ} \mathcal{M}_\tau^{p^{n_\theta}}.\]

Furthermore, if $U_1,U_2 \subset G_\Sigma(\mathbb{A}_f)$ are two sufficiently small open compact subgroups and $g \in G_{\Sigma}(\mathbb{A}^{(p)}_f)$ is such that $g^{-1}U_1 g \subset g$, then the following diagram

 \begin{equation}\label{Hecke equivariance raynaud quat}
 \begin{tikzcd}
     \rho_g^* \ccL_{2,\theta} \arrow[rr, "\rho_g^*(\sigma_2)"] \arrow[d,"\pi_g^*"]&& \rho_g^*  \Psi_{2,IJ}^{*} \mathcal{M}_{2,\theta}^{p^{n_\theta}} = \Psi_{1,IJ}^{*} \rho_{g_{IJ}}^* \mathcal{M}_{2,\theta}^{p^{n_\theta}} \arrow[d," \Psi_{1,IJ}^* (\pi^*_{g_{IJ}})"]\\
     \ccL_{1,\theta} \arrow[rr,"\sigma_1"] && \Psi_{1,IJ}^{*}\mathcal{M}_{1,\theta}^{p^{n_\theta}},
 \end{tikzcd}
 \end{equation}
commutes.

Write $\overline{Y} = Y_U(G_\Sigma)_{\overline{\bF}_p}$, let $J \subset \Theta_{F,\Fp} \setminus \Sigma_{\infty,\Fp}$ be as in section \ref{tame bundle comparison}, and write $I = \Theta_{F,\Fp} \setminus \Sigma_{\infty,\Fp}$. For each $\beta \in \Theta_{F} \setminus \Sigma_\infty$ and line bundle $\ccL_{\tilde{\beta}}$ from section \ref{tame bundle comparison}, let $\ccL_\beta$ denote its restriction to the corresponding quaternionic Shimura variety.  Similarly as above, we deduce from (\ref{tame bundle comparison}) that the isomorphism $\Psi_J: \overline{Y}_J \to \prod_{\beta \in R} \mathbb{P}_{\overline{Y}_{IJ}}(\ccH^1_\beta)$ of theorem \ref{main thm quat tame} induces isomorphisms

\begin{equation}\label{tame bundle comparison quat}  
\Psi_{J,*}\omega_{\beta} \simeq
\mathcal{M}_{\beta}^{p^{n_\beta}}, \textrm{ where } 
\mathcal{M}_\beta =
\begin{cases}
        \omega_{J,\beta^\prime} & \textrm{if } \beta^\prime \notin R, \phi^\prime(\beta) \notin T,\\
        \delta_{J,\beta^\prime}(\omega_{J,\beta^\prime})^{-1} & \textrm{if } \beta^\prime \notin R, \phi^\prime(\beta) \in T,\\
        \delta_{J,\beta^\prime}(-1)_{\beta^\prime} & \textrm{if } \beta^\prime \in R, \phi^\prime(\beta) \in T,\\
        \ccO(1)_{\beta^\prime} & \textrm{if } \beta^\prime \in R, \phi^\prime(\beta) \notin T,
    \end{cases}
\end{equation}
 where we recall that $\beta^\prime = (\phi^{\prime \prime})^{-1}(\phi^\prime(\beta))$.

Furthermore, if $U_1,U_2 \subset G_\Sigma(\mathbb{A}_f)$ are two sufficiently small open compact subgroups and $g \in G_{\Sigma}(\mathbb{A}^{(p)}_f)$ is such that $g^{-1}U_1 g \subset g$, then the following diagram
 \begin{equation}\label{Hecke equivariance bundles tame}
 \begin{tikzcd}
     \rho_{g_{IJ}}^*\Psi_{2,J,*} \omega^0_{2,\beta} = \Psi_{1,J,*} \rho_{g}^*\omega^0_{2,\beta} \arrow[rr, "\rho^*_{g_{IJ}}(\sigma_2)"] \arrow[d,"\Psi_{1,J,*}(\pi_g^*)"']&& \mathcal{M}_{2,\beta}^{p^{n_\beta}} \arrow[d," \pi^*_{g_{IJ}}"]\\
     \Psi_{1,J,*} \omega^0_{1,\beta} \arrow[rr,"\sigma_1"] && \mathcal{M}_{1,\beta}^{p^{n_\beta}},
 \end{tikzcd}
 \end{equation}
commutes.

Finally, recall that for the Hilbert modular variety $\overline{Y} = Y_U(G)_{\overline{\bF}_p}$, we had the isomorphism $i:\overline{Y} \xrightarrow{\sim} \overline{Y}_\emptyset$ from section \ref{model comparison} which respects stratifications and identifies the bundles $\omega_\beta$ and $\delta_\beta$ on either side. Furthermore, by proposition \ref{iwahori quat hilb}, this isomorphism extended to an isomorphism $i: \overline{Y}_0(\Fp) \xrightarrow{\sim} \overline{Y}_{\emptyset,0}(\Fp)$ which respects stratifications. Furthermore, one can see that this isomorphism identifies Raynaud bundles on both sides. 

We thus deduce, for $I,J$ as usual, that the isomorphism $\Psi_{IJ}: \overline{Y}_0(\Fp)_{\phi^\prime(I),J} \to \prod_{\beta \in R} \mathbb{P}_{\overline{Y}_{IJ}}(\ccH^1_\beta)$ of theorem \ref{hilb thm iwahori} induces isomorphisms analogous to those of (\ref{Raynaud comparison quat}) which are Hecke equivariant in the sense that the analogous diagram to diagram (\ref{Hecke equivariance raynaud quat}) commutes. At Tame level, we deduce that the isomorphism $\Psi_{J}: \overline{Y}_J \to \prod_{\beta \in R} \mathbb{P}_{\overline{Y}_{IJ}}(\ccH^1_\beta)$ of theorem \ref{hilb thm tame} induces isomorphisms analogous to those of (\ref{tame bundle comparison quat}) which are Hecke equivariant in the sense that the analogous diagram to diagram (\ref{Hecke equivariance bundles tame}) commutes.

\bibliography{bibliography}{}
\bibliographystyle{alpha}

\end{document}